\documentclass[11pt]{amsart}

\usepackage{geometry}
\usepackage{amsmath,amssymb,amsfonts,amsthm,slashed}
\usepackage[T1]{fontenc}
\usepackage[utf8]{inputenc}
\usepackage{hyperref,graphicx,multicol,eufrak,enumerate,cite}
\usepackage[bottom]{footmisc}
\usepackage[many]{tcolorbox}
\usepackage[english]{babel}

\hypersetup{hidelinks}

\newtheorem{thm}{Theorem}[section]

\newtheorem{lem}[thm]{Lemma}
\newtheorem{prop}[thm]{Proposition}
\newtheorem{rmk}[thm]{Remark}
\newtheorem{defn}[thm]{Definition}

\numberwithin{equation}{section}

\newcommand{\R}{\mathbb{R}}
\newcommand{\D}{\mathcal{D}}

\newcommand{\eps}{\varepsilon}

\newcommand{\norm}[1]{\left\lVert#1\right\rVert}
\newcommand{\mcl}[1]{\mathcal{ #1 }}

\newcommand{\fm}[1]{\begin{align*} #1 \end{align*}}
\newcommand{\eq}[1]{\begin{equation}\begin{aligned} #1 \end{aligned}\end{equation}}
\newcommand{\lra}[1]{\langle #1 \rangle}
\newcommand{\abs}[1]{\left| #1 \right|}
\newcommand{\kh}[1]{\left( #1 \right)}
\newcommand{\wt}[1]{\widetilde{ #1 }}
\newcommand{\wh}[1]{\widehat{ #1 }}

\newcommand{\Tboot}{T_{(\textsf{Boot})}}
\newcommand{\muring}{\mathring{\mu}}
\newcommand{\Dlower}{\partial\D_{\rm low}}
\newcommand{\Dint}{\D_{\rm interior}}

\DeclareMathOperator{\supp}{supp}

\DeclareMathOperator{\divg}{div}

\newcommand{\uLunit}{\underline{L}}
\newcommand{\muuLunit}{\Breve{\underline{L}}}

\title{Global existence for a scalar quasilinear wave equation in two space dimensions beyond the cubic null condition}
\author{Dongxiao Yu}
\address{Department of Mathematics, Vanderbilt University, Nashville, TN 37240, USA}
\email{dongxiao.yu@vanderbilt.edu}
\thanks{}
\date{}

\begin{document}

\begin{abstract}
We prove global existence, both to the future and to the past, for a family of scalar quasilinear wave equations in two space dimensions for sufficiently small $C_c^\infty$ initial data. These equations have cubic leading quasilinear nonlinearities of the form $u^2\partial^2u$ and satisfy a sign condition that allows the cubic null condition to fail.
To explain the sign condition, we derive the geometric reduced system for quasilinear wave equations with cubic leading nonlinearities in two space dimensions and introduce a notion of geometric weak null condition.  We prove that the geometric weak null condition holds for the model scalar equation if and only if the sign condition holds. At the level of the geometric reduced system, the sign condition prevents the corresponding characteristic curves from intersecting in finite time.
We also recover the geometric reduced system from the global solutions in our main theorem. If the model scalar equation satisfies the sign condition but fails the cubic null condition, and if $u$ is a nonzero global solution to the model equation, we show that, for $r=|x|$, the $L^\infty$ norm of $t^{\frac{1}{2}}\partial_r^2u$ in a region where $|r-t|\ll t$ tends to infinity as $t\to\infty$. Thus, its asymptotic behavior differs from that of a solution to the linear wave equation $\Box \psi=0$ with $C_c^\infty$ data.
\end{abstract}

\maketitle

\tableofcontents

\section{Introduction}
Consider a scalar quasilinear wave equation in $\R^{1+2}$ of the form
\eq{\label{qwe} \wt{\Box}_{g(u)}u:= g^{\alpha\beta}(u)\partial_\alpha\partial_\beta u=0}
with small, smooth, and localized initial data
\eq{\label{init} (u,u_t)\restriction_{t=0}=(\eps u_0,\eps u_1)}
for fixed $u_0,u_1\in C_c^\infty(\R^2)$ with $\supp u_0\cup \supp u_1\subset B_{\R^2}(0,1)$,  and for sufficiently small $0<\eps\ll1$. In the equation, we use the Einstein summation convention, taking the sum over $\alpha,\beta\in\{0,1,2\}$, with $\partial_0=\partial_t$ and $\partial_i=\partial_{x_i}$, $i=1,2$. The coefficients $g^{\alpha\beta}$ are smooth functions of $u$ such that $g^{\alpha\beta}=g^{\beta\alpha}$ and that $(g^{\alpha\beta}(0))=(m^{\alpha\beta})=\text{diag}(-1,1,1)$. We assume that the nonlinearity is at least cubic. That is, the Taylor expansions of the coefficients $g^{\alpha\beta}$ near $u=0$ are
\eq{\label{eq:g:taylor}g^{\alpha\beta}(u)=m^{\alpha\beta}+g_0^{\alpha\beta}u^2+O(|u|^3),\qquad \text{for }|u|\leq 1.}
Without loss of generality, we assume that $g^{00}\equiv -1$, because one can divide \eqref{qwe} by $-g^{00}=1+O(|u|^2)$ which never vanishes for small data.

The Cauchy problem \eqref{qwe}--\eqref{init} has a global solution for sufficiently small $\eps$ if the cubic leading nonlinearity in \eqref{qwe} satisfies the \emph{cubic null condition}\footnote{Since there is no quadratic term, it satisfies the quadratic null condition automatically.}. For \eqref{qwe}, this means that the function $G:\mathbb{S}^1\to \R$ defined by \eq{\label{eq:G}G(\omega):=g_0^{\alpha\beta}\wh{\omega}_\alpha\wh{\omega}_\beta,\qquad \text{where }\wh{\omega}=(\wh{\omega}_\alpha)_{\alpha=0,1,2}:=(-1,\omega)} vanishes everywhere, i.e.\ $G\equiv 0$. When the cubic null condition fails, we have an almost global existence result: a solution exists at least for $t\leq e^{\frac{c}{\eps^2}}$. Both results for a general class of equations, including \eqref{qwe}, are contained in Zha \cite{MR3912654}. We will discuss earlier related works in Section~\ref{sec:background}.

In this paper, we prove global existence, both to the future and to the past, for the Cauchy problem \eqref{qwe}--\eqref{init} for $\eps\ll1$ under the following sign condition
\eq{\label{asu:sign:G} G(\omega)\geq 0,\qquad \forall\omega\in\mathbb{S}^1.}
In other words, the cubic null condition is not necessary for global existence. For example, this global existence result applies to
\fm{\Box u+u^2\Delta u=0,} where $G\equiv 1>0$. Before the present paper, in $\R^{1+2}$ there have been several global existence results for semilinear wave equations beyond the cubic null condition, as well as an announced result for certain systems of quasilinear wave equations violating the cubic null condition; see Section~\ref{sec:background}. To the best of our knowledge, the present paper presents the first global existence result for a scalar quasilinear wave equation in $\R^{1+2}$ with a cubic leading nonlinearity of the form $u^2\partial^2u$ that violates the cubic null condition.

The sign condition \eqref{asu:sign:G} is naturally motivated by a  \emph{geometric reduced system} for \eqref{qwe}. Given a general system of quasilinear wave equations in $\R^{1+2}$ with at least cubic
nonlinearities, we derive a two-dimensional version of geometric reduced systems and introduce a corresponding \emph{geometric weak null condition} in this paper. Although we do not prove finite-time blowup for \eqref{qwe} when the sign condition \eqref{asu:sign:G} fails, we prove that \eqref{qwe} satisfies the geometric weak null condition if and only if the sign condition holds. Note that at the level of the geometric reduced system, the sign condition prevents the corresponding characteristic curves from intersecting in finite time.

Moreover, we study the asymptotic behavior of the solutions constructed in the global existence result by recovering the geometric reduced system. Our final conclusion is that if the sign condition is satisfied but we have $G>0$ somewhere, then a nonzero global solution $u$ behaves differently from a solution $\psi$ to $\Box \psi=0$ in the sense that a certain $L^\infty$ norm of $\lra{t}^{\frac{1}{2}}\partial^2u$ diverges as $t\to\infty$.

\subsection{Background}\label{sec:background}
Consider a general scalar quasilinear wave equation in $\R^{1+2}$:
\eq{\label{qwe:general:intro}g^{\alpha\beta}(u,\partial u)\partial_\alpha\partial_\beta u=F(u,\partial u).}
Unlike \eqref{qwe}, we also allow quadratic nonlinearities in \eqref{qwe:general:intro}. In other words, we assume that
\eq{\label{ass:general:intro:taylor}g^{\alpha\beta}(u,\partial u)&=m^{\alpha\beta}+O(|u|+|\partial u|),\\
F(u,\partial u)&=O(|\partial u|^2+|(u,\partial u)|^3).}
Note that we exclude the quadratic semilinear terms $u^2$ and $u\partial u$ in $F$.
For simplicity, we also assume $g^{00}\equiv -1$. Here we only consider an $\R$-valued unknown, but one may extend many of these discussions to an $\R^M$-valued unknown with $M\geq 2$. 

The lifespan for the Cauchy problem \eqref{qwe:general:intro} with small $C_c^\infty$ initial data of size $\eps$ has been studied extensively. The cases in which the leading nonlinearity is quadratic, cubic, and at least quartic were treated in \cite{MR1273703}, \cite{MR1210250,MR1842431}, and \cite{MR1282858}, respectively. Among the references just cited, only \cite{MR1282858} contains a global existence result, under the assumption that the leading nonlinearity in \eqref{qwe:general:intro} is at least quartic. The other three papers give, in general, finite lower bounds for the lifespan. There exist equations of the form \eqref{qwe:general:intro} with quadratic or cubic leading nonlinearities that exhibit finite-time blowup; see, for example, \cite{MR1687180,MR1680539,MR1867312,MR1752748,MR1903627,MR3357493,MR3623547} and the references therein.  In other words, global existence does not hold in general when the leading nonlinearity is quadratic or cubic.

To obtain global existence, one needs to assume extra structural conditions on \eqref{qwe:general:intro}. When the leading nonlinearity is cubic, global existence has been proved under a cubic null condition; see \cite{MR3729480,zhou1992lifespan,MR1371789,MR1325960,MR1218523} and the references therein. Some of these works treat the more general case of systems of wave equations. When the leading nonlinearity is quadratic, Alinhac \cite{MR1856402} proved a global existence result if the quadratic and cubic nonlinearities of \eqref{qwe:general:intro} satisfy the corresponding null conditions. We refer to \cite{MR4600064,dong2024toporder,MR4097245,MR4506567,MR2297944,MR3801755} for several related results. Alinhac also proved an almost global existence result if only the quadratic nonlinearity satisfies a null condition. Both these results by Alinhac, however, only apply to the case where $g^{\alpha\beta}=g^{\alpha\beta}(\partial u)$ and $F\equiv 0$, so the almost global existence result in \cite{MR1856402} cannot be directly applied to the equation \eqref{qwe} in this paper. Later, Zha \cite{MR3912654} extended Alinhac's results to a general class of equations of the form \eqref{qwe:general:intro} under the null conditions. Note that Zha's almost global existence result assumes that the quadratic null condition holds and that the Taylor expansion of $F(u,\partial u)$ contains no $u^3$ or $u^4$ terms. In particular, this result applies to \eqref{qwe}.

Here we state the null conditions for \eqref{qwe:general:intro}. Suppose that we have Taylor expansions
\eq{g^{\alpha\beta}&=m^{\alpha\beta}+g_{0,2}^{\alpha\beta}u+g_{1,2}^{\alpha\beta,\lambda}\partial_\lambda u\\
&\quad+g_{0,3}^{\alpha\beta}u^2+g_{1,3}^{\alpha\beta,\lambda}u\partial_\lambda u+g_{2,3}^{\alpha\beta,\sigma\lambda}\partial_\sigma u\partial_\lambda u+O(|(u,\partial u)|^3),}
\eq{F&=f_{0,2}u^2+f_{1,2}^{\lambda}u\partial_\lambda u+f_{2,2}^{\sigma\lambda}\partial_\sigma u\partial_\lambda u\\
&\quad+f_{0,3}u^3+f_{1,3}^{\lambda}u^2\partial_\lambda u+f_{2,3}^{\sigma\lambda}u\partial_\sigma u\partial_\lambda u+f_{3,3}^{\gamma\sigma\lambda}\partial_\gamma u\partial_\sigma u\partial_\lambda u+O(|(u,\partial u)|^4).}
All the coefficients are real constants. Note that $f_{0,2}=0$ and $f_{1,2}^\lambda=0$ by \eqref{ass:general:intro:taylor}. Set $\wh{\omega}=(-1,\omega)$ for $\omega\in\mathbb{S}^1$. We say that \eqref{qwe:general:intro} satisfies the \emph{quadratic null condition} (also known as the \emph{first null condition}), if
\eq{G_{0,2}(\omega)&:=g_{0,2}^{\alpha\beta}\wh{\omega}_\alpha\wh{\omega}_\beta,\\
G_{1,2}(\omega)&:=g_{1,2}^{\alpha\beta,\lambda}\wh{\omega}_\alpha\wh{\omega}_\beta\wh{\omega}_\lambda,
}
and
\eq{F_{0,2}(\omega)&:=f_{0,2},\\
F_{1,2}(\omega)&:=f_{1,2}^{\lambda}\wh{\omega}_\lambda,\\
F_{2,2}(\omega)&:=f_{2,2}^{\sigma\lambda}\wh{\omega}_\sigma\wh{\omega}_\lambda
}
vanish for all  $\omega\in\mathbb{S}^1$. We say that \eqref{qwe:general:intro} satisfies the \emph{cubic null condition} (also known as the \emph{second null condition}), if
\eq{G_{0,3}(\omega)&:=g_{0,3}^{\alpha\beta}\wh{\omega}_\alpha\wh{\omega}_\beta,\\
G_{1,3}(\omega)&:=g_{1,3}^{\alpha\beta,\lambda}\wh{\omega}_\alpha\wh{\omega}_\beta\wh{\omega}_\lambda,\\
G_{2,3}(\omega)&:=g_{2,3}^{\alpha\beta,\sigma\lambda}\wh{\omega}_\alpha\wh{\omega}_\beta\wh{\omega}_\sigma\wh{\omega}_\lambda,
}
and
\eq{F_{0,3}(\omega)&:=f_{0,3},\\
F_{1,3}(\omega)&:=f_{1,3}^{\lambda}\wh{\omega}_\lambda,\\
F_{2,3}(\omega)&:=f_{2,3}^{\sigma\lambda}\wh{\omega}_\sigma\wh{\omega}_\lambda,\\
F_{3,3}(\omega)&:=f_{3,3}^{\gamma\sigma\lambda}\wh{\omega}_\gamma\wh{\omega}_\sigma\wh{\omega}_\lambda
}
vanish for all  $\omega\in\mathbb{S}^1$. Note that we automatically have $F_{0,2}\equiv 0$ and $F_{1,2}\equiv 0$ because of \eqref{ass:general:intro:taylor}.

It is natural to ask whether the two null conditions are necessary for global existence. To the best of our knowledge, in the scalar case, whether the quadratic null condition is necessary for a general equation \eqref{qwe:general:intro} still appears to be open, but finite-time blowup is known to happen for a large family of equations when the quadratic null condition fails. When there is no semilinear term, i.e.\ $F\equiv 0$, we refer to Alinhac \cite{MR1687180,MR1680539} where $g^{\alpha\beta}$ depends only on $\partial u$ and $G_{1,2}\not\equiv 0$, and to Ding--Witt--Yin \cite{MR3623547} where $G_{0,2}\not\equiv 0$\footnote{This is equivalent to saying that at least one $g^{\alpha\beta}_{0,2}$ is nonzero. Note that $(g^{\alpha\beta})$ and $(g^{\alpha\beta}_{0,2})$ are symmetric. Thus, $G_{0,2}\equiv 0$ if and only if $(g^{\alpha\beta}_{0,2})=C(m^{\alpha\beta})$ for a constant $C$.  And since $g^{00}\equiv -1$, we have $C=0$.} and $G_{1,2}\not\equiv 0$. The case where $G_{0,2}\not\equiv 0$ and $G_{1,2}\equiv0$ (and $F\equiv 0$) appears to be open, but if one allows a nonzero semilinear term $F\not\equiv 0$,  we refer to \cite{MR3357493}, where the authors established blowup for two particular classes of wave equations for nontrivial radially symmetric initial data. We also remark that the conclusion is quite different if the unknown is $ \R^M$-valued for $M>1$. For example, a simple system in $\R^{1+2}$
\fm{\left\{
\begin{array}{ll}
     \Box u=0,  \\
     \Box v=(\partial_tu)^2 
\end{array}\right.}
violates the quadratic null condition but admits global existence for all $C_c^\infty$ data. 

If the quadratic null condition is satisfied, the cubic null condition is not necessary for global existence of \eqref{qwe:general:intro}.
In fact, the equation $\Box u=u_t^3$ admits global existence; see Lions--Strauss \cite{MR199519}. Here, the nonlinearity is a damping term that decreases the energy.
Such a global existence result has been extended to a scalar semilinear equation with the nonlinearity depending only on $\partial u$ (i.e.\ $\Box u=F(\partial u)$), under the assumption that $F_{3,3}(\omega)\leq 0$ for all $\omega\in\mathbb{S}^1$. It was proved by Hoshiga \cite{MR2474170} and Kubo \cite{MR2409896}. This sign condition is commonly referred to as the Agemi condition. Further results on the asymptotic behavior and energy decay of scalar equations under related dissipative conditions can be found in \cite{MR3000461,MR4291437}. We also refer to \cite{MR2474171, MR3385614, MR4170895,MR4484119} for related results on systems of semilinear wave equations. Katayama \cite{MR4385042} announced a global existence result for certain systems of quasilinear wave equations under a structural condition that allows the cubic null condition to fail. His result does not apply to a scalar equation. To the best of our knowledge, the present paper presents the first global existence result for a scalar quasilinear wave equation in $\R^{1+2}$ with a cubic leading nonlinearity of the form $u^2\partial^2u$ that violates the cubic null condition.

We end the current subsection with a comparison with the lifespan estimates in higher space dimensions $d\geq 3$ for small, smooth, and localized initial data. If $d\geq 4$, there is global existence for a general equation \eqref{qwe:general:intro} satisfying \eqref{ass:general:intro:taylor}; see \cite{MR1120284}. For $d=3$, John \cite{MR600571,MR808321} proved that finite-time blowup can occur for solutions to equations such as $\Box u=(\partial_tu)^2$ and $\Box u=(\partial_tu)\Delta u$. We refer to, e.g.,\ \cite{MR1867312,MR2284927,MR3474069,MR3595936,MR3561670,MR3694013,MR4047642} for results on shock formation and other finite-time
breakdown mechanisms for several
quasilinear wave equations. In general, one has an almost global existence result: a solution exists at least for $t\leq e^{\frac{c}{\eps}}$; see \cite{MR697468,MR1047332,MR745325,MR784477,MR897781,MR1466700} and the references therein. Christodoulou \cite{MR820070} and Klainerman \cite{MR837683,MR804771} proved that the null condition is sufficient for global existence. Unlike the two-dimensional case, we only need the null condition on the quadratic leading nonlinearity in $\R^{1+3}$. We remark that there is a distinct nonlinear null condition relevant to improved low-regularity local well-posedness for quasilinear wave equations; see \cite{MR1957533} for its formulation and  \cite{MR4701880,moschidis2025well} for recent progress in this direction for the timelike minimal surface equation. Note that some equations violate the null condition but admit global solutions. One such example is a scalar equation
\eq{\label{qwe:3D}
g^{\alpha\beta}(u)\partial_\alpha\partial_\beta u=0}
where the coefficients have Taylor expansions
\eq{g^{\alpha\beta}(u)=m^{\alpha\beta}+g_0^{\alpha\beta}u +O(|u|^2),\qquad \text{for }|u|\leq 1.}
Global existence for this equation was proved by Lindblad \cite{MR2382144}. Note that \eqref{qwe:3D} is of a similar form to \eqref{qwe}, but the $g^{\alpha\beta}$ in \eqref{qwe:3D} can contain nonzero linear terms in Taylor expansions. Unlike the main theorem in this paper, a sign condition is not necessary for global existence for \eqref{qwe:3D} in $\R^{1+3}$. We also refer to \cite{MR1177476,MR2003417} for some earlier progress on global existence, and \cite{MR4232783,MR4315017,luk2025latetime,MR4929086,MR4772266,MR4078713,MR5036469} for the asymptotic behavior of its global solutions.  Another example is the Einstein vacuum equations in wave coordinates, which form a system of quasilinear wave equations; see Lindblad--Rodnianski \cite{MR2134337,MR2680391}. We also refer to Lindblad \cite{MR3638312} for the asymptotic behavior of the corresponding global solutions in
wave coordinates and to Christodoulou--Klainerman \cite{MR1316662} for an earlier and different proof of the global
nonlinear stability of Minkowski spacetime. Both the scalar example \eqref{qwe:3D} and the Einstein vacuum equations in wave coordinates satisfy the \emph{weak null condition} introduced by Lindblad--Rodnianski \cite{MR1994592}. This condition is formulated in terms of a type of asymptotic equations introduced by H\"ormander \cite{MR897781,MR1120284,MR1466700} for quasilinear wave equations in $\R^{1+3}$. We refer to \cite{keir2018weak,keir2019global,MR4833090} for both positive and negative results on the connection between the weak null condition and global existence for nonlinear wave equations in $\R^{1+3}$.

\subsection{The geometric reduced system}\label{sec:intro:reduced:system}

To motivate and explain the sign condition \eqref{asu:sign:G}, we use the \emph{geometric reduced system} for \eqref{qwe}.

We first discuss a type of asymptotic equations for \eqref{qwe:general:intro}. This method goes back to H\"ormander \cite{MR897781,MR1120284,MR1466700} and was further developed by Alinhac \cite{MR1856402,MR1867312} in two space dimensions. 
Let $u$ be a solution to a general quasilinear wave equation \eqref{qwe:general:intro}. We make an ansatz
\eq{\label{intro:ansatz}u(t,x)\approx\eps r^{-\frac{1}{2}}U(s,q,\omega).}
Here $r=|x|$, $\omega=\frac{x}{r}$, and $q=r-t$. For general quadratic nonlinearities, we set $s=\eps\sqrt{t}$. By plugging the ansatz into \eqref{qwe:general:intro}, one formally obtains an asymptotic equation for $U$ in the  coordinates $(s,q,\omega)$:
\eq{U_{sq}+G_{0,2}(\omega)UU_{qq}+G_{1,2}(\omega)U_qU_{qq}=F_{2,2}(\omega)U_q^2.}
We recall that $F_{0,2}$ and $F_{1,2}$ vanish for all $\omega\in\mathbb{S}^1$. 
When the quadratic null condition holds, we instead set $s=\eps^2\ln t$. By plugging the ansatz into \eqref{qwe:general:intro}, one formally obtains
\eq{\label{hormander:cubic:asyeqn}&2U_{sq}+G_{0,3}(\omega)U^2U_{qq}+G_{1,3}(\omega)U_qUU_{qq}+G_{2,3}(\omega)U_q^2U_{qq}\\
&=F_{0,3}(\omega)U^3+F_{1,3}(\omega)U^2U_q+F_{2,3}(\omega)UU_q^2+F_{3,3}(\omega)U_q^3.}
For convenience, we refer to these asymptotic equations as \emph{classical asymptotic equations}\footnote{The three-dimensional variant is also referred to as H\"ormander's asymptotic equations. We refer to \cite{MR4078713,MR5036469,MR4591580,MR4972728} for some work relying on this type of asymptotic equations in $\R^{1+3}$.} in this paper.

Such a type of asymptotic equation can be used to explain, for example, the Agemi condition. For $\Box u=F(\partial u)$ with cubic nonlinearity, the corresponding asymptotic equation \eqref{hormander:cubic:asyeqn} becomes
\fm{2U_{sq}&=F_{3,3}(\omega)U_q^3.}
We obtain an explicit solution
\fm{  U_q(s,q,\omega) &=\frac{U_q(0,q,\omega)}{\sqrt{1-sF_{3,3}(\omega)U_q^2(0,q,\omega)}}.}
The Agemi condition $F_{3,3}\leq 0$ implies that the denominator is always bounded away from $0$ if $s\geq 0$, so the asymptotic equation has a global solution for $s\geq 0$. This is consistent with the fact that a global solution to the original semilinear wave equation exists; see \cite{MR2474170,MR2409896}.

For our equation \eqref {qwe}, \eqref{hormander:cubic:asyeqn} reduces to
\eq{\label{hormander:cubic:asyeqn:qwe}2V_{s\rho}+G(\omega)V^2V_{\rho\rho}=0}
where $G$ is defined by \eqref{eq:G}. Here we deliberately use $(V,\rho)$ instead of $(U,q)$, since we will compare \eqref{hormander:cubic:asyeqn:qwe} with the geometric reduced system introduced later. We seek to determine whether and when this asymptotic equation admits global existence. The answer to this question can then be used to predict whether and when we may have a global solution to \eqref{qwe}. To solve \eqref{hormander:cubic:asyeqn:qwe}, we define the Lagrangian flow map $s\mapsto Y(s,q,\omega)$ by solving \eq{\label{intro:def:charline}Y_s(s,q,\omega)&=\frac{1}{2}G(\omega)V(s,Y(s,q,\omega),\omega)^2,\qquad Y(0,q,\omega)=q.}
Then,  we have
\eq{\partial_s\kh{V_\rho(s,Y(s,q,\omega),\omega)}=0}
and thus
\eq{\label{intro:charline:Phirhoinvariant}V_\rho(s,Y(s,q,\omega),\omega)=V_\rho(0,Y(0,q,\omega),\omega)=V_\rho(0,q,\omega).}
A necessary condition for the existence of a global solution to \eqref{hormander:cubic:asyeqn:qwe} is that the map $ q\mapsto Y(s,q,\omega)$ is a $C^1$ diffeomorphism for fixed $(s,\omega)$.  Determining this requires an extended analysis, and we shall postpone this discussion until we introduce the geometric reduced system.

When studying the modified scattering theory for the variant \eqref{qwe:3D} in $\R^{1+3}$ of \eqref{qwe} in \cite{MR4232783,MR4315017,MR4772266}, the author found it useful to introduce a new\footnote{However, see the discussion on the connection between the geometric reduced system and the classical asymptotic equation at the end of the current subsection.} type of asymptotic equations, called the \emph{geometric reduced system}, for general quasilinear wave equations. The derivation there can be adapted to \eqref{qwe:general:intro} in $\R^{1+2}$ as follows. In its derivation, we replace $q=r-t$ in the ansatz \eqref{intro:ansatz} with an optical function $q$ that is close to $r-t$. By an optical function, we mean a solution to the eikonal equation
\eq{\label{intro:eik} g^{\alpha\beta}\partial_\alpha q\partial_\beta q=0.}
Eikonal equations have been widely used in the study of nonlinear wave equations and the
Einstein equations; see, e.g., \cite{MR2003417,MR1316662,keir2018weak,MR3402797,MR2382144,MR3638312,MR2178963}.
By setting a new variable $\mu=q_t-q_r$ and plugging $(\mu,U)$ into \eqref{qwe:general:intro}, we formally obtain the following asymptotic equations in $(s,q,\omega)$:
\eq{\label{sec:reduced:system:main:intro}
\partial_s(\mu U_q )&=-  F_{0,3}(\omega)  U^3+\frac{1}{2}  F_{1,3}(\omega) U^2 (\mu U_q) -\frac{1}{4}  F_{2,3}(\omega)  U  (\mu U_q)^2  +\frac{1}{8} F_{3,3}(\omega) (\mu U_q)^3,\\ 
\partial_s\mu&= \mu^2 \partial_q\kh{\frac{1}{4}G_{0,3}(\omega) U^2-\frac{1}{8}G_{1,3}(\omega) U \mu U_q+\frac{1}{16}G_{2,3}(\omega)(\mu U_q)^2}.
}
This is called a \emph{geometric reduced system} for \eqref{qwe:general:intro} since it contains the geometric information from the Lorentzian metric $(g_{\alpha\beta})$ given by the matrix inverse of the coefficients $(g^{\alpha\beta})$ of \eqref{qwe:general:intro}.
We refer to Section~\ref{sec:reduced:system} for a formal derivation of this reduced system. For completeness, there we consider a general system instead of a scalar equation. We also refer to \cite{MR4232783,MR4315017} for discussions on the advantages of this new reduced system and \cite{MR4949961,MR4929086,luk2025latetime} for a few applications in $\R^{1+3}$. 

For \eqref{qwe}, \eqref{sec:reduced:system:main:intro} becomes
\eq{\label{sec:reduced:system:main:intro:qwe}\partial_s(\mu U_q)=0,\qquad \mu_s&=\frac{1}{2}G(\omega)\mu^2 UU_q.}
Let $(\mu,U_q)\restriction_{s=0}=(-2,A)$. Here we choose $\mu\restriction_{s=0}\equiv-2$ because $q\approx r-t$ and $\mu=q_t-q_r\approx -2$. For simplicity, we assume that $A\in C_c^\infty(\R\times\mathbb{S}^1)$. We then have \fm{\mu U_q=-2A,\quad\mu_s=-GA\mu U,\quad U_{sq}=GA UU_q.}
Because of the finite speed of propagation for \eqref{qwe} with compactly supported initial data, we assume that \eq{\label{intro:vanish:outside:U}\lim_{q\to\infty}U(s,q,\omega)=0.} This assumption allows us to uniquely recover $U$ from $U_q$. Recall that $q\approx r-t$, so $q\to\infty$ corresponds to spatial infinity. Similar to \eqref{hormander:cubic:asyeqn:qwe}, we seek to determine whether and when \eqref{sec:reduced:system:main:intro:qwe} admits global existence. The answer to this question can then be used to predict whether and when we may have a global solution to \eqref{qwe}. 

We now explain where the sign condition \eqref{asu:sign:G} comes from.  Set
\eq{\kappa(s,q,\omega)&=\exp\kh{\int_0^s(GAU)(s',q,\omega)\,ds'}>0.}
Then we have $(\mu,U_q)=(\frac{-2}{\kappa},A\kappa)$, $\kappa_s=GAU\kappa=GUU_q$, and
$U(s,q,\omega)=-\int_q^\infty (A\kappa)(s,q',\omega)\,dq'$. In the last identity, we use \eqref{intro:vanish:outside:U}.
Since $A\in C_c^\infty$, we have $|U|\lesssim \int_{[-R_0,R_0]}\kappa(s,q',\omega)\,dq'$ where $A$ vanishes for $|q|\geq R_0$. Now,
\fm{\partial_s\int_{[-R_0,R_0]}\kappa(s,q',\omega)\,dq'&=\int_{[-R_0,R_0]}\frac{1}{2}\partial_q(GU^2)(s,q',\omega)\,dq'=-\frac{1}{2}(GU^2)(s,-R_0,\omega).}
Here we use \eqref{intro:vanish:outside:U} and $A\equiv 0$ for $|q|\geq R_0$. We also emphasize that $U\equiv 0$ does not hold for $q\leq -R_0$ in general. If 
$G\geq 0$, then the quantity $\int_{[-R_0,R_0]}\kappa(s,q',\omega)\,dq'$ is nonincreasing with respect to $s$, so $|U|$ remains bounded for all $(s,q,\omega)$. Then, we have $|\kappa|+|U_q|\lesssim e^{Cs}$, which allows us to extend the solution to all $s\geq 0$. We also have $|\mu|+|\mu|^{-1}\lesssim e^{Cs}$, so the Lagrangian flow map $Y$ introduced below will remain a $C^1$ diffeomorphism. Thus, at the level of the geometric reduced system, we do not expect finite-time shock formation through the intersection of characteristic curves. If $G(\omega^0)<0$ for some $\omega^0\in\mathbb{S}^1$, then we cannot conclude that $|U|$ is bounded. In fact, it can be proved that a blowup may occur. It suggests that $G<0$ is a bad sign.

The discussion in the previous paragraph is formal. We will make it rigorous in Section~\ref{exm:qwe:1.1}; see Lemmas~\ref{lem:reduced:global:qwe} and~\ref{lem:reduced:blowup:qwe}. In Definition~\ref{def:geometric:weak:null}, we introduce a notion of \emph{geometric weak null condition}. It involves both the lifespan and growth control for solutions to the geometric reduced system. By Lemmas~\ref{lem:reduced:global:qwe} and~\ref{lem:reduced:blowup:qwe}, we conclude that the geometric weak null condition holds for \eqref{qwe} if and only if the sign condition \eqref{asu:sign:G} holds.

We now relate the classical asymptotic equation \eqref{hormander:cubic:asyeqn} to the geometric reduced system \eqref{sec:reduced:system:main:intro}. We can derive one from the other under appropriate assumptions. In fact, \eqref{sec:reduced:system:main:intro} is a \emph{Lagrangian} formulation of \eqref{hormander:cubic:asyeqn} and \eqref{hormander:cubic:asyeqn} is an \emph{Eulerian} formulation of \eqref{sec:reduced:system:main:intro}, provided that the corresponding Lagrangian flow map is a $C^1$ diffeomorphism.  We illustrate this by using \eqref{sec:reduced:system:main:intro:qwe} and \eqref{hormander:cubic:asyeqn:qwe} as examples. Suppose that $V(s,\rho,\omega)$ is a solution to \eqref{hormander:cubic:asyeqn:qwe}. Recall the Lagrangian flow map $Y(s,q,\omega)$ defined by \eqref{intro:def:charline}, and suppose that the map $q\mapsto Y(s,q,\omega)$ remains a $C^1$ diffeomorphism. Define
\fm{U(s,q,\omega)&:=V(s,Y(s,q,\omega),\omega),\\
 \mu (s,q,\omega)&:=\frac{-2}{Y_{q}(s,q,\omega)},\\
A(q,\omega)&:=V_{\rho}(0,q,\omega).}
We now check that $(\mu,U)$ solves the geometric reduced system \eqref{sec:reduced:system:main:intro:qwe}.
Recall from \eqref{intro:charline:Phirhoinvariant} that $V_{\rho}(s,Y(s,q,\omega),\omega)=A(q,\omega)$.
We have
\fm{U_q(s,q,\omega)&=V_\rho(s,Y(s,q,\omega),\omega)\cdot Y_{q}(s,q,\omega)=\frac{-2A(q,\omega)}{ \mu(s,q,\omega)}.}
It follows that
\fm{( \mu U_{q})(s,q,\omega)&=-2A(q,\omega),\\
(\partial_s \mu )(s,q,\omega)&=\frac{2Y_{sq}}{Y_{q}^2}=\frac{2G(\omega)(VV_{\rho})(s,Y(s,q,\omega),\omega)}{Y_{q}}\\
&=-(GA\mu U)(s,q,\omega)=\frac{1}{2}G(\omega)( \mu^2UU_q )(s,q,\omega).}
That is, $(\mu,U)$ solves \eqref{sec:reduced:system:main:intro:qwe}. This calculation shows that, as long as the map $q\mapsto Y(s,q,\omega)$ is a $C^1$ diffeomorphism, the geometric reduced system is a Lagrangian formulation of \eqref{hormander:cubic:asyeqn:qwe}.

If we start with a solution $(\mu,U)$ to the geometric reduced system, we recover the profile $V$ as follows. Here we need to assume that $\mu\equiv -2$ for $s=0$ and that $\mu<0$ everywhere. Write $\mu_s=\mu^2\partial_q(\frac{1}{4}G(\omega)U^2)$ and define $Y(s,q,\omega)$ by
\fm{Y(s,q,\omega)&=q+\frac{1}{2}\int_0^sG(\omega)U^2(s',q,\omega)\,ds'.}
Now, we have \fm{Y_{sq}&=G(\omega)UU_q=\frac{2\mu_s}{\mu^2}=\partial_s(-\frac{2}{\mu}),}
and thus \fm{(Y_q+\frac{2}{\mu})(s,q,\omega)=(Y_q+\frac{2}{\mu})(0,q,\omega)=0.}
This $Y$ is the Lagrangian flow map. Suppose that the map $q\mapsto Y(s,q,\omega)$ is a $C^1$ diffeomorphism\footnote{A sufficient condition is, for example, \fm{\int^{\infty}_q\frac{-2}{\mu(s,q',\omega)}\,dq'=\int_{-\infty}^q\frac{-2}{\mu(s,q',\omega)}\,dq'=\infty,\qquad\forall (s,q,\omega)\in I\times\R\times\mathbb{S}^1.}} and we denote the inverse by $\wt{Y}$. Now we can set 
\fm{V(s,\rho,\omega)=U(s,\wt{Y}(s,\rho,\omega),\omega).} One can check that $V$ solves \eqref{hormander:cubic:asyeqn:qwe}.

The discussions above can be extended to a general system of quasilinear wave equations in $\R^{1+2}$ with cubic leading nonlinearities; see Section~\ref{sec:connection}. They can also be extended to the case where the leading nonlinearities are quadratic and satisfy the quadratic null condition, and to the case in~$\R^{1+3}$.

In summary, two types of asymptotic equations describe the same asymptotic dynamics formally. The geometric reduced system is preferred in the present paper because it makes the geometry of characteristic curves an explicit part of the system. For example, the Jacobian of the Lagrangian flow map is $\frac{-2}{\mu}$, so the local invertibility of this flow map is related to the zeros of this quantity. This feature is useful in the study of lifespan, global existence, and long-time dynamics for quasilinear wave equations.

We end this subsection with a comparison between the sign condition \eqref{asu:sign:G} and the Agemi condition. Both are sign conditions on quantities involving coefficients in the Taylor expansions ($G_{0,3}\geq 0$ versus $F_{3,3}\leq 0$). However, the corresponding mechanisms are different. If $(\mu,U)$ is a solution to the geometric reduced system, then first derivatives of a solution $u$ to the original wave equation are expected to formally satisfy $\partial u\approx \eps r^{-\frac{1}{2}}\mu U_q$. Under the Agemi condition, we have $\mu\equiv -2$ while $|U_q|$ is nonincreasing in $s$. In particular, if $F_{3,3}(\omega)<0$, $|U_q|$ decays to zero and $\partial u$ acquires extra decay from it. Under our sign condition, $\mu U_q$ is invariant in $s$, so there is no extra decay for $\partial u$ from $U_q$. Meanwhile, the integral $\int_q^{R_0}\frac{-2}{\mu(s,q',\omega)}\,dq'$ is nonincreasing for $s\geq 0$ when $(q,\omega)$ is fixed\footnote{Here $R_0>0$ is a constant such that $A\equiv 0$ whenever $q\geq R_0$.}. Since $U_q=\frac{-2A}{\mu}$, the monotonicity of the integral above can be used to prove that $U$ is bounded.

\subsection{The main theorems}

Now we state the main theorems. 
Note that \eqref{qwe} has a symmetry $u(t,x)\mapsto \wt{u}(t,x)=u(-t,-x)$, so a backward Cauchy problem for $t\leq 0$ can be converted to a forward one for $t\geq 0$. Since the equation \eqref{qwe} is unchanged, so is the sign condition \eqref{asu:sign:G}.
Here we only state the results for $t\geq 0$ for simplicity.

Let $Z$ denote one of the following vector fields: translations $\partial_\alpha$, scaling $S=t\partial_t+r\partial_r$, rotation $\Omega=x_1\partial_2-x_2\partial_1$, and Lorentz boosts $\Omega_{0i}=x_i\partial_t+t\partial_i$. We use $Z^{\leq m}$ to denote an arbitrary product of at most $m$ vector fields $Z$. We allow $m$ to be a noninteger.

\begin{thm}\label{thm:main}
Fix an integer $N\geq 13$. Consider the equation \eqref{qwe} satisfying the sign condition \eqref{asu:sign:G}. Fix $(u_0,u_1)\in C_c^\infty(\R^2)$ with $\supp u_0\cup \supp u_1\subset B_{\R^2}(0,1)$. Then, there exists a small constant $\eps_0=\eps_0(g^{\alpha\beta},u_0,u_1,N)\in(0,1)$ such that the following facts hold. For all $\eps\in(0,\eps_0)$, the Cauchy problem \eqref{qwe}--\eqref{init} has a global solution $u$ for all $t\geq 0$. Moreover, we have
\fm{\norm{Z^{\leq N}\partial u(t)}_{L^2(\R^2)}&\lesssim \eps\lra{t}^{C\eps^2},}
\fm{\norm{Z^{\leq N-4} u(t)}_{L^\infty(\R^2)}&\lesssim \eps\lra{t}^{-\frac{1}{2}+C\eps^2},}
for all $t\geq 0$. Here $\lra{s}=\sqrt{1+s^2}\sim 1+|s|$ for $s\in\R$. The constants $C$ and implicit constants in these estimates depend on $g^{\alpha\beta},u_0,u_1,N$.
\end{thm}

We have a few remarks.

\begin{rmk}
\rm Theorem~\ref{thm:main} shows that the sign condition \eqref{asu:sign:G} is sufficient for global existence when we consider small, smooth, and localized initial data. It remains unclear whether the sign condition is necessary for global existence. In Section~\ref{exm:qwe:1.1}, however, we show that failure of the sign condition can lead to finite-time blowup for the corresponding geometric reduced system if we choose suitable data. It thus seems reasonable to conjecture that we may have finite-time blowup or instability for the original wave equation \eqref{qwe} if \eqref{asu:sign:G}  fails.
\end{rmk}

\begin{rmk}
\rm We compare our global existence result with the blowup result for small radial data in Li--Witt--Yin \cite{MR3357493}. The authors studied a scalar quasilinear wave equation
\eq{\label{eq:intro:LWY}
-u_{tt}+\divg(c(u)^2\nabla u)=0}
where $c$ is a smooth function of $u$. Without loss of generality, we assume that $c(0)=1$. The authors studied blowup in both the case where $c'(0)\neq 0$, and the case where $c'(0)=0$ and $c''(0)\neq 0$. Here we consider the latter case as it has a  cubic leading nonlinearity.

We rewrite \eqref{eq:intro:LWY} as
\fm{\Box u+(c(u)^2-1)\Delta u=-2c(u)c'(u)|\nabla_xu|^2}
which is of the form \eqref{qwe:general:intro}. If we temporarily ignore the semilinear term, we notice that our sign condition $G(\omega)\geq 0$ holds if and only if $c''(0)\geq 0$. However, Li--Witt--Yin proved a blowup result whenever $c''(0)\neq 0$, so the sign condition is not helpful. This is due to the semilinear term.

To explain this, we notice that the geometric reduced system for \eqref{eq:intro:LWY} is
\fm{\partial_s(\mu U_q)&=\frac{1}{2}c''(0)U(\mu U_q)^2,\qquad \partial_s\mu= \frac{1}{2}c''(0) \mu^2UU_q.}
In Section~\ref{sec:reduced:LWY}, we solve this reduced system and obtain
\fm{\frac{1}{\mu(s,q,\omega)}&=-\frac{1}{2}-\frac{1}{4}s\cdot c''(0)(\partial_qU^2)(0,q,\omega).}
As long as $c''(0)\neq 0$, $U\in C_c^\infty$, and $U\not\equiv 0$, there exists $(q,\omega)$ such that $c''(0)(\partial_qU^2)(0,q,\omega)<0$.  The right side is an increasing linear polynomial which is negative when $s=0$. Thus, finite-time blowup must occur for the geometric reduced system, and the geometric weak null condition fails. This is consistent with the finite-time blowup in \cite{MR3357493}.
\end{rmk}

\bigskip

In addition to global existence, we also study the asymptotic behavior of global solutions constructed in Theorem~\ref{thm:main}.

\begin{thm}\label{thm:asym}
Under the assumptions of Theorem~\ref{thm:main}, we let $u$ be a global solution to the Cauchy problem \eqref{qwe}--\eqref{init}. Let $q=q(t,x)$ be the solution to $\wt{L}q=0$ constructed in
\fm{\D=\{(t,x)\in\R^{1+2}:\ t\geq 1,\ t-t^{\lambda}+3\leq |x|\leq t+2\},\qquad \lambda=\frac{499}{500}.}
Here $\wt{L}$ is defined by \eqref{eq:def:wtL}, with $\wt{L}\approx \partial_t+\partial_r$ and $q\approx r-t$. We refer to Sections~\ref{sec:approxoptical} and~\ref{sec:asym:beha:setup} for details. We have a $C^2$ coordinate change
\fm{(t,x)\mapsto (s,q,\omega)=(\eps^2\ln t,q(t,x),\frac{x}{|x|})}
from $\D$ to 
\fm{\D^*=\{(s,q,\omega)\in[0,\infty)\times(-\infty,2]\times\mathbb{S}^1:\ q\geq 3-e^{\lambda\eps^{-2}s}\}}
and the inverse is also $C^2$. Set $(\mu,U)=(q_t-q_r,\eps^{-1}r^{\frac{1}{2}}u)$ in $\D$ and view it as a function of $(s,q,\omega)$ in $\D^*$.

Then, $(\mu,U)$ is an approximate solution to the geometric reduced system \eqref{sec:reduced:system:main:intro:qwe} in $\D^*$ in the sense of \eqref{est:approx:grs:asym:beha}.
Moreover, there exists an exact solution $(\wh{\mu},\wh{U})$ to  \eqref{sec:reduced:system:main:intro:qwe} such that, for each fixed $(q,\omega)\in(-\infty,2]\times\mathbb{S}^1$, $(\frac{\mu}{\wh{\mu}},U-\wh{U})\to(1,0)$ as $s\to\infty$; see Lemma~\ref{lem:asym:beha:T:existence:limit}.

Finally, if $u$ is a nonzero solution, and if $G\not\equiv 0$ (so $G\geq 0$ on $\mathbb{S}^1$ and $G>0$ somewhere), then
\fm{\lim_{t\to\infty}\lra{t}^{\frac{1}{2}}\norm{\partial_r^2u(t)}_{L^\infty(\D_t)}=\infty.}
Here $\D_{t_0}:=\D\cap\{t=t_0\}$.
In contrast, if $\psi$ is a solution to $\Box \psi=0$ with $C_c^\infty$ initial data, then
\fm{\sup_{t\geq 0}\lra{t}^{\frac{1}{2}}\norm{\partial^2\psi(t)}_{L^\infty(\R^2)}<\infty.}
\end{thm}

\begin{rmk}\rm
In the present paper, the geometric reduced system \eqref{sec:reduced:system:main:intro:qwe} is not just used to predict when there is a global solution to \eqref{qwe}.
Theorem~\ref{thm:asym} shows that one can actually recover the geometric reduced system \eqref{sec:reduced:system:main:intro:qwe} from a global solution $u$ to \eqref{qwe}.
\end{rmk}

\begin{rmk}\rm
If $G\not\equiv 0$ and if $u$ is a nonzero global solution constructed in Theorem~\ref{thm:main}, the last part of Theorem~\ref{thm:asym} indicates that the asymptotic behavior of $\partial^2u$ is different from that of $\partial^2\psi$ where $\psi$ is a global solution to $\Box \psi=0$ with $C_c^\infty$ initial data. It is, however, unclear whether there exists a linear solution $\psi$ such that $u-\psi\to 0$ in a weaker sense. For example, can we have energy scattering $\norm{\partial (u-\psi)(t)}_{L^2(\R^2)}\to 0$? This is unclear from the present paper.
\end{rmk}

\subsection{Proof sketch and paper organization}

We now give a sketch of the proof and explain how the paper is organized. We will follow Lindblad's proof in \cite{MR2382144} for \eqref{qwe:3D}, but there are also many important differences. We will explain the differences throughout this subsection.

In Section~\ref{sec:prelim}, we start with the convention used in the present paper. We then introduce the commuting vector fields and a null frame $\{L,\uLunit,E\}$ under the Minkowski metric. Note that $L$ and $E$ are tangent to the outgoing light cones $r=t+C$ and that tangential derivatives usually decay faster. We then define a vector field $\wt{L}$ by \eqref{eq:def:wtL}. Note that $\wt{L}$ is the vector field $L_2$ in \cite{MR2382144}. We then have a decomposition of the coefficients $(g^{\alpha\beta})$:
\fm{g^{\alpha\beta}=\frac{1}{2}\kh{\wt{L}^\alpha \uLunit^\beta+\wt{L}^\beta \uLunit^\alpha}+\mcl{R}^{\alpha\beta}.}
See \eqref{eq:null-frame-g-decompose:wtL}.
In our computations later, the $\mcl{R}^{\alpha\beta}$ are viewed as error terms. This is because they involve either two tangential derivatives, or one tangential derivative multiplied by $h^{\alpha\beta}:=g^{\alpha\beta}-m^{\alpha\beta}=O(|u|^2)$. We also prove a few lemmas that allow us to convert a wave equation $\wt{\Box}_g\phi=(\dots)$ into a transport equation $\wt{L}(r^{\frac{1}{2}}\uLunit\phi)=(\dots)$. See \eqref{est:transport:r12uLphi} and \eqref{est:transport:r12uLuLphi}.

In Section~\ref{sec:reduced:system}, we discuss the geometric reduced system. While the geometric reduced system is not directly used in the proof of Theorem~\ref{thm:main}, we implicitly use it to predict the pointwise decay rate of a solution $u$ and its derivatives. In this section, 
we first formally derive the geometric reduced system \eqref{sec:reduced:system:main} for a general system \eqref{qwe:general} of quasilinear wave equations in $\R^{1+2}$ with cubic leading nonlinearities. We also define a notion of a geometric weak null condition. Then, in Section~\ref{sec:connection}, we discuss the connection between the geometric reduced system  \eqref{sec:reduced:system:main} and the classical asymptotic equations \eqref{hormander:cubic:asyeqn:system}. In particular, we show that \eqref{sec:reduced:system:main} is a Lagrangian formulation of \eqref{hormander:cubic:asyeqn:system} under an appropriate coordinate change. This part complements the author's earlier work involving the geometric reduced system in $\R^{1+3}$. Finally, in Section~\ref{sec:reduced:exm}, we compute the geometric reduced systems for two scalar equations \eqref{qwe} and \eqref{eq:intro:LWY}. We show that \eqref{qwe} satisfies the geometric weak null condition if and only if the sign condition \eqref{asu:sign:G} holds, and that \eqref{eq:intro:LWY} violates the geometric weak null condition.

In Section~\ref{sec:bootstrap:asu}, we start our proof of Theorem~\ref{thm:main} by setting up the bootstrap argument. Let $N\geq 13$ be the order in Theorem~\ref{thm:main}.  Our bootstrap assumptions involve the pointwise bounds \eqref{asu:bootstrap:ptb}:
\fm{|Z^{\leq N-4}u|\leq M_0\eps  \lra{t}^{-\frac{1}{2}+\delta},\qquad t\in[0,\Tboot],\ x\in\R^2.}
Here $M_0>1$ and $\delta\in(0,1)$ are two constants to be chosen.

In Section~\ref{sec:approxoptical}, we seek to construct an approximate optical function $\wt{q}$ defined for all $[1,\Tboot]\times\R^2$. In our construction, we define a region $\D$ by \eqref{def:D}:
\fm{\D=\{(t,x)\in\R^{1+2}:\ 1\leq t\leq \Tboot,\  t-t^\lambda+3\leq |x|\leq t+ 2\},\qquad \lambda=1-2\delta,} solve the transport equation $\wt{L}q=0$ with $q=r-t$ on the lower boundary of $\D$, and define $\wt{q}$ by gluing $q$ with $r-t$ in a region where $0<t-r\sim t^\lambda$. This involves finding the integral curve $X(s)$ of $\wt{L}$ such that $\dot{X}^0=1$ and $X(t)=(t,x)$ for a given point $(t,x)\in\D$. For simplicity, one can consider
\fm{\D\approx\{t\geq 1,  -t^{\lambda}\leq |x|-t\leq 1\}.}
We choose the specific $\D$ above to guarantee that the corner of the lower boundary, $(t,r)=(1,3)$, lies completely in the interior of the region where $u\equiv 0$. By the finite speed of propagation (see \cite[Corollary I.2.3]{MR2455195}), we have $u\equiv 0$ for all $t\geq 0$ and $r\geq t+1$. After defining $\D$ and $q$, we derive several estimates for $q$, $\muring=-q_t+q_r>0$, their first derivatives, and $\partial^{\leq 2}u$ in $\D$. In particular, we have
\fm{|u|+\lra{q}|\partial u|+\muring^{-1}\lra{q}^2|\partial^2 u|&\lesssim M_0\eps\lra{t}^{-\frac{1}{2}}\lra{q}^{\frac{\delta}{\lambda}},\qquad \text{in }\D.}
The proof of this last bound relies on the transport equations $\wt{L}(r^{\frac{1}{2}}\uLunit u)=(\dots)$ and $\wt{L}(r^{\frac{1}{2}}\muring^{-1}\uLunit \partial u)=(\dots)$ derived in Section~\ref{sec:prelim}. See the proofs of \eqref{est:ptb:u}, \eqref{est:ptb:du}, and \eqref{est:ptb:ddu}. It is easy to transfer the estimates in $\D$ to all points in $[1,\Tboot]\times\R^2$ because of the cutoff function used to glue $q$ and $r-t$ together; see Section~\ref{sec:construction:wtq}. Note that $\wt{q}$ is an approximate solution to the eikonal equation in view of \eqref{est:eik:wtq}, so we call it an approximate optical function.

We have two remarks for Section~\ref{sec:approxoptical}. First, we use $t^\lambda$ in the definition of $\D$ for $\lambda=1-2\delta$. The specific choice of $\lambda$ is not important, and we choose it for convenience. For example, in the proof of \eqref{est:ptb:du}, we derive an estimate of the form $|\wt{L}(r^{\frac{1}{2}}\uLunit u)|\lesssim M_0\eps\lra{t}^{-2+3\delta}$ which is independent of the choice of $\lambda$. By integration along an integral curve of $\wt{L}$, we have in $\D$
\eq{\label{est:pf:stragety:1}r^{\frac{1}{2}}|(\uLunit u)(t,x)|&\lesssim M_0\eps\lra{t_0}^{-\lambda+\delta}+M_0\eps\lra{t_0}^{-1+3\delta}.}
Here $t_0\leq t$ is chosen so that $X(t_0)$ is the first point where the backward integral curve $X(s)$ passing through $(t,x)$ intersects with the lower boundary of $\D$. On the right side of \eqref{est:pf:stragety:1}, the first term comes from using the bootstrap assumptions \eqref{asu:bootstrap:ptb} on the lower boundary of $\D$, while the second comes from integrating $ \wt{L}(r^{\frac{1}{2}}\uLunit u)$ along $X(s)$. We choose $\lambda=1-2\delta$ to make the powers of $\lra{t}$ equal.

Moreover, we emphasize that the sign condition $G\geq 0$ is heavily used in the proofs in Section~\ref{sec:approxoptical}. After defining $h_{LL}=(g^{\alpha\beta}-m^{\alpha\beta})\wh{\omega}_\alpha\wh{\omega}_\beta$, we prove that $-C|u|^3\leq h_{LL}\leq C|u|^2$ in Lemma~\ref{lem:lowerbound:hLL}. The lower bound of $h_{LL}$ is better than its upper bound. Because of this, we have
\begin{itemize}
    \item a better lower bound for $\frac{\lra{q}}{\lra{r-t}}$: \fm{1\lesssim \frac{\lra{q}}{\lra{r-t}}\lesssim  \lra{t}^{CM_0^2\eps^2};}
    \item a better upper bound for $\muring^{-2}\uLunit\muring$:
    \fm{-C\lra{q}^{-2+\frac{2\delta}{\lambda}}\lra{t}^{CM_0^2\eps^2}\leq\frac{\uLunit \muring}{\muring^2}\leq  CM_0^2\eps^2\lra{q}^{-2+\frac{2\delta}{\lambda}} \ln\frac{C\lra{t}}{\lra{q}^{\frac{1}{\lambda}}};}
    \item a better lower bound for $\partial_tq_r$: \fm{-CM_0^2\eps^2\lra{q}^{-2+\frac{2\delta}{\lambda}}\ln(1+t)\cdot q_r^2\leq \partial_tq_r\leq C\lra{q}^{-2+\frac{2\delta}{\lambda}}\lra{t}^{CM_0^2\eps^2}q_r^2;}
    \item a better lower bound for $\partial_t\wt{q}_r$:
    \fm{ -CM_0^2\eps^2\lra{\wt{q}}^{-2+\frac{2\delta}{\lambda}}\ln(1+t)\cdot \wt{q}_r^2\leq \partial_t\wt{q}_r\leq C\lra{\wt{q}}^{-2+\frac{2\delta}{\lambda}}\lra{t}^{CM_0^2\eps^2}\wt{q}_r^2.}
\end{itemize}
The last estimate for $\partial_t\wt{q}_r$ is crucial in Section~\ref{sec:energy:poincare}. 

In Section~\ref{sec:energy:poincare}, we prove the necessary weighted energy estimates \eqref{est:energy} and Poincaré's estimates, both unweighted and weighted, in Lemmas~\ref{lem:poincare:unweighted} and~\ref{lem:poincare:weighted}. We use a weight \eqref{def:weight}:\eq{
w(t,x):=\wt{q}_r^{-1}\exp\kh{\kappa_1 \eps^2\ln(1+t)\cdot (2-\wt{q}(t,x))^{-\kappa_2}}.}
If we remove the factor $\wt{q}_r^{-1}$, this is the one used in Lindblad \cite{MR2382144}. Using his weight, we can still prove energy estimates (and the proof is in fact simpler than the proof of \eqref{est:energy}), but we cannot prove the weighted Poincaré's estimates. In the proof of Lemma~\ref{lem:poincare:weighted}, a key step is to find a good lower bound for $\partial_r(\wt{q}_rw)$. In our setting, it is roughly equal to $-\partial_t(\wt{q}_rw)$ (recall that $\partial_t+\partial_r$ is a tangential derivative which is usually good), so we need a good upper bound for $\partial_t(\wt{q}_rw)$ and $\partial_t\wt{q}_r$. However, in the previous paragraph, we have a good lower bound for $\partial_t\wt{q}_r$, which is a mismatch. If we use the weight \eqref{def:weight} with $\wt{q}_r^{-1}$, the proof of Lemma~\ref{lem:poincare:weighted} becomes much simpler, as we do not need to estimate any second derivative of $q$ there. The proof of \eqref{est:energy} becomes harder. In its proof, we need a good upper bound for 
\fm{-\wt{q}_r^{-1}g^{ij} \wh{\phi}_i \wh{\phi}_j\partial_t\wt{q}_r;} see \eqref{pf:est:energy:r1}. Here $\wh{\phi}_i$ is computed from $\partial\phi$ and $\partial\wt{q}$, and we have $-\wt{q}_r^{-1}g^{ij} \wh{\phi}_i \wh{\phi}_j\leq 0$. Thus, we need a good lower bound for $\partial_t\wt{q}_r$. This is what we have from the previous paragraph.

In Section~\ref{sec:higher:ptb}, we prove higher-order pointwise bounds for $u$:
\fm{|Z^{\leq N-6}u|&\lesssim M_0\eps \lra{t}^{-\frac{1}{2}+CM_0^2\eps^2}\lra{\wt{q}}^{\frac{\delta}{\lambda}}}for all $(t,x)\in[1,\Tboot]\times \R^2$. We also estimate $\wt{\Box}_g\partial^kZ^mu$ with $k+m\leq N$ for all $(t,x)\in[1,\Tboot]\times \R^2$. The proof is similar to that of $\partial^{\leq 2}u$ in Section~\ref{sec:approxoptical}. If we want to estimate $\partial^kZ^mu$ in $\D$ for $k\geq 1$, we use the transport equation $\wt{L}(r^{\frac{1}{2}}\muring^{-1}\uLunit \partial^{k-1}Z^mu)=(\dots)$ and integrate it along an integral curve of $\wt{L}$. To estimate $Z^mu$ in $\D$, we use the estimates for $\partial Z^mu$ and integrate from $r=t+1$ to inside; see, for example, \eqref{est:pf:claim:Zu:dZu} and the context there. We also remark that in the proof of Proposition \ref{prop:est:higher:wtBoxgu}, we use the inequality $\lfloor\frac{N+2}{2}\rfloor\leq N-6$ which is equivalent to $N\geq 13$. This is exactly where we choose $N\geq 13$ in Theorem \ref{thm:main}.

In Section~\ref{sec:main:endpf}, we finish our proof of Theorem~\ref{thm:main}. We first prove the energy bounds 
\fm{\norm{\partial Z^{\leq N}u}_{L^2(\R^2)}\lesssim M_0\eps\lra{t}^{CM_0^2\eps^2},}by using the weighted energy estimates and Poincaré's estimates proved in Section~\ref{sec:energy:poincare}. The proof here is similar to that in \cite[Section 9]{MR2382144}. Then, by these estimates and H\"ormander's $L^1$--$L^\infty$ estimate (see Lemma~\ref{lem:hormander:L1Linfty}), we obtain improved pointwise bounds for $Z^{\leq N-4}u$. This finishes our bootstrap argument.

Finally, in Section~\ref{sec:asym:beha}, we study the asymptotic behavior of global solutions $u$ constructed in Theorem~\ref{thm:main} and prove Theorem~\ref{thm:asym}. The idea here is similar to that in \cite{MR4772266}. We rewrite all the functions using the new coordinates $(s,q,\omega)$, derive the approximate geometric reduced system, and take a limit in a suitable sense. However, note that here we are not using an exact optical function.

\subsection{Acknowledgment}
The author would like to thank Jared Speck for several helpful discussions on the present paper.
The author was supported by a VandyGRAF Fellowship from Vanderbilt University. A generative AI tool (ChatGPT) was used for informal exploratory mathematical discussions and language editing. In particular, these discussions prompted the author to investigate the connection between the geometric reduced systems and the classical asymptotic equations in Section~\ref{sec:reduced:system}. All mathematical statements and proofs were carried out by the author.

\section{Preliminaries}\label{sec:prelim}
\subsection{Convention}
Throughout this paper, the parameter $\eps$ always denotes the size of the initial data in \eqref{init}. If we say that a property holds for $\eps\ll1$, we mean that there exists a small parameter $\eps_0\in(0,1)$ such that this property holds for all $\eps\in(0,\eps_0)$. In Section~\ref{sec:bootstrap:asu}, we list the parameters that this $\eps_0$ depends on. We will silently use $\eps\ll1$ in this paper. For example, we use $-1+C\eps<0$ without clearly stating that $\eps\ll1$.

We use $C$ to denote a universal positive constant. We write $A\lesssim B$, $B\gtrsim A$, or $A=O(B)$ if $|A|\leq CB$ for some $C>1$. The values of the constants can vary from place to place. We write $A\sim B$ if $A\lesssim B$ and $B\lesssim A$.
We use $C_v,\lesssim_v,\gtrsim_v$ if we want to emphasize that the constant depends on a parameter $v$. However, we will never write $C_\eps,\lesssim_\eps,\gtrsim_\eps$ or $C_{\eps_0},\lesssim_{\eps_0},\gtrsim_{\eps_0}$.

Unless specified otherwise, we always assume that the Latin indices $i,j$ take values in $\{1,2\}$ and that the Greek indices $\alpha,\beta$ take values in $\{0,1,2\}$. We use subscripts to denote
partial derivatives unless specified otherwise, so $u_{\alpha\beta}=\partial_\alpha\partial_\beta u$, $q_\alpha=\partial_\alpha q$, etc.  For a fixed integer $k\geq 0$, we use $\partial^k$ to denote either a specific partial derivative of order $k$ with respect to $(t,x)$ or the collection of partial derivatives of order $k$.

Given a function $f=f(\omega)$ defined on $\mathbb{S}^1$, we define $\partial_\omega f:=(\omega_1\partial_{\omega_2}-\omega_2\partial_{\omega_1})f$. To prevent confusion, we will only use $\partial_\omega$ in the coordinates $(s,q,\omega)$, and will never use it in the coordinates $(t,r,\omega)$. Again, for a fixed integer $k\geq 0$, we use $\partial^k_\omega$ to denote either an angular derivative of order $k$ or the collection of angular derivatives of order $k$.

\subsection{Commuting vector fields}
Let $Z$ be any of the following vector fields:
\eq{\label{commuting:vec:field}\partial_\alpha,\ \alpha=0,1,2;\ S=t\partial_t+r\partial_r;\ \Omega=x_1\partial_2-x_2\partial_1;\ \Omega_{0i}=t\partial_i+x_i\partial_t,\ i=1,2.}
Here we set $r=|x|$, $\omega=\frac{x}{|x|}$, and $\partial_r=\omega_1\partial_1+\omega_2\partial_2$. We write the vector fields in \eqref{commuting:vec:field} as $Z_1,\dots,Z_{7}$, respectively. For each integer $k\geq 0$, we use $Z^k$ to denote either a specific product of $k$ vector fields in $\{Z_l\}$, or the collection of all products of $k$ vector fields in $\{Z_l\}$. Similarly, for each $k\geq 0$, we use $Z^{\leq k}$ to denote either a specific product of at most $k$ vector fields in $\{Z_l\}$, or the collection of all products of at most $k$ vector fields in $\{Z_l\}$. We allow $k$ to be a noninteger in $Z^{\leq k}$.

We now state some basic properties of the vector fields in \eqref{commuting:vec:field}. These properties will be used silently later in the proof. First, we have Leibniz's rule:
\eq{Z^k(fg)&=\sum_{k_1+k_2=k} C\cdot Z^{k_1}f\cdot Z^{k_2}g.}
On the right side, we take the sum over all products with $k_1$ vector fields applied to $f$ and $k_2$ vector fields applied to $g$. The constants $C$ depend on $k_1,k_2,k$ and the specific products $Z^{k_1},Z^{k_2}$.

Next, we have the following commutation properties:
\eq{\ [Z,\Box]=C\cdot\Box,\quad [Z_n,Z_m]=C\cdot Z,\quad [Z,\partial]=C\cdot\partial.}
Here $C\cdot Z$ denotes a linear combination of the vector fields $Z$ in \eqref{commuting:vec:field}. Similarly for $C\cdot \partial$. Because of these commutation properties, we refer to the vector fields $Z$ as commuting vector fields. Moreover, because of $[Z,\partial]=C\cdot\partial$, we have
\eq{\sum_{j=0}^m |\partial^kZ^j f|\sim \sum_{j=0}^m |Z^j\partial^k f|,\qquad \forall k,m\geq 0.}
For simplicity, we can even write
\eq{|\partial^kZ^{\leq m} f|\sim |Z^{\leq m}\partial^k f|.}

We also recall several useful estimates involving $Z$. Their proofs
are standard and can be found in, e.g., \cite{MR2455195,MR1466700,MR2382144}.
\begin{lem}
For a function $\phi=\phi(t,x)$, we have
\eq{|\partial^k\phi|&\lesssim \lra{r-t}^{-k}|Z^{\leq k}\phi|,}
\eq{|(\partial_i+\omega_i\partial_t)\phi|&\lesssim \lra{r+t}^{-1}|Z\phi|.}
Here $|Z\phi|=|Z^1\phi|$ and $\lra{s}:=\sqrt{1+s^2}$ for $s\in\R$ is the Japanese bracket.
\end{lem}
\begin{lem}[The Klainerman--Sobolev inequality, \cite{MR784477,MR2455195}]
For $\phi=\phi(t,x)$ in $C^\infty(\R^{1+2})$ that vanishes for large $|x|$, we have
\eq{\lra{r+t}^{\frac{1}{2}}\lra{r-t}^{\frac{1}{2}}|\phi(t,x)|&\lesssim\norm{Z^{\leq 2}\phi(t)}_{L^2_x(\R^2)}.}
\end{lem}

\begin{lem}[The $L^1$--$L^\infty$ estimate, \cite{MR956961}]\label{lem:hormander:L1Linfty}
Let $\phi$ be a solution to $\Box\phi=F$ in $\R^{1+2}$ with zero data $(\phi,\phi_t)\restriction_{t=0}=(0,0)$. Then, for $t\geq 0$,
\eq{\lra{r+t}^{\frac{1}{2}}|\phi(t,x)|&\lesssim \int_0^{t}\int_{\R^2}\lra{\tau}^{-\frac{1}{2}}|Z^{\leq 1}F(\tau,y)|\,dy\,d\tau.}
\end{lem}

\bigskip

We now prove a useful lemma.

\begin{lem}
Let $u$ be a solution to \eqref{qwe}. Suppose that $|Z^{\leq \frac{N}{3}}u|\leq 1$ for an integer $N\geq 2$.\footnote{In Section \ref{sec:bootstrap:asu}, we assume that $|Z^{\leq N-4}u|\leq CM_0\eps\leq 1$ and that $N\geq 13$. In this case, $N-4\geq \frac{N}{3}$, so this lemma is applicable.} Then, for $k+m\leq N$, we have
\eq{
\label{est:pt:box:dZ:u}\abs{\wt{\Box}_g\partial^kZ^m u}&\lesssim  \sum_{k_1\leq k, m_3<m\atop m_1+m_2+m_3\leq m}\lra{r-t}^{-k+k_1}|Z^{m_1}u||Z^{m_2}u||\partial^{2+k_1}Z^{m_3}u|\\
&\quad+1_{k>0}\cdot \sum_{k_1+k_2+k_3\leq k\atop m_1+m_2+m_3\leq m}\lra{r-t}^{-k+k_1+k_2+k_3}|\partial^{k_1}Z^{m_1}u||\partial^{k_2+1}Z^{m_2}u||\partial^{k_3+1}Z^{m_3}u|.}
Here $1_{k>0}=1$ if $k>0$, and $1_{k>0}=0$ otherwise. Note that the first sum vanishes if $m=0$ as we cannot have $m_3<0$. 
\end{lem}
\begin{proof}
The proof is essentially the same as that of \cite[Lemma 3.4]{MR2382144}.  Here we have cubic terms on the right side because  $h^{\alpha\beta}(u)=g^{\alpha\beta}(u)-m^{\alpha\beta}=O(|u|^2)$. Moreover, we carefully keep track of the power of $\lra{r-t}$ in the proof.

By Leibniz's rule and the pointwise bounds $|Z^{\leq \frac{N}{3}}u|\leq 1$, we have
\eq{\label{est:pf:pt:box:dZ:u}|\partial^kZ^m(h^{\alpha\beta}(u))|&\lesssim \sum_{k_1+k_2\leq  k\atop m_1+m_2\leq m}\lra{r-t}^{-k+k_1+k_2}|\partial^{k_1}Z^{m_1}u||\partial^{k_2}Z^{m_2}u|,\qquad k+m\leq N.}
Moreover, if $m=0$ and $k\geq 1$, we have
\eq{\label{est:pf:pt:box:dZ:u:m=0}|\partial^k(h^{\alpha\beta}(u))|&\lesssim \sum_{k_1+k_2\leq k-1}\lra{r-t}^{-k+1+k_1+k_2}|\partial^{k_1}u||\partial^{k_2+1}u|.}
Here we briefly explain the proof. By \eqref{eq:g:taylor}, we can write $h^{\alpha\beta}=u\cdot u\cdot h_0(u)$ where $h_0(u)$ is a smooth function of $u$. By Leibniz's rule, we express $\partial^kZ^m(h^{\alpha\beta}(u))$ as a sum of terms of the form
\fm{C\cdot h_0^{(\ell)}(u)\cdot \prod_{j=1}^{\ell+2} \partial^{k_j}Z^{m_j}u}
with $\ell\geq 0$, $\sum_{j=1}^{\ell+2} k_j=k$, $\sum_{j=1}^{\ell+2} m_j=m$. Without loss of generality, we assume that $k_1+m_1$ and $k_2+m_2$ are the largest among the $k_j+m_j$. Note that for all other $j$, we have $k_j+m_j\leq\frac{N}{3}$, so we use $|\partial^{k_j}Z^{m_j}u|\lesssim\lra{r-t}^{-k_j}|Z^{\leq\frac{N}{3}}u| \lesssim \lra{r-t}^{-k_j}$ for these factors. This explains why we have $\lra{r-t}^{-k+k_1+k_2}$ in \eqref{est:pf:pt:box:dZ:u}. If $m=0$ and $k\geq 1$, then $\max_j\{k_j\}\geq 1$ because $\sum_j k_j=k$. Since $k_1$ and $k_2$ are the largest among $k_j$, we have $k_1+k_2\geq 1$. After exchanging $k_1$ and $k_2$ if necessary, we have $k_2\geq 1$ in \eqref{est:pf:pt:box:dZ:u}. Replacing $k_2$ by $k_2+1$ gives \eqref{est:pf:pt:box:dZ:u:m=0}.

Next, we use 
\fm{[\wt{\Box}_g,\partial]&=-(\partial h )\cdot\partial^2,\\
[\wt{\Box}_g,Z]&=C\cdot \Box +Ch\cdot \partial^2-Zh\cdot \partial^2=C\cdot\wt{\Box}_g +C\cdot Z^{\leq 1}h\cdot \partial^2,}
and
\fm{\wt{\Box}_g\partial^kZ^mu&=\sum_{k'\leq k-1} \partial^{k'}[\wt{\Box}_g,\partial]\partial^{k-1-k'}Z^mu+\sum_{m'\leq m-1}\partial^kZ^{m'}[\wt{\Box}_g,Z]Z^{m-1-m'}u\\
&=-\sum_{k'\leq k-1} \partial^{k'}\kh{\partial h\cdot \partial^{k+1-k'}Z^mu} +\sum_{m'\leq m-1} C\cdot\partial^kZ^{m'} \wt{\Box}_g Z^{m-1-m'}u\\
&\quad+\sum_{m'\leq m-1}C\cdot \partial^kZ^{m'}\kh{Z^{\leq 1}h\cdot \partial^2Z^{m-1-m'}u} \\
&=:S_1+S_2+S_3.}
Now, $S_1$ only appears if $k\geq 1$. By Leibniz's rule and \eqref{est:pf:pt:box:dZ:u:m=0}, we obtain terms of the form
\fm{ C\cdot \partial^{k''+1}h\cdot \partial^{k+1-k'+(k'-k'')}Z^mu =O\kh{\sum_{k_1+k_2\leq k''}\lra{r-t}^{-k''+k_1+k_2}|\partial^{k_1}u||\partial^{k_2+1}u||\partial^{k+1-k''}Z^mu|}}
for $k''\leq k'\leq k-1$. If we set $k_3=k-k''$, then $-k''+k_1+k_2=-k+k_1+k_2+k_3$ and $k_1+k_2+k_3\leq k''+k-k''\leq k$. Thus, we obtain the second term on the right side of \eqref{est:pt:box:dZ:u}.

Next, for $S_3$, it only appears if $m\geq 1$. By Leibniz's rule and \eqref{est:pf:pt:box:dZ:u}, we obtain terms of the form
\fm{&C\cdot \partial^{k'}Z^{\leq m''+1}h\cdot \partial^{2+k-k'}Z^{\leq m-1-m'+(m'-m'')}u \\
&=O\kh{\sum_{k_1+k_2\leq  k'\atop m_1+m_2\leq m''+1}\lra{r-t}^{-k'+k_1+k_2}|\partial^{k_1}Z^{m_1}u||\partial^{k_2}Z^{m_2}u|\cdot |\partial^{2+k-k'}Z^{\leq m-1- m'' }u|}}
for $k'\leq k$ and $m''\leq m'\leq m-1$. In the sum,  we set $m_3=m-1-m''$. This gives
\fm{\lra{r-t}^{-k'+k_1+k_2}|\partial^{k_1}Z^{m_1}u||\partial^{k_2}Z^{m_2}u||\partial^{2+k-k'}Z^{\leq m_3}u|}
where $m_1+m_2+m_3\leq m$ and $m_3<m$. If $k_1=k_2=0$, we set $k_3=k-k'$ and use the first term on the right side of \eqref{est:pt:box:dZ:u} to control it. If $k_2>0$ (and similarly if $k_1>0$), we set $k_3=k-k'+1$ which gives
\fm{\lra{r-t}^{-k+k_1+(k_2-1)+k_3}|\partial^{k_1}Z^{m_1}u||\partial^{1+(k_2-1)}Z^{m_2}u||\partial^{1+k_3}Z^{\leq m_3}u|}
with $k_1+(k_2-1)+k_3\leq k$. We use the second term on the right side of \eqref{est:pt:box:dZ:u} to control it.

Finally, we estimate $S_2$. It only appears if $m\geq 2$ because $\wt{\Box}_gu=0$. Now,
\fm{&\partial^kZ^{m'} \wt{\Box}_g Z^{m-1-m'}u=\sum_{m'\leq m''\leq m-2 }\partial^k  Z^{m''}[\wt{\Box}_g,Z] Z^{m-2-m''}u\\
&=\sum_{m'\leq m''\leq m-2 }C\cdot \partial^k  Z^{m''} \wt{\Box}_g  Z^{m-2-m''}u+\sum_{m'\leq m''\leq m-2 }C\cdot \partial^k  Z^{m''}\kh{Z^{\leq 1}h\cdot \partial^2Z^{m-2-m''}u}.}
The first term is of the same form as $S_2$ with fewer $Z$-derivatives. The second term is of the same form as $S_3$ with fewer $Z$-derivatives. We can then prove by induction. Because of $\partial^k$, we still obtain the same power of $\lra{r-t}$.
\end{proof}

\subsection{A null frame under Minkowski}
Let $\D\subset\R^{1+2}$ be a spacetime region. Assume that for a fixed constant $C_{\D}>1$, we have $\D\subset\{(t,x)\in\R^{1+2}:\ t\geq 1,\ C_{\D}^{-1}t\leq |x|\leq C_{\D}t\}$. We will give a precise definition of $\D$ later by \eqref{def:D}. Note that $\lra{r+t}\sim r\sim t$ in $\D$.

Fix $(t,x)\in\D$. Set $r=|x|$, $\omega=\frac{x}{|x|}$, and 
\eq{L=\partial_t+\partial_r,\quad \uLunit=-\partial_t+\partial_r,\quad E=\omega_1\partial_2-\omega_2\partial_1.}
It is clear that $\{L,\uLunit,E\}$ forms a null frame under the Minkowski metric in the sense that
\fm{m(L,\uLunit)=m(\uLunit,L)=2,\quad m(E,E)=1,\quad m(U,V)=0\ \text{for other pairs }U,V\in \{L,\uLunit,E\}.}
Following the notation in \cite[Section 2]{MR2382144}, for any two vector fields $X,Y$, we define
\eq{\label{def:g_XY}g_{XY}:=g^{\alpha\beta}(m_{\alpha\mu}X^{\mu})(m_{\beta\nu}Y^{\nu}).}
We emphasize that $g_{XY}$ is not the same as $g_{\alpha\beta}X^\alpha Y^\beta$ where $(g_{\alpha\beta})$ is the matrix inverse of $(g^{\alpha\beta})$. Similarly, we define $h_{XY}$ with $(g^{\alpha\beta})$ replaced by \eq{(h^{\alpha\beta})=(g^{\alpha\beta}-m^{\alpha\beta}).}

Moreover, we set
\eq{
\label{eq:null-frame-coefficients}
\left(g^{UV}\right)_{U,V\in\{L,\uLunit,E\}}
=
\begin{pmatrix}
\displaystyle \frac14g_{\uLunit\uLunit}
&
\displaystyle \frac14g_{\uLunit L}
&
\displaystyle \frac12g_{\uLunit E}
\\[1em]
\displaystyle \frac14g_{L\uLunit}
&
\displaystyle \frac14g_{LL}
&
\displaystyle \frac12g_{LE}
\\[1em]
\displaystyle \frac12g_{E\uLunit}
&
\displaystyle \frac12g_{EL}
&
\displaystyle g_{EE}
\end{pmatrix}=
\begin{pmatrix}
\displaystyle \frac14h_{\uLunit\uLunit}
&
\displaystyle \frac12+\frac14h_{\uLunit L}
&
\displaystyle \frac12h_{\uLunit E}
\\[1em]
\displaystyle \frac12+\frac14h_{L\uLunit}
&
\displaystyle \frac14h_{LL}
&
\displaystyle \frac12h_{LE}
\\[1em]
\displaystyle \frac12h_{E\uLunit}
&
\displaystyle \frac12h_{EL}
&
\displaystyle 1+h_{EE}
\end{pmatrix}.
}
Similarly, we define $(h^{UV})$.
This allows us to write 
\eq{\label{eq:null-frame-g-decompose}g^{\alpha\beta}=\sum_{U,V\in\{L,\uLunit,E\}}g^{UV}U^\alpha V^\beta=\frac{1}{2}\kh{L^\alpha\uLunit^\beta+L^\beta\uLunit^\alpha}+E^\alpha E^\beta+\sum_{U,V\in\{L,\uLunit,E\}}h^{UV}U^\alpha V^\beta.}

\begin{lem}
We have
\eq{\label{eq:null-frame-g-decompose:wtL}g^{\alpha\beta}&=\frac{1}{2}\kh{\wt{L}^\alpha\uLunit^\beta+\wt{L}^\beta\uLunit^\alpha}+E^\alpha E^\beta+\sum_{U,V\in\{L,\uLunit,E\}\atop (U,V)\neq (\uLunit,\uLunit)}h^{UV}U^\alpha V^\beta }
where
\eq{\label{eq:def:wtL}\wt{L}:=L+\frac{1}{4}h_{LL} \uLunit=(1-\frac{1}{4}h_{LL})\partial_t+(1+\frac{1}{4}h_{LL})\partial_r.}
\end{lem}
\begin{proof}
This is a direct corollary of \eqref{eq:null-frame-g-decompose} and $h^{\uLunit\uLunit}=\frac{1}{4}h_{LL}$. This $\wt{L}$ is the same as the $L_2$ in \cite{MR2382144}. 
\end{proof}

\begin{lem}
For a function $\phi=\phi(t,x)$, in $\D$ we have
\eq{\label{est:transport:r12uLphi:raw}\abs{r^{\frac{1}{2}}\wt{\Box}_g\phi-\wt{L}\kh{r^{\frac{1}{2}}\uLunit\phi}}&\lesssim \lra{r+t}^{-\frac{1}{2}}\kh{\lra{r+t}^{-1}+\lra{r-t}^{-1}|h|}|Z^{\leq 2}\phi|.}
As a result, if $| h|\lesssim \frac{\lra{r-t}}{\lra{r+t}}$ in $\D$, then
\eq{\label{est:transport:r12uLphi}\abs{r^{\frac{1}{2}}\wt{\Box}_g\phi-\wt{L}\kh{r^{\frac{1}{2}}\uLunit\phi}}&\lesssim   \lra{r+t}^{-\frac{3}{2}}|Z^{1\leq \cdot\leq 2}\phi|.}
By $Z^{1\leq \cdot\leq 2} $, we mean $Z^k $ with $1\leq k\leq 2$.
\end{lem}
\begin{proof}
By \eqref{eq:null-frame-g-decompose:wtL}, we have
\fm{r^{\frac{1}{2}}\wt{\Box}_g\phi&=r^{\frac{1}{2}}g^{\alpha\beta}\partial_\alpha\partial_\beta \phi= r^{\frac{1}{2}}\uLunit^\alpha\cdot\wt{L}\partial_\alpha\phi+r^{\frac{1}{2}}E^\alpha \cdot E\partial_\alpha \phi+\sum_{U,V\in\{L,\uLunit,E\}\atop (U,V)\neq (\uLunit,\uLunit)}r^{\frac{1}{2}}h^{UV}U^\alpha V^\beta\partial_\alpha\partial_\beta\phi.}
We have 
\fm{r^{\frac{1}{2}}\uLunit^\alpha\cdot\wt{L}\partial_\alpha\phi&=\wt{L}\kh{r^{\frac{1}{2}}\uLunit\phi}-\wt{L}\kh{r^{\frac{1}{2}}\uLunit^\alpha}\partial_\alpha\phi=\wt{L}\kh{r^{\frac{1}{2}}\uLunit\phi}-\wt{L}\kh{r^{\frac{1}{2}}}\uLunit\phi-r^{\frac{1}{2}}\wt{L}\kh{\uLunit^\alpha}\partial_\alpha\phi\\
&=\wt{L}\kh{r^{\frac{1}{2}}\uLunit\phi}-\frac{1}{2}r^{-\frac{1}{2}} \kh{1+\frac{1}{4}h_{LL}}\uLunit\phi=\wt{L}\kh{r^{\frac{1}{2}}\uLunit\phi}-\frac{1}{2}r^{-\frac{1}{2}} \uLunit\phi+O(r^{-\frac{1}{2}}|h||\partial\phi|).}
Here we use $\partial_r\omega=0$.  Moreover,
\fm{r^{\frac{1}{2}}E^\alpha \cdot E\partial_\alpha \phi&=r^{-\frac{1}{2}}E^\alpha \cdot \Omega\partial_\alpha \phi=r^{-\frac{3}{2}}\Omega\Omega \phi+r^{-\frac{1}{2}}\partial_r\phi.}
For $U,V\in\{L,\uLunit,E\}$ with $U\neq \uLunit$, we have
\fm{\abs{r^{\frac{1}{2}}h^{UV}U^\alpha V^\beta\partial_\alpha\partial_\beta\phi}&\leq  r^{\frac{1}{2}}|h^{UV}| | U\partial\phi|\lesssim r^{-\frac{1}{2}}|h||Z\partial \phi|\lesssim r^{-\frac{1}{2}}\lra{r-t}^{-1}|h||Z^{\leq 2}\phi|.}
Similarly for $U,V\in\{L,\uLunit,E\}$ with $V\neq \uLunit$. In summary,
\fm{r^{\frac{1}{2}}\wt{\Box}_g\phi&=\wt{L}\kh{r^{\frac{1}{2}}\uLunit\phi}+\frac{1}{2}r^{-\frac{1}{2}} L\phi +O\kh{r^{-\frac{3}{2}}|Z^2\phi|+r^{-\frac{1}{2}}\lra{r-t}^{-1}|h||Z^{\leq 2}\phi|}.}
We end the proof by noticing that $|L\phi|\lesssim \lra{r+t}^{-1}|Z\phi|$ and that $\lra{r+t}\sim r$ in $\D$. This proves \eqref{est:transport:r12uLphi:raw}, and \eqref{est:transport:r12uLphi} follows from $|h|\lesssim\frac{\lra{r-t}}{\lra{r+t}}$.
\end{proof}

\begin{lem}
Suppose that $| h|\lesssim \frac{\lra{r-t}}{\lra{r+t}}$ and that $|\partial h|\lesssim \lra{r+t}^{-1}$ in $\D$.
For a function $\phi=\phi(t,x)$,  in $\D$ we have 
\eq{\label{est:transport:r12uLuLphi}&\abs{r^{\frac{1}{2}}\uLunit\wt{\Box}_g\phi-\wt{L}\kh{r^{\frac{1}{2}}\uLunit \uLunit\phi}-\frac{1}{4}r^{\frac{1}{2}}(\uLunit h)_{LL}\uLunit\uLunit\phi}\lesssim\lra{r+t}^{-\frac{3}{2}}|Z^{1\leq \cdot\leq 2}\partial\phi| .}
\end{lem}
\begin{proof}
We have 
\fm{r^{\frac{1}{2}}\partial_\gamma\wt{\Box}_g\phi&=r^{\frac{1}{2}}\partial_\gamma\kh{g^{\alpha\beta}\partial_\alpha\partial_\beta\phi}=r^{\frac{1}{2}}\wt{\Box}_g\partial_\gamma\phi+r^{\frac{1}{2}}(\partial_\gamma h^{\alpha\beta})\partial_\alpha\partial_\beta\phi.}
By replacing $\phi$ with $\partial_\gamma\phi$ in \eqref{est:transport:r12uLphi}, we have
\fm{\abs{r^{\frac{1}{2}}\wt{\Box}_g\partial_\gamma\phi-\wt{L}\kh{r^{\frac{1}{2}}\uLunit\partial_\gamma\phi}} 
&\lesssim \lra{r+t}^{-\frac{3}{2}} |Z^{1\leq \cdot\leq 2}\partial\phi|.}
By replacing $(g^{\alpha\beta})$ with $(\partial_\gamma h^{\alpha\beta})$ in \eqref{def:g_XY} and \eqref{eq:null-frame-coefficients}, we define $(\partial_\gamma h)_{XY}$ and $(\partial_\gamma h)^{UV}$. We thus obtain
\fm{\partial_\gamma h^{\alpha\beta}&=\sum_{U,V\in\{L,\uLunit,E\}}(\partial_\gamma h)^{UV}U^\alpha V^\beta}
which is analogous to \eqref{eq:null-frame-g-decompose}.  Since $\uLunit(\uLunit^\alpha)=0$, we have
\fm{\uLunit^\alpha \uLunit^\beta \partial_\alpha\partial_\beta\phi&=\uLunit^\alpha\uLunit(\partial_\alpha \phi)=\uLunit\uLunit\phi.}
Moreover, if $U,V\in\{L,\uLunit,E\}$ and $U\neq \uLunit$, we have
\fm{|(\partial h)^{UV}||U^\alpha V^\beta\partial_\alpha\partial_\beta\phi|&\lesssim |\partial h| | U(\partial_\beta \phi)|\lesssim \lra{r+t}^{-2}|Z\partial\phi|.}
Similarly for $U,V\in\{L,\uLunit,E\}$ and $V\neq \uLunit$. Thus,
\fm{ r^{\frac{1}{2}}(\partial_\gamma h^{\alpha\beta})\partial_\alpha\partial_\beta\phi&=r^{\frac{1}{2}}(\partial_\gamma h)^{\uLunit\uLunit}\uLunit^\alpha\uLunit^\beta\partial_\alpha\partial_\beta\phi+r^{\frac{1}{2}}\sum_{U,V\in\{L,\uLunit,E\}\atop (U,V)\neq(\uLunit,\uLunit)}(\partial_\gamma h)^{UV}U^\alpha V^\beta\partial_\alpha\partial_\beta\phi\\
&=\frac{1}{4}r^{\frac{1}{2}}(\partial_\gamma h)_{LL}\uLunit\uLunit\phi+O(\lra{r+t}^{-\frac{3}{2}}|Z\partial\phi|).}
To finish the proof, we notice that $\wt{L}(\uLunit^\alpha)=\uLunit(\uLunit^\alpha)=0$, that $\uLunit^\gamma(\partial_\gamma h)_{LL}=(\uLunit h)_{LL}$, and that $\uLunit^\gamma\wt{L}(r^{\frac{1}{2}}\uLunit\partial_\gamma\phi)=\wt{L} (r^{\frac{1}{2}}\uLunit \uLunit\phi)$.
\end{proof}

\section{The geometric reduced system}
\label{sec:reduced:system}
In this section, we consider a general system of quasilinear wave equations in $\R^{1+2}$:
\eq{\label{qwe:general}g^{\alpha\beta}(u,\partial u)\partial_\alpha\partial_\beta u^I=F^I(u,\partial u),\qquad I=1,2,\dots,M.}
Our unknown $u=(u^I)_{I=1,\dots,M}$ is an $\R^M$-valued function. Assume that the quasilinear coefficients $(g^{\alpha\beta})$ are the same for each component $u^I$ in the equations and that $g^{\alpha\beta}=g^{\beta\alpha}$. Moreover, we assume that the nonlinearities in \eqref{qwe:general} are at least cubic, and that the coefficients $g^{\alpha\beta}$ and $F^I$ have Taylor expansions
\eq{\label{id:general:taylor:g}g^{\alpha\beta}(u,\partial u)&= m^{\alpha\beta}+g^{\alpha\beta}_{0,JK} u^Ju^K+g^{\alpha\beta,\lambda}_{1,JK} u^J\partial_\lambda u^K+g^{\alpha\beta,\sigma\lambda}_{2,JK} \partial_\sigma u^J\partial_\lambda u^K+O(|(u,\partial u)|^3),}
\eq{\label{id:general:taylor:F}F^I(u,\partial u)&= f^I_{0,JKL} u^Ju^Ku^L+f^{I,\lambda}_{1,JKL} u^Ju^K\partial_\lambda u^L+f^{I,\sigma\lambda}_{2,JKL} u^J\partial_\sigma u^K\partial_\lambda u^L\\
&\quad+f^{I,\kappa\sigma\lambda}_{3,JKL}\partial_\kappa u^J\partial_\sigma u^K\partial_\lambda u^L+O(|(u,\partial u)|^4).}
Here we assume the $g^{*}_{*,*}$ and the $f^{*}_{*,*}$ are all real constants. We also use the Einstein summation convention for $\kappa,\sigma,\lambda=0,1,2$ and $J,K,L=1,2,\dots,M$.

For $\omega\in\mathbb{S}^1$, we set $\wh{\omega}=(-1,\omega)$ and
\begin{align}
G_{0,JK}(\omega)&=g^{\alpha\beta}_{0,JK}\wh{\omega}_\alpha\wh{\omega}_\beta,\\
G_{1,JK}(\omega)&=g^{\alpha\beta,\lambda}_{1,JK}\wh{\omega}_\alpha\wh{\omega}_\beta\wh{\omega}_\lambda,\\
G_{2,JK}(\omega)&=g^{\alpha\beta,\sigma\lambda}_{2,JK}\wh{\omega}_\alpha\wh{\omega}_\beta\wh{\omega}_\sigma\wh{\omega}_\lambda.
\end{align}
Similarly, we define
\begin{align}
F_{0,JKL}^I(\omega)&=f^{I}_{0,JKL},\\
F_{1,JKL}^I(\omega)&=f^{I,\lambda}_{1,JKL} \wh{\omega}_\lambda,\\
F_{2,JKL}^I(\omega)&=f^{I,\sigma\lambda}_{2,JKL} \wh{\omega}_\sigma\wh{\omega}_\lambda,\\
F_{3,JKL}^I(\omega)&=f^{I,\kappa\sigma\lambda}_{3,JKL} \wh{\omega}_\kappa\wh{\omega}_\sigma\wh{\omega}_\lambda.
\end{align}
These quantities will be used later.

Here is how the section is organized.
In Section~\ref{sec:reduced:system:derivation}, we formally derive the geometric reduced system \eqref{sec:reduced:system:main} for \eqref{qwe:general} and define the associated geometric weak null condition. In Section~\ref{sec:connection}, we first present the classical asymptotic equations \eqref{hormander:cubic:asyeqn:system} for \eqref{qwe:general}. Then, the main goal of the subsection is to establish an equivalence between \eqref{sec:reduced:system:main} and \eqref{hormander:cubic:asyeqn:system} under suitable assumptions. We conclude that \eqref{sec:reduced:system:main} is a Lagrangian formulation of \eqref{hormander:cubic:asyeqn:system} and that \eqref{hormander:cubic:asyeqn:system} is an Eulerian formulation of \eqref{sec:reduced:system:main}, provided that the corresponding Lagrangian flow map is a $C^1$ diffeomorphism. Finally, in Section~\ref{sec:reduced:exm}, we compute the geometric reduced systems for two scalar equations. One of them is \eqref{qwe}. We discuss in detail when the corresponding geometric reduced system \eqref{sec:reduced:system:main:1.1} has a global solution, and when \eqref{qwe} satisfies the geometric weak null condition.

We remark that most discussions in this section can be extended to the case where quadratic leading nonlinearities satisfying the quadratic null condition are present. In fact, under the quadratic null condition, the quadratic terms in the wave equations contribute trivially to the asymptotic equations.

\subsection{Derivation of the geometric reduced system}\label{sec:reduced:system:derivation}

Let $u=(u^I)$ be a solution to \eqref{qwe:general}, and let $q$ be a solution to the eikonal equation \eq{\label{eik:general} g^{\alpha\beta}(u,\partial u)\partial_\alpha q\partial_\beta q=0,\qquad q\approx r-t.}
We assume that $u$ has an ansatz
\eq{u=(u^I(t,x))&\approx \frac{\eps}{r^{\frac{1}{2}}}U=(\frac{\eps}{r^{\frac{1}{2}}}U^I(s,q,\omega))}
with
\eq{r=|x|,\qquad \omega=\frac{x}{|x|},\qquad s:=\eps^2\ln t,\qquad q=q(t,x).}
We also set
\eq{\mu:=q_t-q_r,\qquad \nu:=q_t+q_r.}

In our derivation, we make the following assumptions.
\begin{enumerate}[(a)]
    \item Every function is smooth.
    \item There is a diffeomorphism between two coordinate systems $(t,x)$ and $(s,q,\omega)$, so any function $H$ can be written as $H(t,x)$ and $H(s,q,\omega)$ at the same time.
    \item The parameter $\eps$ is sufficiently small. The radius $r$ and time $t$ are sufficiently large, and we have $t\approx r$.
    \item Assume that $\mu,U$ are of size $1$, and $\nu$ is of size $\eps^2 t^{-1}$. Similarly for their derivatives with respect to $(s,q,\omega)$.
    \item Assume that $q_i\approx \omega_iq_r$ and that $\mu=q_t-q_r<0$.
\end{enumerate}

We refer our readers to \cite[Section 3]{MR4232783} and \cite[Chapter 2]{MR4315017} for several remarks on these assumptions. We emphasize that these are merely assumptions for our derivation. The actual estimates may be different.

\subsubsection{Derivatives of $u$}
Now, we have
\fm{u^I_t&\approx\eps r^{-\frac{1}{2}}  U^I_q q_t +\eps^3 r^{-\frac{1}{2}} t^{-1}  U^I_s  ,\\
u^I_{i}&\approx\eps r^{-\frac{1}{2}}  U^I_q q_i-\frac{1}{2}\eps r^{-\frac{3}{2}}\omega_iU^I+\eps r^{-\frac{1}{2}}U^I_{\omega_j}\cdot\partial_i\omega_j.}
For $\partial^2u$, if we only consider terms of order $\eps r^{-\frac{1}{2}}$, we have
\fm{u_{\alpha\beta}^I&\approx\partial_\beta\kh{\eps r^{-\frac{1}{2}}U_q^Iq_\alpha}\approx \eps r^{-\frac{1}{2}}U_q^Iq_{\alpha\beta}+\eps r^{-\frac{1}{2}}U_{qq}^Iq_\alpha q_{\beta}.}
If the derivatives fall on $s,\omega,r^{-\frac{1}{2}}$, we obtain lower-order terms. Most of them can be ignored, but we need to compute them for $u_{\alpha\alpha}^I$:
\fm{u_{tt}^I&\approx\eps r^{-\frac{1}{2}} (-U^I_s\cdot \eps^2t^{-2}+U^I_{ss}\cdot \eps^4t^{-2}+2U^I_{sq}\cdot \eps^2t^{-1}q_t+U^I_q q_{tt}+U^I_{qq} q_{t}^2)\\
&\approx\eps  r^{-\frac{1}{2}} (U^I_q q_{tt}+U^I_{qq} q_{t}^2)+\underline{2\eps^3r^{-\frac{1}{2}}t^{-1}U_{sq}^Iq_t},\\
 u_{ii}^I&\approx \partial_i\kh{\eps r^{-\frac{1}{2}}  U^I_q q_i-\frac{1}{2}\eps r^{-\frac{3}{2}}\omega_iU^I+\eps r^{-\frac{1}{2}}U^I_{\omega_j}\cdot\partial_i\omega_j}\\
 &\approx \eps r^{-\frac{1}{2}}  \kh{U^I_{qq} q_{i}^2+U^I_{q} q_{ii}} - \underline{\eps r^{-\frac{3}{2}}\omega_iU_q^Iq_i}+\underline{2\eps r^{-\frac{1}{2}}U^I_{q\omega_j}\cdot q_i\partial_i\omega_j}.}
It follows that
\eq{\label{est:sec:reduced:system:quasi}g^{\alpha\beta}(u,\partial u)\partial_\alpha\partial_\beta u^I&\approx 
\eps r^{-\frac{1}{2}}U^I_q\cdot g^{\alpha\beta}q_{\alpha\beta}+\eps r^{-\frac{1}{2}}U_{qq}^Ig^{\alpha\beta}q_\alpha q_{\beta}\\
&\quad-2\eps^3r^{-\frac{1}{2}}t^{-1}U_{sq}^Iq_t-\eps r^{-\frac{3}{2}} U_q^Iq_r+2\eps r^{-\frac{1}{2}}U^I_{q\omega_j}\cdot \sum_iq_i\partial_i\omega_j\\
&\approx \eps r^{-\frac{1}{2}}U_q^I\cdot g^{\alpha\beta}q_{\alpha\beta}- \eps^3r^{-\frac{1}{2}}t^{-1}\mu U_{sq}^I+\frac{1}{2}\eps r^{-\frac{3}{2}}\mu U_q^I.} 
Here we use the eikonal equation, $\sum_i q_i\partial_i\omega_j\approx q_r\sum_i\omega_i\partial_i\omega_j\approx q_r\partial_r\omega\approx0$, and $q_\alpha\approx-\frac{1}{2}\mu\wh{\omega}_\alpha$. Thus, there is no angular derivative of $U$ involved at the end. It thus makes sense to simply assume that the angular derivatives of $U$ are negligible. We did so in \cite{MR4232783,MR4315017}.

Moreover, using $u^I\approx\eps r^{-\frac{1}{2}}U^I$ and $\partial_\alpha u^I\approx \eps r^{-\frac{1}{2}}U^I_qq_\alpha\approx -\frac{1}{2}\eps r^{-\frac{1}{2}}U^I_q\mu\wh{\omega}_\alpha$, we have
\eq{\label{est:sec:reduced:system:semi}F^I(u,\partial u)&\approx \eps^3 r^{-\frac{3}{2}} f^I_{0,JKL}  U^JU^KU^L+\eps^3 r^{-\frac{3}{2}} f^{I,\lambda}_{1,JKL} U^JU^K U^L_q \cdot\kh{-\frac{1}{2}\mu\wh{\omega}_\lambda}\\
&\quad+\eps^3 r^{-\frac{3}{2}} f^{I,\sigma\lambda}_{2,JKL} U^J  U^K_q U^L_q\cdot \kh{\frac{1}{4}\mu^2\wh{\omega}_\sigma \wh{\omega}_\lambda}+\eps^3 r^{-\frac{3}{2}}f^{I,\kappa\sigma\lambda}_{3,JKL}U^J_qU^K_qU^L_q\cdot\kh{-\frac{1}{8}\mu^3\wh{\omega}_\kappa\wh{\omega}_\sigma \wh{\omega}_\lambda}\\
&\approx \eps^3 r^{-\frac{3}{2}} F^I_{0,JKL}(\omega)  U^JU^KU^L-\frac{1}{2}\eps^3 r^{-\frac{3}{2}} F^{I }_{1,JKL}(\omega) U^JU^K (\mu U^L_q) \\
&\quad+\frac{1}{4}\eps^3 r^{-\frac{3}{2}} F^{I}_{2,JKL}(\omega) U^J  (\mu U^K_q) (\mu U^L_q)  -\frac{1}{8}\eps^3 r^{-\frac{3}{2}}F^{I}_{3,JKL}(\omega) (\mu U^J_q)(\mu U^K_q)(\mu U^L_q). }

\subsubsection{Derivatives of $q$} We have
\fm{q_t&\approx \frac{\mu+\nu}{2}\approx \frac{1}{2}\mu,\\
q_r&\approx\frac{-\mu+\nu}{2}\approx -\frac{1}{2}\mu,\\
q_i&\approx\omega_iq_r\approx \frac{-\mu+\nu}{2}\cdot \omega_i\approx -\frac{1}{2}\mu\omega_i.} 
By \eqref{eik:general}, we have
\fm{0&\approx m^{\alpha\beta}q_\alpha q_\beta+g^{\alpha\beta}_{0,JK} \cdot\frac{1}{4}\mu^2\wh{\omega}_\alpha\wh{\omega}_\beta\cdot \eps^2r^{-1}U^JU^K+g^{\alpha\beta,\lambda}_{1,JK}\cdot\frac{1}{4}\mu^2\wh{\omega}_\alpha\wh{\omega}_\beta\cdot \eps^2r^{-1} U^JU^K_q\cdot\kh{-\frac{1}{2}\mu\wh{\omega}_\lambda}\\
&\quad+g^{\alpha\beta,\sigma\lambda}_{2,JK}\cdot\frac{1}{4}\mu^2\wh{\omega}_\alpha\wh{\omega}_\beta\cdot\eps^2r^{-1} U_q^J U_q^K\cdot\kh{ \frac{1}{4}\mu^2\wh{\omega}_\sigma\wh{\omega}_\lambda}\\
&\approx -\mu\nu +\eps^2 r^{-1}\kh{\frac{1}{4}G_{0,JK}(\omega)\mu^2U^JU^K-\frac{1}{8}G_{1,JK}(\omega)\mu^3U^JU^K_q+\frac{1}{16}G_{2,JK}(\omega)\mu^4 U^J_qU^K_q}.}
Divide both sides by $\mu$. We obtain
\eq{ \label{est:sec:reduced:system:nu} \nu  &\approx\eps^2 r^{-1} \kh{\frac{1}{4}G_{0,JK}(\omega)\mu U^JU^K-\frac{1}{8}G_{1,JK}(\omega)\mu^2U^JU^K_q+\frac{1}{16}G_{2,JK}(\omega)\mu^3 U^J_qU^K_q}.}
We only keep those terms of order $\eps^2t^{-1}$, so we do not consider the case when $\partial_q$ falls on $r^{-1}$.

For $\partial^2q$, if we only consider terms of order $1$, we have
\fm{q_{\alpha\beta}&\approx\partial_\beta\kh{-\frac{1}{2}\mu\wh{\omega}_\alpha}\approx -\frac{1}{2}\mu_qq_\beta\cdot \wh{\omega}_\alpha \approx \frac{1}{4}\mu\mu_q \wh{\omega}_\alpha \wh{\omega}_\beta .}
Recall that $\wh{\omega}=(-1,\omega)$. For $q_{\alpha\alpha}$, we also consider terms of order $t^{-1}$:
\fm{q_{tt}&\approx \frac{\mu_t+\nu_t}{2}\approx\frac{1}{2}\kh{(\mu_s+\nu_s)\cdot \eps^2t^{-1}+(\mu_q+\nu_q) \cdot\frac{\mu+\nu}{2}}\\
&\approx \frac{1}{4}\mu_q\mu+\underline{\frac{1}{4}\kh{2\mu_s\eps^2t^{-1}+\mu_q\nu+\nu_q\mu}},\\
q_{ii}&\approx \frac{-\mu_i+\nu_i}{2}\omega_i+\frac{-\mu+\nu}{2}\partial_i\omega_i\approx\frac{1}{2}\kh{ (-\mu_q+\nu_q) \cdot\frac{-\mu+\nu}{2}\omega_i-\mu_{\omega_j}\partial_i\omega_j}\omega_i-\frac{1}{2}\mu\partial_i\omega_i\\
&\approx \frac{1}{4}\omega_i^2\mu_q\mu+\underline{\frac{1}{4}\kh{-2\mu_{\omega_j}\omega_i\partial_i\omega_j-(\mu_q\nu+\mu\nu_q)\omega_i^2-2\mu\partial_i\omega_i}}.}
We then have
\fm{g^{\alpha\beta}\partial_\alpha\partial_\beta q&\approx \frac{1}{4}\mu_q\mu\cdot g^{\alpha\beta} \wh{\omega}_\alpha \wh{\omega}_\beta-\frac{1}{4}\kh{2\mu_s\eps^2t^{-1}+\mu_q\nu+\nu_q\mu}\\
&\quad+\sum_i\frac{1}{4}\kh{-2\mu_{\omega_j}\omega_i\partial_i\omega_j-(\mu_q\nu+\mu\nu_q)\omega_i^2-2\mu\partial_i\omega_i}\\
&\approx\frac{1}{4}\mu_q\mu\cdot g^{\alpha\beta} \wh{\omega}_\alpha \wh{\omega}_\beta-\frac{1}{2}\kh{\mu_s\eps^2t^{-1}+(\mu_q\nu+\nu_q\mu)+r^{-1}\mu}.}
Here we use $\sum_i\omega_i\partial_i\omega_j=0$ and $\sum_i\partial_i\omega_i=r^{-1}$.
Also notice that
\fm{g^{\alpha\beta} \wh{\omega}_\alpha \wh{\omega}_\beta&\approx m^{\alpha\beta} \wh{\omega}_\alpha \wh{\omega}_\beta+\eps^2 r^{-1}g^{\alpha\beta}_{0,JK} \wh{\omega}_\alpha \wh{\omega}_\beta U^JU^K+\eps^2 r^{-1}g^{\alpha\beta,\lambda}_{1,JK}\wh{\omega}_\alpha \wh{\omega}_\beta U^J  U^K_q\cdot\kh{-\frac{1}{2}\mu\wh{\omega}_\lambda}\\
&\quad+\eps^2 r^{-1}g^{\alpha\beta,\sigma\lambda}_{2,JK} \wh{\omega}_\alpha \wh{\omega}_\beta  U^J_q U_q^K\cdot\kh{\frac{1}{4}\mu^2\wh{\omega}_\sigma\wh{\omega}_\lambda}\\
&\approx  \eps^2 r^{-1}\kh{G_{0,JK}(\omega)  U^JU^K-\frac{1}{2}G_{1,JK}(\omega)  U^J  (\mu U^K_q ) +\frac{1}{4}G_{2,JK}(\omega) (\mu U^J_q )(\mu U_q^K)}.}

\subsubsection{Derivation of the geometric reduced system}

Since $\mu_t+\mu_r=q_{tt}-q_{rr}=\nu_t-\nu_r$, we have
\fm{\mu_q(q_t+q_r)+\eps^2t^{-1}\mu_s&\approx \nu_q(q_t-q_r)+\eps^2t^{-1}\nu_s.}
By keeping those terms of order $\eps^2t^{-1}$, we obtain
\eq{ \label{est:sec:reduced:system:munuq:minus}\mu_q\nu+\eps^2t^{-1}\mu_s\approx \nu_q\mu.}
By \eqref{est:sec:reduced:system:nu}, we have
\fm{\mu_s&\approx \eps^{-2}t(\nu_q\mu-\mu_q\nu)\approx \eps^{-2}t\mu^2\partial_q\kh{\frac{\nu}{\mu}}\\
&\approx \mu^2 \partial_q\kh{\frac{1}{4}G_{0,JK}(\omega) U^JU^K-\frac{1}{8}G_{1,JK}(\omega)\mu U^JU^K_q+\frac{1}{16}G_{2,JK}(\omega)\mu^2 U^J_qU^K_q}.}
This is the equation for $\mu_s$ in the reduced system. Because of the factor $\eps^{-2}t$ in the first row, we only consider terms of size $\eps^2t^{-1}$ in \eqref{est:sec:reduced:system:nu}.

Next, by \eqref{qwe:general}, \eqref{est:sec:reduced:system:quasi},  and \eqref{est:sec:reduced:system:semi}, we have
\eq{\label{est:pf:reduced:system:approxiamte:id:medium}&\eps^3 r^{-\frac{3}{2}}U_q^I\cdot  \frac{1}{4}\mu_q\mu\cdot \kh{G_{0,JK}(\omega)  U^JU^K-\frac{1}{2}G_{1,JK}(\omega)  U^J  (\mu U^K_q ) +\frac{1}{4}G_{2,JK}(\omega) (\mu U^J_q )(\mu U_q^K)}\\
&\quad-  \kh{\underline{\underline{\frac{1}{2}\eps^3 r^{-\frac{1}{2}}t^{-1}U_q^I\mu_s +\frac{1}{2}\eps r^{-\frac{1}{2}}U_q^I(\mu_q\nu+\nu_q\mu)}}+\underline{\frac{1}{2}\eps r^{-\frac{1}{2}}U_q^I\cdot r^{-1}\mu}} - \eps^3r^{-\frac{1}{2}}t^{-1}\mu U_{sq}^I+\underline{\frac{1}{2}\eps r^{-\frac{3}{2}}\mu U_q^I}\\
&\approx\eps^3 r^{-\frac{3}{2}} F^I_{0,JKL}(\omega)  U^JU^KU^L-\frac{1}{2}\eps^3 r^{-\frac{3}{2}} F^{I }_{1,JKL}(\omega) U^JU^K (\mu U^L_q) \\
&\quad+\frac{1}{4}\eps^3 r^{-\frac{3}{2}} F^{I}_{2,JKL}(\omega) U^J  (\mu U^K_q) (\mu U^L_q)  -\frac{1}{8}\eps^3 r^{-\frac{3}{2}}F^{I}_{3,JKL}(\omega) (\mu U^J_q)(\mu U^K_q)(\mu U^L_q).}
The underlined parts are canceled, while all other terms are of order $\eps^3r^{-\frac{3}{2}}$. By \eqref{est:sec:reduced:system:nu} and \eqref{est:sec:reduced:system:munuq:minus}, the doubly underlined part equals approximately
\fm{&\frac{1}{2}\eps^3  r^{-\frac{1}{2}}t^{-1} U_q^I\mu_s +\frac{1}{2}\eps r^{-\frac{1}{2}}U_q^I(\mu_q\nu+\nu_q\mu)\approx \eps^3  r^{-\frac{1}{2}}t^{-1} U_q^I\mu_s + \eps r^{-\frac{1}{2}}U_q^I \mu_q\nu\\
&\approx\eps^3  r^{-\frac{1}{2}}t^{-1} U_q^I\mu_s + \frac{1}{4}\eps^3 r^{-\frac{3}{2}}U_q^I \mu_q\kh{G_{0,JK}(\omega)\mu U^JU^K-\frac{1}{2}G_{1,JK}(\omega)\mu^2U^JU^K_q+\frac{1}{4}G_{2,JK}(\omega)\mu^3 U^J_qU^K_q}.}
We multiply this approximate identity by $-1$ and plug it into the left side of \eqref{est:pf:reduced:system:approxiamte:id:medium}. The first row gets canceled.
In summary, we obtain
\fm{\partial_s(\mu U_q^I)&\approx -  F^I_{0,JKL}(\omega)  U^JU^KU^L+\frac{1}{2}  F^{I }_{1,JKL}(\omega) U^JU^K (\mu U^L_q) \\
&\quad-\frac{1}{4}  F^{I}_{2,JKL}(\omega) U^J  (\mu U^K_q) (\mu U^L_q)  +\frac{1}{8} F^{I}_{3,JKL}(\omega) (\mu U^J_q)(\mu U^K_q)(\mu U^L_q). }
This is the equation for $(\mu U_q^I)_s$ in the reduced system.

\subsubsection{Definitions}
We make two definitions. They are analogous to \cite[Definitions 2.1 and 2.2]{MR4315017}.

\begin{defn}\rm
For a system \eqref{qwe:general} of quasilinear wave equations in $\R^{1+2}$ with at least cubic nonlinearities, the system
\eq{\label{sec:reduced:system:main}
\partial_s(\mu U_q^I)&=-  F^I_{0,JKL}(\omega)  U^JU^KU^L+\frac{1}{2}  F^{I }_{1,JKL}(\omega) U^JU^K (\mu U^L_q) \\
&\quad-\frac{1}{4}  F^{I}_{2,JKL}(\omega) U^J  (\mu U^K_q) (\mu U^L_q)  +\frac{1}{8} F^{I}_{3,JKL}(\omega) (\mu U^J_q)(\mu U^K_q)(\mu U^L_q),\\ 
\partial_s\mu&= \mu^2 \partial_q\kh{\frac{1}{4}G_{0,JK}(\omega) U^JU^K-\frac{1}{8}G_{1,JK}(\omega)\mu U^JU^K_q+\frac{1}{16}G_{2,JK}(\omega)\mu^2 U^J_qU^K_q}
} for $(\mu, U)$ in the coordinates $(s,q,\omega)$
is called a \emph{geometric reduced system}.

For convenience, we rewrite the geometric reduced system \eqref{sec:reduced:system:main} as
\eq{\label{sec:reduced:system:main:abbrv}\partial_s(\mu U_q^I)&=\mcl{S}^I_{\rm grs},\qquad \partial_s \mu =\mcl{Q}_{\rm grs}.}
Here $\mcl{S}^I_{\rm grs}$ is given explicitly in terms of $F^{I}_{*,JKL}$, $\mu$, $U^{I'}$, and their $q$-derivatives, while $\mcl{Q}_{\rm grs}$ is given explicitly in terms of $G_{*,JK}$, $\mu$, $U^{I'}$, and their $q$-derivatives.\end{defn}

\begin{defn}\label{def:geometric:weak:null}\rm
We say that a system \eqref{qwe:general} of quasilinear wave equations in $\R^{1+2}$ with at least cubic nonlinearities satisfies the \emph{geometric weak null condition} if, for any initial data with $\mu\restriction_{s=0}\equiv -2$ and $U\restriction_{s=0}\in C_c^\infty(\R\times\mathbb{S}^1)$, the corresponding geometric reduced system has a global solution $(\mu,U)$ for all $s\geq 0$ satisfying
\eq{\label{asu:vanishing:outside}\lim_{q\to\infty}U(s,q,\omega)=0,\qquad \forall (s,\omega)\in[0,\infty)\times\mathbb{S}^1,}
and
\eq{|\mu|+|\mu|^{-1}+\sum_{1\leq a+ b+c\leq N}|\partial_s^a\partial_q^b\partial_\omega^c \mu |+\sum_{a+b+c\leq N}|\partial_s^a\partial_q^b\partial_\omega^cU|\lesssim_N e^{C_Ns},\qquad \forall s\geq 0,(q,\omega)\in \R\times\mathbb{S}^1}
for all $N\geq 1$.
Here the implicit constant and $C_N$ may depend on the initial data.
\end{defn}

\begin{rmk}
\rm In Definition~\ref{def:geometric:weak:null}, we impose strong assumptions on the initial data. That is, we assume that $\mu\restriction_{s=0}\equiv -2$ and that $U\restriction_{s=0}\in C_c^\infty(\R\times\mathbb{S}^1)$. Such assumptions are imposed for convenience. Following the spirit in \cite{MR1994592} (and \cite{MR4315017}), one could state a variant of the geometric weak null condition by assuming that $(\mu,U)\restriction_{s=0}$ is a perturbation of $(-2,0)$ that decays sufficiently fast in $q$.

We do not claim that Definition~\ref{def:geometric:weak:null} is equivalent to this variant. In some cases, it can be weaker because we consider a smaller class of data for $(\mu,U)$. However, the advantage of Definition~\ref{def:geometric:weak:null} is that the family of data considered is explicit, so one can directly check whether a system satisfies the geometric weak null condition.
Moreover, in the present paper, we do not seek to prove global existence for small, smooth, and localized data directly from the geometric weak null condition. Instead, we view the condition as a structural indicator to predict global existence and asymptotic behavior. For this purpose, it is natural to work with a more restrictive class of initial data.

We also emphasize the importance of the assumption $|\mu|+|\mu|^{-1}\lesssim e^{Cs}$, which was not included in the corresponding definition in \cite{MR4315017}. In fact, this guarantees that the Lagrangian flow map introduced below remains a $C^1$ diffeomorphism and prevents finite-time shock formation through an intersection of characteristic curves.  
\end{rmk}

\begin{rmk}\label{rmk:on:vanishing:asu}\rm
The assumption \eqref{asu:vanishing:outside} allows us to uniquely recover $U$ from $U_q$. It is related to the finite speed of propagation for \eqref{qwe:general} with small, smooth, and localized initial data. Note that $q\approx r-t$, so for a fixed $t$, we have $r\to\infty$ as $q\to\infty$.

For a general system \eqref{qwe:general}, we may not have $\lim_{q\to\infty}U(s,q,\omega)=0$ and $\lim_{q\to-\infty}U(s,q,\omega)=0$ at the same time. As a result, we believe that it is not necessary to assume that $U\restriction_{s=0}$ has compact support in Definition~\ref{def:geometric:weak:null}. It is even possible that one of these assumptions leads to a global solution while the other leads to finite-time blowup. We choose \eqref{asu:vanishing:outside} because of the finite speed of propagation. See Section~\ref{exm:qwe:1.1}.
\end{rmk}

\subsection{Connection between geometric reduced systems and classical asymptotic equations}\label{sec:connection}

In Section~\ref{sec:intro:reduced:system}, we mentioned the classical asymptotic equations \eqref{hormander:cubic:asyeqn} for general scalar quasilinear wave equations of the form \eqref{qwe:general:intro}. In this subsection, we first extend the classical asymptotic equations to general systems of the form \eqref{qwe:general}; see \eqref{hormander:cubic:asyeqn:system} below. We then discuss the connection between \eqref{sec:reduced:system:main} and \eqref{hormander:cubic:asyeqn:system}.

\subsubsection{The classical asymptotic equations}
We return to the derivation in Section~\ref{sec:reduced:system:derivation} and now make the ansatz \eq{u=(u^I(t,x))&\approx\frac{\eps}{r^{\frac{1}{2}}}V=(\frac{\eps}{r^{\frac{1}{2}}}V^I(s,\rho,\omega))}
with
\eq{r=|x|,\qquad \omega=\frac{x}{|x|},\qquad s:=\eps^2\ln t.}
Here we set $\rho=r-t$, which is the optical function under the usual Minkowski metric. We make the analogues of assumptions (a)--(e) in Section~\ref{sec:reduced:system:derivation}, with $q$ replaced by $\rho$. In particular, $\rho_t-\rho_r=-2$ and $\rho_t+\rho_r=0$. Many computations still work after we replace $(U,q,\mu)$ with $(V,\rho,-2)$, as long as the eikonal equation \eqref{eik:general} is not used. For example, \eqref{est:sec:reduced:system:semi} still holds after making the replacements above. We also have \eqref{est:sec:reduced:system:quasi}, but now we need to add $\eps r^{-\frac{1}{2}}V_{\rho\rho}^Ig^{\alpha\beta}\rho_\alpha\rho_\beta$. By plugging these expressions into \eqref{qwe:general}, we have
\fm{&\eps r^{-\frac{1}{2}}V_\rho^I\cdot g^{\alpha\beta}\rho_{\alpha\beta}+2\eps^3r^{-\frac{1}{2}}t^{-1}V_{s\rho}^I-\eps r^{-\frac{3}{2}}V_\rho^I+\eps r^{-\frac{1}{2}}V_{\rho\rho}^Ig^{\alpha\beta}\rho_\alpha\rho_\beta\\
&\approx\eps^3 r^{-\frac{3}{2}} F^I_{0,JKL}(\omega)  V^JV^KV^L+\eps^3 r^{-\frac{3}{2}} F^{I }_{1,JKL}(\omega) V^JV^K V^L_\rho \\
&\quad+\eps^3 r^{-\frac{3}{2}} F^{I}_{2,JKL}(\omega) V^JV^K_\rho V^L_\rho+\eps^3 r^{-\frac{3}{2}}F^{I}_{3,JKL}(\omega) V^J_\rho V^K_\rho V^L_\rho.}
Moreover, we have $\rho_{\alpha}=\wh{\omega}_\alpha$, $\rho_{ij}=r^{-1}(\delta_{ij}-\omega_i\omega_j)$ if $i,j\in\{1,2\}$, and $\rho_{0\alpha}=\rho_{\alpha 0}=0$. Then, by only keeping terms of order $\eps^2r^{-1}$, we have
\fm{g^{\alpha\beta}\rho_{\alpha}\rho_{\beta}&\approx\kh{m^{\alpha\beta}+\eps^2r^{-1}g^{\alpha\beta}_{0,JK} V^JV^K+\eps^2r^{-1}g^{\alpha\beta,\lambda}_{1,JK} \wh{\omega}_\lambda V^J V_\rho^K+\eps^2r^{-1}g^{\alpha\beta,\sigma\lambda}_{2,JK} \wh{\omega}_\sigma \wh{\omega}_\lambda V_\rho^JV_\rho^K}\wh{\omega}_\alpha\wh{\omega}_\beta\\
&\approx \eps^2r^{-1}G_{0,JK}(\omega) V^JV^K+\eps^2r^{-1}G_{1,JK} (\omega)  V^J V_\rho^K+\eps^2r^{-1}G_{2,JK}(\omega) V_\rho^JV_\rho^K;}
by only keeping terms of order $r^{-1}$, we have
\fm{g^{\alpha\beta}\rho_{\alpha\beta}&=g^{ij}\rho_{ij}\approx  \delta^{ij}\cdot r^{-1}(\delta_{ij}-\omega_i\omega_j)=r^{-1}.}
In summary, we have
\fm{& 2 \eps^3r^{-\frac{1}{2}}t^{-1} V_{s\rho}^I+\eps^3r^{-\frac{3}{2}}\kh{G_{0,JK}(\omega) V^JV^K+ G_{1,JK} (\omega)  V^J V_\rho^K+ G_{2,JK}(\omega) V_\rho^JV_\rho^K}V_{\rho\rho}^I\\
&\approx\eps^3 r^{-\frac{3}{2}} F^I_{0,JKL}(\omega)  V^JV^KV^L+\eps^3 r^{-\frac{3}{2}} F^{I }_{1,JKL}(\omega) V^JV^KV^L_\rho  \\
&\quad+ \eps^3 r^{-\frac{3}{2}} F^{I}_{2,JKL}(\omega) V^JV^K_\rho V^L_\rho+\eps^3 r^{-\frac{3}{2}}F^{I}_{3,JKL}(\omega)V^J_\rho V^K_\rho V^L_\rho.}

We now make the following definition.
\begin{defn}\rm
For a system \eqref{qwe:general} of quasilinear wave equations in $\R^{1+2}$ with at least cubic nonlinearities, the system
\eq{\label{hormander:cubic:asyeqn:system} &2V_{s\rho}^I+\kh{G_{0,JK}(\omega) V^JV^K+ G_{1,JK} (\omega)  V^J V_\rho^K+ G_{2,JK}(\omega) V_\rho^JV_\rho^K}V_{\rho\rho}^I\\
&=F^I_{0,JKL}(\omega)V^JV^KV^L+F^{I }_{1,JKL}(\omega)V^JV^KV^L_\rho+F^{I}_{2,JKL}(\omega)V^JV^K_\rho V^L_\rho+F^{I}_{3,JKL}(\omega)V^J_\rho V^K_\rho V^L_\rho}
for $V$ in the coordinates $(s,\rho,\omega)$ is a system of \emph{classical asymptotic equations}.

For convenience, we rewrite \eqref{hormander:cubic:asyeqn:system} as
\eq{\label{hormander:cubic:asyeqn:system:abbrv}2\partial_sV^{I}_{\rho}+\mcl{Q}_{\rm cae}\partial_\rho V^I_{\rho}&=\mcl{S}^I_{\rm cae}.}
Here $\mcl{Q}_{\rm cae}$ and $\mcl{S}^I_{\rm cae}$ are given explicitly in terms of $G_{*,JK}$, $F_{*,JKL}$, $V^{I'}$, and their $\rho$-derivatives.
\end{defn}

\subsubsection{Derivation from \eqref{hormander:cubic:asyeqn:system} to \eqref{sec:reduced:system:main}}

Let $V=(V^I(s,\rho,\omega))$ be a solution to the classical asymptotic equations \eqref{hormander:cubic:asyeqn:system} for $(s,\rho,\omega)\in I\times\R\times\mathbb{S}^1$, where $I$ is an interval containing $0$. Recall an abbreviated version \eqref{hormander:cubic:asyeqn:system:abbrv} of \eqref{hormander:cubic:asyeqn:system} where we introduced $\mcl{S}^I_{\rm cae}$ and $\mcl{Q}_{\rm cae}$.
Let $Y=Y(s,q,\omega)$ be the Lagrangian flow map given by
\eq{\label{def:lagrangian:flow}(\partial_sY)(s,q,\omega)&=\frac{1}{2}\mcl{Q}_{\rm cae}(s,Y(s,q,\omega),\omega),\qquad Y(0,q,\omega)=q.}

Assume that for fixed $(s,\omega)\in I\times\mathbb{S}^1$, the map $q\mapsto Y(s,q,\omega)$ is a $C^1$ diffeomorphism. Since $Y(0,q,\omega)=q$, this at least requires that $Y_q>0$ everywhere. We use $\wt{Y}$ to denote the inverse flow map. This gives us two coordinates, $(s,q,\omega)$ and $(s,\rho,\omega)$, both in $I\times\R\times\mathbb{S}^1$, with the coordinate changes
\eq{\rho=Y(s,q,\omega),\qquad q=\wt{Y}(s,\rho,\omega).}
We call the coordinates $(s,\rho,\omega)$ the Eulerian ones, and $(s,q,\omega)$ the Lagrangian ones.
For $(s,q,\omega)\in I\times\R\times\mathbb{S}^1$, we set
\eq{U^I(s,q,\omega)&=V^I(s,Y(s,q,\omega),\omega),\\
\mu(s,q,\omega)&=\frac{-2}{Y_q(s,q,\omega)}.}
We now check that $(\mu,U)$ is a solution to the geometric reduced system \eqref{sec:reduced:system:main}.
In the following identities, $\rho=Y(s,q,\omega)$. It follows that
\fm{1&=\wt{Y}_\rho(s,\rho,\omega)Y_q(s,q,\omega)=\frac{-2\wt{Y}_\rho(s,\rho,\omega)}{\mu(s,q,\omega)},}
and
\fm{U^I_q(s,q,\omega)&=V^I_\rho(s,\rho,\omega)Y_q(s,q,\omega)=\frac{-2V^I_\rho(s,\rho,\omega)}{\mu(s,q,\omega)},\\
\partial_q(\mu U_{q}^I)(s,q,\omega)&=\partial_q\kh{-2V_\rho^I(s,Y(s,q,\omega),\omega)}\\
&=-2V_{\rho\rho}^I(s,\rho,\omega)Y_q(s,q,\omega)=\frac{4V_{\rho\rho}^I(s,\rho,\omega)}{\mu(s,q,\omega)}.}
For the evolution equations, we obtain
\fm{\mu_s(s,q,\omega)&=\frac{2Y_{sq}}{Y_q^2}=\frac{(\partial_\rho\mcl{Q}_{\rm cae})(s,\rho,\omega)}{Y_q(s,q,\omega)}=-\frac{1}{2}\mu(s,q,\omega)(\partial_\rho\mcl{Q}_{\rm cae})(s,\rho,\omega),\\
\partial_s(\mu U^I_q)(s,q,\omega)&=-2\partial_s\kh{V_\rho^I(s,Y(s,q,\omega),\omega)}\\
&=-2\kh{V^I_{s\rho}+\frac{1}{2}\mcl{Q}_{\rm cae}V^I_{\rho\rho}}(s,\rho,\omega)=-\mcl{S}_{\rm cae}^I(s,\rho,\omega).}
Finally, we replace
\eq{\label{est:grs:cae:equivalence:V:U:mu}(V^I,V_\rho^I,V_{\rho\rho}^I)(s,\rho,\omega)&=\kh{U^I,-\frac{1}{2}\mu U_q^I,\frac{1}{4}\mu\partial_q(\mu U_{q}^I)}(s,q,\omega)}
in the expressions for $\partial_\rho\mcl{Q}_{\rm cae}$ and $\mcl{S}^I_{\rm cae}$. One can check that 
\eq{\label{est:grs:cae:equivalence:V:U:mu:cor}-\frac{1}{2}\mu(s,q,\omega)(\partial_\rho\mcl{Q}_{\rm cae})(s,\rho,\omega)&=\mcl{Q}_{\rm grs}(s,q,\omega),\\
-\mcl{S}_{\rm cae}^I(s,\rho,\omega)&=\mcl{S}^I_{\rm grs}(s,q,\omega).}
For example, in the expansion of $ \partial_\rho\mcl{Q}_{\rm cae}$, we have a term 
\fm{\partial_\rho(G_{2,JK}(\omega)V_\rho^JV_\rho^K).}
By \eqref{est:grs:cae:equivalence:V:U:mu}, this term becomes
\fm{-\frac{1}{2}\mu\partial_q\kh{G_{2,JK}(\omega)(-\frac{1}{2}\mu U_q^J)(-\frac{1}{2}\mu U_q^K)}&=-\frac{1}{8}\mu\partial_q\kh{G_{2,JK}(\omega)( \mu U_q^J)(\mu U_q^K)}(s,q,\omega).}
Multiply this term by $-\frac{1}{2}\mu$, and we obtain the last term in the equation for $\mu_s$ in \eqref{sec:reduced:system:main}. Similarly for other terms. We skip the details here. 

We conclude that $(\mu,U)$ satisfies the geometric reduced system \eqref{sec:reduced:system:main} in the coordinates $(s,q,\omega)$. In summary, under the assumption that the Lagrangian flow map $Y$ is a $C^1$ diffeomorphism, we derive the geometric reduced system from the classical asymptotic equations.

\subsubsection{Derivation from \eqref{sec:reduced:system:main} to \eqref{hormander:cubic:asyeqn:system}}\label{sec:reduced:system:3.2.3}

Let $(\mu,U)(s,q,\omega)$ be a solution to the geometric reduced system \eqref{sec:reduced:system:main} for $(s,q,\omega)\in I\times\R\times\mathbb{S}^1$, where $I$ is an interval containing $0$. We further assume that $\mu\restriction_{s=0}\equiv -2$ and that $\mu<0$ everywhere. Recall that $\mu_s=\mcl{Q}_{\rm grs}=\mu^2\partial_q\wh{Q}$. Here $\wh{Q}$ is given explicitly in terms of $\mu$, $U^I$, $U^I_q$, and $G_{*,JK}$, so it is a known function of $(s,q,\omega)$. We now define $Y(s,q,\omega)$ by 
\eq{Y(s,q,\omega)&=q+2\int_0^s\wh{Q}(s',q,\omega)\,ds'.} Moreover, we have
\fm{Y_{sq}&=2\partial_q\wh{Q}=\frac{2\mcl{Q}_{\rm grs}}{\mu^2}=\frac{2\mu_s}{\mu^2}=\partial_s\kh{-\frac{2}{\mu}}.}
We thus have
\fm{Y_q(s,q,\omega)+\frac{2}{\mu(s,q,\omega)}&=Y_q(0,q,\omega)+\frac{2}{\mu(0,q,\omega)}=0.}

Again, assume that the flow map $q\mapsto Y(s,q,\omega)$ is a $C^1$ diffeomorphism\footnote{One can add an explicit assumption on $\mu$ as follows:
\fm{\int^{\infty}_q\frac{-2}{\mu(s,q',\omega)}\,dq'=\int_{-\infty}^q\frac{-2}{\mu(s,q',\omega)}\,dq'=\infty,\qquad\forall (s,q,\omega)\in I\times\R\times\mathbb{S}^1.} } for fixed $(s,\omega)\in I\times\mathbb{S}^1$. Define the inverse as $\wt{Y}(s,\rho,\omega)$ and obtain two coordinates $(s,q,\omega)$ and $(s,\rho,\omega)$ with the coordinate changes given by $Y$ and $\wt{Y}$. Now, we set
\eq{V^I(s,\rho,\omega)&=U^I(s,q,\omega).}
Then,
\fm{
1&=\wt{Y}_\rho(s,\rho,\omega)\cdot Y_q(s,q,\omega)=\frac{-2\wt{Y}_\rho(s,\rho,\omega)}{\mu(s,q,\omega)},\\
V^I_\rho(s,\rho,\omega)&=U_q^I(s,q,\omega) \wt{Y}_\rho(s,\rho,\omega)=\frac{(\mu U_q^I)(s,q,\omega)}{-2},\\
V^I_{\rho\rho}(s,\rho,\omega)&=\frac{\partial_q(\mu U_q^I)(s,q,\omega)}{-2}\cdot \wt{Y}_\rho(s,\rho,\omega)=\frac{(\mu\partial_q(\mu U_q^I))(s,q,\omega)}{4}.}
In other words, we have \eqref{est:grs:cae:equivalence:V:U:mu} and thus \eqref{est:grs:cae:equivalence:V:U:mu:cor}. By \eqref{sec:reduced:system:main}, in the coordinates $(s,q,\omega)$, we have
\fm{\partial_s(\mu U_q^I)(s,q,\omega)&=\mcl{S}^I_{\rm grs}(s,q,\omega)=-\mcl{S}^I_{\rm cae}(s,\rho,\omega),\\
\partial_s(\mu U_q^I)(s,q,\omega)&=\partial_s\kh{-2V_\rho^I(s,Y(s,q,\omega),\omega)}=-2V_{s\rho}^I(s,\rho,\omega)-2V^I_{\rho\rho}(s,\rho,\omega)Y_s(s,q,\omega)\\
&=-2V_{s\rho}^I(s,\rho,\omega)-4V^I_{\rho\rho}(s,\rho,\omega)\wh{Q}(s,q,\omega).
}
Finally, we notice that $4\wh{Q}(s,q,\omega)=\mcl{Q}_{\rm cae}(s,\rho,\omega)$ by \eqref{est:grs:cae:equivalence:V:U:mu}.
In summary,  $V$ satisfies the system of classical asymptotic equations \eqref{hormander:cubic:asyeqn:system}. Under the assumption that the Lagrangian flow map $Y$ is a $C^1$ diffeomorphism, we derive the classical asymptotic equations from the geometric reduced system.

\subsubsection{Summary}
We finish this subsection with two remarks. First, the transformations between these two types of asymptotic equations do not rely on the space dimensions, so the same derivations work between the geometric reduced system and H\"ormander's asymptotic equations in $\R^{1+3}$. 

Moreover, what we have proved indicates that the geometric reduced system is a \emph{Lagrangian} formulation of the classical asymptotic equations, which are \emph{Eulerian}. One obtains the same geometric and asymptotic information from both types of asymptotic equations. 

\subsection{Examples}\label{sec:reduced:exm}
We now consider two examples. For a semilinear system, we have $g^{\alpha\beta}=m^{\alpha\beta}$ and $\mu=-2$. The geometric reduced system then reduces to the classical asymptotic equations for semilinear wave equations. Thus, we only consider quasilinear examples here.

\subsubsection{The equation \eqref{qwe}}\label{exm:qwe:1.1}
We return to \eqref{qwe}. In this case, we have $M=1$, $F^I\equiv 0$, and $g^{*}_{j,*}=0$ for $j=1,2$ in \eqref{id:general:taylor:g}. The geometric reduced system for \eqref{qwe} is 
\eq{\label{sec:reduced:system:main:1.1}\partial_s(\mu U_q)&=0,\qquad
\partial_s\mu =\mu^2\partial_q\kh{\frac{1}{4}G(\omega)U^2}=\frac{1}{2}G(\omega)\mu^2UU_q,\qquad \lim_{q\to\infty}U(s,q,\omega)=0.}
We now study long-time existence for \eqref{sec:reduced:system:main:1.1}.

We choose the initial data
\eq{\label{sec:reduced:system:main:1.1:A}(\mu,U_q)\restriction_{s=0}=(-2,A).}
To check the geometric weak null condition, we assume that 
\fm{A\in C_c^\infty(\R\times\mathbb{S}^1),\qquad \int_{\R}A(q,\omega)\,dq=0,\ \forall\omega\in\mathbb{S}^1.}
The integral identity is used to guarantee  $U\restriction_{s=0}\in C_c^\infty$. It is related to Definition~\ref{def:geometric:weak:null}.

By the first equation in \eqref{sec:reduced:system:main:1.1}, we have $\mu U_q\equiv -2A$ for all $s$. It then follows that
\fm{\mu_s&=-G(\omega)A\mu U.}
We set \eq{\label{eqn:kappa}\kappa(s,q,\omega):=\exp\kh{G(\omega)A(q,\omega)\int_0^sU(s',q,\omega)\,ds'}.}
It follows that $\mu=-2\kappa^{-1}$, $U_q=A\kappa$, and
\eq{\label{eqn:kappa:s}\kappa_s(s,q,\omega)&=(GAU\kappa)(s,q,\omega)=-G(\omega)A(q,\omega)\kappa(s,q,\omega)\int_q^{\infty}A(q',\omega)\kappa(s,q',\omega)\,dq',\\
\kappa(0,q,\omega)&=1.}
This is an ODE for an unknown $s\mapsto \kappa(s,\cdot)$ valued in a Banach space (e.g.,\ in $C(K\times\mathbb{S}^1)$ where $K\times\mathbb{S}^1\supset\supp A$ and $K$ is a closed interval).
By a standard fixed-point argument, we obtain a local existence result for \eqref{eqn:kappa:s}. Moreover, if a solution $\kappa$ has a finite lifespan $s_*<\infty$, we have
\fm{\norm{\kappa}_{L^\infty([0,s_*)\times K\times\mathbb{S}^1)}=\norm{\kappa}_{L^\infty([0,s_*)\times \R\times\mathbb{S}^1)}=\infty.}
If $\norm{U}_{L^\infty([0,s_*)\times \R\times\mathbb{S}^1)}<\infty$, by the definition of $\kappa$ we also have
\fm{\norm{\kappa}_{L^\infty([0,s_*)\times \R\times\mathbb{S}^1)}\lesssim \exp\kh{C\norm{U}_{L^\infty([0,s_*)\times \R\times\mathbb{S}^1)}s_*}<\infty.}
If instead $\norm{U_q}_{L^\infty([0,s_*)\times \R\times\mathbb{S}^1)}<\infty$, then since $A$ has compact support and $U$ vanishes for $q\geq C$ for some $C>0$, we have $\norm{U}_{L^\infty([0,s_*)\times \R\times\mathbb{S}^1)}<\infty$. Therefore,  if a solution $\kappa$ has a finite lifespan $s_*<\infty$, we have
\eq{\label{eqn:kappa:s:blowup}\norm{\kappa}_{L^\infty([0,s_*)\times \R\times\mathbb{S}^1)}=\norm{U}_{L^\infty([0,s_*)\times \R\times\mathbb{S}^1)}=\norm{U_q}_{L^\infty([0,s_*)\times \R\times\mathbb{S}^1)}=\infty.}
A corollary of the blowup criterion \eqref{eqn:kappa:s:blowup} is that, if any of $\kappa,U,U_q$ has a finite $L^\infty([0,s_*)\times \R\times\mathbb{S}^1)$ norm for some $s_*<\infty$, so do the other two. In this case, the solution can be extended to $s>s_*$.  

Lemmas~\ref{lem:reduced:global:qwe} and~\ref{lem:reduced:blowup:qwe} below illustrate the following fact. \emph{The equation \eqref{qwe} satisfies the geometric weak null condition if and only if the sign condition \eqref{asu:sign:G} holds: $G(\omega)\geq 0$ for all $\omega\in\mathbb{S}^1$.} In particular, Lemma~\ref{lem:reduced:global:qwe} motivates our global existence result, Theorem~\ref{thm:main}, for \eqref{qwe} under the sign condition.

\begin{lem}\label{lem:reduced:global:qwe}
Suppose that $G(\omega)\geq 0$ for all $\omega\in\mathbb{S}^1$. Then, for all $A\in C_c^\infty(\R\times\mathbb{S}^1)$, the Cauchy problem \eqref{sec:reduced:system:main:1.1}--\eqref{sec:reduced:system:main:1.1:A} has a global solution for all $s\geq 0$. Moreover, the global solution $(\mu,U)$ satisfies the estimates in Definition~\ref{def:geometric:weak:null}, so \eqref{qwe} satisfies the geometric weak null condition.

Note that the assumption $\int_{\R}A(q,\omega)\,dq=0$ for all $\omega\in\mathbb{S}^1$ will not be used in the proof. It is only used to guarantee that our data $(\mu,U)\restriction_{s=0}$ matches the assumptions in Definition~\ref{def:geometric:weak:null}.  
\end{lem}
\begin{proof}
We first prove global existence. By \eqref{eqn:kappa:s:blowup} and the discussion above, we only need to prove 
\fm{\norm{U}_{L^\infty([0,s_*)\times \R\times\mathbb{S}^1)}<\infty,\qquad \forall s_*>0.}
We use a bootstrap argument. Suppose that a solution $(\mu,U)$ exists for $s\in[0,s_*)$ and
\fm{|U(s,q,\omega)|\leq 2M_0,\qquad \forall (s,q,\omega)\in [0,s_*)\times\mathbb{R}\times\mathbb{S}^1.}
Since $\norm{U(0)}_{L^\infty(\R\times\mathbb{S}^1)}\lesssim \norm{A}_{L^1_qL^\infty_\omega(\R\times\mathbb{S}^1)}$, we can choose a sufficiently large $M_0$ so that $|U(0,q,\omega)|\leq M_0$ for all $(q,\omega)$.
The local existence result for \eqref{eqn:kappa:s} implies that there exists $s_*>0$ satisfying the bootstrap assumptions. Moreover, we have
$|\ln\kappa|\leq |GA|\cdot 2M_0 s\lesssim s$
and thus $|\mu|+|\mu^{-1}|+|\kappa|+|\kappa^{-1}|+|U_q|\lesssim e^{Cs}$.

Next, we choose $R_0>0$ so that $A\equiv 0$ whenever $|q|\geq R_0$. Then,
\fm{|U(s,q,\omega)|&\lesssim 1_{q\leq R_0}\cdot \int_q^{R_0} |(A\kappa)(s,q',\omega)|\,dq'\lesssim  \int_{-R_0}^{R_0} \kappa(s,q',\omega)\,dq'.}
By \eqref{eqn:kappa:s}, we have
\fm{\partial_s\int_{-R_0}^{R_0} \kappa(s,q',\omega)\,dq'&=\int_{-R_0}^{R_0} (GAU\kappa)(s,q',\omega)\,dq'=\int_{-R_0}^{R_0} (G UU_q)(s,q',\omega)\,dq'\\
&=\frac{1}{2}\int_{-R_0}^{R_0} \partial_q(G U^2)(s,q',\omega)\,dq'=-\frac{1}{2}G(\omega)U^2(s,-R_0,\omega)\leq 0.}
Here we use $U(s,R_0,\omega)=0$ and $G(\omega)\geq 0$. It follows that
\fm{|U(s,q,\omega)|&\lesssim\int_{-R_0}^{R_0} \kappa(s,q',\omega)\,dq'\lesssim\int_{-R_0}^{R_0} \kappa(0,q',\omega)\,dq'\lesssim R_0.}
The implicit constants are independent of $M_0$, so we choose a sufficiently large $M_0$ to conclude that
\fm{|U(s,q,\omega)|\leq M_0,\qquad \forall (s,q,\omega)\in [0,s_*)\times\mathbb{R}\times\mathbb{S}^1.} This improves our bootstrap assumptions. 
The blowup criterion implies that the solution can be extended to $s\in[0,s_*+\epsilon)$ for some $\epsilon>0$ with \fm{\norm{U}_{L^\infty([0,s_*+\epsilon)\times \R\times\mathbb{S}^1)}<\infty.} By continuity, we have
\fm{|U(s,q,\omega)|\leq 2M_0,\qquad \forall (s,q,\omega)\in [0,s_*+\epsilon)\times\mathbb{R}\times\mathbb{S}^1.}
This finishes our bootstrap argument.

Now, we verify the pointwise estimates in Definition~\ref{def:geometric:weak:null}. We will prove that $|\partial^{a}_s\partial^b_q\partial^c_\omega(\mu,\kappa,U)|\lesssim e^{Cs}$ by inducting on $a+b+c$. The case $a+b+c=0$ has been proved above. Now, fix $(a,b,c)$ with $a+b+c\geq 1$. Suppose that we have proved this estimate for all $(a',b',c')$ with $a'+b'+c'<a+b+c$. If $a>0$, we have
\fm{\partial^{a}_s\partial^b_q\partial^c_\omega\kappa&=\partial^{a-1}_s\partial^b_q\partial^c_\omega\kh{GAU\kappa},\\
\partial^{a}_s\partial^b_q\partial^c_\omega U&=1_{q\leq R_0}\cdot\partial^{a-1}_s\partial^b_q\partial^c_\omega(-\int_{q}^{R_0}(GA^2U\kappa)(s,q',\omega)\,dq').}
By Leibniz's rule and the induction hypotheses, we conclude that $|\partial^{a}_s\partial^b_q\partial^c_\omega(\kappa,U)|\lesssim e^{Cs}$ for $a>0$. If $a=0$, then
\fm{&\partial_{q}^b\partial_\omega^c\kappa\\
&=1_{b>0}\cdot\partial_q^{b-1}\partial_\omega^c\kh{\int_0^s \partial_{q}(GAU)(s',q,\omega)\,ds' \cdot \kappa} +1_{b=0}\cdot  \partial_\omega^{c-1}\kh{\int_0^s \partial_{\omega}(GAU)(s',q,\omega)\,ds' \cdot \kappa}\\
&=1_{b>0}\cdot\partial_q^{b-1}\partial_\omega^c\kh{\int_0^s \partial_{q}(GAU)(s',q,\omega)\,ds' \cdot \kappa}\\
&\quad+1_{b=0}\cdot\kh{\kappa\cdot \int_0^s\partial_\omega^c(GAU)(s',q,\omega)\,ds'+  \sum_{1\leq c'\leq c-1} C\cdot \int_0^s \partial_{\omega}^{c-c'}(GAU)(s',q,\omega)\,ds' \cdot \partial_\omega^{c'}\kappa },}
\fm{ \partial_{q}^b\partial_\omega^c U 
&=1_{b>0}\cdot \partial_{q}^{b-1}\partial_\omega^c (A\kappa)+1_{b=0}\cdot 1_{q\leq R_0}\cdot \partial_\omega^{c }\kh{-\int_q^{R_0}(A\kappa)(s,q',\omega)\,dq'}.}
If $b>0$, we use $U_q=A\kappa$ to notice that all derivatives of $(\kappa,U)$ are of order $<b+c$. 
By Leibniz's rule and the induction hypotheses, we have $|\partial_{q}^b\partial_\omega^c(\kappa,U)|\lesssim e^{Cs}$ for $b>0$. If $b=0$, then the only term that involves $\partial^c_\omega U$ in the expression for $\partial^c_\omega\kappa$ is
\fm{\kappa \cdot \int_0^s(GA\partial_\omega^c U)(s',q,\omega)\,ds'&=GA\kappa \cdot \int_0^s(\partial_\omega^c U)(s',q,\omega)\,ds'.}
All the other terms involve derivatives of $(\kappa,U)$ of order $<c$. In summary,
\eq{\label{est:pf:kappa:domega:higher}  (\partial_\omega^c\kappa)(s,q,\omega) &=(GA\kappa)(s,q,\omega) \cdot \int_0^s(\partial_\omega^c U)(s',q,\omega)\,ds' +O(e^{Cs}).}
Similarly,
we have
\fm{ (\partial_\omega^cU)(s,q,\omega)&=-1_{q\leq R_0}\cdot \int_{q}^{R_0} (A\partial_\omega^c\kappa)(s,q',\omega) \,dq'+O(e^{Cs}\cdot 1_{q\leq R_0}).}
In the last step, we use Leibniz's rule and the induction hypotheses. As a result,
\fm{(\partial_\omega^cU)(s,q,\omega)&=-1_{q\leq R_0}\cdot  \int_0^s\int_{q}^{R_0}(A^2G\kappa)(s,q',\omega)\cdot(\partial_\omega^cU)(s',q',\omega)  \,dq'\,ds'+O(e^{Cs}\cdot 1_{q\leq R_0})\\
&=O\kh{  \int_0^s\norm{\partial_\omega^cU(s')}_{L^\infty( [-R_0,R_0]\times\mathbb{S}^1)}\cdot\int_{-R_0}^{R_0} \kappa (s,q',\omega)  \,dq'\,ds'+e^{Cs} }\\
&=O\kh{  \int_0^s\norm{\partial_\omega^cU(s')}_{L^\infty( [-R_0,R_0]\times\mathbb{S}^1)}\,ds'+e^{Cs} }.}
This estimate holds for all $s\geq 0$ and $(q,\omega)\in\R\times\mathbb{S}^1$. 
Here we use  $\int_{-R_0}^{R_0} \kappa (s,q',\omega)  \,dq'=O(1)$. Thus,
\fm{\norm{\partial_\omega^cU(s)}_{L^\infty( \R\times\mathbb{S}^1)}&\lesssim \int_0^s\norm{\partial_\omega^cU(s')}_{L^\infty(\R\times\mathbb{S}^1)}\,ds'+e^{Cs},\qquad\forall s\geq 0.}
By Gronwall's inequality, we conclude that $|\partial_\omega^cU|\lesssim e^{Cs}$. By \eqref{est:pf:kappa:domega:higher}, we have $|\partial_\omega^c\kappa|\lesssim e^{Cs}$. Finally, since $\mu\kappa=-2$, we have
\fm{0&=\partial_s^a\partial_q^b\partial_\omega^c(\mu\kappa)=\kappa\partial_s^a\partial_q^b\partial_\omega^c \mu+\sum_{a'\leq a,\ b'\leq b,\ c'\leq c\atop a'+b'+c'<a+b+c} C\cdot \partial_s^{a-a'}\partial_q^{b-b'}\partial_\omega^{c-c'}\kappa\cdot \partial_s^{a'}\partial_q^{b'}\partial_\omega^{c'} \mu.}
We apply the induction hypothesis to conclude that $|\partial_s^a\partial_q^b\partial_\omega^c \mu|\lesssim e^{Cs}$. This finishes our proof.
\end{proof}

\begin{lem}\label{lem:reduced:blowup:qwe}
Suppose that $G(\omega^0)< 0$ for some $\omega^0\in\mathbb{S}^1$. Then, there exists $A\in C^\infty_c(\R\times\mathbb{S}^1)$ with $\supp A\subset(-\infty,1]\times\mathbb{S}^1$ such that the Cauchy problem \eqref{sec:reduced:system:main:1.1}--\eqref{sec:reduced:system:main:1.1:A} blows up at some finite $s_*>0$, where the $L^\infty$ norms of $|U|,\kappa,|U_q|$ tend to infinity as $s\to s_*$; see \eqref{eqn:kappa:s:blowup}.
\end{lem}
\begin{proof}
We choose $A\in C^\infty_c(\R\times\mathbb{S}^1)$ such that $A_q(q,\omega^0)<0$ for $q\in(0,1)$ and $A\equiv 1$ for $q\in[-1,0]$.
We assume by contradiction that a solution $(\mu,U)$ with the data $(-2,A)$ exists for all $s\geq 0$. Here we have
\eq{\label{est:pf:sign:AAqUUq}A(q,\omega^0)>0,\ A_q(q,\omega^0)\leq 0,\ U_q(s,q,\omega^0)>0,\ U(s,q,\omega^0)<0,\ G(\omega^0)<0}
for all $q\in[-1,1)$ and $s\geq 0$.

By \eqref{sec:reduced:system:main:1.1}, we have
\fm{U_{sq}&=\mu^{-1}\cdot (-\mu_sU_q)=-\frac{1}{2}G \mu UU_q^2=GAUU_q=\partial_q\kh{\frac{1}{2}GAU^2}-\frac{1}{2}GA_qU^2}
for $s\geq 0$ and $q\in[-1,1)$. 
As a result,
\fm{\partial_q(U_s-\frac{1}{2}GAU^2)(s,q,\omega^0)=-\frac{1}{2}GA_qU^2}
and thus
\fm{U_s(s,q,\omega^0)-\frac{1}{2}(GAU^2)(s,q,\omega^0)&= U_s(s,1,\omega^0)-\frac{1}{2}(GAU^2)(s,1,\omega^0)-\frac{1}{2}\int_1^q GA_q U^2(s,q',\omega^0)\,dq'\\
&=\frac{1}{2}\int_{[q_+,1]} GA_qU^2(s,q',\omega^0)\,dq',}
for all $s\geq 0$ and $q\in[-1,1)$. Here $q_+:=\max\{q,0\}$, and we use $A_q\equiv 0$ for $q\in[-1,0]$.
Since $U<0$, we divide both sides by $U^2$ to obtain
\fm{\partial_s |U|^{-1}=-\partial_sU^{-1}=\frac{U_s}{U^2}= \frac{1}{2}G(\omega^0)A(q,\omega^0)+\frac{1}{2}\int_{q_+}^1 G(\omega^0)A_q(q',\omega^0) \cdot\frac{U^2(s,q',\omega^0)}{U^2(s,q,\omega^0)}\,dq'}
and
\fm{ \frac{1}{|U(s,q,\omega^0)|}&=\frac{1}{2}G(\omega^0)A(q,\omega^0)s+\frac{1}{|U(0,q,\omega^0)|}+\frac{1}{2}\int_{q_+}^1 |G(\omega^0)A_q(q',\omega^0)| \cdot\int_{0}^s\frac{U^2(s',q',\omega^0)}{U^2(s',q,\omega^0)}\,ds'\,dq'}
for all $s\geq 0$ and $q\in[-1,1)$. If we set $q=-1$, we have
\eq{\label{est:lem:blowup:pf} \frac{1}{|U(s,-1,\omega^0)|}&=\frac{1}{2}G(\omega^0)s+\frac{1}{|U(0,-1,\omega^0)|}+\frac{1}{2}\int_{0}^1 |G(\omega^0)A_q(q',\omega^0)| \cdot\int_{0}^s\frac{U^2(s',q',\omega^0)}{U^2(s',-1,\omega^0)}\,ds'\,dq'\\
&\leq \frac{1}{2}G(\omega^0)s+\frac{1}{|U(0,-1,\omega^0)|}+\frac{1}{2}\int_{0}^1 |G(\omega^0)A_q(q',\omega^0)|\,dq' \cdot\int_{0}^s\frac{U^2(s',0,\omega^0)}{U^2(s',-1,\omega^0)}\,ds'\\
&=\frac{1}{2}G(\omega^0)s+\frac{1}{|U(0,-1,\omega^0)|}-\frac{1}{2} G(\omega^0)  \cdot\int_{0}^s\frac{U^2(s',0,\omega^0)}{U^2(s',-1,\omega^0)}\,ds'.}In the second last step, we use \eqref{est:pf:sign:AAqUUq} to obtain $|U(s',q',\omega^0)|\leq |U(s',0,\omega^0)|$ for $q'\in[0,1)$.
In the last step, we use 
\fm{\int_0^1|G(\omega^0)A_q(q',\omega^0)|\,dq'=G(\omega^0)\int_0^1 A_q(q',\omega^0) \,dq'=G(\omega^0)(A(1,\omega^0)-A(0,\omega^0))=-G(\omega^0).}
Now, since $A=1$ and $U_q=\kappa$ for $q\in[-1,0]$,
\fm{&|U(s,-1,\omega^0)|-|U(s,0,\omega^0)|=\int_{-1}^0 \exp\kh{|G(\omega^0)|\int_0^s |U(s',q,\omega^0)|\,ds'}\,dq\\
&\geq\int_{-1}^0 \exp\kh{|G(\omega^0)|\int_0^s |U(s',0,\omega^0)|\,ds'}\,dq= \exp\kh{|G(\omega^0)|\int_0^s |U(s',0,\omega^0)|\,ds'},}
\fm{&0<|U(s,0,\omega^0)|=\int_{0}^1 A(q,\omega^0) \exp\kh{|G(\omega^0)|A(q,\omega^0)\int_0^s |U(s',q,\omega^0)|\,ds'}\,dq\\
&\leq\int_{0}^1 \exp\kh{|G(\omega^0)|\int_0^s |U(s',0,\omega^0)|\,ds'}\,dq=\exp\kh{|G(\omega^0)|\int_0^s |U(s',0,\omega^0)|\,ds'}.}
It follows that $|U(s,-1,\omega^0)|\geq 2| U(s, 0,\omega^0) |$. We return to  \eqref{est:lem:blowup:pf} to obtain
\fm{\frac{1}{|U(s,-1,\omega^0)|}&\leq \frac{1}{2}G(\omega^0)s+\frac{1}{|U(0,-1,\omega^0)|}-\frac{1}{2} G(\omega^0)  \cdot\frac{1}{4}s \leq -\frac{3}{8}|G(\omega^0)|s+\frac{1}{|U(0,-1,\omega^0)|}.}
Letting $s\to\infty$, we notice that the right side is negative, which is a contradiction.
\end{proof}

\begin{rmk}\rm
The argument in this proof will not lead to a contradiction if instead we assume that $G(\omega^0)\geq 0$.  We still have $U<0<U_q$, but now \eqref{est:lem:blowup:pf} is replaced by
\fm{\frac{1}{|U(s,-1,\omega^0)|}&=\frac{1}{2}G(\omega^0)s+\frac{1}{|U(0,-1,\omega^0)|}-\frac{1}{2}\int_{0}^1  |G(\omega^0)A_q(q',\omega^0)|  \cdot\int_{0}^s\frac{U^2(s',q',\omega^0)}{U^2(s',-1,\omega^0)}\,ds'\,dq'\\
&\geq \frac{1}{2}G(\omega^0)s+\frac{1}{|U(0,-1,\omega^0)|}-\frac{1}{2}\int_{0}^1  |G(\omega^0)A_q(q',\omega^0)|  \,dq'\cdot\int_{0}^s\frac{U^2(s',0,\omega^0)}{U^2(s',-1,\omega^0)}\,ds'\\
&= \frac{1}{2}G(\omega^0)s+\frac{1}{|U(0,-1,\omega^0)|}-\frac{1}{2}G(\omega^0)\cdot\int_{0}^s\frac{U^2(s',0,\omega^0)}{U^2(s',-1,\omega^0)}\,ds'.}
We cannot reach a contradiction even if the right side is negative.
\end{rmk}

\begin{rmk}
\rm This lemma implies the following fact. If $G(\omega^0)<0$ for some $\omega^0\in\mathbb{S}^1$, then the wave equation \eqref{qwe} violates the geometric weak null condition defined in Definition~\ref{def:geometric:weak:null}. The proof only relies on $A\restriction_{q\geq -1}$, so we can choose $A\restriction_{q\leq -1}$ to make $\int_\R A(q,\omega)\,dq=0$ for all $\omega\in\mathbb{S}^1$.
\end{rmk}

\begin{rmk}\rm
Lemmas~\ref{lem:reduced:global:qwe} and~\ref{lem:reduced:blowup:qwe} rely on \eqref{asu:vanishing:outside}. If we assume $(\mu,U)$ solves \eqref{sec:reduced:system:main:1.1} with the last limit replaced by \fm{\lim_{q\to-\infty}U(s,q,\omega)=0,}
then by setting $(\wt{\mu},\wt{U})(s,q,\omega)=(\mu,U)(s,-q,\omega)$, we have
\eq{\partial_s(\wt{\mu} \wt{U}_q)&=0,\qquad
\partial_s\wt{\mu} =-\wt{\mu}^2\partial_q\kh{\frac{1}{4}G(\omega)\wt{U}^2}=-\frac{1}{2}G(\omega)\wt{\mu}^2\wt{U}\wt{U}_q,\qquad \lim_{q\to\infty}\wt{U}(s,q,\omega)=0.}
Then we have global existence if $-G(\omega)\geq 0$ for all $\omega\in\mathbb{S}^1$ and have finite-time blowup if $-G(\omega^0)<0$ for some $\omega^0\in\mathbb{S}^1$. That is, we obtain opposite results! However, as discussed in Remark~\ref{rmk:on:vanishing:asu}, we believe that \eqref{asu:vanishing:outside} is the correct choice because of the finite speed of propagation.
\end{rmk}

\subsubsection{The equation in \cite{MR3357493}}\label{sec:reduced:LWY}

In \cite{MR3357493}, Li, Witt, and Yin proved finite-time blowup for small radial data for the following scalar quasilinear equation in $\R^{1+2}$
\eq{\label{eq:sec:reduced:LWY}
-u_{tt}+\divg(c(u)^2\nabla u)=0.}
They assumed that $c(0)\neq 0$ and that either $c'(0)\neq 0$ or $c'(0)=0$ and $c''(0)\neq 0$. Here we study the case where the nonlinearities are at least cubic, and we assume \eq{c(u)=1+c_2u^2+O(|u|^3),\qquad c_2\neq 0.} Our goal is to show that this equation violates the geometric weak null condition, which is consistent with the blowup results proved in their paper.

We rewrite \eqref{eq:sec:reduced:LWY} as
\fm{\Box u+(c(u)^2-1)\Delta u+2c(u)c'(u)|\nabla_xu|^2=0.}
This equation is similar to \eqref{qwe}, and the only difference is that it has a nonzero semilinear term.
By ignoring quartic and higher-order terms in the equation, we obtain
\fm{\Box u+2c_2u^2\Delta u=-4c_2u|\nabla_xu|^2+\text{error}.}
The only nonzero Taylor coefficients in \eqref{id:general:taylor:g} are $g^{11}_0=g^{22}_0=2c_2$, and the only nonzero Taylor coefficients in \eqref{id:general:taylor:F} are $f^{11}_2=f^{22}_2=-4c_2$. We have $G_0\equiv 2c_2$ and $F_2\equiv-4c_2$. 
The corresponding geometric reduced system is
\eq{\partial_s(\mu U_q)&=c_2U(\mu U_q)^2,\qquad \partial_s\mu= c_2 \mu^2UU_q.}
The equation for $\mu_s$ is of the same form as \eqref{sec:reduced:system:main:1.1}, and $G_0\equiv 2c_2$ even satisfies the sign condition \eqref{asu:sign:G} if $c_2>0$. We assume by contradiction that the reduced system has a global solution for $s\geq 0$.
Fix initial data $(\mu,U_q)\restriction_{s=0}=(-2,A)$ with $\lim_{q\to\pm\infty}U=0$. Here we require $\int_{\R}A(q,\omega)\,dq=0$. Since $U_{sq}=0$, we have $U_q=A$ and thus $U(s,q,\omega)=U(0,q,\omega)$ for all $s\geq 0$. Then,
\fm{\mu_s&=c_2 (AU)(0,q,\omega)\mu^2,\\
\partial_s\kh{\frac{1}{\mu}}&=-\frac{\mu_s}{\mu^2}=-c_2(AU)(0,q,\omega),\\
\frac{1}{\mu(s,q,\omega)}&=-\frac{1}{2}-s\cdot c_2(AU)(0,q,\omega).}
We now choose the initial data $U\restriction_{s=0}$ so that $-c_2(AU)(0,q,\omega)>0$ at some $(q,\omega)$. In fact, since $2AU=\partial_q(U^2)$ at $s=0$, for any nonzero initial data $U\restriction_{s=0}\in C_c^\infty(\R\times\mathbb{S}^1)$, we can find $(q^0,\omega^0)$ such that  $-c_2(AU)(0,q^0,\omega^0)>0$.
Thus, there exists $s>0$ such that $\mu(s,q^0,\omega^0)^{-1}=0$, which is a contradiction.

\section{Bootstrap assumptions}\label{sec:bootstrap:asu}
In the proof of Theorem~\ref{thm:main}, we use a bootstrap argument. In this section, we set up the argument used in the paper.

Fix $(u_0,u_1)\in C_c^\infty(\R^2)$ with $\supp u_0\cup \supp u_1\subset B_{\R^2}(0,1)$. We fix several constants
\eq{\label{bootstrap:constants}
\begin{array}{c}
    \displaystyle N\geq 13,\ \delta=\frac{1}{1000},\ \lambda=1-2\delta=\frac{499}{500},  \\[1em]
    \displaystyle M_0\geq 1,\ \kappa_1\geq 1,\ \kappa_2\in(0,1),\ \eps_0\in(0,1)
\end{array}}
whose values will be chosen later.
Here $N$ is an integer. For all $\eps\in(0,\eps_0)$, we assume that a solution $u$ to the Cauchy problem \eqref{qwe}--\eqref{init} exists for all $t\in[0,\Tboot]$ and $x\in\R^2$. Moreover, we assume that
\eq{\label{asu:bootstrap:ptb}|Z^{\leq N-4}u|\leq M_0\eps\lra{t}^{-\frac{1}{2}+\delta},\qquad \forall (t,x)\in[0,\Tboot]\times\R^2.}
The constants $\lambda,\kappa_1,\kappa_2$ do not appear in the bootstrap assumptions, but they will appear later in this paper.
None of the constants in \eqref{bootstrap:constants} depends on $\Tboot$. We choose $\eps_0$ at the end, so $\eps_0$ depends on all other constants in \eqref{bootstrap:constants}. 

We seek to prove   \eqref{asu:bootstrap:ptb} with $M_0$ replaced by $\frac{M_0}{2}$ for all $\eps\in(0,\eps_0)$ as long as $\eps_0$ is sufficiently small.
Later, we simply say that we choose $\eps\ll1$. In our proof, we will also prove that
\eq{\label{asu:bootstrap:energy}\norm{Z^{\leq N}\partial u(t)}_{L^2(\R^2)}\lesssim M_0\eps \lra{t}^{CM_0^2\eps^2},\qquad \forall t\in[0,\Tboot].}
If this is achieved, we obtain the global existence result by the local existence theory (for example, \cite[Theorem 6.4.11]{MR1466700}) and a standard bootstrap argument.

In the rest of the paper, when we write $\lesssim$, the implicit constant can depend on $N,\delta,\lambda,\kappa_2,\frac{\kappa_1}{M_0^2}$\footnote{We will set $\kappa_1=\wt{C}_1M_0^2$ where $\wt{C}_1>0$ is a constant independent of $M_0$ in Section~\ref{sec:energy:poincare}.}, the coefficients
$g^{\alpha\beta}$, and the fixed functions $(u_0,u_1)$. They
are independent of $M_0,\eps,\eps_0,\Tboot$.

By the standard local existence theory, we first prove that, for $\eps\ll1$, there exists a solution $u$ for $t\in[0,1]$ such that
\eq{\label{est:local:existence:t01:ene}\norm{Z^{\leq N}\partial u(t)}_{L^2(\R^2)}\lesssim \eps ,\qquad \forall t\in[0,1],}
\eq{\label{est:local:existence:t01:ptb}|Z^{\leq N-2}u|\lesssim\eps\lra{t}^{-\frac{1}{2}},\qquad \forall (t,x)\in[0,1]\times\R^2.}
Therefore, after choosing a sufficiently large $M_0$, the bootstrap assumption \eqref{asu:bootstrap:ptb} holds for $\Tboot=1$.
Note that \eqref{est:local:existence:t01:ptb} follows from \eqref{est:local:existence:t01:ene}, the Klainerman--Sobolev inequalities, and the fact that $u\equiv 0$ whenever $t\geq 0$ and $r\geq t+1$.

\section{An approximate optical function and lower-order pointwise estimates}\label{sec:approxoptical}

In this section, we seek to define an approximate optical function $\wt{q}$ everywhere in $[1,\Tboot]\times\R^2$.
We first define a localized version $q$ of $\wt{q}$ by solving $\wt{L}q=0$ in $\D$ where
\eq{\label{def:D}\D=\D_{\Tboot}^\lambda:=\{(t,x)\in\R^{1+2}:\ 1\leq t\leq \Tboot,\  t-t^\lambda+3\leq |x|\leq t+ 2\}.}
Recall that $\lambda=1-2\delta$ by  \eqref{bootstrap:constants}. We also set
\eq{\label{def:Dlowerboundary}\Dlower:=\{(t,x)\in\R^{1+2}:\ t\in[1,\Tboot],\ |x|=t-t^\lambda+3\}}
and
\eq{\label{def:D:interior}\Dint:=\{(t,x)\in\R^{1+2}:\ t\in[1,\Tboot],\ |x|<t-t^\lambda+3\}.}
See Section~\ref{sec:construction:q}.

In Section~\ref{sec:construction:q:lowerorder}, we derive pointwise estimates for $\partial^{\leq 2}(u, q)$ in $\D$. Here we emphasize that the sign condition \eqref{asu:sign:G} leads to 
\begin{enumerate}[(a)]
    \item a better lower bound for $h_{LL}$ in Lemma~\ref{lem:lowerbound:hLL}, 
    \item $\lra{r-t}\lesssim \lra{q}$ in Lemma~\ref{lem:small:r-t:q:D} (cf.\ $\lra{q}\lesssim \lra{r-t}\lra{t}^{CM_0^2\eps^2}$), 
    \item a better upper bound for $\muring^{-2}\uLunit\muring$ in \eqref{est:wtLmu/mu}, 
    \item a better lower bound for $\partial_tq_r$ in \eqref{est:dtqr/qr2}.
\end{enumerate}
These better bounds are crucial in our proof.

In Section~\ref{sec:construction:wtq}, we define $\wt{q}$ for all $t\in[1,\Tboot]$ and $x\in\R^2$. Here we glue $q$ and $r-t$ using a cutoff function; see \eqref{eq:def:wtq:construction}. We then prove pointwise estimates for $\wt{q}$ using those for $q$. The better bounds for $q$ mentioned above are transferred to $\wt{q}$.  In particular, we have \eqref{est:eik:wtq} which indicates that $\wt{q}$ is an approximate solution to the eikonal equation. This is why we call it an approximate optical function.

\subsection{Construction of $q$}\label{sec:construction:q}
Recall the definition of $\wt{L}$ in \eqref{eq:def:wtL}. 
We construct a function $q$ uniquely by solving a transport equation
\eq{\wt{L} q=0\ \text{ in }\D;\qquad q=r-t \ \text{ on }\Dlower.}
We have \fm{\wt{L}(t-r-t^\lambda+3)&=\frac{1}{4}h_{LL}\uLunit (t-r)-(1-\frac{1}{4}h_{LL})\lambda t^{\lambda-1} =-\frac{1}{4}(2-\lambda t^{\lambda-1})h_{LL}-\lambda t^{\lambda-1}.}
Since $t\geq 1$ and $\lambda<1$, we have $1<2-\lambda t^{\lambda-1}\leq 2$.
By \eqref{asu:bootstrap:ptb} and $h_{LL}=O(|u|^2)$, we have
\fm{\wt{L}(t-r-t^\lambda+3)&\leq -\lambda t^{\lambda-1}+C|u|^2\leq -\lambda t^{\lambda-1}+CM_0^2\eps^2\lra{t}^{-1+2\delta}<0.}
Here we use $\lambda-1=-2\delta>-1+2\delta$ and $\eps\ll1$. And since $\wt{L}^0>0$, the vector field $\wt{L}$ points strictly into $\D$ along $\Dlower$.
The coefficients of $\wt L$ have the same regularity as $u^2$, so the
solution $q$ is at least $C^2$ in $\D$. Since $u$ vanishes for $r\geq t+1$, we have $q=r-t$ for $r\geq t+1$. It is natural to extend $q$ to $\{t\in[1,\Tboot],\ |x|\geq t+2\}$ by setting $q=r-t$ there.

To estimate $q$, we fix $(t,x)\in\D$ and consider the integral curve $X(s)\in\D$ defined by
\eq{\label{eqn:integralcurve:wtL}\dot{X}^\alpha(s)&=\frac{\wt{L}^\alpha}{\wt{L}^0}(X(s)),\quad X(t)=(t,x).}
Here $\wt{L}^\alpha=L^\alpha+O(|u|^2)=L^\alpha+O(M_0^2\eps^2)$, so the quotient is well-defined.  
By \eqref{eq:def:wtL}, we notice that $(X^1,X^2)(s)$ is in the same direction as $x$:
\eq{\frac{(X^1,X^2)}{r\circ X}=\omega=\frac{x}{|x|},\qquad r(X)=\sqrt{(X^1)^2+(X^2)^2}.}
Moreover, for each $(t,x)\in\D$, there exists a unique $t_0\in[1,t]$ such that
\eq{\label{def:t0:X}X(t_0)\in\Dlower,\qquad X(s)\in\D\setminus \Dlower \text{ for all }s\in(t_0,t].}
Indeed, we have $X^0(s)=s$. For each $(t,x)\in\D$, the integral curve $X(s)$ reaches $\Dlower\cup\{t\in[1,\Tboot],\ |x|=t+2\}$ as $s$ decreases. If $r-t=2$, the integral curve $X(s)$ remains on $\{r=t+2\}$ until it reaches the intersection of $\Dlower$ and $\{t\in[1,\Tboot],\ |x|=t+2\}$. If $r-t<2$, then $r(X(s))-s< 2$ for all $s$. Here we use the uniqueness of the integral curves of $(\wt{L}^0)^{-1}\wt{L}$ passing through a given point. The coefficients of $(\wt{L}^0)^{-1}\wt{L}$ have the same regularity as $u^2$. 

Explicitly, in $\D$ we have
\eq{X(t_0)&=\kh{t_0,\kh{t_0-t_0^\lambda+3}\omega},}
\eq{q(t,x)&=q(X(t_0))=3-t_0^\lambda\in[3-t^\lambda,2].}
As a result, we have
\eq{\label{est:ratio:qtlambda:D}\lra{q}\lesssim \lra{t}^\lambda,\qquad \text{in }\D.}

We also define in $\D$
\eq{\muring :=\uLunit q=-q_t+q_r.}
Recall that we set $\mu=q_t-q_r=-\muring$ in Section~\ref{sec:reduced:system:derivation} when we derived the geometric reduced system. Here we find it more convenient to use a positive quantity $\muring$ when we prove pointwise estimates.
It follows that
\eq{\label{eqn:transport:mu}\wt{L}\muring&=\wt{L}\uLunit q=[\wt{L},\uLunit]q=[L+\frac{1}{4}h_{LL}\uLunit,\uLunit]q=-\frac{1}{4}(\uLunit h_{LL})\uLunit q=-\frac{1}{4}(\uLunit h)_{LL}\muring.} 
Here we use $\uLunit (L^\alpha)=0$ to obtain $(\uLunit h)_{LL}=(\uLunit h_{LL})$.
We also define
\eq{\muuLunit:=\muring^{-1}\uLunit.}
Later, we will check that $\muring> 0$, so this is a valid definition.

\subsection{Lower-order estimates}\label{sec:construction:q:lowerorder}

In this subsection, we derive pointwise estimates for $\partial^{\leq 2}(u, q)$ in $\D$.

From the sign condition \eqref{asu:sign:G}, we have the following key estimate.
\begin{lem}\label{lem:lowerbound:hLL}
There exists a large constant $C>1$, depending on the coefficients $g^{\alpha\beta}$, such that 
\eq{-C|u|^3\leq h_{LL}\leq C |u|^2.}
\end{lem}
\begin{proof}
Note that $m_{\alpha\beta}L^\beta=\wh{\omega}_\alpha=(-1,\omega)$. By \eqref{eq:g:taylor}, we have
\fm{h_{LL}&=g^{\alpha\beta}_0u^2\wh{\omega}_\alpha\wh{\omega}_\beta +O(|u|^3)=(G(\omega)+O(|u|))u^2.}
By \eqref{asu:bootstrap:ptb}, we have $h_{LL}\leq C|u|^2$.
Since $G(\omega)\geq 0$, we have $h_{LL}\geq -C|u|^3$.
\end{proof}

\begin{lem}\label{lem:small:r-t:q:D}
In $\D$, we have $q\leq r-t+CM_0^3\eps^{3}$ and $\lra{r-t}\lesssim \lra{q}$.
\end{lem}
\begin{proof}
In $\D$, we apply Lemma~\ref{lem:lowerbound:hLL} to obtain
\fm{\frac{\wt{L}(r-t-q)}{\wt{L}^0}&=\frac{\frac{1}{2}h_{LL}}{1-\frac{1}{4}h_{LL}}=\frac{2h_{LL}}{4+O(|u|^2)}\geq -C|u|^3\geq -CM_0^3\eps^3\lra{t}^{-\frac{3}{2}+3\delta}.}
Since $(r-t-q)(X(t_0))=0$, we have
\fm{r-t-q&=\int_{t_0}^t\frac{\wt{L}(r-t-q)}{\wt{L}^0}\circ X(\tau)\,d\tau\geq -CM_0^3\eps^3\int_{t_0}^t\lra{\tau}^{-\frac{3}{2}+3\delta}\,d\tau\geq -CM_0^3\eps^3\lra{t_0}^{-\frac{1}{2}+3\delta} \geq -CM_0^3\eps^3.} It also follows that in $\D$
\fm{\lra{q}\sim (3-q)\geq (t-r+3-CM_0^3\eps^3)\sim\lra{t-r}.}
\end{proof}

\begin{lem}
In $\D$, we have
\eq{\label{est:ptb:du}|\partial u|&\lesssim M_0\eps\lra{t}^{-\frac{1}{2}}\lra{q}^{-1+\frac{\delta}{\lambda}}\lesssim M_0\eps\lra{t}^{-\frac{1}{2}}\lra{r-t}^{-1+\frac{\delta}{\lambda}},}
\eq{\label{est:ptb:u}|u|&\lesssim M_0\eps\lra{t}^{-\frac{1}{2}}\lra{r-t}^{\frac{\delta}{\lambda}}\lesssim M_0\eps\lra{t}^{-\frac{1}{2}}\lra{q}^{\frac{\delta}{\lambda}}.}
By \eqref{eq:g:taylor}, we have $|h|\lesssim |u|^2\lesssim M_0^2\eps^2\lra{t}^{-1}\lra{r-t}^{\frac{2\delta}{\lambda}} $ and $|\partial h|\lesssim |u||\partial u|\lesssim M_0^2\eps^2\lra{t}^{-1}\lra{r-t}^{-1+\frac{2\delta}{\lambda}} $, so both \eqref{est:transport:r12uLphi} and \eqref{est:transport:r12uLuLphi} are applicable in subsequent arguments.

It also follows that
\eq{\label{est:diff:q:r-t} -CM_0^3\eps^3\leq r-t-q\leq C M_0^2\eps^2\lra{q}^{\frac{2\delta}{\lambda}}\cdot \ln\frac{C\lra{t}}{\lra{q}^{\frac{1}{\lambda}}}\leq C M_0^2\eps^2 \lra{t}^{2\delta}.}
\end{lem}
\begin{proof}
Since $|h|\lesssim |u|^2$, by setting $\phi=u$ in \eqref{est:transport:r12uLphi:raw} and applying \eqref{asu:bootstrap:ptb}, in $\D$ we have
\eq{\label{est:wtLr12uLu}\abs{\frac{1}{\wt{L}^0}\wt{L}\kh{r^{\frac{1}{2}}\uLunit u}}&\lesssim \lra{r+t}^{-\frac{1}{2}}\kh{\lra{r+t}^{-1}+\lra{r-t}^{-1}|u|^2}|Z^{\leq 2}u|\\
&\lesssim \lra{t}^{-\frac{1}{2}}\kh{\lra{t}^{-1}+M_0^2\eps^2\lra{t}^{-1+2\delta}}\cdot M_0\eps\lra{t}^{-\frac{1}{2}+\delta} \lesssim M_0\eps\lra{t}^{-2+3\delta}.}
On $\Dlower$, we have $\lra{r-t}\sim \lra{t}^{\lambda}$ and thus
\fm{r^{\frac{1}{2}}|\uLunit u|&\lesssim \lra{t+r}^{\frac{1}{2}}\lra{r-t}^{-1}|Zu|\lesssim \lra{t}^{\frac{1}{2}-\lambda}\cdot M_0\eps\lra{t}^{-\frac{1}{2}+\delta}\lesssim M_0\eps\lra{t}^{-\lambda+\delta} .}
It follows that
\fm{\abs{ r^{\frac{1}{2}}\uLunit u(t,x)}&\lesssim \abs{ r^{\frac{1}{2}}\uLunit u(X(t_0))}+\int_{t_0}^t\abs{\frac{1}{\wt{L}^0}\wt{L}\kh{r^{\frac{1}{2}}\uLunit u}}\circ X(s)\,ds\\
&\lesssim M_0\eps\lra{t_0}^{-\lambda+\delta}+M_0\eps\lra{t_0}^{-1+3\delta}\lesssim M_0\eps\lra{t_0}^{-\lambda+\delta}\lesssim M_0\eps\lra{q}^{-1+\frac{\delta}{\lambda}}.}
Here we notice that $\lambda= 1-2\delta$ implies that $-\lambda+\delta=-1+3\delta$. Then,
\fm{|\partial u|&\lesssim|\uLunit u|+\lra{r+t}^{-1}|Zu|\lesssim M_0\eps\lra{t}^{-\frac{1}{2}}\lra{q}^{-1+\frac{\delta}{\lambda}}+M_0\eps\lra{t}^{-\frac{3}{2}+\delta}\lesssim M_0\eps\lra{t}^{-\frac{1}{2}}\lra{q}^{-1+\frac{\delta}{\lambda}}.}
In the last step, we have $\lra{q}\lesssim \lra{t}^\lambda$ by \eqref{est:ratio:qtlambda:D} and thus $\lra{t}^{-1+\delta}\lesssim \lra{q}^{-\frac{1}{\lambda}+\frac{\delta}{\lambda}}\lesssim \lra{q}^{-1+\frac{\delta}{\lambda}}$ in $\D$.

To prove \eqref{est:ptb:u}, we use $\lra{r-t}\lesssim \lra{q}$ and integrate $\partial_ru$ from $r=t+1$ to inside.

Finally, we have
\fm{\frac{\wt{L}(r-t-q)}{\wt{L}^0}&=\frac{\frac{1}{2}h_{LL}}{1-\frac{1}{4}h_{LL}}\lesssim |u|^2\lesssim M_0^2\eps^2\lra{t}^{-1}\lra{q}^{\frac{2\delta}{\lambda}}.}
By integrating along $X(s)$, we conclude that
\fm{r-t-q&\leq CM_0^2\eps^2\ln\frac{1+t}{1+t_0}\cdot\lra{q}^{\frac{2\delta}{\lambda}}\leq C M_0^2\eps^2\lra{q}^{\frac{2\delta}{\lambda}}\cdot \ln\frac{C\lra{t}}{\lra{q}^{\frac{1}{\lambda}}}.}
We finish the proof by noticing that $\ln\frac{C\lra{t}}{\lra{q}^{\frac{1}{\lambda}}}\lesssim_\delta (\frac{C\lra{t}}{\lra{q}^{\frac{1}{\lambda}}})^{2\delta}$.
\end{proof}

\begin{lem}
In $\D$, we have $\muring>0$ and
\eq{\label{est:mu:D} 
\muring&\geq \kh{2-CM_0^2\eps^2\lra{q}^{-1+\frac{2\delta}{\lambda}}}\cdot \kh{\frac{C\lra{t}}{\lra{q}^{\frac{1}{\lambda}}}}^{-CM_0^2\eps^2\lra{q}^{-1+\frac{2\delta}{\lambda}}},\\
\muring &\leq \kh{2+CM_0^3\eps^3 \lra{q}^{-\frac{3}{2}+\frac{2\delta}{\lambda}}}\cdot \kh{\frac{C\lra{t}}{\lra{q}^{\frac{1}{\lambda}}}}^{CM_0^2\eps^2\lra{q}^{-1+\frac{2\delta}{\lambda}}},}
\eq{\label{est:q:r-t:ratio}\kh{\frac{C\lra{t}}{\lra{q}^{\frac{1}{\lambda}}}}^{-CM_0^2\eps^2}\lesssim \frac{\lra{r-t}}{\lra{q}}\lesssim 1.}
Since $\muring\neq 0$, the vector field $\muuLunit=\muring^{-1}\uLunit$ is well-defined in $\D$.
\end{lem}
\begin{proof}
Fix $(t,x)\in\D$ and let $X(s)$ be the integral curve defined by \eqref{eqn:integralcurve:wtL}. Choose $t_0 $ by \eqref{def:t0:X}. Set
\eq{\label{def:H(s)}H(s):=-\int_{s}^t\frac{(\uLunit h)_{LL}}{4-h_{LL}}(X(\tau))\,d\tau.}
By \eqref{eqn:transport:mu}, we have
\fm{\frac{d}{ds}(\muring(X(s)) e^{H(s)})&=e^{H(s)}\frac{\wt{L}\muring}{\wt{L}^0}(X(s))+\muring(X(s))e^{H(s)}\cdot\frac{(\uLunit h)_{LL}}{4-h_{LL}}(X(s))=0.}
Thus, we have
\fm{\muring(t,x)&=e^{H(t_0)}\muring(X(t_0)).}
Since $\uLunit h=O(|u||\partial u|)=O(M_0^2\eps^2\lra{t}^{-1}\lra{q}^{-1+\frac{2\delta}{\lambda}})$ by \eqref{est:ptb:u} and \eqref{est:ptb:du}, we have
\eq{\label{est:H(s)}|H(t_0)|&\lesssim\int_{t_0}^tM_0^2\eps^2\lra{\tau}^{-1}\lra{q}^{-1+\frac{2\delta}{\lambda}}\, d\tau\lesssim M_0^2\eps^2\lra{q}^{-1+\frac{2\delta}{\lambda}}\ln\frac{1+t}{1+t_0}\lesssim M_0^2\eps^2\lra{q}^{-1+\frac{2\delta}{\lambda}}\cdot   \ln\frac{C\lra{t}}{\lra{q}^{\frac{1}{\lambda}}} .}
On $\Dlower$, the vector field \fm{L-\lambda t^{\lambda-1}\partial_r=(1-\frac{1}{2}\lambda t^{\lambda-1})L-\frac{1}{2}\lambda t^{\lambda-1}\uLunit } is tangent to $\Dlower$. It follows that
\fm{(1-\frac{1}{2}\lambda t^{\lambda-1})Lq-\frac{1}{2}\lambda t^{\lambda-1}\uLunit q=(1-\frac{1}{2}\lambda t^{\lambda-1})L(r-t)-\frac{1}{2}\lambda t^{\lambda-1}\uLunit(r-t)=-\lambda t^{\lambda-1}.}
Since 
\fm{\wt{L}q=Lq+\frac{1}{4}h_{LL}\uLunit q=0,}
on $\Dlower$ we have
\eq{\label{est:uLq:dD} \uLunit q&=\frac{\lambda t^{\lambda-1}}{\frac{1}{4}h_{LL}(1-\frac{1}{2}\lambda t^{\lambda-1})+\frac{1}{2}\lambda t^{\lambda-1}}=\frac{2}{1+O(t^{1-\lambda}|u|^2)}\\
&=\frac{2}{1+O(M_0^2\eps^2t^{-\lambda}\lra{r-t}^{\frac{2\delta}{\lambda}})}=\frac{2}{1+O(M_0^2\eps^2 t^{-\lambda+2\delta})}=2+O(M_0^2\eps^2 \lra{t}^{-\lambda+2\delta}).}
Since $h_{LL}\geq -C|u|^3\geq -CM_0^3\eps^3\lra{t}^{-\frac{3}{2}+3\delta}$ and $\frac{1}{2}\lambda t^{\lambda-1}\leq \frac{1}{2}$, we have on $\Dlower$
\eq{\label{est:uLq:dD2}\uLunit q&\leq \frac{\lambda t^{\lambda-1}}{\frac{1}{2}\lambda t^{\lambda-1}-C|u|^3}\leq\frac{2}{1-Ct^{1-\lambda}\cdot M_0^3\eps^3 \lra{t}^{-\frac{3}{2}+3\delta}}\leq 2+C M_0^3\eps^3 \lra{t}^{-\frac{3}{2}\lambda+2\delta}.}
Here we use $\lambda=1-2\delta$.
Thus, 
\fm{-CM_0^2\eps^2\lra{t_0}^{-\lambda+2\delta}\leq \muring(X(t_0))-2\leq CM_0^3\eps^3 \lra{t_0}^{-\frac{3}{2}\lambda+2\delta}.}
Since $\lra{q}\sim\lra{t_0}^\lambda$, we have
\fm{
\muring&\leq \kh{2+CM_0^3\eps^3 \lra{t_0}^{-\frac{3}{2}\lambda+2\delta}}e^{ |H(t_0) |}\leq \kh{2+CM_0^3\eps^3 \lra{q}^{-\frac{3}{2}+\frac{2\delta}{\lambda}}}\cdot \kh{\frac{C\lra{t}}{\lra{q}^{\frac{1}{\lambda}}}}^{CM_0^2\eps^2\lra{q}^{-1+\frac{2\delta}{\lambda}}},\\
\muring&\geq \kh{2-CM_0^2\eps^2\lra{t_0}^{-\lambda+2\delta}}e^{-|H(t_0)|}\geq \kh{2-CM_0^2\eps^2\lra{q}^{-1+\frac{2\delta}{\lambda}}}\cdot \kh{\frac{C\lra{t}}{\lra{q}^{\frac{1}{\lambda}}}}^{-CM_0^2\eps^2\lra{q}^{-1+\frac{2\delta}{\lambda}}}.}
We thus obtain \eqref{est:mu:D}. Here we also have
\eq{\label{est:mu-2}|\muring-2|&\lesssim|\muring -2e^{H(t_0)}|+|e^{H(t_0)}-1|\lesssim  M_0^2\eps^2\lra{t_0}^{-\lambda+2\delta}e^{|H(t_0)|}+|H(t_0)|e^{|H(t_0)|}\\
&\lesssim \kh{M_0^2\eps^2\lra{q}^{-1+\frac{2\delta}{\lambda}}+M_0^2\eps^2\lra{q}^{-1+\frac{2\delta}{\lambda}}\cdot   \ln\frac{C\lra{t}}{\lra{q}^{\frac{1}{\lambda}}}}\cdot \kh{\frac{C\lra{t}}{\lra{q}^{\frac{1}{\lambda}}}}^{CM_0^2\eps^2\lra{q}^{-1+\frac{2\delta}{\lambda}}}\\
&\lesssim  M_0^2\eps^2\lra{q}^{-1+\frac{2\delta}{\lambda}} \cdot\ln\frac{C\lra{t}}{\lra{q}^{\frac{1}{\lambda}}}\cdot \kh{\frac{C\lra{t}}{\lra{q}^{\frac{1}{\lambda}}}}^{CM_0^2\eps^2\lra{q}^{-1+\frac{2\delta}{\lambda}}}.}
This estimate will be useful later in the paper.

To prove \eqref{est:q:r-t:ratio}, we notice that
\fm{\frac{\wt{L}\ln(3-r+t)}{\wt{L}^0}&=\frac{-\frac{1}{2}h_{LL}}{(1-\frac{1}{4}h_{LL})(3-r+t)}\geq  -\frac{C|u|^2}{\lra{r-t}}\geq - CM_0^2\eps^2\lra{t}^{-1}\lra{r-t}^{-1+\frac{2\delta}{\lambda}}.}
Integrate along $X(s)$, and notice that $\ln(3-r+t)=\ln(t^\lambda)$ on $\Dlower$. We conclude that
\fm{\ln(3-r+t)&\geq \ln(t_0^\lambda)-\int_{t_0}^tCM_0^2\eps^2\lra{\tau}^{-1}\,d\tau \geq \ln(3-q)-CM_0^2\eps^2\ln\frac{1+t}{1+t_0}\\
&\geq\ln(3-q) -CM_0^2\eps^2\ln\frac{C\lra{t}}{\lra{q}^{\frac{1}{\lambda}}},}
\fm{\lra{r-t}\gtrsim (3-r+t)\geq (3-q) \kh{\frac{C\lra{t}}{\lra{q}^{\frac{1}{\lambda}}}}^{-CM_0^2\eps^2}\gtrsim \lra{q}\kh{\frac{C\lra{t}}{\lra{q}^{\frac{1}{\lambda}}}}^{-CM_0^2\eps^2}.}
We also recall that $\lra{r-t}\lesssim \lra{q}$.
\end{proof}

\begin{lem}
In $\D$, we have
\eq{\label{est:ptb:ddu}|\partial^2u|\lesssim \muring\cdot M_0\eps\lra{t}^{-\frac{1}{2}} \lra{q}^{-2+\frac{\delta}{\lambda}}.}
\end{lem}
\begin{proof}
By applying \eqref{est:transport:r12uLuLphi} with $\phi=u$ and \eqref{eqn:transport:mu}, we have
\fm{\wt{L}\kh{r^{\frac{1}{2}}\muuLunit\uLunit u}&=-\muring^{-2}\wt{L}\muring\cdot\kh{r^{\frac{1}{2}}\uLunit\uLunit u}+\muring^{-1}\wt{L}\kh{r^{\frac{1}{2}}\uLunit\uLunit u}\\
&=\frac{1}{4}(\uLunit h)_{LL}\muring^{-1}\kh{r^{\frac{1}{2}}\uLunit\uLunit u}-\frac{1}{4}(\uLunit h)_{LL}\muring^{-1}\kh{r^{\frac{1}{2}}\uLunit\uLunit u}+O\kh{ \muring^{-1}\lra{r+t}^{-\frac{3}{2}}|Z^{1\leq \cdot\leq 2}\partial u|}\\
&=O\kh{ \muring^{-1}\lra{t}^{-\frac{3}{2}}\lra{r-t}^{-1}|Z^{\leq 3}u|}.}
By \eqref{asu:bootstrap:ptb}, \eqref{est:mu:D}, and \eqref{est:q:r-t:ratio}, we have
\fm{\abs{\wt{L}\kh{r^{\frac{1}{2}}\muuLunit\uLunit u}}&\lesssim \kh{\frac{C\lra{t}}{\lra{q}^{\frac{1}{\lambda}}}}^{CM_0^2\eps^2\lra{q}^{-1+\frac{2\delta}{\lambda}}}\cdot \lra{t}^{-\frac{3}{2}}\cdot\lra{q}^{-1}\cdot\kh{\frac{C\lra{t}}{\lra{q}^{\frac{1}{\lambda}}}}^{CM_0^2\eps^2} \cdot M_0\eps\lra{t}^{-\frac{1}{2}+\delta}\\
&\lesssim M_0\eps\lra{t}^{-2+\delta+CM_0^2\eps^2}\lra{q}^{-1}.}
Divide both sides by $\wt{L}^0$ and integrate along $X(s)$. By \eqref{est:uLq:dD} and \eqref{est:uLq:dD2}, we have
\fm{\abs{r^{\frac{1}{2}}\muuLunit\uLunit u}(X(t_0))&\lesssim \abs{r^{\frac{1}{2}}\uLunit\uLunit u}(X(t_0))\lesssim \lra{t_0}^{\frac{1}{2}}\lra{t_0^\lambda}^{-2}|(Z^{\leq 2} u)(X(t_0))|\\
&\lesssim M_0\eps \lra{t_0}^{\delta-2\lambda}\lesssim M_0\eps \lra{q}^{-2+\frac{\delta}{\lambda}}.}
Thus, by integrating along $X(s)$, we obtain
\eq{\abs{r^{\frac{1}{2}}\muuLunit\uLunit u}&\lesssim \int_{t_0}^tM_0\eps \lra{\tau}^{-2+\delta+CM_0^2\eps^2}\lra{q}^{-1}\,d\tau+M_0\eps\lra{q}^{-2+\frac{\delta}{\lambda}}\\
&\lesssim M_0\eps\lra{t_0}^{-1+\delta+CM_0^2\eps^2} \lra{q}^{-1}+M_0\eps\lra{q}^{-2+\frac{\delta}{\lambda}}\lesssim M_0\eps\lra{q}^{-2+\frac{\delta}{\lambda}}.}
In fact, since $\lambda= 1-2\delta$, we have
\eq{\label{est:t0:-1+delta}\lra{t_0}^{-1+\delta+CM_0^2\eps^2}\lra{q}^{1-\frac{\delta}{\lambda}}\sim \lra{t_0}^{-1+\delta+CM_0^2\eps^2}\lra{t_0}^{\lambda- \delta}\sim \lra{t_0}^{-2\delta+CM_0^2\eps^2}\lesssim 1.}
By choosing $\eps\ll1$,  we have $-2\delta+CM_0^2\eps^2<0$. 

In general, to estimate $\partial_{\alpha_0}\partial_{\beta_0} u$, we define a matrix $(a^{\alpha\beta}):=(\frac{1}{2}\delta_{\alpha_0\beta_0}^{\alpha\beta}+\frac{1}{2}\delta_{\beta_0\alpha_0}^{\alpha\beta})_{\alpha,\beta=0,1,2}$. Define $a_{XY}$ by \eqref{def:g_XY} and define $a^{UV}$ by \eqref{eq:null-frame-coefficients}, with $g$ replaced by $a$. It follows that $a^{UV}=O(1)$ and thus
\fm{|\partial_{\alpha_0}\partial_{\beta_0} u|&=\abs{\sum_{U,V\in\{L,\uLunit,E\}}a^{UV}U^\alpha V^\beta \partial_\alpha\partial_\beta u}\lesssim |\uLunit\uLunit u |+\sum_{U,V\in\{L,\uLunit,E\}\atop (U,V)\neq (\uLunit,\uLunit)}\abs{U^\alpha V^\beta \partial_\alpha\partial_\beta u}.}
In the last sum, if $U\neq \uLunit$, we have
\fm{\abs{U^\alpha V^\beta \partial_\alpha\partial_\beta u}&\lesssim |U\partial u|\lesssim\lra{r+t}^{-1}|Z\partial u|\lesssim M_0\eps\lra{t}^{-\frac{3}{2}+\delta}\lra{r-t}^{-1}\\
&\lesssim M_0\eps\lra{t}^{-\frac{3}{2}+\delta}\lra{q}^{-1}\kh{\frac{C\lra{t}}{\lra{q}^{\frac{1}{\lambda}}}}^{CM_0^2\eps^2}\lesssim M_0\eps\lra{t}^{-\frac{3}{2}+\delta+CM_0^2\eps^2}\lra{q}^{-1}.}
Here we use \eqref{est:q:r-t:ratio}.
Similarly for $V\neq \uLunit$. In summary,
\fm{|\partial^2u|&\lesssim \muring  |\muuLunit\uLunit u|+M_0\eps\lra{t}^{-\frac{3}{2}+\delta+CM_0^2\eps^2}\lra{q}^{-1 }\\
&\lesssim \muring\cdot M_0\eps\lra{t}^{-\frac{1}{2}}\lra{q}^{-2+\frac{\delta}{\lambda}}+ M_0\eps\lra{t}^{-\frac{3}{2}+\delta+CM_0^2\eps^2}\lra{q}^{-1 }\\
&\lesssim \muring\cdot M_0\eps\lra{t}^{-\frac{1}{2}} \lra{q}^{-2+\frac{\delta}{\lambda}}.}
To obtain the last step, we use $t_0\leq t$ and \eqref{est:t0:-1+delta}.
\end{proof}

\begin{lem}
In $\D$, we have
\eq{ \label{est:wtLmu/mu} -C\lra{q}^{-2+\frac{2\delta}{\lambda}}\lra{t}^{CM_0^2\eps^2}\leq \frac{\muuLunit \muring}{\muring}=\frac{\uLunit \muring}{\muring^2}\leq  CM_0^2\eps^2\lra{q}^{-2+\frac{2\delta}{\lambda}} \ln\frac{C\lra{t}}{\lra{q}^{\frac{1}{\lambda}}}.}
\end{lem}
\begin{proof}
By \eqref{eqn:transport:mu} and since $\muring>0$, we have
\fm{\wt{L}\ln( \muring)&=\frac{\wt{L} \muring}{\muring}=-\frac{1}{4}(\uLunit h)_{LL}.}
We have \eq{\ [\muuLunit,\wt{L}]&=[\muring^{-1}\uLunit,L+\frac{1}{4}h_{LL}\uLunit]=\frac{1}{4}\muuLunit(h_{LL})\uLunit-\wt{L}\muring^{-1}\cdot \uLunit=\kh{\frac{1}{4}(\muuLunit h)_{LL}+\muring^{-2}\wt{L}\muring}\uLunit=0.}
Note that $\muuLunit(h_{LL})=(\muuLunit h)_{LL}$ since $\muuLunit(L^\alpha)=0$.
Thus,
\fm{\wt{L}\muuLunit \ln(\muring)&=-\frac{1}{4}\muuLunit(\uLunit h)_{LL}=-\frac{1}{4}(\muuLunit \uLunit h)_{LL}.}
By \eqref{eq:g:taylor}, we have $(g^{\alpha\beta})''(u)=2g^{\alpha\beta}_0+O(|u|)$ and thus
\fm{\muuLunit \uLunit h^{\alpha\beta}&=\muuLunit ((g^{\alpha\beta})'(u)\cdot \uLunit u)=(g^{\alpha\beta})'(u)\cdot \muuLunit\uLunit u+(g^{\alpha\beta})''(u)\cdot \uLunit u\muuLunit u\\
&=2g^{\alpha\beta}_0\muring^{-1}(\uLunit u)^2+O(|u||\muuLunit\uLunit u|+|u||\uLunit u||\muuLunit u|)\\
&=2g^{\alpha\beta}_0\muring^{-1}(\uLunit u)^2+O(M_0^2\eps^2\lra{t}^{-1}\lra{q}^{-2+\frac{2\delta}{\lambda}}+M_0^3\eps^3\lra{t}^{-\frac{3}{2}}\lra{q}^{-2+\frac{3\delta}{\lambda}}\muring^{-1})\\
&=2g^{\alpha\beta}_0\muring^{-1}(\uLunit u)^2+O(M_0^2\eps^2\lra{t}^{-1}\lra{q}^{-2+\frac{2\delta}{\lambda}}).}
In the last step, we use $\muring^{-1}\lesssim \lra{t}^{CM_0^2\eps^2}$. By $m_{\alpha\beta}L^\beta=\wh{\omega}_\alpha$, we have
\fm{-\frac{1}{4}(\muuLunit \uLunit h)_{LL}&=-\frac{1}{2}g^{\alpha\beta}_0\wh{\omega}_\alpha\wh{\omega}_\beta\muring^{-1}(\uLunit u)^2+O(M_0^2\eps^2\lra{t}^{-1}\lra{q}^{-2+\frac{2\delta}{\lambda}})\\
&=-\frac{1}{2}G(\omega)\muring^{-1}(\uLunit u)^2+O(M_0^2\eps^2\lra{t}^{-1}\lra{q}^{-2+\frac{2\delta}{\lambda}}).}
By \eqref{est:ptb:du}, we have
\fm{-\frac{1}{4}(\muuLunit \uLunit h)_{LL}&\geq -\muring^{-1} \cdot CM_0^2\eps^2\lra{t}^{-1}\lra{q}^{-2+\frac{2\delta}{\lambda}}-CM_0^2\eps^2\lra{t}^{-1}\lra{q}^{-2+\frac{2\delta}{\lambda}}\\
&\geq - CM_0^2\eps^2\lra{t}^{-1+CM_0^2\eps^2}\lra{q}^{-2+\frac{2\delta}{\lambda}}.}
Since $G(\omega)\geq 0$ and $\muring^{-1}>0$, we have
\fm{-\frac{1}{4}(\muuLunit \uLunit h)_{LL}\leq  CM_0^2\eps^2\lra{t}^{-1}\lra{q}^{-2+\frac{2\delta}{\lambda}}.}
By integrating along $X(s)$, we have
\fm{ (\muuLunit\ln\muring)(t,x)-(\muuLunit\ln\muring)(X(t_0))&\geq -\int_{t_0}^tCM_0^2\eps^2\lra{\tau}^{-1+CM_0^2\eps^2}\lra{q}^{-2+\frac{2\delta}{\lambda}}\,d\tau\geq -C \lra{q}^{-2+\frac{2\delta}{\lambda}} \lra{t}^{CM_0^2\eps^2},
\\(\muuLunit \ln\muring)(t,x)-(\muuLunit \ln\muring)(X(t_0))&\leq \int_{t_0}^tCM_0^2\eps^2\lra{\tau}^{-1}\lra{q}^{-2+\frac{2\delta}{\lambda}}\,d\tau\leq CM_0^2\eps^2\lra{q}^{-2+\frac{2\delta}{\lambda}}\ln\frac{C\lra{t}}{\lra{q}^{\frac{1}{\lambda}}}. }
Here we use $\int_0^t\lra{\tau}^{-1+CM_0^2\eps^2}\lesssim M_0^{-2}\eps^{-2}\lra{t}^{CM_0^2\eps^2}$.

On $\Dlower$, by \eqref{est:uLq:dD} and \eqref{est:uLq:dD2}, we have
\fm{\muring= \uLunit q= \frac{8\lambda t^{\lambda-1}}{ h_{LL}(2- \lambda t^{\lambda-1})+4\lambda t^{\lambda-1}}=\frac{ 8\lambda }{ h_{LL}(2t^{1-\lambda}- \lambda )+4\lambda }=2+O(M_0^2\eps^2) .}
Since \fm{L-\lambda t^{\lambda-1}\partial_r=\partial_t+(1-\lambda t^{\lambda-1})\partial_r=(1-\frac{1}{2}\lambda t^{\lambda-1})L-\frac{1}{2}\lambda t^{\lambda-1}\uLunit } is tangent to $\Dlower$, we have
\eq{\label{est:wtLmu/mu:pf1}&\kh{(1-\frac{1}{2}\lambda t^{\lambda-1})L-\frac{1}{2}\lambda t^{\lambda-1}\uLunit}\muring\\
&=-\frac{8\lambda }{( h_{LL}(2t^{1-\lambda}- \lambda )+4\lambda)^2 }\cdot\kh{ (1-\frac{1}{2}\lambda t^{\lambda-1})L-\frac{1}{2}\lambda t^{\lambda-1}\uLunit }(h_{LL}(2t^{1-\lambda}-\lambda))\\
&=-\frac{\muring^2}{8\lambda}\cdot\kh{h_{LL}\cdot 2(1-\lambda)t^{-\lambda}+ (2t^{1-\lambda}-\lambda)\cdot \kh{(1-\frac{1}{2}\lambda t^{\lambda-1})Lh_{LL}-\frac{1}{2}\lambda t^{\lambda-1}\uLunit h_{LL}}}\\
&=O\kh{\muring^2\cdot\kh{t^{-\lambda}|h_{LL}|+t^{1-\lambda}|(Lh)_{LL}|+|(\uLunit h)_{LL}|}}=O\kh{\muring^2\cdot \kh{t^{-\lambda}|u|^2+t^{1-\lambda}|u||Lu|+|u||\uLunit u|}}\\
&= O(  M_0^2\eps^2\lra{t}^{-1}\lra{q}^{-1+\frac{2\delta}{\lambda}}).}
To obtain the last step, we use $\muring\sim 1$,
\fm{t^{-\lambda}|u|^2&\lesssim M_0^2\eps^2\lra{t}^{-1-\lambda}\lra{q}^{\frac{2\delta}{\lambda}},\\
t^{1-\lambda}|u||Lu|&\lesssim t^{-\lambda}|u||Zu|\lesssim M_0^2\eps^2\lra{t}^{-1-\lambda+\delta}\lra{q}^{\frac{\delta}{\lambda}},\\
|u||\uLunit u|&\lesssim M_0^2\eps^2\lra{t}^{-1}\lra{q}^{-1+\frac{2\delta}{\lambda}}}
and notice that $\lra{q}\sim \lra{t}^\lambda$ on $\Dlower$.
Moreover, on $\Dlower$ we have \eqref{eqn:transport:mu}:
\eq{\label{est:wtLmu/mu:pf2}L\muring+\frac{1}{4}h_{LL}\uLunit\muring=\wt{L}\muring&=-\frac{1}{4}(\uLunit h)_{LL}\muring= O(\muring\cdot|u||\partial u|)= O( M_0^2\eps^2\lra{t}^{-1}\lra{q}^{-1+\frac{2\delta}{\lambda}}).}
In the last step, we again use $\muring\sim 1$.
From \eqref{est:wtLmu/mu:pf1} and \eqref{est:wtLmu/mu:pf2}, we cancel $L\muring$ and obtain
\fm{\kh{\frac{1}{2}\lambda t^{\lambda-1}+\frac{1}{4}(1-\frac{1}{2}\lambda t^{\lambda-1})h_{LL}}\cdot \uLunit\muring&=O(M_0^2\eps^2\lra{t}^{-1}\lra{q}^{-1+\frac{2\delta}{\lambda}}).}
The coefficient on the left side satisfies 
\fm{\frac{1}{2}\lambda t^{\lambda-1}+\frac{1}{4}(1-\frac{1}{2}\lambda t^{\lambda-1})h_{LL}&= \frac{1}{2}\lambda t^{\lambda-1}+O(M_0^2\eps^2\lra{t}^{-1}\lra{q}^{\frac{2\delta}{\lambda}})=\frac{1}{2}\lambda t^{\lambda-1}+O(M_0^2\eps^2\lra{t}^{\lambda-1}\lra{q}^{-1+\frac{2\delta}{\lambda}})\\
&\sim \lra{t}^{\lambda-1}\sim \lra{t}^{-1}\lra{q}.}
Again, we use $\lra{q}\sim \lra{t}^\lambda$ on $\Dlower$.
As a result, on $\Dlower$ we have
\fm{\abs{\frac{\muuLunit\muring}{\muring}}&=\abs{\frac{\uLunit\muring}{\muring^2}}\lesssim |\uLunit\muring|\lesssim \lra{t}\lra{q}^{-1}\cdot\ M_0^2\eps^2\lra{t}^{-1}\lra{q}^{-1+\frac{2\delta}{\lambda}}\lesssim  M_0^2\eps^2\lra{q}^{-2+\frac{2\delta}{\lambda}} .}

In summary, in $\D$ we have
\fm{ (\muuLunit\ln\muring)(t,x) &\geq(\muuLunit\ln\muring)(X(t_0)) - C \lra{q}^{-2+\frac{2\delta}{\lambda}} \lra{t}^{CM_0^2\eps^2}\\
&\geq -CM_0^2\eps^2\lra{q}^{-2+\frac{2\delta}{\lambda}}- C \lra{q}^{-2+\frac{2\delta}{\lambda}} \lra{t}^{CM_0^2\eps^2}\geq - C \lra{q}^{-2+\frac{2\delta}{\lambda}} \lra{t}^{CM_0^2\eps^2},
\\(\muuLunit \ln\muring)(t,x)&\leq(\muuLunit \ln\muring)(X(t_0))+  CM_0^2\eps^2\lra{q}^{-2+\frac{2\delta}{\lambda}}\ln\frac{C\lra{t}}{\lra{q}^{\frac{1}{\lambda}}}\\
&\leq CM_0^2\eps^2\lra{q}^{-2+\frac{2\delta}{\lambda}}+  CM_0^2\eps^2\lra{q}^{-2+\frac{2\delta}{\lambda}}\ln\frac{C\lra{t}}{\lra{q}^{\frac{1}{\lambda}}}\leq CM_0^2\eps^2\lra{q}^{-2+\frac{2\delta}{\lambda}}\ln\frac{C\lra{t}}{\lra{q}^{\frac{1}{\lambda}}}. }
\end{proof}

\begin{lem}
\label{lem:ptang:q:eikonal}
In $\D$, we have
\eq{-C|u|^2\muring\leq Lq=-\frac{1}{4}h_{LL}\muring\leq C|u|^3\muring,}
\eq{|Eq|=r^{-1}|\Omega q|\lesssim \muring\cdot M_0^2\eps^2\lra{t}^{-1+\delta+CM_0^2\eps^2}\lra{q}^{\frac{\delta}{\lambda}}.}
As a result, we have $|\partial q|\lesssim \muring$.

Moreover,
\eq{|g^{\alpha\beta}\partial_\alpha q \partial_\beta q|&\lesssim \muring^2\cdot M_0^4\eps^4\lra{t}^{-2+2\delta+CM_0^2\eps^2}\lra{q}^{\frac{2\delta}{\lambda}}.}
\end{lem}
\begin{proof}
Since $\wt{L}q=0$, we have
\fm{L q=-\frac{1}{4}h_{LL}\uLunit q=-\frac{1}{4}h_{LL}\muring.}
We now use $-C|u|^3\leq h_{LL}\leq C|u|^2$.

Next, we have
\fm{\wt{L}\Omega q&=[\wt{L},\Omega]q=[\frac{1}{4}h_{LL}\uLunit,\Omega]q=-\frac{1}{4}\Omega (h_{LL})\uLunit q=O(\muring |u||Z^{\leq 1} u|)\\
&=O(\kh{\frac{C\lra{t}}{\lra{q}^{\frac{1}{\lambda}}}}^{CM_0^2\eps^2\lra{q}^{-1+\frac{2\delta}{\lambda}}}\cdot M_0^2\eps^2\lra{t}^{-1+\delta}\lra{q}^{\frac{\delta}{\lambda}})=O(M_0^2\eps^2\lra{t}^{-1+\delta+CM_0^2\eps^2}\lra{q}^{\frac{\delta}{\lambda}}).}
Here we use $[\Omega,\pm\partial_t+\partial_r]=0$, $\Omega(\omega_i)=O(1)$, and $\Omega (h_{LL})=O(|h|+|\Omega h|)=O(|u|(|u|+|\Omega u|))=O(|u||Z^{\leq 1}u|)$. Since $\Omega q=0$ on $\Dlower$, we have 
\fm{|\Omega q|\lesssim  M_0^2\eps^2\lra{t}^{\delta+CM_0^2\eps^2}\lra{q}^{\frac{\delta}{\lambda}}\lesssim  M_0^2\eps^2\lra{t}^{\delta+CM_0^2\eps^2}\lra{q}^{\frac{\delta}{\lambda}}\cdot \muring.}
We thus have $|\partial q|\sim \muring$.

By \eqref{eq:null-frame-g-decompose:wtL} and since $|Lq|\lesssim |u|^2\muring$, we have
\fm{&\abs{g^{\alpha\beta}\partial_\alpha q\partial_\beta q}=\abs{\wt{L}q\cdot\uLunit q+(Eq)^2+\sum_{U,V\in\{L,\uLunit,E\}\atop (U,V)\neq(\uLunit,\uLunit)}h^{UV} Uq\cdot  Vq}\\
&\lesssim |Eq|^2+|u|^2 |\partial q| (|L q|+|Eq|)  \lesssim |Eq|^2+|u|^2\muring (|u|^2\muring+|Eq|)\lesssim |Eq|^2+|u|^4\muring+|u|^4\muring^2\\
&\lesssim M_0^4\eps^4\lra{t}^{-2+2\delta+CM_0^2\eps^2}\lra{q}^{\frac{2\delta}{\lambda}}\muring^2+M_0^4\eps^4\lra{t}^{-2}\lra{q}^{\frac{4\delta}{\lambda}}(1+\muring^{-1})\muring^2\\
&\lesssim M_0^4\eps^4\lra{t}^{-2+2\delta+CM_0^2\eps^2}\lra{q}^{\frac{2\delta}{\lambda}}\muring^2.}
Here we use $\lra{q}\lesssim \lra{t}^\lambda$ and $\muring^{-1}\lesssim \lra{t}^{CM_0^2\eps^2}$ in $\D$.
\end{proof}
\begin{rmk}\label{rmk:lem:ptang:q:eikonal}
\rm We have
\eq{-2q_t&= \muring-Lq=(1+O(|u|^2))\muring>0,\\
2q_r&=\muring+Lq=(1+O(|u|^2))\muring>0}
in $\D$ for $\eps\ll1$. It also follows that
\eq{\frac{-q_t}{q_r},\frac{q_r}{-q_t}&=1+O(|u|^2).}

Moreover, we have
\eq{\label{est:dq:D}q_\alpha-\frac{1}{2}\wh{\omega}_\alpha \muring&=O(|Lq|+|Eq|)=O(M_0^2\eps^2\lra{t}^{-1+\delta+CM_0^2\eps^2}\lra{q}^{\frac{\delta}{\lambda}})\cdot \muring.}
\end{rmk}

\begin{lem}
\label{lem:dmu}
In $\D$, we have
\eq{|L\muring+\frac{1}{4}h_{LL}\uLunit\muring|+|\partial_tLq-\frac{1}{8}h_{LL}\uLunit\muring|+|\partial_rLq+\frac{1}{8}h_{LL}\uLunit\muring|&\lesssim M_0^2\eps^2\lra{t}^{-1}
\lra{q}^{-1+\frac{2\delta}{\lambda}} \muring,}
\eq{|E\muring|=r^{-1}|\Omega\muring|&\lesssim M_0^2\eps^2\lra{t}^{-1+\delta+CM_0^2\eps^2}
\lra{q}^{-1+\frac{\delta}{\lambda}}\cdot \muring,}
\eq{\abs{\Omega(Lq)+\frac{1}{4}h_{LL}\Omega\muring}&\lesssim M_0^2\eps^2\lra{t}^{-1+\delta}\lra{q}^{\frac{\delta}{\lambda}}\cdot \muring.}
\end{lem}
\begin{proof}
By \eqref{eqn:transport:mu}, we have
\fm{L\muring+\frac{1}{4}h_{LL}\uLunit\muring=\wt{L} \muring=-\frac{1}{4}(\uLunit h)_{LL} \muring= O(\muring|u||\partial u|)=O(M_0^2\eps^2\lra{t}^{-1}
\lra{q}^{-1+\frac{2\delta}{\lambda}}\muring).}
By \eqref{est:wtLmu/mu}, we have a crude estimate \eq{\label{est:crude:uLmu}|\uLunit\muring|\lesssim \lra{q}^{-2+\frac{2\delta}{\lambda}}
\lra{t}^{CM_0^2\eps^2}\muring^2\lesssim \lra{q}^{-2+\frac{2\delta}{\lambda}}
\lra{t}^{CM_0^2\eps^2}}
and thus
\fm{|L\muring|&\lesssim |u|^2\lra{q}^{-2+\frac{2\delta}{\lambda}}
\lra{t}^{CM_0^2\eps^2}+M_0^2\eps^2\lra{t}^{-1}
\lra{q}^{-1+\frac{2\delta}{\lambda}}\muring\\
&\lesssim M_0^2\eps^2\lra{t}^{-1+CM_0^2\eps^2}\lra{q}^{-2+\frac{4\delta}{\lambda}}
+M_0^2\eps^2\lra{t}^{-1+CM_0^2\eps^2}
\lra{q}^{-1+\frac{2\delta}{\lambda}} \\
&\lesssim M_0^2\eps^2\lra{t}^{-1+CM_0^2\eps^2}
\lra{q}^{-1+\frac{2\delta}{\lambda}}.}
Since $Lq=-\frac{1}{4}h_{LL}\muring$, we have
\fm{-2\partial_tLq&= \uLunit Lq-LLq= L\muring+\frac{1}{4}L(h_{LL}\muring)=L\muring +O(|u|^2|L\muring|+|u||Lu|\muring)\\
&=-\frac{1}{4}h_{LL}\uLunit\muring+O(M_0^2\eps^2\lra{t}^{-1}
\lra{q}^{-1+\frac{2\delta}{\lambda}}\muring+|u|^2|L\muring|+\lra{t+r}^{-1}|u||Zu|\muring )\\
&=-\frac{1}{4}h_{LL}\uLunit\muring+O(M_0^2\eps^2\lra{t}^{-1}
\lra{q}^{-1+\frac{2\delta}{\lambda}}\muring+M_0^4\eps^4\lra{t}^{-2+CM_0^2\eps^2}
\lra{q}^{-1+\frac{4\delta}{\lambda}}+M_0^2\eps^2\lra{t}^{-2+\delta}\lra{q}^{\frac{\delta}{\lambda}}\muring )\\
&=-\frac{1}{4}h_{LL}\uLunit \muring+O( M_0^2\eps^2 \lra{t}^{-1}
\lra{q}^{-1+\frac{2\delta}{\lambda}}\muring).}
Moreover, we have $\lra{q}\lesssim\lra{t}^{\lambda}$  in $\D$. Similarly,
\fm{2\partial_rLq&=\uLunit Lq+LLq= L\muring-\frac{1}{4}L(h_{LL}\muring)= L\muring+O(|u|^2|L\muring|+|u||Lu|\muring)\\
&=-\frac{1}{4}h_{LL}\uLunit \muring+O( M_0^2\eps^2 \lra{t}^{-1}
\lra{q}^{-1+\frac{2\delta}{\lambda}}\muring).}

Next, we have
\fm{\wt{L}\Omega\muring&=\Omega(\wt{L}\muring)+[\wt{L},\Omega]\muring=-\frac{1}{4}\Omega((\uLunit h)_{LL}\muring)-\frac{1}{4}\Omega(h_{LL})\uLunit\muring\\
&=-\frac{1}{4}(\uLunit h)_{LL}\Omega\muring+O\kh{|\Omega((\uLunit h)_{LL})|\muring+|\Omega(h_{LL})||\uLunit\muring|}\\
&=-\frac{1}{4}(\uLunit h)_{LL}\Omega\muring+O\kh{(|u||\partial u|+|\Omega u||\partial u|+|u||\Omega\partial u|)\muring+|u|(|u|+|\Omega u|)|\uLunit\muring|}
\\
&=-\frac{1}{4}(\uLunit h)_{LL}\Omega\muring+O\kh{(|\partial u|+\lra{r-t}^{-1}|u|)|Z^{\leq 2} u|\lra{t}^{CM_0^2\eps^2}+|u||Z^{\leq 1}u| \lra{q}^{-2+\frac{2\delta}{\lambda}}
\lra{t}^{CM_0^2\eps^2}}
\\
&=-\frac{1}{4}(\uLunit h)_{LL}\Omega\muring+O\kh{M_0^2\eps^2\lra{t}^{-1+\delta+CM_0^2\eps^2}\lra{q}^{-1+\frac{\delta}{\lambda}}+M_0^2\eps^2\lra{t}^{-1+\delta +CM_0^2\eps^2}  \lra{q}^{-2+\frac{3\delta}{\lambda}}
}\\
&=-\frac{1}{4}(\uLunit h)_{LL}\Omega\muring+O\kh{M_0^2\eps^2\lra{t}^{-1+\delta+CM_0^2\eps^2}\lra{q}^{-1+\frac{\delta}{\lambda}}}.}
On $\Dlower$,  we apply $\Omega$ to \eqref{est:uLq:dD}:
\fm{\abs{\Omega\muring}&=\abs{\frac{8\lambda(2t^{1-\lambda}-\lambda)\Omega(h_{LL})}
{\kh{h_{LL}(2t^{1-\lambda}-\lambda)+4\lambda}^2}}=\frac{\muring^2}{8\lambda}\cdot|(2t^{1-\lambda}-\lambda)\Omega(h_{LL})|\\
&\lesssim t^{1-\lambda}|u||Z^{\leq 1}u|\lesssim M_0^2\eps^2\lra{t}^{-\lambda+\delta}\lra{q}^{\frac{\delta}{\lambda}}\lesssim M_0^2\eps^2\lra{q}^{-1+\frac{2\delta}{\lambda}}.}
If we define $H(s)$ by \eqref{def:H(s)}, by \eqref{est:H(s)} we have $e^{|H(s)|}\leq \lra{t}^{CM_0^2\eps^2}$. Thus, by estimating $\frac{d}{ds}(e^{H(s)}\Omega\muring \circ X(s))$, we have
\fm{|\Omega\muring|&\lesssim \lra{t}^{CM_0^2\eps^2}|\Omega\muring(X(t_0))|+\int_{t_0}^tM_0^2\eps^2\lra{\tau}^{-1+\delta+CM_0^2\eps^2}\lra{q}^{-1+\frac{\delta}{\lambda}}\,d\tau\\
&\lesssim M_0^2\eps^2\lra{t}^{CM_0^2\eps^2}\lra{q}^{-1+\frac{2\delta}{\lambda}}+M_0^2\eps^2\lra{t}^{\delta+CM_0^2\eps^2}\lra{q}^{-1+\frac{\delta}{\lambda}}\\
&\lesssim M_0^2\eps^2\lra{t}^{\delta+CM_0^2\eps^2}\lra{q}^{-1+\frac{\delta}{\lambda}}\cdot \muring.}
Here we use $\lra{q}\lesssim\lra{t}^\lambda$ and $\muring\gtrsim \lra{t}^{-CM_0^2\eps^2}$. It also follows that
\fm{\abs{\Omega(Lq)+\frac{1}{4}h_{LL}\Omega\muring}&\lesssim \abs{\Omega(-\frac{1}{4}h_{LL}\muring)+\frac{1}{4}h_{LL}\Omega\muring}\lesssim |\Omega(h_{LL})|\muring \lesssim |u||Z^{\leq 1}u|\muring  \lesssim M_0^2\eps^2\lra{t}^{-1+\delta}\lra{q}^{\frac{\delta}{\lambda}}\cdot \muring.}
\end{proof}
\begin{rmk}
\rm By Lemma~\ref{lem:dmu}, we have
\eq{\label{est:dtqr/qr2}-CM_0^2\eps^2\lra{q}^{-2+\frac{2\delta}{\lambda}}\ln(1+t)\cdot q_r^2\leq \partial_tq_r\leq C\lra{q}^{-2+\frac{2\delta}{\lambda}}\lra{t}^{CM_0^2\eps^2}q_r^2.}
In fact,
\fm{4\partial_tq_{r}&=  2\partial_tLq-\uLunit\muring+L\muring =  -\uLunit \muring+O(M_0^2\eps^2\lra{t}^{-1}
\lra{q}^{-1+\frac{2\delta}{\lambda}}\muring).}
By \eqref{est:wtLmu/mu}, in $\D$ we have
\fm{4\partial_tq_{r}
&\geq -CM_0^2\eps^2\lra{q}^{-2+\frac{2\delta}{\lambda}}\ln\frac{C\lra{t}}{\lra{q}^{\frac{1}{\lambda}}}\cdot \muring^2-CM_0^2\eps^2\lra{t}^{-1+CM_0^2\eps^2}
\lra{q}^{-1+\frac{2\delta}{\lambda}}\muring^2\\
&\geq -CM_0^2\eps^2\lra{q}^{-2+\frac{2\delta}{\lambda}}\ln(1+t)\cdot q_r^2,\\
4\partial_tq_r
&\leq C\lra{q}^{-2+\frac{2\delta}{\lambda}}\lra{t}^{CM_0^2\eps^2}\muring^2+CM_0^2\eps^2
\lra{q}^{-1-\frac{1}{\lambda}+\frac{2\delta}{\lambda}}\lra{t}^{CM_0^2\eps^2}\muring^2\\
&\leq C\lra{q}^{-2+\frac{2\delta}{\lambda}}\lra{t}^{CM_0^2\eps^2}q_r^2.}
Here we use $\muring\sim q_r $ by Remark~\ref{rmk:lem:ptang:q:eikonal}, $\lra{q}\lesssim \lra{t}^\lambda$ by \eqref{est:ratio:qtlambda:D}, and $\ln(1+t)\gtrsim \ln(C\lra{t})$.

Furthermore, we have
\eq{\label{est:dqr}\partial_\alpha q_r&=\frac{1}{2}\wh{\omega}_\alpha\uLunit q_r+O(|Lq_r|+|Eq_r|)=\frac{1}{2}\wh{\omega}_\alpha\partial_r\muring+O(|\partial_rLq|+|E\muring|+|ELq|)\\
&=\frac{1}{4}\wh{\omega}_\alpha\uLunit\muring+O(|L\muring|+|\partial_rLq|+|E\muring|+|ELq|)\\
&=\frac{1}{4}\wh{\omega}_\alpha\uLunit\muring+O(M_0^2\eps^2\lra{t}^{-1+\delta+CM_0^2\eps^2}\lra{q}^{-1+\frac{\delta}{\lambda}}).}
Here we use \eqref{est:crude:uLmu} to obtain
\fm{|L\muring|+|\partial_rLq|&\lesssim |h_{LL}||\uLunit\muring|+M_0^2\eps^2\lra{t}^{-1}\lra{q}^{-1+\frac{2\delta}{\lambda}}\muring\\
&\lesssim M_0^2\eps^2\lra{t}^{-1}\lra{q}^{\frac{2\delta}{\lambda}}\cdot \lra{q}^{-2+\frac{2\delta}{\lambda}}\lra{t}^{CM_0^2\eps^2}+M_0^2\eps^2\lra{t}^{-1+CM_0^2\eps^2}\lra{q}^{-1+\frac{2\delta}{\lambda}}\\
&\lesssim M_0^2\eps^2\lra{t}^{-1+CM_0^2\eps^2}\lra{q}^{-1+\frac{2\delta}{\lambda}},}
\fm{|E\muring|+ |E Lq|&\lesssim r^{-1}(1+|u|^2)|\Omega\muring|+M_0^2\eps^2\lra{t}^{-2+\delta}\lra{q}^{\frac{\delta}{\lambda}}\muring\\
&\lesssim M_0^2\eps^2\lra{t}^{-1+\delta+CM_0^2\eps^2}\lra{q}^{-1+\frac{\delta}{\lambda}}+M_0^2\eps^2\lra{t}^{-2+\delta+CM_0^2\eps^2}\lra{q}^{\frac{\delta}{\lambda}}\\
&\lesssim M_0^2\eps^2\lra{t}^{-1+\delta+CM_0^2\eps^2}\lra{q}^{-1+\frac{\delta}{\lambda}}.}
\end{rmk}

\subsection{Construction of $\wt{q}$}\label{sec:construction:wtq}
So far, we have defined $q$ in $\D$. If we define $q=r-t$ in $([1,\Tboot]\times\R^2)\setminus\D$, then we obtain a continuous function on $[1,\Tboot]\times\R^2$. However, such an extension is not $C^1$ across $\Dlower$.

Set $q=r-t$ for $t\in[1,\Tboot]$ and $r\geq t+2$. Thus, $q$ is at least $C^2$ for $t\in[1,\Tboot]$ and $r\geq t-t^\lambda+3$. Fix a cutoff function $\chi=\chi(s)\in C^\infty(\R)$ such that $\chi\in[0,1]$,  $\chi\equiv 1$ for $s\geq \frac{1}{2}$, and $\chi\equiv 0$ for $s\leq \frac{1}{4}$. We then define
\eq{y=y(t,r)&:=\frac{t-r+3}{t^\lambda},}
\eq{\label{eq:def:wtq:construction}\wt{q}&:=(r-t)\cdot\chi\kh{y(t,r)}+q\cdot\kh{1-\chi\kh{y(t,r)}} =q+(r-t-q)\chi\kh{y(t,r)}.}
We have $(t,x)\in\D\cup\{r-t\geq 2\}$ whenever $1-\chi\neq 0$, so $\wt{q}$ is well-defined.
Note that $\wt{q}=q$ for $r\geq t-\frac{1}{4}t^\lambda+3$ (i.e.\ $y\leq \frac{1}{4}$) and that $\wt{q}=r-t$ for $r\leq t-\frac{1}{2}t^\lambda+3$ (i.e.\ $y\geq \frac{1}{2}$). Moreover, for $y\in[\frac{1}{4},\frac{1}{2}]$, we have $\lra{t}^\lambda\sim\lra{r-t}\lesssim \lra{q}\lesssim \lra{t}^\lambda$. It follows that in this region,
\eq{\label{est:q:wtq:equivalence} \lra{t}^\lambda\sim\lra{r-t}\sim \lra{q}\sim \lra{\wt{q}}.}

To continue, we notice that $\partial_{t,x} y =O(t^{-\lambda})$ whenever $y(t,r)\in\supp\chi'\cup \supp\chi''$. Also recall from \eqref{est:diff:q:r-t} that $|r-t-q|\lesssim M_0^2\eps^2\lra{t}^{2\delta}$ in $\D$. Then, by direct computations, we have
\eq{\label{est:dwtq}
\wt{q}_\alpha &=q_\alpha +(\wh{\omega}_\alpha-q_\alpha)\chi+(r-t-q)\chi'\cdot\partial_\alpha y\\
&=(1-\chi) q_\alpha + \wh{\omega}_\alpha \chi+O(M_0^2\eps^2\lra{t}^{-\lambda+2\delta})\cdot1_{y\in[\frac{1}{4},\frac{1}{2}]}.}
Moreover, we have
\eq{\label{est:ddwtq}\partial_\alpha\wt{q}_r&=\partial_\alpha\kh{q_r +(1-q_r)\chi -t^{-\lambda}(r-t-q)\chi' }\\
&=(1-\chi)\partial_\alpha q_r+(1-q_r)\chi'\cdot\partial_\alpha y -t^{-\lambda}(r-t-q)\chi''\cdot \partial_\alpha y+\partial_\alpha\kh{\frac{t-r+q}{t^\lambda}}\cdot\chi'\\
&=(1-\chi)\partial_\alpha q_r+O\kh{ \lra{t}^{-\lambda}|\partial(r-t-q)|+\lra{t}^{-2\lambda}|r-t-q|}\cdot1_{y\in[\frac{1}{4},\frac{1}{2}]}\\
&=(1-\chi)\partial_\alpha q_r+O(M_0^2\eps^2\lra{t}^{-2\lambda+2\delta})\cdot1_{y\in[\frac{1}{4},\frac{1}{2}]}.}
In fact, for $y\in[\frac{1}{4},\frac{1}{2}]$, we have  
\fm{\partial_\alpha\kh{\frac{t-r+q}{t^\lambda}}=O(t^{-\lambda-1}|r-t-q|+t^{-\lambda}|\partial(r-t-q)|).} Moreover, by  \eqref{est:mu-2}  and Lemma~\ref{lem:ptang:q:eikonal}, we have 
\fm{|\partial(r-t-q)|&\lesssim |\uLunit(r-t-q)|+|L(r-t-q)|+|E(r-t-q)|\lesssim |\muring-2|+|Lq|+|Eq|\\
&\lesssim M_0^2\eps^2\lra{q}^{-1+\frac{2\delta}{\lambda}} \cdot\ln\frac{C\lra{t}}{\lra{q}^{\frac{1}{\lambda}}}\cdot \kh{\frac{C\lra{t}}{\lra{q}^{\frac{1}{\lambda}}}}^{CM_0^2\eps^2\lra{q}^{-1+\frac{2\delta}{\lambda}}}+M_0^2\eps^2\lra{t}^{-1+\delta+CM_0^2\eps^2}\lra{q}^{\frac{\delta}{\lambda}}.}
Since $t-r+3\in[\frac{1}{4}t^\lambda,\frac{1}{2}t^\lambda]$, we have $\lra{t}^\lambda\sim\lra{r-t}\sim \lra{q}\sim \lra{\wt{q}}$ by \eqref{est:q:wtq:equivalence}. It follows that 
\eq{\label{est:d(r-t-q):chi'}|\partial(r-t-q)|&\lesssim M_0^2\eps^2\lra{t}^{-\lambda+ 2\delta } \cdot\ln C \cdot  \lra{t} ^{CM_0^2\eps^2\lra{t}^{-\lambda+2\delta}}+M_0^2\eps^2\lra{t}^{-1+2\delta+CM_0^2\eps^2} \lesssim M_0^2\eps^2\lra{t}^{-\lambda+ 2\delta }.}

We now summarize the estimates for $\wt{q}$.  Note that we no longer require $(t,x)\in\D$.
\begin{lem}
We have $\wt{q}_r>0$, $-\wt{q}_t>0$, and
\eq{\label{est:sgn:wtqtr:ratio}|\frac{\wt{q}_t}{\wt{q}_r}+1|+|\frac{\wt{q}_r}{\wt{q}_t}+1|+|\wt{q}_\alpha-\frac{1}{2}(1-\chi)\wh{\omega}_\alpha \muring-\wh{\omega}_\alpha\chi|&\lesssim M_0^2\eps^2\lra{t}^{2\delta-\lambda}. }
\end{lem}
\begin{proof}
If $r\geq t+1$ or $y>\frac{1}{2}$, we have $\wt{q}=r-t$. There is nothing to prove.

If $r\leq t+1$ and $y<\frac{1}{4}$, we have $\wt{q}=q$. In this case, we apply Remark~\ref{rmk:lem:ptang:q:eikonal} and \eqref{est:dq:D}. Note that $|u|^2\lesssim  M_0^2\eps^2\lra{t}^{-1+2\delta}$ and that $M_0^2\eps^2\lra{t}^{-1+\delta+CM_0^2\eps^2}\lra{q}^{\frac{\delta}{\lambda}}\lesssim M_0^2\eps^2\lra{t}^{-1+2\delta+CM_0^2\eps^2}\lesssim M_0^2\eps^2\lra{t}^{-\lambda+2\delta}$.

If $r\leq t+1$ and $y\in[ \frac{1}{4},\frac{1}{2}]$, we first notice that $|\partial(r-t-q)|\lesssim M_0^2\eps^2\lra{t}^{-\lambda+2\delta}$ by \eqref{est:d(r-t-q):chi'}. Thus, $|q_t+1|+|q_r-1|\lesssim M_0^2\eps^2\lra{t}^{-\lambda+2\delta}$. It follows from \eqref{est:dwtq} that
\fm{\wt{q}_t&=q_t+\chi(-q_t-1)+O(M_0^2\eps^2\lra{t}^{-\lambda+2\delta})=-1+O(M_0^2\eps^2\lra{t}^{-\lambda+2\delta}),\\
\wt{q}_r&=q_r+\chi(-q_r+1)+O(M_0^2\eps^2\lra{t}^{-\lambda+2\delta})=1+O(M_0^2\eps^2\lra{t}^{-\lambda+2\delta}).}
By \eqref{est:dq:D}, we also have
\fm{\wt{q}_\alpha&=(1-\chi) q_\alpha + \wh{\omega}_\alpha \chi+O(M_0^2\eps^2\lra{t}^{-\lambda+2\delta}) \\
&=\frac{1}{2}(1-\chi)\wh{\omega}_\alpha\muring + \wh{\omega}_\alpha \chi+O(M_0^2\eps^2\lra{t}^{-\lambda+2\delta}+M_0^2\eps^2\lra{t}^{-1+\delta+CM_0^2\eps^2}\lra{q}^{\frac{\delta}{\lambda}}) \\
&=\frac{1}{2}(1-\chi)\wh{\omega}_\alpha\muring + \wh{\omega}_\alpha \chi+O(M_0^2\eps^2\lra{t}^{-\lambda+2\delta}).}
Here we use \eqref{est:q:wtq:equivalence} in this region.
\end{proof}

\begin{lem}We have
\eq{\label{est:eik:wtq} \abs{\wt{q}_t^{-1}g^{\alpha\beta}\wt{q}_\alpha \wt{q}_\beta}\lesssim \frac{M_0^2\eps^2\lra{\wt{q}}^{\frac{5\delta }{\lambda}}}{\lra{t}^{1+\delta}}.}
\end{lem}
\begin{proof}
By \eqref{est:dwtq}, we have
\fm{ g^{\alpha\beta}\wt{q}_\alpha \wt{q}_\beta&=(1-\chi)^2g^{\alpha\beta}q_\alpha q_\beta+\chi^2g^{\alpha\beta}\wh{\omega}_\alpha \wh{\omega}_\beta+2\chi(1-\chi)g^{\alpha\beta}q_\alpha \wh{\omega}_\beta\\
&\quad+O\kh{M_0^2\eps^2\lra{t}^{-\lambda+2\delta}\cdot\kh{|\partial q|+1}}\cdot 1_{y\in[\frac{1}{4},\frac{1}{2}]}.}

If $r\geq t+1$ or $y> \frac{1}{2}$, we have $\wt{q}=r-t$. In this case, we have $\wt{q}_t=-1$ and 
\fm{\abs{g^{\alpha\beta}\wt{q}_\alpha \wt{q}_\beta}&=\abs{g^{\alpha\beta}\wh{\omega}_\alpha \wh{\omega}_\beta}=\abs{(g^{\alpha\beta}-m^{\alpha\beta})\wh{\omega}_\alpha \wh{\omega}_\beta}\\
&\lesssim |u|^2\lesssim M_0^2\eps^2\lra{t}^{-1}\lra{r-t}^{\frac{2\delta}{\lambda}}\cdot 1_{r<t+1}.}
For $y>\frac{1}{2}$, we have $3-r+t\geq \frac{1}{2}t^\lambda$ and thus $\lra{r-t}\gtrsim \lra{t}^\lambda$. As a result, we have
\fm{M_0^2\eps^2\lra{t}^{-1}\lra{r-t}^{\frac{2\delta}{\lambda}}\cdot 1_{r<t+1}&\lesssim M_0^2\eps^2\lra{t}^{-1-\delta}\lra{r-t}^{\frac{3\delta}{\lambda}}.}

If $r\leq t+1$ and $y<\frac{1}{4}$, we have $\wt{q}=q$ and $\chi=0$.  By Lemma~\ref{lem:ptang:q:eikonal}, we have
\fm{|\wt{q}_t|^{-1}\cdot |g^{\alpha\beta}q_\alpha q_\beta|&\lesssim \muring^{-1}\cdot \muring^2\cdot  M_0^2\eps^2\lra{t}^{-2+2\delta+CM_0^2\eps^2}\lra{q}^{\frac{2\delta}{\lambda}}\lesssim \muring\cdot M_0^2\eps^2\lra{t}^{-2+2\delta+CM_0^2\eps^2}\lra{q}^{\frac{2\delta}{\lambda}}\\
&\lesssim M_0^2\eps^2\lra{t}^{-1-\delta}\lra{q}^{\frac{2\delta}{\lambda}}.}
Here we use $|\wt{q}_t|\sim\muring\lesssim \lra{t}^{CM_0^2\eps^2}$.

If $r\leq t+1$ and $y\in[\frac{1}{4},\frac{1}{2}]$, we first recall $|\wt{q}_t|\sim|q_t|\sim  1$ from the proof of \eqref{est:sgn:wtqtr:ratio}. We also notice that $\lra{t}^\lambda\sim \lra{r-t}\sim\lra{q}\sim \lra{\wt{q}}$ in this region. Moreover,
\fm{\abs{g^{\alpha\beta}\wt{q}_\alpha \wt{q}_\beta}&\lesssim \abs{g^{\alpha\beta}q_\alpha q_\beta}+\abs{g^{\alpha\beta}\wh{\omega}_\alpha \wh{\omega}_\beta}+\abs{g^{\alpha\beta}q_\alpha \wh{\omega}_\beta}+M_0^2\eps^2\lra{t}^{-\lambda+2\delta}. }
By \eqref{est:d(r-t-q):chi'}, we have $|q_\alpha-\wh{\omega}_\alpha|\lesssim M_0^2\eps^2\lra{t}^{-\lambda+2\delta}$ and thus
\fm{g^{\alpha\beta}q_\alpha \wh{\omega}_\beta&=m^{\alpha\beta}q_\alpha \wh{\omega}_\beta+O(|u|^2|\partial q|)=Lq+O(|u|^2)=-\frac{1}{4}h_{LL}\muring+O(|u|^2)=O(M_0^2\eps^2\lra{t}^{-\lambda+2\delta}).}
Following the proof above, we still have
\fm{\abs{g^{\alpha\beta}q_\alpha q_\beta}+\abs{g^{\alpha\beta}\wh{\omega}_\alpha \wh{\omega}_\beta}&\lesssim    M_0^2\eps^2\lra{t}^{-1-\delta}\lra{q}^{\frac{2\delta}{\lambda}}+M_0^2\eps^2\lra{t}^{-1-\delta}\lra{r-t}^{\frac{3\delta}{\lambda}}\lesssim M_0^2\eps^2\lra{t}^{-1-\delta}\lra{\wt{q}}^{\frac{3\delta}{\lambda}}.}
We also have 
\fm{M_0^2\eps^2\lra{t}^{-\lambda+2\delta}\lesssim M_0^2\eps^2\lra{t}^{-1-\delta}\cdot\lra{t}^{3\delta+1-\lambda}\lesssim M_0^2\eps^2\lra{t}^{-1-\delta}\cdot\lra{\wt{q}}^{\frac{3\delta+1-\lambda}{\lambda}}\lesssim M_0^2\eps^2\lra{t}^{-1-\delta}\cdot\lra{\wt{q}}^{\frac{5\delta}{\lambda}}.}
\end{proof}

\begin{lem}
We have
\eq{\label{est:second:wtq}  -CM_0^2\eps^2\lra{\wt{q}}^{-2+\frac{2\delta}{\lambda}}\wt{q}_r^2\cdot \ln(1+t)\leq \partial_t\wt{q}_r\leq C\lra{\wt{q}}^{-2+\frac{2\delta}{\lambda}}\lra{t}^{CM_0^2\eps^2}\wt{q}_r^2.}
\end{lem}
\begin{proof}
By \eqref{est:ddwtq}, we have
\fm{\partial_t\wt{q}_r&=(1-\chi)\partial_tq_r+O(M_0^2\eps^2\lra{t}^{-2\lambda+2\delta})\cdot1_{y\in[\frac{1}{4},\frac{1}{2}]}.}

If $r\geq t+1$ or $y>\frac{1}{2}$, we have $\wt{q}=r-t$ and thus $\partial_\alpha \wt{q}_r=0$. 

If $r\leq t+1$ and $y<\frac{1}{4}$, we have $\chi=0$ and $\wt{q}=q$. We recall from \eqref{est:dtqr/qr2} that
\fm{-CM_0^2\eps^2\lra{q}^{-2+\frac{2\delta}{\lambda}}q_r^2\cdot \ln(1+t)\leq \partial_t q_r\leq C\lra{q}^{-2+\frac{2\delta}{\lambda}}\lra{t}^{CM_0^2\eps^2}q_r^2.}

If $r\leq t+1$ and $y\in[\frac{1}{4},\frac{1}{2}]$, we first recall that $\wt{q}_r\sim q_r\sim  1$ and that $\lra{t}^\lambda\sim\lra{r-t}\sim\lra{q}\sim\lra{\wt{q}}$. By \eqref{est:dtqr/qr2} and since $\chi\in[0,1]$, we have \fm{\partial_t\wt{q}_r&=(1-\chi)\partial_tq_r+O(M_0^2\eps^2\lra{t}^{-2\lambda+2\delta})\leq (1-\chi)\cdot  C\lra{q}^{-2+\frac{2\delta}{\lambda}}\lra{t}^{CM_0^2\eps^2} +CM_0^2\eps^2\lra{t}^{-2\lambda+2\delta}\\
&\leq C\lra{\wt{q}}^{-2+\frac{2\delta}{\lambda}}\lra{t}^{C M_0^2\eps^2} +CM_0^2\eps^2\lra{\wt{q}}^{-2+\frac{2\delta}{\lambda}}\leq C\lra{\wt{q}}^{-2+\frac{2\delta}{\lambda}}\lra{t}^{C M_0^2\eps^2} \wt{q}_r^2,\\ 
\partial_t\wt{q}_r&=(1-\chi)\partial_tq_r+O(M_0^2\eps^2\lra{t}^{-2\lambda+2\delta})\geq -(1-\chi)\cdot  CM_0^2\eps^2\lra{q}^{-2+\frac{2\delta}{\lambda}}\ln(1+t)-CM_0^2\eps^2\lra{t}^{-2\lambda+2\delta}\\
&\geq- CM_0^2\eps^2\lra{\wt{q}}^{-2+\frac{2\delta}{\lambda}}\ln(1+t)-CM_0^2\eps^2\lra{\wt{q}}^{-2+\frac{2\delta}{\lambda}}\geq - CM_0^2\eps^2\lra{\wt{q}}^{-2+\frac{2\delta}{\lambda}}\ln(1+t)\cdot \wt{q}_r^2.}
\end{proof}

\begin{lem}
We have
\eq{\label{est:dwtq:ddwtq:ptang}\abs{\wt{q}_\alpha \partial_t\wt{q}_r- \wt{q}_t\partial_\alpha\wt{q}_r}&\lesssim M_0^2\eps^2\lra{t}^{-1+5\delta}\lra{\wt{q}}^{-1+\frac{2\delta}{\lambda}},}
\eq{\label{est:dwtq:ddwtq:ptang:2} \abs{g^{\alpha\beta}\wt{q}_\alpha\partial_\beta\wt{q}_r}&\lesssim M_0^2\eps^2\lra{t}^{-1}\wt{q}_r^{2}.}
\end{lem}
\begin{proof}
By \eqref{est:ddwtq} and \eqref{est:dqr}, we have
\eq{\label{est:ddwtq:uLmu}\partial_\alpha\wt{q}_r&=(1-\chi)\cdot\frac{1}{4}\wh{\omega}_\alpha \uLunit\muring+O(M_0^2\eps^2\lra{t}^{-1+\delta+CM_0^2\eps^2}\lra{q}^{-1+\frac{\delta}{\lambda}})\cdot 1_{y\leq\frac{1}{2}}+O(M_0^2\eps^2\lra{t}^{-2\lambda+2\delta})\cdot1_{y\in[\frac{1}{4},\frac{1}{2}]}\\
&=(1-\chi)\cdot\frac{1}{4}\wh{\omega}_\alpha \uLunit\muring+O(M_0^2\eps^2\lra{t}^{-1+2\delta}\lra{\wt{q}}^{-1+\frac{2\delta}{\lambda}})\cdot 1_{y\leq\frac{1}{2}}.}
To obtain the second row, we recall that $\lra{r-t}\sim\lra{q}\sim\lra{\wt{q}}\sim\lra{t}^{\lambda}$ for $y\in[\frac{1}{4},\frac{1}{2}]$. Thus, we have
\fm{\lra{t}^{-2\lambda+2\delta}\lesssim \lra{t}^{-1+2\delta}\cdot\lra{\wt{q}}^{\frac{1-2\lambda}{\lambda}}\lesssim \lra{t}^{-1+2\delta}\cdot\lra{\wt{q}}^{-1+\frac{2\delta}{\lambda}}}
for $y\in[\frac{1}{4},\frac{1}{2}]$.
By \eqref{est:crude:uLmu}, we also have
\fm{\partial_\alpha\wt{q}_r&=O(\lra{t}^{CM_0^2\eps^2}\lra{q}^{-2+\frac{2\delta}{\lambda}}+M_0^2\eps^2\lra{t}^{-1+2\delta}\lra{\wt{q}}^{-1+\frac{2\delta}{\lambda}})\cdot 1_{y\leq\frac{1}{2}}\\
&=O(\lra{t}^{CM_0^2\eps^2}\lra{q}^{-2+\frac{2\delta}{\lambda}}+M_0^2\eps^2 \lra{\wt{q}}^{-2+\frac{2\delta}{\lambda}})\cdot 1_{y\leq\frac{1}{2}}=O(\lra{t}^{CM_0^2\eps^2}\lra{\wt{q}}^{-2+\frac{2\delta}{\lambda}})\cdot 1_{y\leq\frac{1}{2}}.}
Here we use $\lambda=1-2\delta$ and $\lra{q}\lesssim \lra{t}^\lambda$ in $\D$.

By \eqref{est:dwtq} and \eqref{est:sgn:wtqtr:ratio}, we have
\fm{\wt{q}_\alpha \partial_t\wt{q}_r- \wt{q}_t\partial_\alpha\wt{q}_r&=\kh{\frac{1}{2}(1-\chi)\muring+\chi}\wh{\omega}_\alpha\partial_t\wt{q}_r-\kh{\frac{1}{2}(1-\chi)\muring+\chi}\cdot (-1)\partial_\alpha\wt{q}_r+O(M_0^2\eps^2\lra{t}^{2\delta-\lambda}|\partial \wt{q}_r|)\\
&=\kh{\frac{1}{2}(1-\chi)\muring+\chi}\cdot (\partial_\alpha \wt{q}_r+\wh{\omega}_\alpha \partial_t\wt{q}_r)+O(M_0^2\eps^2\lra{t}^{2\delta-\lambda+CM_0^2\eps^2}\lra{q}^{-2+\frac{2\delta}{\lambda}})\cdot 1_{y\leq\frac{1}{2}}\\
&=O\kh{\lra{t}^{CM_0^2\eps^2}|\partial_\alpha \wt{q}_r+\wh{\omega}_\alpha \partial_t\wt{q}_r|+M_0^2\eps^2\lra{t}^{-1+4\delta+CM_0^2\eps^2}\lra{\wt{q}}^{-2+\frac{2\delta}{\lambda}}}\cdot 1_{y\leq\frac{1}{2}}.}
Here we use $\muring\lesssim \lra{t}^{CM_0^2\eps^2}$ and $\lambda=1-2\delta$. 
We also have
\fm{\lra{t}^{CM_0^2\eps^2}|\partial_\alpha \wt{q}_r+\wh{\omega}_\alpha \partial_t\wt{q}_r|&\lesssim\lra{t}^{CM_0^2\eps^2} \abs{(1-\chi)\cdot\frac{1}{4}\wh{\omega}_\alpha \uLunit\muring+\wh{\omega}_\alpha(1-\chi)\cdot\frac{1}{4}\wh{\omega}_0 \uLunit\muring} \\
&\quad+ M_0^2\eps^2\lra{t}^{-1+2\delta+CM_0^2\eps^2}\lra{\wt{q}}^{-1+\frac{2\delta}{\lambda}} \cdot 1_{y\leq\frac{1}{2}}\\
&\lesssim M_0^2\eps^2\lra{t}^{-1+3\delta }\lra{\wt{q}}^{-1+\frac{2\delta}{\lambda}} \cdot 1_{y\leq\frac{1}{2}}.}
We thus obtain \eqref{est:dwtq:ddwtq:ptang}.

Next, to prove \eqref{est:dwtq:ddwtq:ptang:2}, we apply \eqref{eq:null-frame-g-decompose:wtL} to obtain
\fm{g^{\alpha\beta}\wt{q}_\alpha\partial_\beta \wt{q}_r&=\frac{1}{2}\kh{\wt{L}\wt{q}\cdot \uLunit\wt{q}_r+\uLunit\wt{q}\cdot\wt{L}\wt{q}_r}+E\wt{q} \cdot E\wt{q}_r+\sum_{U,V\in\{L,\uLunit,E\}\atop (U,V)\neq (\uLunit,\uLunit)}h^{UV}U\wt{q} \cdot V\wt{q}_r.}
By \eqref{est:dwtq} and Lemma~\ref{lem:ptang:q:eikonal}, we have
\fm{\wt{L}\wt{q}&=(1-\chi)\wt{L}q+\wt{L}(r-t)\chi+O(M_0^2\eps^2\lra{t}^{-\lambda+2\delta})\cdot 1_{y\in[\frac{1}{4},\frac{1}{2}]}\\
&=\frac{1}{2}h_{LL}\chi+O(M_0^2\eps^2\lra{t}^{-\lambda+2\delta})\cdot 1_{y\in[\frac{1}{4},\frac{1}{2}]}=O(M_0^2\eps^2\lra{t}^{-\lambda+2\delta})\cdot 1_{y\geq \frac{1}{4}},}
\fm{\uLunit\wt{q} &= (1-\chi)\uLunit  q+2\chi+O(M_0^2\eps^2\lra{t}^{-\lambda+2\delta})\cdot 1_{y\in[\frac{1}{4},\frac{1}{2}]}=(1-\chi) \muring+2\chi+O(M_0^2\eps^2\lra{t}^{-\lambda+2\delta})\cdot 1_{y\in[\frac{1}{4},\frac{1}{2}]}\\
&=O(\lra{t}^{CM_0^2\eps^2}),}
\fm{|E\wt{q}|+|L\wt{q}|&\lesssim (1-\chi)(|Eq|+|Lq|)+ M_0^2\eps^2\lra{t}^{-\lambda+2\delta} \cdot 1_{y\in[\frac{1}{4},\frac{1}{2}]}\\
&\lesssim \muring\cdot M_0^2\eps^2\lra{t}^{-1+\delta+CM_0^2\eps^2}\lra{q}^{\frac{\delta}{\lambda}}\cdot 1_{y\leq \frac{1}{2}}+ M_0^2\eps^2\lra{t}^{-1+4\delta} \cdot 1_{y\in[\frac{1}{4},\frac{1}{2}]}\\
&\lesssim \muring\cdot M_0^2\eps^2\lra{t}^{-1+4\delta} \cdot 1_{y\leq \frac{1}{2}}.}
By \eqref{est:ddwtq} and \eqref{est:ddwtq:uLmu}, we have
\fm{\wt{L}\wt{q}_r&=(1-\chi)\wt{L}q_r+O(M_0^2\eps^2\lra{t}^{-2\lambda+2\delta})\cdot1_{y\in[\frac{1}{4},\frac{1}{2}]}\\
&=(1-\chi)[L+\frac{1}{4}h_{LL}\uLunit,\partial_r]q+O(M_0^2\eps^2\lra{t}^{-2\lambda+2\delta})\cdot1_{y\in[\frac{1}{4},\frac{1}{2}]}\\
&=-\frac{1}{4}(1-\chi)\partial_rh_{LL}\cdot \muring+O(M_0^2\eps^2\lra{t}^{-2\lambda+2\delta})\cdot1_{y\in[\frac{1}{4},\frac{1}{2}]}\\
&=O(M_0^2\eps^2\lra{t}^{-1}\cdot\muring)\cdot 1_{y\leq\frac{1}{2}},}
\fm{|\uLunit \wt{q}_r|&\lesssim \lra{t}^{CM_0^2\eps^2}\lra{\wt{q}}^{-2+\frac{2\delta}{\lambda}}\cdot 1_{y\leq\frac{1}{2}},}
\fm{|E\wt{q}_r|+|L\wt{q}_r|&\lesssim  M_0^2\eps^2\lra{t}^{-1+2\delta}\lra{\wt{q}}^{-1+\frac{2\delta}{\lambda}} \cdot 1_{y\leq\frac{1}{2}}.}

Multiply these estimates together. We have
\fm{|\wt{L}\wt{q}||\uLunit\wt{q}_r|&\lesssim M_0^2\eps^2\lra{t}^{-\lambda+2\delta}\cdot \lra{t}^{CM_0^2\eps^2}\lra{\wt{q}}^{-2+\frac{2\delta}{\lambda}} 1_{y\in[ \frac{1}{4},\frac{1}{2}]}\\
&\lesssim M_0^2\eps^2\lra{t}^{-\lambda+2\delta+CM_0^2\eps^2}\cdot \lra{t}^{-2\lambda+2\delta}\cdot 1_{y\in[ \frac{1}{4},\frac{1}{2}]}\lesssim M_0^2\eps^2\lra{t}^{-1}\cdot \wt{q}_r^2,}
\fm{|\uLunit\wt{q}||\wt{L}\wt{q}_r|&\lesssim \kh{(1-\chi) \muring+\chi+M_0^2\eps^2\lra{t}^{-\lambda+2\delta} \cdot 1_{y\in[\frac{1}{4},\frac{1}{2}]}}\cdot M_0^2\eps^2\lra{t}^{-1}\muring\cdot 1_{y\leq\frac{1}{2}}\\
&\lesssim M_0^2\eps^2\lra{t}^{-1}\muring^2\cdot 1_{y\leq\frac{1}{2}}+M_0^2\eps^2\lra{t}^{-1}\muring\cdot 1_{y\in[\frac{1}{4},\frac{1}{2}]}\\
&\lesssim M_0^2\eps^2\lra{t}^{-1}\muring^2\cdot 1_{y\leq\frac{1}{4}}+M_0^2\eps^2\lra{t}^{-1} \cdot 1_{y\in[\frac{1}{4},\frac{1}{2}]}\lesssim M_0^2\eps^2\lra{t}^{-1}\wt{q}_r^2,}
\fm{|E\wt{q}||E\wt{q}_r|&\lesssim \muring\cdot M_0^2\eps^2\lra{t}^{-1+4\delta}\cdot M_0^2\eps^2\lra{t}^{-1+2\delta}\lra{\wt{q}}^{-1+\frac{2\delta}{\lambda}}\cdot 1_{y\leq\frac{1}{2}}\\
&\lesssim \muring^2\cdot M_0^2\eps^2\lra{t}^{-2+6\delta+CM_0^2\eps^2}\cdot 1_{y\leq\frac{1}{2}}\lesssim M_0^2\eps^2\lra{t}^{-1}\wt{q}_r^2.}
Moreover, for $(U,V)\in\{L,\uLunit,E\}$ with $U\neq \uLunit$, we have
\fm{|h^{UV}||U\wt{q}||V\wt{q}_r|&\lesssim |u|^2\cdot  \muring\cdot M_0^2\eps^2\lra{t}^{-1+4\delta}  \cdot \lra{t}^{CM_0^2\eps^2}\lra{\wt{q}}^{-2+\frac{2\delta}{\lambda}} \cdot 1_{y\leq\frac{1}{2}}\\
&\lesssim M_0^4\eps^4\lra{t}^{-2+6\delta+CM_0^2\eps^2}\cdot 1_{y\leq\frac{1}{2}}\lesssim M_0^2\eps^2\lra{t}^{-1}\wt{q}_r^2.}
For $(U,V)\in\{L,\uLunit,E\}$ with $V\neq \uLunit$, we have
\fm{|h^{UV}||U\wt{q}||V\wt{q}_r|&\lesssim |u|^2\cdot \lra{t}^{CM_0^2\eps^2} \cdot M_0^2\eps^2\lra{t}^{-1+2\delta}\lra{\wt{q}}^{-1+\frac{2\delta}{\lambda}} \cdot 1_{y\leq\frac{1}{2}}\\
&\lesssim M_0^4\eps^4\lra{t}^{-2+2\delta+CM_0^2\eps^2}\cdot 1_{y\leq\frac{1}{2}}\lesssim M_0^2\eps^2\lra{t}^{-1}\wt{q}_r^2.}
We thus obtain \eqref{est:dwtq:ddwtq:ptang:2}.
\end{proof}

\begin{lem}\label{lem:ptb:everywhere:u:du:ddu}
For all $t\in[1,\Tboot]$ and $x\in\R^2$, we have
\eq{  |u|&\lesssim M_0\eps\lra{t}^{-\frac{1}{2}}\min\{\lra{\wt{q}}^{\frac{\delta}{\lambda}},\lra{r-t}^{\frac{\delta}{\lambda}}\},}
\eq{  |\partial u|&\lesssim M_0\eps\lra{t}^{-\frac{1}{2}}\lra{\wt{q}}^{-1+\frac{\delta}{\lambda}} ,}
\eq{  |\partial^2u|&\lesssim M_0\eps\lra{t}^{-\frac{1}{2}}\lra{\wt{q}}^{-2+\frac{\delta}{\lambda}}\wt{q}_r.}
\end{lem}
\begin{proof}
If $r\geq t+1$, we have $u\equiv 0$.

If $r\leq t+1$ and $y<\frac{1}{4}$, we have $\wt{q}=q$ and $(t,x)\in\D$. These estimates are the same as \eqref{est:ptb:u}, \eqref{est:ptb:du}, and \eqref{est:ptb:ddu}. 

If $y\geq \frac{1}{4}$, we have $\lra{\wt{q}}\sim\lra{r-t}$, $\wt{q}_r\sim 1$, and $\lra{r-t}\gtrsim\lra{t}^\lambda$. In this case, we use the bootstrap assumption \eqref{asu:bootstrap:ptb} to obtain
\fm{|\partial^ku|&\lesssim\lra{r-t}^{-k}|Z^{\leq k}u|\lesssim M_0\eps\lra{t}^{-\frac{1}{2}+\delta}\lra{r-t}^{-k}\lesssim M_0\eps\lra{t}^{-\frac{1}{2}}\lra{r-t}^{-k+\frac{\delta}{\lambda}}\lesssim M_0\eps\lra{t}^{-\frac{1}{2}}\lra{\wt{q}}^{-k+\frac{\delta}{\lambda}}}
for $k=0,1,2$.
\end{proof}

\section{Energy estimates and Poincaré's estimates}\label{sec:energy:poincare}

For $(t,x)\in[1,\Tboot]\times\R^2$ with $r-t< 2$, we define a weight function
\eq{\label{def:weight}
w(t,x):=\wt{q}_r^{-1}\exp\kh{\kappa_1 \eps^2\ln(1+t)\cdot (2-\wt{q}(t,x))^{-\kappa_2}}.}
Here $\kappa_1:=\wt{C}_1M_0^2$ where $\wt{C}_1>1$ is a large constant independent of $M_0,\eps$, and 
\eq{\kappa_2:=1-\frac{2\delta}{\lambda}.}
We will explain the reason for this choice in the proof of \eqref{est:energy}. In this section, before we choose $\wt{C}_1$, all the constants $C$ and the implicit constants in $\lesssim$ are independent of $\wt{C}_1$.

The function $\wt{q}$ is defined in Section~\ref{sec:construction:wtq}. Note that $w$ is not defined if $r-t\geq 2$ (where $\wt{q}=r-t$). We can extend $w$ to $[1,\Tboot]\times\R^2$ as follows. Fix a cutoff function $\chi\in C^\infty(\R)$ with $\chi\geq 0$, $\chi\restriction_{s\leq 1}\equiv 1$, and $\chi\restriction_{s\geq \frac{3}{2}}\equiv 0$. Replace $w$ with $\wt{w}=w\cdot\chi(\wt{q})$.
However, $w$ will always be multiplied by a function that vanishes for $\wt{q}\geq 1$ in the rest of the paper, so we still use $w$ instead of $\wt{w}$ in the following computations.

For a function $\psi=\psi(t,x)$ with $\supp\psi\subset\{r\leq t+1\}$, we set
\eq{\norm{\psi(t)}_{L^2(w)}:=\kh{\int_{\R^2} |\psi(t,x)|^2w\,dx}^{\frac{1}{2}}.}

In this section, we derive weighted energy estimates and Poincaré's estimates. 

\subsection{Energy estimates}
\begin{lem}
Let $\phi$ be a function of $(t,x)\in\R^{1+2}_+$ that vanishes for $r\geq t+1$. 
Then, for all $1\leq t_1\leq t_2\leq\Tboot$, we have
\eq{\label{est:energy}\norm{\partial \phi(t_2)}_{L^2(w)}&\lesssim\norm{\partial \phi(t_1)}_{L^2(w)} +\int_{t_1}^{t_2}M_0^2\eps^2\lra{\tau}^{-1}\norm{\partial \phi(\tau)}_{L^2(w)}\, d\tau +\int_{t_1}^{t_2}\norm{(\wt{\Box}_g\phi)(\tau)}_{L^2(w)}\, d\tau.}
\end{lem}
\begin{proof}
The proof is similar to that of \cite[Lemma 7.1]{MR2382144}. First, by \cite[(7.7)]{MR2382144}, we have
\eq{\label{est:7.7:lindblad}
\frac{d}{dt}\int_{\R^2}(-g^{00}\phi_t^2+g^{ij}\phi_i\phi_j)w\,dx&=\int_{\R^2} \kh{-2\phi_t\wt{\Box}_g\phi+(\partial_tg^{\alpha\beta})\phi_\alpha\phi_\beta-2(\partial_\alpha g^{\alpha\beta})\phi_t\phi_\beta}w\,dx\\
&\quad+\int_{\R^2} \kh{g^{\alpha\beta}\phi_\alpha\phi_\beta w_t-2g^{\alpha\beta}\phi_t\phi_\alpha w_\beta} \,dx.
}
By Lemma~\ref{lem:ptb:everywhere:u:du:ddu}, we have $|\partial g^{\alpha\beta}|\lesssim |u||\partial u|\lesssim M_0^2\eps^2\lra{t}^{-1}\lra{\wt{q}}^{-1+\frac{2\delta}{\lambda}}\lesssim M_0^2\eps^2\lra{t}^{-1}$ everywhere.  Moreover, we have \fm{0\leq e_0(\phi):= -g^{00}\phi_t^2+g^{ij}\phi_i\phi_j\sim |\partial\phi|^2} because $|g^{\alpha\beta}-m^{\alpha\beta}|\lesssim |u|^2$. Our goal is to prove
\fm{\int_{\R^2} \kh{g^{\alpha\beta}\phi_\alpha\phi_\beta w_t-2g^{\alpha\beta}\phi_t\phi_\alpha w_\beta} \,dx\leq CM_0^2\eps^2\lra{t}^{-1}\int_{\R^2}|\partial\phi|^2 w\,dx.}
Note that we will only prove this upper bound. If this goal is achieved, we have
\fm{\frac{d}{dt}\norm{e_0(\phi)^{\frac{1}{2}}}_{L^2(w)}^{2}&\leq C \norm{\partial\phi}_{L^2(w)} \norm{\wt{\Box}_g\phi}_{L^2(w)} +CM_0^2\eps^2\lra{t}^{-1}\norm{\partial\phi}_{L^2(w)}^{2}\\
&\leq C \norm{e_0(\phi)^{\frac{1}{2}}}_{L^2(w)} \norm{\wt{\Box}_g\phi}_{L^2(w)} +CM_0^2\eps^2\lra{t}^{-1}\norm{e_0(\phi)^{\frac{1}{2}}}_{L^2(w)}^{2}.}
Divide both sides by $\norm{e_0(\phi)^{\frac{1}{2}}}_{L^2(w)}$, take the integral with respect to time, and use $e_0(\phi)\sim |\partial \phi|^2$. This yields \eqref{est:energy}.

Next,
\eq{w_t&=-\wt{q}_r^{-1}\partial_t\wt{q}_r\cdot w+\kappa_1\kappa_2\eps^2\ln(1+t)(2-\wt{q})^{-1-\kappa_2}\wt{q}_t\cdot w+\kappa_1\eps^2(1+t)^{-1}(2-\wt{q})^{-\kappa_2}\cdot w,\\
w_i&=-\wt{q}_r^{-1}\partial_i\wt{q}_r\cdot w+\kappa_1\kappa_2\eps^2\ln(1+t)(2-\wt{q})^{-1-\kappa_2}\wt{q}_i\cdot w.}
We thus have
\fm{g^{\alpha\beta}\phi_\alpha\phi_\beta w_t-2g^{\alpha\beta}\phi_t\phi_\alpha w_\beta&=\wt{q}_r^{-1}\kh{-g^{\alpha\beta}\phi_\alpha\phi_\beta \partial_t\wt{q}_r+2g^{\alpha\beta}\phi_t\phi_\alpha \partial_\beta\wt{q}_r}\cdot w\\
&\quad+\kappa_1\kappa_2\eps^2\ln(1+t)\cdot(2-\wt{q})^{-1-\kappa_2}\kh{g^{\alpha\beta}\phi_\alpha\phi_\beta \wt{q}_t-2g^{\alpha\beta}\phi_t\phi_\alpha \wt{q}_\beta}\cdot w\\
&\quad+\kappa_1\eps^2(1+t)^{-1}(2-\wt{q})^{-\kappa_2}\kh{g^{\alpha\beta}\phi_\alpha\phi_\beta -2g^{\alpha 0}\phi_t\phi_\alpha  }\cdot w\\
&:=\mcl{R}_1w+\mcl{R}_2w+\mcl{R}_3w.}
We claim that\footnote{Strictly speaking, since $\wt{q}_i=\omega_i$, the left side of \eqref{est:energy:claim} is not well-defined at $r=0$. However, we note that $|\partial\wt{q}|\lesssim 1$ and $\partial\wt{q}_r=0$ near $r=0$. Moreover, we will show that \eqref{est:energy:claim} holds everywhere in $\supp\phi$ except at $r=0$. We then take an integral in \eqref{est:7.7:lindblad}, so it is not an issue that $\partial\wt{q}$ is not defined at $r=0$.}
\eq{\label{est:energy:claim}\mcl{R}_1 +\mcl{R}_2 +\mcl{R}_3&\leq CM_0^2\eps^2\lra{t}^{-1}|\partial\phi|^2}
if we choose appropriate $\wt{C}_1,\kappa_2$; recall that $\kappa_1=\wt{C}_1M_0^2$. The constant $C$ is independent of $M_0,\eps$. This will finish the proof of \eqref{est:energy}.

We start by computing $\mcl{R}_2$. 
Set $\wh{q}_\alpha=\frac{\wt{q}_\alpha}{\wt{q}_t}$ and $\wh{\phi}_\alpha=\phi_\alpha-\wh{q}_\alpha\phi_t$. We now have
\fm{&g^{\alpha\beta}\phi_\alpha\phi_\beta\wt{q}_t-2g^{\alpha\beta}\phi_\alpha\phi_t \wt{q}_\beta=g^{\alpha\beta}(\wh{\phi}_\alpha+\wh{q}_\alpha\phi_t)(\wh{\phi}_\beta+\wh{q}_\beta\phi_t)\wt{q}_t-2g^{\alpha\beta}(\wh{\phi}_\alpha+\wh{q}_\alpha\phi_t)\phi_t\wh{q}_\beta\wt{q}_t\\
&=g^{\alpha\beta}\wh{\phi}_\alpha\wh{\phi}_\beta\wt{q}_t+2g^{\alpha\beta}\wh{q}_\alpha\phi_t\wh{\phi}_\beta\wt{q}_t+g^{\alpha\beta}\wh{q}_\alpha \wh{q}_\beta \phi_t^2\wt{q}_t-2g^{\alpha\beta} \wh{\phi}_\alpha \phi_t\wh{q}_\beta\wt{q}_t-2g^{\alpha\beta} \wh{q}_\alpha\phi_t^2\wh{q}_\beta\wt{q}_t\\
&=g^{ij}\wh{\phi}_i\wh{\phi}_j\wt{q}_t-g^{\alpha\beta}\wh{q}_\alpha \wh{q}_\beta \phi_t^2\wt{q}_t=-g^{ij}\wh{\phi}_i\wh{\phi}_j|\wt{q}_t|+g^{\alpha\beta}\wt{q}_\alpha \wt{q}_\beta \phi_t^2\abs{\wt{q}_t}^{-1}.}
Here we use $\wh{\phi}_0=0$. By \eqref{est:eik:wtq}, we have
\fm{\abs{g^{\alpha\beta}\wt{q}_\alpha \wt{q}_\beta \phi_t^2\abs{\wt{q}_t}^{-1}}&\lesssim\frac{M_0^2\eps^2\lra{\wt{q}}^{\frac{5\delta}{\lambda}}}{\lra{t}^{1+\delta}}\cdot  \phi_t^2.}
Since $g^{ij}=m^{ij}+O(|u|^2)$, we have $g^{ij} \wh{\phi}_i \wh{\phi}_j\sim|\wh{\phi}|^2\geq 0$. Thus,
\fm{\mcl{R}_2&= \kappa_1\kappa_2\eps^2\ln(1+t)\cdot (2-\wt{q})^{-1-\kappa_2}\cdot \kh{-g^{ij}\wh{\phi}_i\wh{\phi}_j|\wt{q}_t|+g^{\alpha\beta}\wt{q}_\alpha \wt{q}_\beta \phi_t^2\abs{\wt{q}_t}^{-1}}\\
&\leq -C^{-1}\kappa_1\kappa_2\eps^2\ln(1+t)\cdot (2-\wt{q})^{-1-\kappa_2}\cdot|\wt{q}_t||\wh{\phi}|^2\\
&\quad+C\kappa_1\kappa_2\eps^2\ln(1+t)\cdot (2-\wt{q})^{-1-\kappa_2}\cdot \frac{M_0^2\eps^2\lra{\wt{q}}^{\frac{5\delta}{\lambda}}}{\lra{t}^{1+\delta}}\cdot  \phi_t^2\\
&\leq  -C^{-1}\kappa_1\kappa_2\eps^2\ln(1+t)\cdot \lra{\wt{q}}^{-2+\frac{2\delta}{\lambda}}\cdot|\wt{q}_t||\wh{\phi}|^2\\
&\quad+C\kappa_1\kappa_2\eps^2\cdot  \frac{\ln(1+t)}{\lra{t}^{1+\delta}}\cdot M_0^2\eps^2\lra{\wt{q}}^{-2+\frac{7\delta}{\lambda}}\cdot  \phi_t^2.}
Here we use $\kappa_2=1-\frac{2\delta}{\lambda}$. Since $\ln(1+t)\lesssim \lra{t}^{\frac{\delta}{2}}$, we have
\eq{\label{est:energy:R2}\mcl{R}_2&\leq  -C^{-1}\kappa_1\kappa_2\eps^2\ln(1+t)\cdot \lra{\wt{q}}^{-2+\frac{2\delta}{\lambda}}\cdot|\wt{q}_t||\wh{\phi}|^2\\
&\quad+C\kappa_1\kappa_2\eps^2\lra{t}^{-1}\cdot M_0^2\eps^2\lra{\wt{q}}^{-2+\frac{7\delta}{\lambda}}\lra{t}^{-\frac{\delta}{2}}\cdot  \phi_t^2\\
&\leq -C^{-1}\kappa_1\kappa_2\eps^2\ln(1+t)\cdot \lra{\wt{q}}^{-2+\frac{2\delta}{\lambda}}\cdot|\wt{q}_t||\wh{\phi}|^2 +CM_0^2\kappa_1\kappa_2\eps^4\lra{t}^{-1}\cdot  \phi_t^2.}

Next, we compute $\mcl{R}_1$. We have 
\eq{\label{pf:est:energy:r1}&\mcl{R}_1=\wt{q}_r^{-1}\kh{-g^{\alpha\beta}\phi_\alpha\phi_\beta\partial_t\wt{q}_r+2g^{\alpha\beta}\phi_t\phi_\alpha\partial_\beta\wt{q}_r}\\
&=\wt{q}_r^{-1}\kh{-g^{\alpha\beta}(\wh{\phi}_\alpha+\wh{q}_\alpha\phi_t)(\wh{\phi}_\beta+\wh{q}_\beta\phi_t)\partial_t\wt{q}_r+2g^{\alpha\beta}\phi_t(\wh{\phi}_\alpha+\wh{q}_\alpha\phi_t)\partial_\beta\wt{q}_r}\\
&=\wt{q}_r^{-1}\kh{-g^{\alpha\beta} \wh{\phi}_\alpha \wh{\phi}_\beta\partial_t\wt{q}_r-g^{\alpha\beta} \wh{q}_\alpha \wh{q}_\beta\phi_t^2\partial_t\wt{q}_r -2g^{\alpha\beta}\phi_t\wh{\phi}_\alpha (\wh{q}_\beta \partial_t\wt{q}_r- \partial_\beta\wt{q}_r)+2g^{\alpha\beta}\phi_t^2 \wh{q}_\alpha  \partial_\beta\wt{q}_r}\\
&=-\wt{q}_r^{-1}g^{ij} \wh{\phi}_i \wh{\phi}_j\partial_t\wt{q}_r-\wt{q}_r^{-1}\phi_t^2|\wt{q}_t|^{-2}g^{\alpha\beta} \wt{q}_\alpha \wt{q}_\beta\partial_t\wt{q}_r \\
&\quad+2\wt{q}_r^{-1}|\wt{q}_t|^{-1}g^{\alpha\beta}\phi_t\wh{\phi}_\alpha (\wt{q}_\beta \partial_t\wt{q}_r- \wt{q}_t\partial_\beta\wt{q}_r)-2\wt{q}_r^{-1}|\wt{q}_t|^{-1}\phi_t^2 g^{\alpha\beta}\wt{q}_\alpha  \partial_\beta\wt{q}_r.}
By \eqref{est:second:wtq} and since $|\wt{q}_t|\sim\wt{q}_r>0$, we have
\fm{-\wt{q}_r^{-1}g^{ij} \wh{\phi}_i \wh{\phi}_j\partial_t\wt{q}_r&\leq \wt{q}_r\cdot g^{ij} \wh{\phi}_i \wh{\phi}_j\cdot CM_0^2\eps^2\lra{\wt{q}}^{-2+\frac{2\delta}{\lambda}}\cdot \ln(1+t)\\
&\leq  CM_0^2\eps^2\ln(1+t)\cdot \lra{\wt{q}}^{-2+\frac{2\delta}{\lambda}}\cdot |\wt{q}_t||\wh{\phi}|^2.}
We will use the nonpositive part in \eqref{est:energy:R2} to absorb this term. We thus need $-1-\kappa_2\geq -2+\frac{2\delta}{\lambda}$.
Moreover,
\fm{\wt{q}_r^{-1}\phi_t^2|\wt{q}_t|^{-2}\abs{g^{\alpha\beta} \wt{q}_\alpha \wt{q}_\beta\partial_t\wt{q}_r}&\lesssim \wt{q}_r^{-1}\phi_t^2\abs{\wt{q}_t}^{-2}|g^{\alpha\beta} \wt{q}_\alpha \wt{q}_\beta|\cdot \lra{\wt{q}}^{-2+\frac{2\delta}{\lambda}}\lra{t}^{CM_0^2\eps^2}\wt{q}_r^2\\
&\lesssim \phi_t^2 \cdot |\wt{q}_t|^{-1}|g^{\alpha\beta} \wt{q}_\alpha \wt{q}_\beta|\cdot \lra{\wt{q}}^{-2+\frac{2\delta}{\lambda}}\lra{t}^{CM_0^2\eps^2}\\
&\lesssim \phi_t^2 \cdot \frac{M_0^2\eps^2\lra{\wt{q}}^{\frac{5\delta }{\lambda}}}{\lra{t}^{1+\delta}} \cdot \lra{\wt{q}}^{-2+\frac{2\delta}{\lambda}}\lra{t}^{CM_0^2\eps^2}\\
&\lesssim \phi_t^2 \cdot  M_0^2\eps^2\lra{\wt{q}}^{-2+\frac{7\delta}{\lambda}}  \cdot  \lra{t}^{-1-\delta+CM_0^2\eps^2}\lesssim   M_0^2\eps^2\lra{t}^{-1 }\phi_t^2.}
By \eqref{est:dwtq:ddwtq:ptang}, we have
\fm{&2\wt{q}_r^{-1}|\wt{q}_t|^{-1}\abs{g^{\alpha\beta}\phi_t\wh{\phi}_\alpha (\wt{q}_\beta \partial_t\wt{q}_r- \wt{q}_t\partial_\beta\wt{q}_r)}\lesssim |\wt{q}_t|^{-2}|\phi_t||\wh{\phi}|\sum_\beta|\wt{q}_\beta \partial_t\wt{q}_r- \wt{q}_t\partial_\beta\wt{q}_r|\\
&\lesssim   M_0^2\eps^2 \ln(1+t) \cdot \lra{\wt{q}}^{-2+\frac{2\delta}{\lambda}}\cdot |\wt{q}_t| |\wh{\phi}|^2+\sum_\beta\frac{|\wt{q}_\beta \partial_t\wt{q}_r- \wt{q}_t\partial_\beta\wt{q}_r|^2\cdot \lra{\wt{q}}^{2-\frac{2\delta}{\lambda}}}{M_0^2\eps^2\ln(1+t)\cdot |\wt{q}_t|^5}\phi_t^2\\
&\lesssim   M_0^2\eps^2 \ln(1+t) \cdot\lra{\wt{q}}^{-2+\frac{2\delta}{\lambda}}\cdot  |\wt{q}_t||\wh{\phi}|^2+\frac{M_0^2\eps^2\lra{t}^{-2+10\delta}\lra{\wt{q}}^{ \frac{2\delta}{\lambda}}}{ \ln(1+t)\cdot |\wt{q}_t|^5}\phi_t^2.}
In $\supp\phi_t$, we have $\lra{\wt{q}}\lesssim \lra{t}$ and $|\wt{q}_t|\sim\wt{q}_r\gtrsim\lra{t}^{-CM_0^2\eps^2}$. As a result, 
\fm{& 2\wt{q}_r^{-1}|\wt{q}_t|^{-1}\abs{g^{\alpha\beta}\phi_t\wh{\phi}_\alpha (\wt{q}_\beta \partial_t\wt{q}_r- \wt{q}_t\partial_\beta\wt{q}_r)}
\\
&\lesssim M_0^2\eps^2 \ln(1+t) \cdot\lra{\wt{q}}^{-2+\frac{2\delta}{\lambda}}\cdot  |\wt{q}_t||\wh{\phi}|^2+\frac{M_0^2\eps^2\lra{t}^{-2+10\delta+\frac{2\delta}{\lambda}+CM_0^2\eps^2} }{ \ln(1+t) }\phi_t^2\\
&\lesssim M_0^2\eps^2 \ln(1+t) \cdot\lra{\wt{q}}^{-2+\frac{2\delta}{\lambda}}\cdot  |\wt{q}_t||\wh{\phi}|^2+ M_0^2\eps^2\lra{t}^{-1}  \phi_t^2.}
By \eqref{est:dwtq:ddwtq:ptang:2}, we have
\fm{\abs{-2\wt{q}_r^{-1}|\wt{q}_t|^{-1}\phi_t^2 g^{\alpha\beta}\wt{q}_\alpha  \partial_\beta\wt{q}_r}&\lesssim M_0^2\eps^2\lra{t}^{-1} \phi_t^2.}
In summary, we have proved that
\fm{\mcl{R}_1&\leq CM_0^2\eps^2\ln(1+t)\cdot \lra{\wt{q}}^{-2+\frac{2\delta}{\lambda}}\cdot |\wt{q}_t||\wh{\phi}|^2 +CM_0^2\eps^2\lra{t}^{-1} \phi_t^2.}
We emphasize that the constants $C$ in \eqref{est:energy:R2} and in the estimate for $\mcl{R}_1$ are independent of $\wt{C}_1$. In fact, we keep $\kappa_1$ unchanged in \eqref{est:energy:R2} and the definition of $\mcl{R}_1$ does not involve $\kappa_1$ and $w$ at all.
By comparing this with \eqref{est:energy:R2}, we set $\kappa_1=\wt{C}_1M_0^2$ where $\wt{C}_1>1$ is sufficiently large so that \eq{C^{-1}\wt{C}_1 (1-\frac{2\delta}{\lambda})\geq C.}
This guarantees that the first term in the estimate for $\mcl{R}_1$ can be absorbed by the nonpositive term in \eqref{est:energy:R2}. In this case, we also have
\fm{|\mcl{R}_3|&\lesssim M_0^2\eps^2\lra{t}^{-1}\lra{\wt{q}}^{-\kappa_2}|\partial\phi|^2\lesssim M_0^2\eps^2\lra{t}^{-1} |\partial\phi|^2.}
We thus finish the proof of \eqref{est:energy:claim} and \eqref{est:energy}.
\end{proof}

\subsection{Poincaré's estimates}
We now present two versions of Poincaré's estimates that will be used later in our proof. One is unweighted, and the other is weighted. See, e.g., \cite{MR2003417,MR2382144}.

\begin{lem}\label{lem:poincare:unweighted}
Let $\phi$ be a function of $(t,x)\in\R^{1+2}_+$ that vanishes for $r\geq t+1$. Then, we have
\eq{\int_{\R^2}\frac{|\phi(t,x)|^2}{\lra{|x|-t}^{2}}\,dx\lesssim  \int_{\R^2} |\phi_r(t,x)|^2 \,dx.}
The implicit constant is universal.

Moreover, since $\lra{t}^{-CM_0^2\eps^2}\lesssim w\lesssim \lra{t}^{CM_0^2\eps^2}$ in $\supp\phi$, we have
\eq{\norm{\frac{\phi}{\lra{r-t}}}_{L^2(w)}^2\lesssim \lra{t}^{CM_0^2\eps^2}\int_{\R^2}\frac{|\phi(t,x)|^2}{\lra{|x|-t}^{2}}\,dx\lesssim \lra{t}^{CM_0^2\eps^2} \int_{\R^2} |\phi_r(t,x)|^2 \,dx\lesssim \lra{t}^{CM_0^2\eps^2}\norm{\partial_r\phi}_{L^2(w)}^2.}
\end{lem}
\begin{proof}
First, notice that $\lra{r-t}\sim (2-r+t)$ in $\supp\phi$. 
We have
\fm{\int_{\R^2}\frac{|\phi(t,x)|^2}{(2-|x|+t)^2}\,dx&=\int_{\mathbb{S}^1}\int_0^{t+1} r \phi(t,r \omega)^2\cdot(2-r +t)^{-2}\,dr dS_\omega\\
&=\int_{\mathbb{S}^1}\int_0^{t+1} r \phi(t,r \omega)^2\cdot\partial_{r }(2-r +t)^{-1}\,dr dS_\omega\\
&=\int_{\mathbb{S}^1}\int_0^{t+1} -\partial_r\kh{r \phi^2}\cdot(2-r +t)^{-1}\,dr dS_\omega\\
&=\int_{\mathbb{S}^1}\int_0^{t+1} -\kh{\phi^2+2r\phi\phi_r}\cdot(2-r +t)^{-1}\,dr dS_\omega\\
&\leq -2\int_{\mathbb{S}^1}\int_0^{t+1} r\phi\phi_r \cdot(2-r +t)^{-1}\,dr dS_\omega\\
&=-2\int_{\R^2} \phi\phi_r \cdot (2-|x| +t)^{-1}\,dx.}
By the Cauchy--Schwarz inequality, we have
\fm{-2\int_{\R^2} \phi\phi_r \cdot (2-|x| +t)^{-1}\,dx\leq 2\kh{\int_{\R^2}\frac{|\phi(t,x)|^2}{(2-|x|+t)^{2}}\,dx}^{\frac{1}{2}}\kh{\int_{\R^2} |\phi_r(t,x)|^2 \,dx}^{\frac{1}{2}}.}
This finishes the proof.
\end{proof}

\begin{lem}\label{lem:poincare:weighted}
Let $\phi$ be a function of $(t,x)\in\R^{1+2}_+$ that vanishes for $r\geq t+1$. Then, we have
\eq{\norm{\phi\lra{\wt{q}}^{-1}\wt{q}_r}_{L^2(w)}^2\sim\int_{\R^2}\frac{|\phi|^2}{(2-\wt{q})^2}\wt{q}_r^2w\,dx\lesssim  \int_{\R^2} |\phi_r(t,x)|^2w \,dx\sim \norm{\partial_r\phi}_{L^2(w)}^2.}
The implicit constant is universal.
\end{lem}
\begin{proof}
First, notice that $\lra{\wt{q}}\sim (2-\wt{q})$ in $\supp\phi$.  We have
\fm{\int_{\R^2}\frac{|\phi|^2}{(2-\wt{q})^2}\wt{q}_r^2w\,dx&=\int_{\mathbb{S}^1}\int_0^{t+1}\frac{|\phi|^2}{(2-\wt{q})^2}r\wt{q}_r^2w\,drdS_\omega\\
&=\int_{\mathbb{S}^1}\int_0^{t+1} \phi^2 r\wt{q}_r w\cdot\partial_r\kh{(2-\wt{q})^{-1}}\,drdS_\omega\\
&=\int_{\mathbb{S}^1}\int_0^{t+1} - \partial_r\kh{\phi^2 r\wt{q}_r w}\cdot  (2-\wt{q})^{-1}\,drdS_\omega.}
By \eqref{def:weight}, we have
\fm{\wt{q}_rw&=\exp\kh{\kappa_1 \eps^2\ln(1+t)\cdot (2-\wt{q}(t,x))^{-\kappa_2}}}
and thus
\fm{&- \partial_r\kh{\phi^2 r\wt{q}_r w}=-2\phi\phi_r r\wt{q}_rw- \phi^2\wt{q}_rw-\phi^2r \partial_r\exp\kh{\kappa_1 \eps^2\ln(1+t)\cdot (2-\wt{q})^{-\kappa_2}}\\
&\leq -2\phi\phi_r r\wt{q}_rw- \phi^2r \cdot \kappa_1\kappa_2\eps^2\ln(1+t)\cdot (2-\wt{q})^{-1-\kappa_2}\wt{q}_r\cdot\exp\kh{\kappa_1 \eps^2\ln(1+t)\cdot (2-\wt{q} )^{-\kappa_2}}\\
&\leq-2\phi\phi_r r\wt{q}_rw.}
Here we use $\wt{q}_r>0$. Then,
\fm{\int_{\R^2}\frac{|\phi|^2}{(2-\wt{q})^2}\wt{q}_r^2w\,dx&\leq \int_{\mathbb{S}^1}\int_0^{t+1} -2\phi\phi_r r\wt{q}_rw\cdot (2-\wt{q})^{-1} \,drdS_\omega= \int_{\R^2} -2\phi\phi_r  \wt{q}_rw\cdot (2-\wt{q})^{-1} \,dx\\
&\leq 2\kh{\int_{\R^2}\frac{|\phi|^2}{(2-\wt{q})^2}\wt{q}_r^2w\,dx}^{\frac{1}{2}}\kh{\int_{\R^2} |\phi_r|^2 w\,dx}^{\frac{1}{2}}.}
This finishes the proof.
\end{proof}

\section{Higher-order pointwise estimates}\label{sec:higher:ptb}

We now prove higher-order pointwise estimates for $u$. At the end of this section, we will obtain bounds for $Z^{\leq N-6} u$ for all $t\in[1,\Tboot]$ and $x\in\R^2$. We will also estimate $\wt{\Box}_g\partial^kZ^{m}u$ with $k+m\leq N$.

\subsection{Estimates in $\D$}

\begin{prop}
In $\D$, for all $k+m\leq N-6$, we have
\eq{\label{est:pt:higher:u}|\partial^{k}Z^{m}u|&\lesssim M_0\eps\lra{t}^{-\frac{1}{2}+CM_0^2\eps^2}\lra{q}^{\frac{\delta}{\lambda}-k}.}
\end{prop}

\bigskip
We will prove \eqref{est:pt:higher:u} when $k\geq 1$ and $k+m\leq N-5$. In fact, if we have proved
\fm{|\partial Z^{m}u|&\lesssim M_0\eps\lra{t}^{-\frac{1}{2}+CM_0^2\eps^2}\lra{q}^{\frac{\delta}{\lambda}-1}} for $m\leq N-6$ in $\D$, then by Lemma~\ref{lem:small:r-t:q:D}, we have
\fm{|\partial Z^{m}u|&\lesssim M_0\eps\lra{t}^{-\frac{1}{2}+CM_0^2\eps^2}\lra{r-t}^{\frac{\delta}{\lambda}-1}.}
Since $u\equiv 0$ for $r\geq t+1$, we obtain \eqref{est:pt:higher:u} with $k=0$ by integrating from $r=t+1$ to inside and applying $\lra{r-t}\lesssim \lra{q}$ again. 

\subsubsection{The case $k+m\leq 2$ and $m\leq 1$}\label{sec:ptb:pf:m=1,k=1}
By \eqref{est:ptb:u}, \eqref{est:ptb:du}, and \eqref{est:ptb:ddu}, we have proved \eqref{est:pt:higher:u} in $\D$ for $m=0$ and $k\leq 2$. Here we notice that $0<\muring\lesssim \lra{t}^{CM_0^2\eps^2}$. It remains to estimate $|\partial Zu|$ and $|Zu|$. The proof is similar to that in \cite[Section 6.3]{MR2382144}.

\begin{lem}
For a function $\phi=\phi(t,x)$,  in $\D$ we have 
\eq{\label{est:transport:r12muuLphi}\abs{r^{\frac{1}{2}}\muring^{-1}\wt{\Box}_g\phi-\wt{L}\kh{r^{\frac{1}{2}}\muuLunit\phi}}&\lesssim M_0^2\eps^2\lra{t}^{-\frac{1}{2}}\lra{q}^{-1+\frac{2\delta}{\lambda}}|\muuLunit\phi|+
\muring^{-1}\lra{t}^{-\frac{3}{2}}|Z^{1\leq \cdot\leq 2}\phi|.}
\end{lem}
\begin{proof}
By \eqref{est:ptb:u}, we have $|h|\lesssim|u|^2\lesssim M_0^2\eps^2\lra{t}^{-1}\lra{r-t}^{\frac{2\delta}{\lambda}}\lesssim\frac{\lra{r-t}}{\lra{r+t}}$, so we can apply \eqref{est:transport:r12uLphi} to obtain
\fm{\abs{r^{\frac{1}{2}}\muring^{-1}\wt{\Box}_g\phi-\muring^{-1}\wt{L}\kh{r^{\frac{1}{2}}\uLunit\phi}}&\lesssim   \muring^{-1}\lra{r+t}^{-\frac{3}{2}}|Z^{1\leq \cdot\leq 2}\phi|.}
By \eqref{eqn:transport:mu}, we have
\fm{\wt{L}\kh{r^{\frac{1}{2}}\muuLunit\phi}&=\muring^{-1}\wt{L}\kh{r^{\frac{1}{2}}\uLunit\phi}-\muring^{-2}\wt{L}\muring \cdot \kh{r^{\frac{1}{2}}\uLunit\phi}\\
&=r^{\frac{1}{2}}\muring^{-1}\wt{\Box}_g\phi+\frac{1}{4}(\uLunit h)_{LL} \cdot  r^{\frac{1}{2}}\muuLunit\phi +\mcl{O}\kh{\muring^{-1}\lra{r+t}^{-\frac{3}{2}}|Z^{1\leq \cdot\leq 2}\phi|}\\
&=r^{\frac{1}{2}}\muring^{-1}\wt{\Box}_g\phi+\mcl{O}\kh{M_0^2\eps^2\lra{t}^{-\frac{1}{2}}\lra{q}^{-1+\frac{2\delta}{\lambda}}|\muuLunit\phi| +\muring^{-1}\lra{r+t}^{-\frac{3}{2}}|Z^{1\leq \cdot\leq 2}\phi|}.}
Here we use $|\partial h|\lesssim |u||\partial u|$.
\end{proof}

By setting $\phi=Zu$ in \eqref{est:transport:r12muuLphi}, in $\D$ we have
\fm{\abs{\wt{L}\kh{r^{\frac{1}{2}}\muuLunit Zu}}&\lesssim \abs{r^{\frac{1}{2}}\muring^{-1}\wt{\Box}_gZu}+M_0^2\eps^2\lra{t}^{-\frac{1}{2}}\lra{q}^{-1+\frac{2\delta}{\lambda}}|\muuLunit Zu|+
\muring^{-1}\lra{ t}^{-\frac{3}{2}}|Z^{ \leq 3}u|\\
&\lesssim \lra{t}^{\frac{1}{2}}\muring^{-1}\abs{[\wt{\Box}_g,Z]u}+M_0^2\eps^2\lra{t}^{-\frac{1}{2}}\lra{q}^{-1+\frac{2\delta}{\lambda}}|\muuLunit Zu|+
M_0\eps\lra{t}^{-2+\delta+CM_0^2\eps^2}
\\
&\lesssim \lra{t}^{\frac{1}{2}}\muring^{-1} |Z^{\leq 1}h|\cdot |\partial^2u|+M_0^2\eps^2\lra{t}^{-\frac{1}{2}}\lra{q}^{-1+\frac{2\delta}{\lambda}}|\muuLunit Zu|+
M_0\eps\lra{t}^{-2+\delta+CM_0^2\eps^2}.}
By \eqref{est:ptb:ddu} and \eqref{est:ptb:u}, we have 
\fm{\muring^{-1}\cdot \lra{t}^{\frac{1}{2}} |Z^{\leq 1}h|\cdot |\partial^2u|&\lesssim \lra{t}^{\frac{1}{2}}|u|\kh{|u|+|Zu|}\cdot \muring^{-1}|\partial^2u|\\
&\lesssim M_0\eps\lra{q}^{\frac{\delta}{\lambda}}\kh{M_0\eps\lra{t}^{-\frac{1}{2}}\lra{q}^{\frac{\delta}{\lambda}}+|Zu|}\cdot M_0\eps\lra{t}^{-\frac{1}{2}}\lra{q}^{-2+\frac{\delta}{\lambda}}\\
&\lesssim M_0^2\eps^2\lra{t}^{-\frac{1}{2}}\lra{q}^{-2+\frac{2\delta}{\lambda}}|Zu|+M_0^3\eps^3\lra{t}^{-1+CM_0^2\eps^2}\lra{q}^{-2+\frac{3\delta}{\lambda}}.}
Although we can replace $\lra{t}^{-1+CM_0^2\eps^2}$ with $\lra{t}^{-1}$ in the second term, we deliberately use this weaker bound to simplify the proof in Section~\ref{sec:ptb:pf:m>0}.
In summary, we have
\eq{\label{est:transport:higher:r12muuLu}\abs{\wt{L}\kh{r^{\frac{1}{2}}\muuLunit Zu}}&\lesssim M_0^2\eps^2\lra{t}^{-\frac{1}{2}}\lra{q}^{-1+\frac{2\delta}{\lambda}}\kh{|\muuLunit Zu|+\lra{q}^{-1}|Zu|}\\
&\quad+
M_0\eps\lra{t}^{-2+\delta+CM_0^2\eps^2}+M_0^3\eps^3\lra{t}^{-1+CM_0^2\eps^2}\lra{q}^{-2+\frac{3\delta}{\lambda}}.}
On $\Dlower$, we have 
\fm{r^{\frac{1}{2}}|\muuLunit Zu|&\lesssim \lra{t}^{\frac{1}{2}}\cdot 1\cdot |\partial Zu|\lesssim \lra{t}^{\frac{1}{2}}\lra{r-t}^{-1} \cdot |Z^{\leq 2}u|\lesssim M_0\eps\lra{t}^{\delta}\lra{r-t}^{-1}\lesssim M_0\eps\lra{t}^{-\lambda+\delta}.}
Here we use $\muring\sim 1$ and $\lra{r-t}\sim \lra{t}^\lambda$ on $\Dlower$. By integrating along $X(s)$, we obtain
\fm{&\lra{q}^{1-\frac{\delta}{\lambda}}\abs{\lra{t}^{\frac{1}{2}}\muuLunit Zu(t,x)}\lesssim \lra{q}^{1-\frac{\delta}{\lambda}}\abs{r^{\frac{1}{2}}\muuLunit Zu(t,x)}\\
&\lesssim \int_{t_0}^tM_0^2\eps^2\lra{\tau}^{-\frac{1}{2}}\lra{q}^{-1+\frac{2\delta}{\lambda}}\kh{\lra{q}^{1-\frac{\delta}{\lambda}}|\muuLunit Zu(X(\tau))|+\lra{q}^{-\frac{\delta}{\lambda}}|Zu(X(\tau))|}\,d\tau\\
&  +\lra{q}^{1-\frac{\delta}{\lambda}}\cdot M_0\eps\lra{t_0}^{-\lambda+\delta}+M_0\eps\lra{q}^{1-\frac{\delta}{\lambda}}\lra{t_0}^{-1+\delta+CM_0^2\eps^2}+M_0 \eps \lra{t}^{CM_0^2\eps^2}\cdot \lra{q}^{-1+\frac{2\delta}{\lambda}}\\
&\lesssim \int_{t_0}^tM_0^2\eps^2\lra{\tau}^{-1}\kh{\lra{\tau}^{\frac{1}{2}}|(\lra{q}^{1-\frac{\delta}{\lambda}}\muuLunit Zu)(X(\tau))|+\lra{\tau}^{\frac{1}{2}}|( \lra{q}^{-\frac{\delta}{\lambda}}Zu)(X(\tau))|}\,d\tau+M_0\eps \lra{t}^{CM_0^2\eps^2}.}
Here we use $\lra{q}\sim \lra{t_0}^\lambda$.

We now claim that
\eq{\label{est:pf:claim:Zu:dZu}\norm{   \lra{q}^{-\frac{\delta}{\lambda}}|Zu|}_{L^\infty(\D_t)}&\lesssim \norm{ \lra{q}^{1-\frac{\delta}{\lambda}}|\muuLunit Zu|}_{L^\infty(\D_t)}+M_0\eps\lra{t}^{-\frac{1}{2}-2\delta+CM_0^2\eps^2}.}
If this is true, then by setting $\mcl{I}(t):=\lra{t}^{\frac{1}{2}}\norm{   \lra{q}^{1-\frac{\delta}{\lambda}}|\muuLunit Zu|}_{L^\infty(\D_t)}$, we have
\fm{\mcl{I}(t)&\lesssim\int_{1}^t M_0^2\eps^2\lra{\tau}^{-1}\mcl{I}(\tau)+M_0^2\eps^2\lra{\tau}^{-1}\cdot M_0\eps\lra{\tau}^{-2\delta+CM_0^2\eps^2}\,d\tau+M_0\eps\lra{t}^{CM_0^2\eps^2}\\
&\lesssim\int_{1}^t M_0^2\eps^2\lra{\tau}^{-1}\mcl{I}(\tau) \,d\tau+M_0\eps\lra{t}^{CM_0^2\eps^2}.}
We apply Gronwall's inequality to obtain
\fm{\norm{\lra{t}^{\frac{1}{2}}   \lra{q}^{1-\frac{\delta}{\lambda}}|\muuLunit Zu|}_{L^\infty(\D_t)}&\lesssim M_0\eps\lra{t}^{CM_0^2\eps^2},\quad\forall t\in[1,\Tboot].}
That is, in $\D$ we have
\fm{|\uLunit Zu|&\lesssim \muring\cdot M_0\eps\lra{t}^{-\frac{1}{2}+CM_0^2\eps^2}\lra{q}^{-1+\frac{\delta}{\lambda}}\lesssim  M_0\eps\lra{t}^{-\frac{1}{2}+CM_0^2\eps^2}\lra{q}^{-1+\frac{\delta}{\lambda}}.}
Since \fm{\lra{r+t}^{-1}|Z^{\leq 2}u|\lesssim M_0\eps\lra{t}^{-\frac{3}{2}+\delta}\lesssim M_0\eps\lra{t}^{-\frac{1}{2}-\lambda-\delta}\lesssim M_0\eps\lra{t}^{-\frac{1}{2}}\lra{q}^{-1-\frac{\delta}{\lambda}},}
we obtain \eqref{est:pt:higher:u} for $(k,m)=(1,1)$. The case when $(k,m)=(0,1)$ also follows.

It remains to prove \eqref{est:pf:claim:Zu:dZu}. Recall that $q_r\gtrsim\lra{t}^{-CM_0^2\eps^2}>0$. Moreover, we have
\fm{ \abs{q_r^{-1}\partial_rZu}&\lesssim  \muring^{-1} (|LZu|+|\uLunit Zu|) \lesssim |\muuLunit Zu|+\lra{t}^{-1+CM_0^2\eps^2}|Z^{\leq 2}u| \lesssim |\muuLunit Zu|+M_0\eps\lra{t}^{-\frac{3}{2}+\delta+CM_0^2\eps^2}.}
Since $u\equiv 0$ for $r\geq t+1$, we have
\fm{|Zu(t,x)|&\lesssim\int_r^{t+1}|(\partial_r Zu)(t,r'\frac{x}{|x|})|\,dr'\lesssim\int_r^{t+1}|(q_r\muuLunit Zu)(t,r'\frac{x}{|x|})|\,dr' +\int_r^{t+1}M_0\eps\lra{t}^{-\frac{3}{2}+\delta+CM_0^2\eps^2}q_{r'}\,dr'\\
&\lesssim \int_r^{t+1}|(\lra{q}^{1-\frac{\delta}{\lambda}}\muuLunit Zu)(t,r'\frac{x}{|x|})|\cdot \lra{q}^{-1+\frac{\delta}{\lambda}}q_{r'}\,dr'+M_0\eps\lra{t}^{-\frac{3}{2}+\delta+CM_0^2\eps^2}\lra{q}\\
&\lesssim  \norm{\lra{q}^{1-\frac{\delta}{\lambda}}|\muuLunit Zu|}_{L^\infty(\D_t)} \cdot \int_r^{t+1}\lra{q}^{-1+\frac{\delta}{\lambda}}q_{r'}\,dr'+M_0\eps\lra{t}^{\lambda-\delta-\frac{3}{2}+\delta+CM_0^2\eps^2}\lra{q}^{\frac{\delta}{\lambda}}\\
&\lesssim  \norm{\lra{q}^{1-\frac{\delta}{\lambda}}|\muuLunit Zu|}_{L^\infty(\D_t)} \cdot \frac{\lambda}{\delta}\lra{q}^{\frac{\delta}{\lambda}}+M_0\eps\lra{t}^{-\frac{1}{2}-2\delta+CM_0^2\eps^2}\lra{q}^{\frac{\delta}{\lambda}}.}
Here $\frac{\lambda}{\delta}$ is a universal constant as we have chosen $\delta=\frac{1}{1000}$ and $\lambda=1-2\delta$. This finishes our proof.

\subsubsection{The case $m=0$}\label{sec:ptb:pf:m=0}
Consider the case where we only have partial derivatives. That is, $m=0$ and $k\leq N-5$. We induct on $k$. By \eqref{est:ptb:u}, \eqref{est:ptb:du}, and \eqref{est:ptb:ddu}, we have proved the case for $k=0,1,2$. Now fix $3\leq k\leq N-5$ and suppose that we have proved \eqref{est:pt:higher:u} for all $k'<k$ and $m=0$. By \eqref{est:pt:box:dZ:u} and \eqref{asu:bootstrap:ptb}, we have
\fm{\abs{\wt{\Box}_g\partial^{k-1}u}&\lesssim \sum_{k_1+k_2+k_3\leq k-1}\lra{r-t}^{1-k+k_1+k_2+k_3}|\partial^{k_1}u||\partial^{k_2+1}u||\partial^{k_3+1}u|\\
&\lesssim |u||\partial u||\partial^ku|+\sum_{k_1+k_2+k_3\leq k-1\atop k_2,k_3<k-1}\lra{r-t}^{1-k+k_1+k_2+k_3}|\partial^{k_1}u||\partial^{k_2+1}u||\partial^{k_3+1}u|.}
By the induction hypotheses, we have
\fm{\abs{\wt{\Box}_g\partial^{k-1}u}&\lesssim M_0^2\eps^2\lra{t}^{-1}\lra{q}^{-1+\frac{2\delta}{\lambda}}|\partial^ku|+M_0^3\eps^3\lra{t}^{-\frac{3}{2}+CM_0^2\eps^2}\lra{q}^{-1-k+\frac{3\delta}{\lambda}}.}
Here we use $\lra{r-t}^{-1}\lesssim \lra{t}^{CM_0^2\eps^2}\lra{q}^{-1}$.
By applying \eqref{est:transport:r12uLphi} to $\phi=\partial^{k-1}u$, we have
\eq{\label{est:transport:higher:r12uLu}&\abs{\wt{L}\kh{r^{\frac{1}{2}}\uLunit \partial^{k-1}u}}\lesssim r^{\frac{1}{2}}\abs{\wt{\Box}_g\partial^{k-1}u}+\lra{t}^{-\frac{3}{2}}|Z^{\leq 2}\partial^{k-1}u|\\
&\lesssim M_0^2\eps^2\lra{t}^{-1}\lra{q}^{-1+\frac{2\delta}{\lambda}}\cdot r^{\frac{1}{2}}|\partial^ku|+M_0^3\eps^3\lra{t}^{-1+CM_0^2\eps^2}\lra{q}^{-1-k+\frac{3\delta}{\lambda}}+M_0\eps\lra{t}^{-2+\delta+CM_0^2\eps^2}\lra{q}^{1-k}.}
Since $k\leq N-5$, we use \eqref{asu:bootstrap:ptb} to obtain \fm{|Z^{\leq 2}\partial^{k-1}u|\lesssim \lra{r-t}^{1-k}|Z^{\leq k+1}u|\lesssim M_0\eps\lra{t}^{-\frac{1}{2}+\delta+CM_0^2\eps^2}\lra{q}^{1-k}.}  Now, fix $(t,x)\in\D$. Recall the definitions of $X(s)$ and $t_0$ from \eqref{eqn:integralcurve:wtL} and \eqref{def:t0:X}. On $\Dlower$, we have
\fm{|\uLunit\partial^{k-1}u|&\lesssim |\partial^ku|\lesssim \lra{r-t}^{-k}|Z^{\leq k}u|\lesssim\lra{q}^{-k}\cdot  M_0\eps\lra{t}^{-\frac{1}{2}+\delta}\lesssim M_0\eps\lra{t}^{-\frac{1}{2}}\lra{q}^{-k+\frac{\delta}{\lambda}}.}
We integrate \eqref{est:transport:higher:r12uLu} along $X(s)$ to obtain
\fm{\lra{q}^{k-\frac{\delta}{\lambda}}\abs{r^{\frac{1}{2}}\uLunit \partial^{k-1}u}&\lesssim \lra{q}^{k-\frac{\delta}{\lambda}}\abs{r^{\frac{1}{2}}\uLunit \partial^{k-1}u}(X(t_0))\\
&\quad+\int_{t_0}^tM_0^2\eps^2\lra{\tau}^{-1}\lra{q}^{-1+\frac{2\delta}{\lambda}}\norm{r^{\frac{1}{2}}\lra{q}^{k-\frac{\delta}{\lambda}}\partial^ku}_{L^\infty(\D_\tau)}\,d\tau\\
&\quad+M_0 \eps\lra{t}^{ CM_0^2\eps^2}\lra{q}^{-1 +\frac{2\delta}{\lambda}}+M_0\eps\lra{t_0}^{-1+\delta+CM_0^2\eps^2}\lra{q}^{1-\frac{\delta}{\lambda}}\\
&\lesssim \int_{t_0}^tM_0^2\eps^2\lra{\tau}^{-1} \norm{r^{\frac{1}{2}}\lra{q}^{k-\frac{\delta}{\lambda}}\partial^ku}_{L^\infty(\D_\tau)}\,d\tau +M_0 \eps\lra{t}^{CM_0^2\eps^2}. }
See Theorem~\ref{thm:asym} for the definition of $\D_t$.
We use $\lra{t_0}^{-1+\delta}\lra{q}^{1-\frac{\delta}{\lambda}}\lesssim \lra{t_0}^{-2\delta }\lesssim 1$. Since \fm{\lra{q}^{k-\frac{\delta}{\lambda}}|r^{\frac{1}{2}}\partial^ku|&\lesssim \lra{q}^{k-\frac{\delta}{\lambda}}|r^{\frac{1}{2}}\uLunit\partial^{k-1}u|+\lra{r-t}^{k-\frac{\delta}{\lambda}}\lra{t}^{CM_0^2\eps^2}\cdot \lra{r+t}^{-\frac{1}{2}}|Z\partial^{k-1}u|\\
&\lesssim \lra{q}^{k-\frac{\delta}{\lambda}}|r^{\frac{1}{2}}\uLunit\partial^{k-1}u|+\lra{r-t}^{ 1-\frac{\delta}{\lambda}}\cdot \lra{t}^{-\frac{1}{2}+CM_0^2\eps^2}|Z^{\leq k}u|
\\
&\lesssim \lra{q}^{k-\frac{\delta}{\lambda}}|r^{\frac{1}{2}}\uLunit\partial^{k-1}u|+M_0\eps\lra{r-t}^{1 -\frac{\delta}{\lambda}}\cdot \lra{t}^{-1+\delta+CM_0^2\eps^2}\\
&\lesssim \lra{q}^{k-\frac{\delta}{\lambda}}|r^{\frac{1}{2}}\uLunit\partial^{k-1}u|+M_0\eps  \lra{t}^{ -2\delta+CM_0^2\eps^2},}
we obtain  
\fm{\norm{r^{\frac{1}{2}}\lra{q}^{k-\frac{\delta}{\lambda}} \partial^ku}_{L^\infty(\D_t)}&\lesssim \int_{1}^tM_0^2\eps^2\lra{\tau}^{-1}\norm{r^{\frac{1}{2}}\lra{q}^{k-\frac{\delta}{\lambda}}\partial^ku}_{L^\infty(\D_\tau)}\,d\tau +M_0 \eps\lra{t}^{CM_0^2\eps^2}.}
By Gronwall's inequality, we conclude that
\fm{\norm{r^{\frac{1}{2}}\lra{q}^{k-\frac{\delta}{\lambda}} \partial^ku}_{L^\infty(\D_t)}&\lesssim M_0\eps\lra{t}^{CM_0^2\eps^2},\qquad \forall t\in[1,\Tboot].}
By induction, we obtain \eqref{est:pt:higher:u} for $m=0$ and $k\leq N-5$.

\subsubsection{The case $k,m>0$}\label{sec:ptb:pf:m>0} Fix $(k,m)$ with $ k+m\leq N-5$ and $k,m\geq 1$. We have proved the case for $k=1$ and $m=1$, so now we assume $k+m\geq 3$.

Suppose that we have proved \eqref{est:pt:higher:u} for all $(k',m')$ with $k'+m'<k+m$ and $k'\geq 1$, or for all $k'+m'=k+m$ and $m'<m$. The argument below \eqref{est:pt:higher:u} also implies that \eqref{est:pt:higher:u} holds for $k'=0$ and $m'<k+m-1$. In particular, if $k\geq 2$, then \eqref{est:pt:higher:u}  holds for all $(0,m')$ with $m'\leq m$. If $k=1$, then \eqref{est:pt:higher:u}  holds for all $(0,m')$ with $m'\leq m-1$.

By \eqref{est:pt:box:dZ:u} and \eqref{asu:bootstrap:ptb}, we have
\fm{\abs{\wt{\Box}_g\partial^{k-1}Z^mu}&\lesssim \sum_{k_1\leq k-1, m_3<m\atop m_1+m_2+m_3\leq m}\lra{r-t}^{1-k+k_1}|Z^{m_1}u||Z^{m_2}u||\partial^{2+k_1}Z^{m_3}u|\\
&\quad+1_{k\geq 2}\cdot \sum_{k_1+k_2+k_3\leq k-1\atop m_1+m_2+m_3\leq m}\lra{r-t}^{1-k+k_1+k_2+k_3}|\partial^{k_1}Z^{m_1}u||\partial^{k_2+1}Z^{m_2}u||\partial^{k_3+1}Z^{m_3}u|.}
In the first sum, if $k\geq 2$, or if $k=1$ and $m_1,m_2<m$,  we apply the induction hypotheses to obtain an upper bound
\fm{&C\cdot\lra{r-t}^{1-k+k_1}\cdot M_0\eps\lra{t}^{-\frac{1}{2}+CM_0^2\eps^2}\lra{q}^{\frac{\delta}{\lambda}}\cdot M_0\eps\lra{t}^{-\frac{1}{2}+CM_0^2\eps^2}\lra{q}^{\frac{\delta}{\lambda}}\cdot M_0\eps\lra{t}^{-\frac{1}{2}+CM_0^2\eps^2}\lra{q}^{-2-k_1+\frac{\delta}{\lambda}}
\\&\lesssim M_0^3\eps^3\lra{t}^{-\frac{3}{2}+CM_0^2\eps^2}\lra{q}^{-1-k+\frac{3\delta}{\lambda}}.}
Note that $k_1+2+m_3\leq k+m$ and $m_3<m$.
If $k=1$ and $m\in\{m_1,m_2\}$, then we obtain terms of the form
\fm{|Z^mu||u||\partial^2u|&\lesssim \muring\cdot M_0^2\eps^2 
\lra{t}^{-1}\lra{q}^{-2+\frac{2\delta}{\lambda}}|Z^mu|.}
In the second sum, since it vanishes when $k=1$, we can apply the induction hypotheses to estimate $Z^{m'}u$ for all $m'\leq m$. If $(k-1,m)\notin\{(k_2,m_2),(k_3,m_3)\}$,  we apply the induction hypotheses to obtain an upper bound
\fm{&C\cdot\lra{r-t}^{1-k+k_1+k_2+k_3}\cdot M_0\eps\lra{t}^{-\frac{1}{2}+CM_0^2\eps^2}\lra{q}^{-k_1+\frac{\delta}{\lambda}}\cdot M_0\eps\lra{t}^{-\frac{1}{2}+CM_0^2\eps^2}\lra{q}^{-k_2-1+\frac{\delta}{\lambda}}\\
&\cdot M_0\eps\lra{t}^{-\frac{1}{2}+CM_0^2\eps^2}\lra{q}^{-1-k_3+\frac{\delta}{\lambda}}\\
&\lesssim  M_0^3\eps^3\lra{t}^{-\frac{3}{2}+CM_0^2\eps^2}\lra{q}^{-1-k+\frac{3\delta}{\lambda}}.}
If $(k-1,m)\in\{(k_2,m_2),(k_3,m_3)\}$, we obtain terms of the form
\fm{|u||\partial u||\partial^kZ^mu|&\lesssim M_0^2\eps^2\lra{t}^{-1}\lra{q}^{-1+\frac{2\delta}{\lambda}}|\partial^kZ^mu|.}
In summary,
\fm{\abs{\wt{\Box}_g\partial^{k-1}Z^mu}&\lesssim 1_{k=1}\cdot\muring\cdot M_0^2\eps^2 
\lra{t}^{-1}\lra{q}^{-2+\frac{2\delta}{\lambda}}|Z^mu|+1_{k\geq 2}\cdot M_0^2\eps^2\lra{t}^{-1}\lra{q}^{-1+\frac{2\delta}{\lambda}}|\partial^kZ^mu|\\
&\quad+M_0^3\eps^3\lra{t}^{-\frac{3}{2}+CM_0^2\eps^2}\lra{q}^{-1-k+\frac{3\delta}{\lambda}}.}

If $k\geq 2$, we apply \eqref{est:transport:r12uLphi} to obtain
\eq{\label{est:transport:higher:r12uLu:higher}&\abs{\wt{L}\kh{r^{\frac{1}{2}}\uLunit \partial^{k-1}Z^mu}}\lesssim r^{\frac{1}{2}}\abs{\wt{\Box}_g\partial^{k-1}Z^mu}+\lra{t}^{-\frac{3}{2}}|Z^{\leq 2}\partial^{k-1}Z^{m}u|\\
&\lesssim M_0^2\eps^2\lra{t}^{-1}\lra{q}^{-1+\frac{2\delta}{\lambda}}\cdot r^{\frac{1}{2}}|\partial^kZ^mu| +M_0^3\eps^3\lra{t}^{-1+CM_0^2\eps^2}\lra{q}^{-1-k+\frac{3\delta}{\lambda}} +\lra{t}^{-\frac{3}{2}}\lra{r-t}^{1-k}|Z^{\leq k+1+m}u|\\
&\lesssim M_0^2\eps^2\lra{t}^{-1}\lra{q}^{-1+\frac{2\delta}{\lambda}}\cdot r^{\frac{1}{2}}|\uLunit\partial^{k-1}Z^mu|+M_0^3\eps^3\lra{t}^{-1+CM_0^2\eps^2}\lra{q}^{-1-k+\frac{3\delta}{\lambda}}+M_0\eps\lra{t}^{-2+\delta+CM_0^2\eps^2}\lra{q}^{1-k}.}
Here we use
\fm{&M_0^2\eps^2\lra{t}^{-1}\lra{q}^{-1+\frac{2\delta}{\lambda}}\cdot r^{\frac{1}{2}}\lra{r+t}^{-1}|Z\partial^{k-1}Z^mu| \lesssim M_0^2\eps^2\lra{t}^{-\frac{3}{2}}\lra{r-t}^{1-k} |Z^{\leq k+m}u|.} Note that we need $k+m+1\leq N-4$ here to apply \eqref{asu:bootstrap:ptb}. Since \eqref{est:transport:higher:r12uLu:higher} is of the same form as \eqref{est:transport:higher:r12uLu}, and since on $\Dlower$
\fm{|\uLunit\partial^{k-1}Z^{m}u|&\lesssim |\partial^kZ^mu|\lesssim \lra{r-t}^{-k}|Z^{\leq k+m}u|\lesssim M_0\eps\lra{t}^{-\frac{1}{2}+\delta}\lra{q}^{-k}\lesssim M_0\eps\lra{t}^{-\frac{1}{2}}\lra{q}^{-k+\frac{\delta}{\lambda}},}
we follow the same proof in Section~\ref{sec:ptb:pf:m=0} to conclude that
\fm{\norm{r^{\frac{1}{2}}\lra{q}^{k-\frac{\delta}{\lambda}} \partial^kZ^mu}_{L^\infty(\D_t)}&\lesssim M_0\eps\lra{t}^{CM_0^2\eps^2},\qquad \forall t\in[1,\Tboot].}
By induction, we obtain \eqref{est:pt:higher:u} for $k\geq 2$, $m\geq 1$, and $m+k\leq N-5$.

Finally, if $k=1$, we apply \eqref{est:transport:r12muuLphi} to obtain
\fm{\abs{\wt{L}\kh{r^{\frac{1}{2}}\muuLunit Z^{m}u }}&\lesssim r^{\frac{1}{2}}\muring^{-1}\abs{\wt{\Box}_gZ^{m}u}+M_0^2\eps^2\lra{t}^{-\frac{1}{2}}\lra{q}^{-1+\frac{2\delta}{\lambda}}|\muuLunit Z^{m}u|+
\muring^{-1}\lra{t}^{-\frac{3}{2}}|Z^{\leq m+2}u|\\
&\lesssim  M_0^2\eps^2 
\lra{t}^{-\frac{1}{2}}\lra{q}^{-1+\frac{2\delta}{\lambda}}(|\muuLunit Z^mu|+\lra{q}^{-1}|Z^mu|)\\
&\quad+M_0^3\eps^3\lra{t}^{-1+CM_0^2\eps^2}\lra{q}^{-2+\frac{3\delta}{\lambda}}+M_0\eps \lra{t}^{-2+\delta+CM_0^2\eps^2}.}
We notice that this estimate is of the same form as \eqref{est:transport:higher:r12muuLu}. On $\Dlower$, we have
\fm{r^{\frac{1}{2}}|\muuLunit Z^mu|&\lesssim \lra{t}^{\frac{1}{2}}\cdot 1\cdot |\partial Z^mu|\lesssim \lra{t}^{\frac{1}{2}}\lra{r-t}^{-1} \cdot |Z^{\leq m+1}u|\lesssim M_0\eps\lra{t}^{\delta}\lra{r-t}^{-1}\lesssim M_0\eps\lra{t}^{-\lambda+\delta}.}
We then follow the rest of the proof in Section~\ref{sec:ptb:pf:m=1,k=1} to finish the proof.

\subsection{Estimates in the whole spacetime}
\begin{prop}
For all $ t\in[1, \Tboot]$ and $x\in\R^2$, we have
\eq{\label{est:ptb:everywhere:higher:u}  \abs{Z^{\leq N-6}u} &\lesssim M_0\eps\lra{t}^{-\frac{1}{2}+CM_0^2\eps^2}\lra{r-t}^{\frac{\delta}{\lambda}}.}

Meanwhile, we recall the following better bounds from Lemma~\ref{lem:ptb:everywhere:u:du:ddu}:
\eq{\label{est:ptb:everywhere:u} |u|&\lesssim M_0\eps\lra{t}^{-\frac{1}{2}}\min\{\lra{\wt{q}}^{\frac{\delta}{\lambda}},\lra{r-t}^{\frac{\delta}{\lambda}}\},}
\eq{\label{est:ptb:everywhere:du} |\partial u|&\lesssim M_0\eps\lra{t}^{-\frac{1}{2}}\lra{\wt{q}}^{-1+\frac{\delta}{\lambda}} ,}
\eq{\label{est:ptb:everywhere:ddu} |\partial^2u|&\lesssim M_0\eps\lra{t}^{-\frac{1}{2}}\lra{\wt{q}}^{-2+\frac{\delta}{\lambda}}\wt{q}_r.}
Recall that $\wt{q}$ was defined in Section~\ref{sec:construction:wtq}.
\end{prop}
\begin{proof}
For $r\geq t+1$, we have $u\equiv 0$, so there is nothing to prove.

For $(t,x)\in\D$, we apply \eqref{est:pt:higher:u} to obtain
\fm{| Z^{\leq N-6}u|&\lesssim M_0\eps\lra{t}^{-\frac{1}{2}+CM_0^2\eps^2}\lra{q}^{ \frac{\delta}{\lambda}}\lesssim M_0\eps\lra{t}^{-\frac{1}{2}+CM_0^2\eps^2}\lra{r-t}^{ \frac{\delta}{\lambda}}.}
Here we use $\lra{q}\lesssim \lra{r-t}\lra{t}^{CM_0^2\eps^2}$ in $\D$ by \eqref{est:q:r-t:ratio}.

For $(t,x)\in\Dint$, we have $\lra{r-t}\gtrsim\lra{t}^\lambda$ and thus 
\fm{|Z^{\leq N-6}u|&\lesssim M_0
\eps\lra{t}^{-\frac{1}{2}+\delta}\lesssim M_0\eps\lra{t}^{-\frac{1}{2}}\lra{r-t}^{\frac{\delta}{\lambda}}}
by the bootstrap assumption \eqref{asu:bootstrap:ptb}.
\end{proof}

\subsection{Applications to the wave equation}
We now use \eqref{est:ptb:everywhere:higher:u}--\eqref{est:ptb:everywhere:ddu} to estimate $\wt{\Box}_gZ^{\leq N}u$.
\begin{prop}\label{prop:est:higher:wtBoxgu}
For all $k,m\geq 0$ with $k+m\leq N$, we have
\eq{\label{est:higher:wtBoxgu}\abs{\wt{\Box}_g\partial^kZ^{m}u}&\lesssim M_0^2\eps^2\lra{t}^{-1}\kh{\abs{\partial^{k+1}Z^mu}+\lra{\wt{q}}^{-1}\wt{q}_r\abs{Z^mu}}\\
&+M_0^2\eps^2\lra{t}^{-1+CM_0^2\eps^2}\kh{\sum_{k'+m'=k+m, k'>k\atop \text{or }k'+m'<k+m}\abs{\partial\partial^{k'}Z^{m'}u}+\frac{|Z^{ m}u|\cdot 1_{k>0}+|Z^{\leq m-1}u|}{\lra{r-t} }}.}
This estimate holds for all $t\in[1,\Tboot]$ and $x\in\R^2$.
Here we use the convention $Z^{\leq -1}u=0$.
\end{prop}
\begin{proof}
Fix $(k,m)$ with $k+m\leq N$.
By \eqref{est:pt:box:dZ:u}, we have
\fm{\abs{\wt{\Box}_g\partial^kZ^{m}u}&\lesssim \sum_{k_1\leq k, m_3<m\atop m_1+m_2+m_3\leq m}|Z^{m_1}u||Z^{m_2}u||\partial^{2+k_1}Z^{m_3}u|\\
&\quad+1_{k>0}\cdot \sum_{k_1+k_2+k_3\leq k\atop m_1+m_2+m_3\leq m}|\partial^{k_1}Z^{m_1}u||\partial^{k_2+1}Z^{m_2}u||\partial^{k_3+1}Z^{m_3}u|.}
We no longer need the extra decay in $\lra{r-t}$.

We start with the first sum. Since $m_1+m_2+m_3+k_1+2\leq N+2$, there exists at most one element in $\{m_1,m_2,m_3+k_1+2\}$ that is strictly larger than $\frac{N+2}{2}$. Since $N\geq 13$, we have $\lfloor\frac{ N+2}{2}\rfloor\leq N-6$. If $m_1$ is the largest, we have
\fm{|Z^{m_1}u||Z^{m_2}u||\partial^{2+k_1}Z^{m_3}u|&\lesssim |Z^{m_1}u|\cdot\lra{r-t}^{-2-k_1}|Z^{\leq N-6}u|^2\\
&\lesssim M_0^2\eps^2\lra{t}^{-1+CM_0^2\eps^2}\lra{r-t}^{-2+\frac{2\delta}{\lambda}}|Z^{\leq m}u|.}
Here we use \eqref{est:ptb:everywhere:higher:u}. The only case where this bound is not enough is when $m_1=m$ and $k_1=k=0$. In this case, we obtain
\fm{|Z^mu||u||\partial^{2} u|&\lesssim M_0^2\eps^2\lra{t}^{-1}\lra{\wt{q}}^{-2+\frac{2\delta}{\lambda}}\wt{q}_r|Z^{m}u|.}
Here we use \eqref{est:ptb:everywhere:u} and \eqref{est:ptb:everywhere:ddu}. The case where $m_2$ is the largest is the same. If $m_3+k_1+2$ is the largest, we have
\fm{|Z^{m_1}u||Z^{m_2}u||\partial^{2+k_1}Z^{m_3}u|&\lesssim |Z^{\leq \frac{N+2}{2}}u|^2\cdot \lra{r-t}^{-1}|\partial^{1+k_1}Z^{\leq m_3+1}u| \\
&\lesssim M_0^2\eps^2\lra{t}^{-1+CM_0^2\eps^2} |\partial^{1+k_1}Z^{\leq m_3+1}u|.}
The only case where this bound is not enough is when $m_3=m-1$ and $k_1=k$. In this case, we have $m_1+m_2\leq 1$ and
\fm{&| u||Z^{\leq 1}u||\partial^{2+k }Z^{m -1}u| \lesssim M_0^2\eps^2\lra{t}^{-1+CM_0^2\eps^2}\lra{r-t}^{\frac{2\delta}{\lambda}}\cdot|\partial^{2+k}Z^{m-1}u|\\
&\lesssim M_0^2\eps^2\lra{t}^{-1 }\lra{r-t}\cdot|\partial^{2+k}Z^{m-1}u|+M_0^2\eps^2\lra{t}^{-1+CM_0^2\eps^2}\lra{r-t}^{-1+\frac{4\delta}{\lambda}}\cdot|\partial^{2+k}Z^{m-1}u| \\&\lesssim M_0^2\eps^2\lra{t}^{-1}|\partial^{1+k}Z^{\leq m}u|+M_0^2\eps^2\lra{t}^{-1+CM_0^2\eps^2}|\partial^{2+k}Z^{m-1}u|.}
Here we use $\lra{t}^{CM_0^2\eps^2}\lra{r-t}^{\frac{2\delta}{\lambda}}\lesssim \lra{r-t}+\lra{t}^{2CM_0^2\eps^2}\lra{r-t}^{-1+\frac{4\delta}{\lambda}}$.
We conclude that the first sum is bounded by the right side of \eqref{est:higher:wtBoxgu}.

Next, we consider the second sum where $k\geq 1$. Since $(k_1+m_1)+(k_2+m_2+1)+(k_3+m_3+1)\leq k+m+2\leq N+2$, there exists at most one element in $\{k_1+m_1,k_2+m_2+1,k_3+m_3+1\}$ that is strictly larger than $\frac{N+2}{2}$. If $k_1+m_1$ is the largest, we have
\fm{|\partial^{k_1}Z^{m_1}u||\partial^{k_2+1}Z^{m_2}u||\partial^{k_3+1}Z^{m_3}u|&\lesssim \lra{r-t}^{-2-k_2-k_3}|Z^{\leq \lfloor\frac{N+2}{2}\rfloor}u|^2\cdot|\partial^{k_1}Z^{m_1}u|\\
&\lesssim M_0^2\eps^2\lra{t}^{-1+CM_0^2\eps^2}\lra{r-t}^{-2+\frac{2\delta}{\lambda}}|\partial^{k_1}Z^{m_1}u|.}
This is bounded above by the second row of \eqref{est:higher:wtBoxgu}. Here we use $k\geq 1$. Either we have $k_1>0$ and $|\partial^{k_1}Z^{m_1}u|\leq \sum_{k'+m'< k+m}|\partial\partial^{k'}Z^{m'}u|$, or $k_1=0$ and $\lra{r-t}^{-2+\frac{2\delta}{\lambda}}|Z^{m_1}u|\leq \lra{r-t}^{-1}|Z^{\leq m}u|$.  If $k_2+m_2+1$ is the largest,  we have
\fm{|\partial^{k_1}Z^{m_1}u||\partial^{k_2+1}Z^{m_2}u||\partial^{k_3+1}Z^{m_3}u|&\lesssim \lra{r-t}^{-1-k_3}|Z^{\leq \lfloor\frac{N+2}{2}\rfloor}u|^2|\partial^{k_2+1}Z^{m_2}u|\\
&\lesssim M_0^2\eps^2\lra{t}^{-1+CM_0^2\eps^2}|\partial^{k_2+1}Z^{m_2}u|.}
This is enough unless $k_2=k$ and $m_2=m$, in which case we have
\fm{| u||\partial^{k+1}Z^{m}u||\partial u|&\lesssim M_0^2\eps^2\lra{t}^{-1}|\partial^{k+1}Z^{m}u|. }
We use \eqref{est:ptb:everywhere:u} and \eqref{est:ptb:everywhere:du} instead of \eqref{est:ptb:everywhere:higher:u}. If $k_3+m_3+1$ is the largest, the proof is the same.
\end{proof}

\section{End of the proof of Theorem~\ref{thm:main}}\label{sec:main:endpf}

We now end the proof of Theorem~\ref{thm:main}. We first prove estimates for the $L^2$ norms of $\partial Z^{\leq N} u$. Here we need the weighted energy estimates and Poincaré's estimates proved in Section~\ref{sec:energy:poincare}. Then, by these estimates and the $L^1$--$L^\infty$ estimate by H\"ormander (see Lemma~\ref{lem:hormander:L1Linfty}), we obtain improved pointwise bounds for $Z^{\leq N-4}u$. We thus conclude that the constant $M_0$ in \eqref{asu:bootstrap:ptb} is strictly improved.

\subsection{The proof of \eqref{asu:bootstrap:energy}}

For $k,m\geq 0$ and $k+m\leq N$, we define
\eq{\label{def:energy:pf:bootstrap} E_{k,m}(t)&:=\norm{\partial^{k+1} Z^{m}u(t)}_{L^2(w)}.}
Following \cite{MR2382144}, we set $E_{-1,0}(t)=E_{0,-1}(t)=0$.
We seek to prove \eqref{asu:bootstrap:energy}. Because of \eqref{est:local:existence:t01:ene}, it suffices to prove
\eq{\label{est:pf:bootstrap:energy:claim} E_{k,m}(t)&\lesssim M_0\eps\lra{t}^{CM_0^2\eps^2},\qquad \forall k+m\leq N,\ t\in[1,\Tboot]. }
If this goal is achieved, we first notice from \eqref{def:weight} and $\kappa_1\sim M_0^2$ that $\lra{t}^{-CM_0^2\eps^2}\lesssim w\lesssim \lra{t}^{CM_0^2\eps^2}$. Thus,
\fm{\norm{ Z^{\leq N}\partial u(t)}_{L^2(\R^2)} \lesssim \lra{t}^{CM_0^2\eps^2}\norm{\partial Z^{\leq N}u(t)}_{L^2(w)}\lesssim M_0\eps\lra{t}^{CM_0^2\eps^2},\qquad t\in [1,\Tboot].}
We now prove \eqref{est:pf:bootstrap:energy:claim} by induction on $k+m$ and then on $m$. For $k=m=0$, we apply \eqref{est:energy} to $\phi=u$, $t_1=1$, and $t_2=t$. This yields
\fm{E_{0,0}(t)&\lesssim E_{0,0}(1)+\int_{1}^{t}M_0^2\eps^2\lra{\tau}^{-1}E_{0,0}(\tau)\, d\tau\lesssim \eps+\int_{1}^{t}M_0^2\eps^2\lra{\tau}^{-1}E_{0,0}(\tau)\, d\tau,\qquad \forall t\in[1,\Tboot].}
By Gronwall's inequality, we obtain \eqref{est:pf:bootstrap:energy:claim} with $k=m=0$.

In general, suppose we have proved \eqref{est:pf:bootstrap:energy:claim} for $(k',m')$ with $k'+m'<k+m$, or with $k'+m'=k+m$ and $m'<m$. By \eqref{est:higher:wtBoxgu}, we have
\fm{&\norm{\wt{\Box}_g\partial^{k}Z^mu}_{L^2(w)}\\
&\lesssim M_0^2\eps^2\lra{t}^{-1}\kh{E_{k,m}+\norm{\lra{\wt{q}}^{-1}\wt{q}_r\abs{Z^mu}}_{L^2(w)}}\\
&+M_0^2\eps^2\lra{t}^{-1+CM_0^2\eps^2}\kh{\sum_{k'+m'=k+m, k'>k\atop \text{or }k'+m'<k+m}E_{k',m'}+\norm{\frac{|Z^mu|}{\lra{r-t}}}_{L^2(w)}\cdot 1_{k>0}+\norm{\frac{|Z^{\leq m-1}u|}{\lra{r-t}}}_{L^2(w)}}.}
By Lemma~\ref{lem:poincare:weighted}, we have
\fm{\norm{\lra{\wt{q}}^{-1}\wt{q}_r\abs{Z^mu}}_{L^2(w)}&\lesssim \norm{\partial Z^mu }_{L^2(w)}\lesssim E_{0,m}.}
By Lemma~\ref{lem:poincare:unweighted}, we have
\fm{\norm{\frac{|Z^mu|}{\lra{r-t}}}_{L^2(w)}\cdot 1_{k>0}+\norm{\frac{|Z^{\leq m-1}u|}{\lra{r-t}}}_{L^2(w)}&\lesssim \lra{t}^{CM_0^2\eps^2}\kh{\norm{\partial Z^{m}u}_{L^2(w)}\cdot 1_{k>0}+ \norm{\partial Z^{\leq m-1}u}_{L^2(w)}}\\
&\lesssim\lra{t}^{CM_0^2\eps^2}\kh{E_{0,m}\cdot 1_{k>0}+\sum_{m'\leq m-1}E_{0,m'}}.}
By the induction hypotheses, we have
\fm{\norm{\wt{\Box}_g\partial^{k}Z^mu}_{L^2(w)}
&\lesssim M_0^2\eps^2\lra{t}^{-1}E_{k,m}+M_0^2\eps^2\lra{t}^{-1+CM_0^2\eps^2}\cdot M_0\eps\lra{t}^{CM_0^2\eps^2}\\
&\lesssim M_0^2\eps^2\lra{t}^{-1}E_{k,m}+M_0^3\eps^3\lra{t}^{-1+CM_0^2\eps^2} .}
By \eqref{est:energy}, we have
\fm{E_{k,m}(t)&\lesssim E_{k,m}(1)+\int_{1}^t M_0^2\eps^2\lra{\tau}^{-1}E_{k,m}(\tau)+M_0^3\eps^3\lra{\tau}^{-1+CM_0^2\eps^2} \,d\tau\\
&\lesssim M_0 \eps \lra{t}^{ CM_0^2\eps^2}+\int_{1}^t M_0^2\eps^2\lra{\tau}^{-1}E_{k,m}(\tau)\,d\tau.}
We obtain \eqref{est:pf:bootstrap:energy:claim} and this ends the induction.

\subsection{Improving the constants in \eqref{asu:bootstrap:ptb}}
Next, we improve \eqref{asu:bootstrap:ptb}. Fix $m\leq N-4$, and set $\phi=\phi_m=Z^mu$. 
Write
\fm{Z^{m}u&=\phi_m=\phi_{m,\rm lin}+\phi_{m,\rm inh}}
where \fm{\Box\phi_{m,\rm lin}=0,\qquad (\phi_{m,\rm lin},\partial_t\phi_{m,\rm lin})\restriction_{t=1}=(Z^mu,\partial_tZ^{m}u)\restriction_{t=1}.}
It is easy to show that $|\phi_{m,\rm lin}|\lesssim \eps\lra{t}^{-\frac{1}{2}}$, where the implicit constant is independent of $M_0,\eps$.
Moreover, since $\phi_{m,\rm inh}$ has zero data at $t=1$, we apply Lemma~\ref{lem:hormander:L1Linfty} (after a translation in time). It suffices to estimate
\fm{\int_1^{t}\int_{\R^2}\lra{\tau}^{-\frac{1}{2}}|(Z^{\leq 1}\Box Z^mu)(\tau,y)|\,dy\,d\tau&\lesssim \int_1^{t}\int_{\R^2}\lra{\tau}^{-\frac{1}{2}}|(Z^{\leq m+1}\Box u)(\tau,y)|\,dy\,d\tau.}
Here we use $[\Box,Z]=C\cdot\Box$. Now, since $\wt{\Box}_gu=0$, we have
\fm{|Z^{\leq m+1}\Box u|&\lesssim |Z^{\leq m+1}(h(u)\cdot\partial^2u)| }
where $h(u)=(h^{\alpha\beta}(u))$. Recall that we have proved \eqref{est:pf:pt:box:dZ:u}:
\fm{|Z^{\leq m+1}(h(u))|&\lesssim \sum_{m_1+m_2\leq m+1}|Z^{m_1}u||Z^{m_2}u|.}
Thus, 
\fm{|Z^{\leq m+1}\Box u|&\lesssim \sum_{m_1+m_2+m_3\leq m+1}|Z^{m_1}u||Z^{m_2}u||Z^{m_3}\partial^2u|\\
&\lesssim \sum_{m_1+m_2+m_3\leq m+1}\lra{r-t}^{-1}|Z^{m_1}u||Z^{m_2}u||\partial Z^{\leq m_3+1}u|. }
We now use \eqref{est:ptb:everywhere:higher:u}. If $m_3+2=\max\{m_1,m_2,m_3+2\}$, we have \fm{m_1,m_2\leq \frac{m_1+m_2+m_3+2}{2}\leq \frac{m+3}{2}\leq \frac{N-1}{2}\leq N-6.} Here we use $m\leq N-4$ and $N\geq 13$. Then, \fm{&\norm{\lra{r-t}^{-1}|Z^{m_1}u||Z^{m_2}u||\partial Z^{\leq m_3+1}u|}_{L^1(\R^2)}\lesssim \norm{\lra{r-t}^{-1}|Z^{\leq N-6}u|^2 }_{L^2(\R^2)}\norm{\partial Z^{\leq N}u }_{L^2(\R^2)}\\
&\lesssim M_0^2\eps^2\lra{t}^{-1+CM_0^2\eps^2}\norm{\lra{r-t}^{-1+\frac{2\delta}{\lambda}}}_{L^2(\R^2:\ r\leq t+1)}\cdot M_0\eps\lra{t}^{CM_0^2\eps^2}\\
&\lesssim M_0^3\eps^3\lra{t}^{-\frac{1}{2}+CM_0^2\eps^2}.}
Here we use the improved energy bound  \eqref{asu:bootstrap:energy}.
In the last step, we use
\fm{&\norm{\lra{r-t}^{-1+\frac{2\delta}{\lambda}}}_{L^2(\R^2:\ r\leq t+1)}^2\lesssim \int_0^{t+1}r\lra{t-r}^{-2+\frac{4\delta}{\lambda}}\,dr\\
&\lesssim \lra{t}\int_{\frac{t+1}{2}}^{t+1}\lra{t-r}^{-2+\frac{4\delta}{\lambda}}\,dr+\lra{t}^{-2+\frac{4\delta}{\lambda}}\int_0^{\frac{t+1}{2}}r\,dr\lesssim \lra{t}+\lra{t}^{\frac{4\delta}{\lambda}}\lesssim \lra{t}.}
If $m_1=\max\{m_1,m_2,m_3+2\}$, we have $m_2,m_3+2\leq\frac{m+3}{2}\leq N-6$ and thus
\fm{&\norm{\lra{r-t}^{-1}|Z^{m_1}u||Z^{m_2}u||\partial Z^{\leq m_3+1}u|}_{L^1(\R^2)}\\
&\lesssim \norm{ \lra{r-t}^{-1}|Z^{\leq N-6}u|^2}_{L^2(\R^2)}\norm{\lra{r-t}^{-1}|Z^{\leq N}u| }_{L^2(\R^2)}\\
&\lesssim M_0^2\eps^2\lra{t}^{-1+CM_0^2\eps^2}\norm{\lra{r-t}^{-1+\frac{2\delta}{\lambda}}}_{L^2(\R^2:\ r\leq t+1)}\cdot \norm{\partial Z^{\leq N}u }_{L^2(\R^2)}\lesssim M_0^3\eps^3\lra{t}^{-\frac{1}{2}+CM_0^2\eps^2}. }
Here we use Lemma~\ref{lem:poincare:unweighted}. In summary,
\fm{&\int_1^{t}\int_{\R^2}\lra{\tau}^{-\frac{1}{2}}|(Z^{\leq 1}\Box Z^mu)(\tau,y)|\,dy\,d\tau \lesssim \int_1^t M_0^3\eps^3\lra{\tau}^{-1+CM_0^2\eps^2} \, d\tau\\
&\lesssim \int_1^t M_0^3\eps^3\lra{t}^{CM_0^2\eps^2}(1+\tau)^{-1} \, d\tau\lesssim M_0^3\eps^3\lra{t}^{CM_0^2\eps^2}\ln(1+t).}
By Lemma~\ref{lem:hormander:L1Linfty}, we conclude that
\fm{|Z^{\leq N-4}u|&\lesssim |\phi_{m,\rm lin}|+|\phi_{m,\rm inh}|\lesssim \eps\lra{t}^{-\frac{1}{2}}+M_0^3\eps^3\lra{t}^{-\frac{1}{2}+CM_0^2\eps^2}\ln(1+t),\quad \forall (t,x)\in[1,\Tboot]\times\R^2.}
By \eqref{est:local:existence:t01:ptb}, we have
\fm{|Z^{\leq N-2}u|\lesssim \eps\lra{t}^{-\frac{1}{2}},\qquad \forall (t,x)\in[0,1]\times\R^2.}
In summary, for all $t\in[0,\Tboot]$ and $x\in\R^2$, we have
\fm{|Z^{\leq N-4}u|&\lesssim \eps\lra{t}^{-\frac{1}{2}}+M_0^3\eps^3\lra{t}^{-\frac{1}{2}+\delta}\cdot\lra{t}^{-\delta+CM_0^2\eps^2}\ln(1+t)\lesssim (\eps+M_0^3\eps^3)\lra{t}^{-\frac{1}{2}+\delta}.}
By first choosing a sufficiently large $M_0$ and then a sufficiently small $\eps_0$, we obtain \eqref{asu:bootstrap:ptb} with $M_0$ replaced by $\frac{M_0}{2}$ for all $\eps\in(0,\eps_0)$ and $t\in[0,\Tboot]$. Thus, the constant $M_0$ in \eqref{asu:bootstrap:ptb} is improved. Once the bootstrap argument is completed, for all $(t,x)\in[0,\Tboot]\times\R^2$,
\fm{|Z^{\leq N-4}u|&\lesssim \eps\lra{t}^{-\frac{1}{2}}+ \eps^3\lra{t}^{-\frac{1}{2}+C \eps^2}\ln(1+t)\lesssim \eps \lra{t}^{-\frac{1}{2}+C \eps^2} .}
This finishes the proof of Theorem~\ref{thm:main}.

\section{Asymptotic behavior}\label{sec:asym:beha}

In this section, we study the asymptotic behavior of the global solutions obtained in Theorem~\ref{thm:main}. In Section~\ref{sec:asym:beha:setup}, we recall the definitions of $\D$ and $q$ from Section~\ref{sec:approxoptical}. Most importantly, we have a coordinate change between $(t,x)\in\D$ and $(s,q,\omega)\in\D^*$. Here $\D^*$ is a subset of $[0,\infty)\times\R\times\mathbb{S}^1$ such that the $C^2$ map $(t,x)\mapsto (s,q,\omega)=(\eps^2\ln t,q(t,x),\frac{x}{|x|})$ is invertible with a $C^2$ inverse. 

In Section~\ref{sec:asym:beha:recover}, we recover the geometric reduced system \eqref{sec:reduced:system:main:1.1} from \eqref{qwe}. We will show that $(\mu,U)=(q_t-q_r,\eps^{-1} r^{\frac{1}{2}}u)$, when viewed as a function on $\D^*$, is an approximate solution to \eqref{sec:reduced:system:main:1.1}. Moreover, we can find an exact solution $(\wh{\mu},\wh{U})$ in $\D^*$ to \eqref{sec:reduced:system:main:1.1} such that, for each fixed $(q,\omega)\in(-\infty,2]\times\mathbb{S}^1$, $(\frac{\mu}{\wh{\mu}},U-\wh{U})\to (1,0)$ as $s\to\infty$.

Finally, in Section~\ref{sec:asym:beha:different}, we show that the global solution constructed in Theorem~\ref{thm:main} has asymptotic behavior different from that of a solution to $\Box \psi=0$ under suitable assumptions. To be more precise, if $u$ is a nonzero global solution to \eqref{qwe}--\eqref{init} for $t\geq 0$ and if $G\not\equiv 0$, we have
\fm{\lim_{t\to\infty}\lra{t}^{\frac{1}{2}}\norm{\partial_r^2u(t)}_{L^\infty(\D_t)}=\infty.}
See Theorem~\ref{thm:asym} for the definition of $\D_t$. In contrast, we notice that if $\psi$ is a global solution to $\Box \psi=0$ with $C_c^\infty$ initial data, then the Klainerman--Sobolev inequality shows that
\fm{\sup_{t\geq 0}\lra{t}^{\frac{1}{2}}\norm{\partial^2\psi(t)}_{L^\infty(\D_t)}\lesssim 1}
where the implicit constant depends only on $\norm{Z^{\leq 2}\partial^2 \psi(0)}_{L^2(\R^2)}$.

\subsection{Setup}\label{sec:asym:beha:setup}
Let $u$ be a global solution to the Cauchy problem \eqref{qwe}--\eqref{init}. We have
\eq{\label{est:new:ptb:sec:asym:beha}|Z^{\leq N-4}u|&\lesssim \eps\lra{t}^{-\frac{1}{2}+C\eps^2},\qquad \forall t\geq 1,\ x\in\R^2.}
Set $\lambda=1-2\delta=\frac{499}{500}$ and
\eq{\D=\D^{\lambda}_{\infty}&:=\{(t,x)\in\R^{1+2}:\ t\geq 1, t-t^\lambda+3\leq|x|\leq t+2\},\\
\Dlower&:=\{(t,x)\in\R^{1+2}:\ t\geq 1,\ |x|=t-t^\lambda+3\}.}
Fix $(t,x)\in\D$ and now let $X(\tau)$ be the \emph{entire} integral curve of $\frac{\wt{L}}{\wt{L}^0}$ that lies in $\D$. That is, we solve \eqref{eqn:integralcurve:wtL} for $\tau\geq t_0$ where $X(t_0)\in\Dlower$ as in \eqref{def:t0:X}. We also have $X(\tau)\in\D$ for all $\tau\geq t_0$. Such an integral curve is unique.

Moreover, we define $q$ in $\D$ as before:
\eq{\wt{L}q=0\ \text{ in }\D;\qquad q=r-t\ \text{ on }\Dlower.}
Our computations in Section~\ref{sec:approxoptical} show that $q$ is at least $C^2$ in $\D$ and that $q_r>0$. Thus, the map
\eq{(t,x)\in\D\mapsto (s,q,\omega)=(\eps^2\ln t,q(t,x),\frac{x}{|x|})\in\D^*}
where
\eq{\D^*:=\{(s,q,\omega)\in[0,\infty)\times(-\infty,2]\times \mathbb{S}^1:\ q\geq 3-e^{\lambda\eps^{-2}s}\}}
is at least $C^2$. Moreover, it is invertible and has a $C^2$ inverse. We also have 
\eq{\partial_s&=\eps^{-2}t\partial_t-\eps^{-2}tq_tq_r^{-1}\partial_r=\frac{\eps^{-2}t}{1-\frac{1}{4}h_{LL}}\wt{L},\\
\partial_q&= q_r^{-1}\partial_r.}
Note that
\fm{[\partial_s,\partial_q]&=\kh{\eps^{-2}t\partial_tq_r^{-1}-\eps^{-2}tq_tq_r^{-1}\partial_rq_r^{-1}}\partial_r-q_r^{-1}\partial_r(-\eps^{-2}tq_tq_r^{-1})\partial_r\\
&=\eps^{-2}t\kh{-q_r^{-2} \partial_tq_r+ q_tq_r^{-3}\partial_rq_r+q_r^{-2}\partial_rq_t-q_tq_r^{-3}\partial_r q_r }\partial_r=0.}
We can also define an angular derivative in the coordinates $(s,q,\omega)$, but it is not necessary in the present paper.

\subsection{Recovering the geometric reduced system}\label{sec:asym:beha:recover}

Set
\eq{(\mu,U)(t,x):=(-\uLunit q,\eps^{-1}r^{\frac{1}{2}}u)(t,x),\qquad (t,x)\in\D}
and view it as a function of $(s,q,\omega)\in\D^*$.

\begin{lem}
For $(s,q,\omega)\in\D^*$, we have
\eq{\label{est:approx:grs:asym:beha}
\left\{
\begin{array}{l}
   \displaystyle \partial_s(\mu U_q)=O(\eps^{-2}e^{-(\eps^{-2}-C)s}),  \\[1em]
   \displaystyle \partial_s\mu=\frac{1}{2}G(\omega)\mu^2 UU_q+O\kh{ e^{-(\eps^{-2}-C)s} |\mu|+\eps e^{-(\frac{1}{2}\eps^{-2}-C)s}\lra{q}^{-1}|\mu|}.  
\end{array}
\right.
}
In other words, $(\mu,U)$ is an approximate solution to the geometric reduced system \eqref{sec:reduced:system:main:1.1}.
\end{lem}
\begin{proof}
We first recall an improved version of \eqref{est:wtLr12uLu}:
\eq{\label{sec:asy:beh:wtLr12uLu}\abs{\frac{1}{\wt{L}^0}\wt{L}\kh{r^{\frac{1}{2}}\uLunit u}}&\lesssim \lra{t}^{-\frac{3}{2}}|Z^{\leq 2}u|\lesssim\eps\lra{t}^{-2+C\eps^2},\qquad \forall (t,x)\in\D.}
Since $\wt{L}q=\wt{L}\omega=0$ and $\wt{L}(\eps^2\ln t)=\eps^2t^{-1}(1-\frac{1}{4}h_{LL})$, we have
\fm{&\wt{L}\kh{r^{\frac{1}{2}}\uLunit u}=\wt{L}\kh{r^{\frac{1}{2}}\uLunit (\eps r^{-\frac{1}{2}}U)}=\wt{L}\kh{ \uLunit (\eps  U)}-\frac{1}{2}\wt{L}\kh{r^{-\frac{1}{2}}u},}
\fm{\abs{\wt{L}\kh{r^{-\frac{1}{2}}u}}\lesssim |\wt{L}(r^{-\frac{1}{2}})||u|+r^{-\frac{1}{2}}|\wt{L}u|\lesssim r^{-\frac{3}{2}}|u|+r^{-\frac{1}{2}}(|Lu|+|u|^2|\uLunit u|)\lesssim \eps\lra{t}^{-2+C\eps^2},}
\fm{\wt{L}(\uLunit(\eps U))&=\wt{L}(\eps U_q\uLunit q-\eps^3t^{-1} U_s)=-\eps\wt{L}(\mu U_q)-\eps^3\wt{L}(t^{-1}U_s)\\
&=-\eps^3t^{-1}(1-\frac{1}{4}h_{LL})\partial_s(\mu U_q)-\eps^3\wt{L}(t^{-1}U_s),}
\fm{\abs{\eps^3\wt{L}(t^{-1}U_s)}&\lesssim \abs{ \wt{L}(\frac{\wt{L}( r^{\frac{1}{2}}u)}{1-\frac{1}{4}h_{LL}})}\lesssim  \abs{ \wt{L}(\frac{r^{\frac{1}{2}}\wt{L} u }{1-\frac{1}{4}h_{LL}})}+\abs{ \wt{L}(\frac{ r^{-\frac{1}{2}}(1+\frac{1}{4}h_{LL})u }{1-\frac{1}{4}h_{LL}})}\lesssim \eps\lra{t}^{-2+C\eps^2}.}
To obtain the last step, we use Leibniz's rule and \fm{|h_{LL}|&\lesssim |u|^2\lesssim \eps^2\lra{t}^{-1+C\eps^2},\\
|\wt{L}u|&\lesssim |Lu|+|h_{LL}||\uLunit u|\lesssim \eps\lra{t}^{-\frac{3}{2}+C\eps^2},\\
|\wt{L}(h_{LL})|&\lesssim |u||\wt{L}u|\lesssim \eps^2\lra{t}^{-2+C\eps^2},\\
|\wt{L}r^{m}|&\lesssim r^{m-1}|1+\frac{1}{4}h_{LL}|\lesssim r^{m-1},\qquad m=\pm\frac{1}{2},\\
|\wt{L}\wt{L}u|
&\lesssim |L \wt{L}u|+|h_{LL}||\uLunit\wt{L}u|\lesssim (\lra{r+t}^{-1} +\eps^2\lra{t}^{-1+C\eps^2})|Z\wt{L}u|\\
&\lesssim \lra{t}^{-1+C\eps^2}\kh{|ZLu|+\abs{Z\kh{h_{LL}\uLunit u}}}\lesssim \lra{t}^{-1+C\eps^2}\kh{|ZLu|+|Zh_{LL}||\uLunit u|+|h_{LL}|\abs{Z\uLunit u}}\\
&\lesssim  \lra{t}^{-1+C\eps^2} \abs{Z\kh{\frac{Su+\omega_1\Omega_{01}u+\omega_2\Omega_{02}u}{t+r}}}+\lra{t}^{-1+C\eps^2}\kh{|u||Zu|^2+|u|^2(|Z\partial u|+|\partial u|)}\\
&\lesssim \eps\lra{t}^{-\frac{5}{2}+C\eps^2}+\eps^3\lra{t}^{-\frac{5}{2}+C\eps^2}\lesssim \eps\lra{t}^{-\frac{5}{2}+C\eps^2}.}
Thus,
\fm{|\partial_s(\mu U_q)|&\lesssim \eps^{-3}t\cdot \eps \lra{t}^{-2+C\eps^2}\lesssim \eps^{-2}\lra{t}^{-1+C\eps^2}\lesssim \eps^{-2}e^{-(\eps^{-2}-C)s}.}

Next, we compute $\mu_s$. Since $\wt{L}q=0$, we have
\fm{\wt{L}\mu=-[\wt{L},\uLunit]q=\frac{1}{4}\uLunit h_{LL}\uLunit q=-\frac{1}{4}(\uLunit h_{LL})\mu.}
The left side equals $\eps^2t^{-1}\wt{L}^0\partial_s\mu$, while the right side equals
\fm{&-\frac{1}{4}(2g^{\alpha\beta}_0u\uLunit u+O(|u|^2|\uLunit u|) )\wh{\omega}_\alpha\wh{\omega}_\beta \mu\\
&=-\frac{1}{2}G(\omega)\mu\cdot \eps r^{-\frac{1}{2}}U\cdot(-\frac{1}{2}\eps r^{-\frac{3}{2}}\cdot U-\eps r^{-\frac{1}{2}}\mu U_q-\eps r^{-\frac{1}{2}}\cdot \eps^2t^{-1}U_s)+O(|u|^2|\uLunit u||\mu|)\\
&=\frac{1}{2}\eps^2r^{-1}G(\omega)\mu^2 UU_q+O(r^{-1}|u|^2|\mu|+\eps r^{-\frac{1}{2}}|u||\wt{L}U||\mu|+|u|^2|\uLunit u||\mu|)\\
&=\frac{1}{2}\eps^2r^{-1}G(\omega)\mu^2 UU_q+O(\eps^2\lra{t}^{-2+C\eps^2} |\mu|+\eps^3\lra{t}^{-\frac{3}{2}+C\eps^2}\lra{q}^{-1}|\mu|).}
Here we use \eqref{eq:g:taylor} and \eqref{est:q:r-t:ratio}. As a result, we have
\fm{&\partial_s\mu=\frac{t}{2r(1-\frac{1}{4}h_{LL})} G(\omega)\mu^2 UU_q+O(\lra{t}^{-1+C\eps^2} |\mu|+\eps \lra{t}^{-\frac{1}{2}+C\eps^2}\lra{q}^{-1}|\mu|).}
To continue, we use
\fm{|U|&\lesssim \eps^{-1}r^{\frac{1}{2}}|u|\lesssim \lra{t}^{C\eps^2},\\
|\mu U_q|&\lesssim |\mu|q_r^{-1}
\cdot |\partial_r(\eps^{-1} r^{\frac{1}{2}}u)|\lesssim \eps^{-1} r^{-\frac{1}{2}}|u|+\eps^{-1} r^{\frac{1}{2}}|\partial u|\\
&\lesssim \lra{t}^{-1+C\eps^2}+\lra{t}^{C\eps^2}\lra{r-t}^{-1}\lesssim \lra{t}^{C\eps^2}\lra{r-t}^{-1}.}
We thus have
\fm{&\partial_s\mu=\frac{1}{2} G(\omega)\mu^2 UU_q+O((\lra{t}^{-1}\lra{r-t}+\eps^2\lra{t}^{-1+C\eps^2})|\mu^2UU_q|+\lra{t}^{-1+C\eps^2} |\mu|+\eps \lra{t}^{-\frac{1}{2}+C\eps^2}\lra{q}^{-1}|\mu|)\\
&=\frac{1}{2} G(\omega)\mu^2 UU_q+O( \lra{t}^{-1+C\eps^2} |\mu|+\eps \lra{t}^{-\frac{1}{2}+C\eps^2}\lra{q}^{-1}|\mu|).}
We finally recall that $s=\eps^2\ln t$.
\end{proof}

\begin{lem}\label{lem:limit:muUq:Psi}
The limit 
\eq{\label{est:lem:limit:muUq:Psi}\Psi(q,\omega)&:=\lim_{s\to\infty}(\mu U_q)(s,q,\omega) }
exists for all $(q,\omega)\in(-\infty,2]\times\mathbb{S}^1$. By setting $\Psi\restriction_{q\geq 2}\equiv 0$, we obtain a continuous function on $\R\times\mathbb{S}^1$. Moreover, we have
\eq{|\Psi(q,\omega)|&\lesssim \lra{q}^{-1+C\eps^2},\quad \forall (q,\omega)\in\R\times\mathbb{S}^1;}
\eq{\label{est:diff:muUq:Psi}|(\mu U_q)(s,q,\omega)-\Psi(q,\omega)|&\lesssim  e^{-(\eps^{-2}-C)s},\qquad \forall (s,q,\omega)\in\D^*.}
\end{lem}
\begin{proof}
By the first equation in \eqref{est:approx:grs:asym:beha}, we have
\fm{|(\mu U_q)(s_2,q,\omega)-(\mu U_q)(s_1,q,\omega)|&\lesssim \int_{s_1}^\infty\eps^{-2}e^{-(\eps^{-2}-C)s}\,ds\lesssim e^{-(\eps^{-2}-C)s_1}}
whenever $(s_1,q,\omega)\in\D^*$ and $s_2\geq s_1$. The definition of $\D^*$ implies that $(s,q,\omega)\in\D^*$ for all $s\geq s_1$. As a result, the limit
\eqref{est:lem:limit:muUq:Psi}
exists for all $q\leq 2$. We can extend the definition of $\Psi$ by setting $\Psi\restriction_{q\geq 2}\equiv 0$. This is well-defined as $U\equiv 0$ whenever $q\geq 1$. The estimate
\eqref{est:diff:muUq:Psi} also follows if we send $s_2\to\infty$ above.
The implicit constant here is uniform in all $(s,q,\omega)$, so the convergence above is uniform. This implies that $\Psi$ is continuous. Moreover, since $|\mu U_q|\lesssim \lra{t}^{C\eps^2}\lra{r-t}^{-1}\lesssim e^{Cs}\lra{q}^{-1}$ by \eqref{est:q:r-t:ratio}, for each $(q,\omega)\in(-\infty,2]\times\mathbb{S}^1$, we set $s=s(q)=\lambda^{-1}\eps^2\ln(3-q)$ (so $e^{\lambda\eps^{-2}s(q)}=3-q$ and $(s(q),q,\omega)\in\D^*$) and obtain
\eq{|\Psi(q,\omega)|&\lesssim |(\mu U_q)(s(q),q,\omega)|+e^{-(\eps^{-2}-C)s(q)}\lesssim e^{Cs(q)}\lra{q}^{-1}+e^{-(\eps^{-2}-C)s(q)}\\
&\lesssim e^{-(\lambda\eps^{-2}-C) s(q)}\lesssim \lra{q}^{-1+C\eps^2}.}
\end{proof}

We now seek to find an exact solution $(\wh{\mu},\wh{U})$ to the geometric reduced system \eqref{sec:reduced:system:main:1.1} in $\D^*$ such that $\wh{\mu}\wh{U}_q=\Psi$, $\wh{\mu}<0$, $\wh{U}\restriction_{q\geq 2}\equiv 0$, and $(\frac{\wh{\mu}}{\mu}, \wh{U}-U )\to (1,0)$ as $s\to\infty$. To achieve this goal, we fix a large time $T>1$ and consider the Cauchy problem 
\eq{\label{eq:grs:asym:beha:T}\partial_s(\wh{\mu}^T\wh{U}_q^T)=0,\quad \partial_s\wh{\mu}^T =\frac{1}{2}G(\omega)\Psi\wh{\mu}^T \wh{U}^T,\quad \wh{U}^T\restriction_{q\geq 2}\equiv 0,\qquad \forall(s,q,\omega)\in\D^*\cap\{s\leq T\} }
along with data
\eq{\label{eq:grs:asym:beha:T:data}(\wh{\mu}^T,\wh{U}^T_q)(T,q,\omega)=(\mu(T,q,\omega),\frac{\Psi(q,\omega)}{\mu(T,q,\omega)}),\qquad\text{whenever }(T,q,\omega)\in\D^*.}

We first show that such a Cauchy problem has a solution in $\D^*\cap\{s\leq T\}$. 
\begin{lem}\label{lem:asym:beha:T:existence}
The Cauchy problem \eqref{eq:grs:asym:beha:T}--\eqref{eq:grs:asym:beha:T:data} has a solution in $\D^*\cap\{s\leq T\}$. Moreover, in $\D^*\cap\{s\leq T\}$, we have
\eq{\label{lem:asym:beha:T:existence:est}|\wh{U}^T-U|&\lesssim e^{-(\frac{1}{2}\eps^{-2}-C)s}\lra{q},\\
|\wh{U}_q^T-U_q|
&\lesssim \eps^2 \lra{q}^{-1+C\eps^2}e^{-(\frac{1}{2}\eps^{-2}-C) s}+e^{-(\eps^{-2}-C)s},\\
\frac{\mu}{\wh{\mu}^T}&=\exp\kh{O\kh{\eps^2 \lra{q}^{C\eps^2} e^{-(\frac{1}{2}\eps^{-2}-C) s}}}=1+O\kh{\eps^2 \lra{q}^{C\eps^2} e^{-(\frac{1}{2}\eps^{-2}-C)s}}.}
\end{lem}
\begin{proof}
We will use a bootstrap argument. Let $\Tboot\in[0,T]$ and suppose that there exists a solution to \eqref{eq:grs:asym:beha:T} and \eqref{eq:grs:asym:beha:T:data} for $(s,q,\omega)\in\D^*\cap\{\Tboot\leq s\leq T\}$ such that
\eq{\label{asu:bootstrap:grs:T}|\wh{U}^T-U|&\leq M_0 e^{-(\frac{1}{2}\eps^{-2}-M_1)s}\lra{q},\qquad\text{ in }\D^*\cap\{s\in[\Tboot,T]\}.}
Both $M_0,M_1>1$ are sufficiently large constants to be chosen. We will prove \eqref{asu:bootstrap:grs:T} with $M_0$ replaced by $\frac{M_0}{2}$ while $M_1$ remains fixed. If this is done, then this lemma follows from the local existence theory for \eqref{eq:grs:asym:beha:T}.

We first check the bootstrap assumptions for $T=\Tboot$. In fact, by \eqref{eq:grs:asym:beha:T:data} and \eqref{est:diff:muUq:Psi}, and since $|\mu|\gtrsim e^{-Cs}$, we have
\fm{|(U_q-\wh{U}^T_q)(T,q,\omega)|&\lesssim |\mu|^{-1}e^{-(\eps^{-2}-C)T}\lesssim e^{-(\eps^{-2}-C)T}.}
Since $U\equiv \wh{U}^T\equiv 0$ for $q\geq 2$, we have 
\fm{|(U-\wh{U}^T)(T,q,\omega)|&\lesssim \int_q^2e^{-(\eps^{-2}-C)T}\,dq'\lesssim e^{-(\eps^{-2}-C)T}\lra{q}.}
By choosing a sufficiently large $M_0$, we obtain \eqref{asu:bootstrap:grs:T} with $M_0$ replaced by $\frac{M_0}{2}$.

Now, suppose that $\Tboot\in[0,T)$. For $(s,q,\omega)\in\D^*\cap\{\Tboot\leq s\leq T\}$, we define
\eq{\wh{\kappa}^T(s,q,\omega)&=\exp\kh{\frac{1}{2}G(\omega)\Psi(q,\omega)\int_{s}^T\wh{U}^T(s',q,\omega)\, ds'}.}
It is clear  that $\partial_s(\wh{\mu}^T\wh{\kappa}^T)=0$ and that
\eq{\label{est:asym:beha:UqT}
 \wh{U}_q^T (s,q,\omega)&=\frac{\Psi(q,\omega)}{\wh{\mu}^T(s,q,\omega)} =\frac{\Psi(q,\omega) \wh{\kappa}^T(s,q,\omega)}{\mu (T,q,\omega)}.}
Moreover, by \eqref{est:approx:grs:asym:beha} and \eqref{est:diff:muUq:Psi}, we have
\fm{\mu_s&=\frac{1}{2}G(\omega)\mu  U\kh{\Psi+O(e^{-(\eps^{-2}-C)s})}+O\kh{ e^{-(\eps^{-2}-C)s} |\mu|+\eps e^{-(\frac{1}{2}\eps^{-2}-C)s}\lra{q}^{-1}|\mu|}\\
&=\frac{1}{2}G(\omega)\mu  U \Psi +O\kh{ e^{-(\eps^{-2}-C)s} |\mu|+\eps e^{-(\frac{1}{2}\eps^{-2}-C)s}\lra{q}^{-1}|\mu|}.}
Here we use $|U|\lesssim \lra{t}^{C\eps^2}\lesssim e^{Cs}$. It follows that
\eq{\label{est:asym:beha:muTmu:ratio}\partial_s\ln(-\mu)&=\frac{1}{2}G(\omega)  U \Psi +O\kh{ e^{-(\eps^{-2}-C)s}  +\eps e^{-(\frac{1}{2}\eps^{-2}-C)s}\lra{q}^{-1} },\\
\ln\frac{\mu(T,q,\omega)}{\mu(s,q,\omega)}&= \frac{1}{2}G(\omega)\Psi(q,\omega)\int_s^T  U (s',q,\omega)\,ds' +\int_s^TO\kh{ e^{-(\eps^{-2}-C)s'}  +\eps e^{-(\frac{1}{2}\eps^{-2}-C)s'}\lra{q}^{-1} }\,ds'\\
&=\frac{1}{2}G(\omega)\Psi(q,\omega)\int_s^T  U (s',q,\omega)\,ds' + O\kh{ \eps^2e^{-(\eps^{-2}-C)s }  +\eps^3 e^{-(\frac{1}{2}\eps^{-2}-C)s }\lra{q}^{-1} }.}
That is,
\fm{\wh{U}_q^T(s,q,\omega)&=\frac{\Psi(q,\omega)}{\mu(s,q,\omega)}\\
&\cdot   \exp\kh{\frac{1}{2}G(\omega)\Psi(q,\omega)\int_{s}^T(\wh{U}^T-U)(s',q,\omega)\, ds' + O\kh{ \eps^2e^{-(\eps^{-2}-C)s }  +\eps^3 e^{-(\frac{1}{2}\eps^{-2}-C)s }\lra{q}^{-1} }}.}
By \eqref{asu:bootstrap:grs:T}, we have
\fm{\frac{1}{2}|G(\omega)\Psi(q,\omega)|\int_{s}^T|U-\wh{U}^T|(s',q,\omega)\, ds' &\lesssim \lra{q}^{-1+C\eps^2}\int_{s}^T M_0e^{-(\frac{1}{2}\eps^{-2}-M_1) s'}\lra{q}\,ds' \\
&\lesssim M_0\eps^2\lra{q}^{C\eps^2}  e^{-(\frac{1}{2}\eps^{-2}-M_1)s}.}
In summary, we have
\fm{\wh{U}_q^T(s,q,\omega)&=\frac{\Psi(q,\omega)}{\mu(s,q,\omega)}\cdot\exp\kh{O\kh{M_0\eps^2\lra{q}^{C\eps^2}  e^{-(\frac{1}{2}\eps^{-2}-M_1)s} }}\\
&=\frac{\Psi(q,\omega)}{\mu(s,q,\omega)}\cdot \kh{1+O\kh{M_0\eps^2\lra{q}^{C\eps^2}  e^{-(\frac{1}{2}\eps^{-2}-M_1)s}}}. }
To obtain the last row, we notice that $s\geq s(q)=\lambda^{-1}\eps^2\ln(3-q)$ and thus $\lra{q}\sim(3-q)=e^{\lambda\eps^{-2}s(q)}\leq e^{\lambda \eps^{-2}s}$ in $\D^*$.

Moreover, we have
\fm{U_q(s,q,\omega)&=\frac{\Psi(q,\omega)+O\kh{e^{-(\eps^{-2}-C)s}}}{\mu(s,q,\omega)}=\frac{\Psi(q,\omega)}{\mu(s,q,\omega)}+O\kh{e^{-(\eps^{-2}-C)s}}.}
It thus follows that
\fm{|\wh{U}_q^T-U_q|&\lesssim \abs{\frac{\Psi(q,\omega)}{\mu(s,q,\omega)}}\cdot M_0\eps^2\lra{q}^{C\eps^2}  e^{-(\frac{1}{2}\eps^{-2}-M_1)s} + e^{-(\eps^{-2}-C)s} \\
&\lesssim (-\mu(s,q,\omega))^{-1}\cdot  M_0\eps^2 \lra{q}^{-1+C\eps^2}e^{-(\frac{1}{2}\eps^{-2}-M_1) s}+e^{-(\eps^{-2}-C)s},}
\fm{|\wh{U}^T-U|&\lesssim  M_0\eps^2e^{-(\frac{1}{2}\eps^{-2}-M_1) s}\int_{q}^2(-\mu(s,q',\omega))^{-1}\cdot  (3-q')^{-1+C\eps^2}\,dq'+e^{-(\eps^{-2}-C)s}\lra{q}.}
To estimate the integral, we use
\fm{&\partial_s\int_{q}^2(-\mu(s,q',\omega))^{-1}\cdot  (3-q')^{-1+C\eps^2}\,dq'=\int_{q}^2(\mu^{-2}\mu_s)(s,q',\omega)\cdot  (3-q')^{-1+C\eps^2}\,dq'\\
&=\int_{q}^2\kh{\frac{1}{2}G(\omega)  UU_q(s,q',\omega)+O\kh{ e^{-(\eps^{-2}-C)s}  +\eps e^{-(\frac{1}{2}\eps^{-2}-C)s}\lra{q'}^{-1}}}\cdot  (3-q')^{-1+C\eps^2}\,dq'\\
&=\frac{G(\omega)}{4}\int_{q}^2   \partial_q(U^2)(s,q',\omega)\cdot  (3-q')^{-1+C\eps^2}\,dq'+O\kh{\min\{\eps^{-2}\lra{q}^{C\eps^2},\lra{q}\}e^{-(\eps^{-2}-C)s}+\eps e^{-(\frac{1}{2}\eps^{-2}-C)s}}.}
Here we use \eqref{est:approx:grs:asym:beha} and $|\mu|\gtrsim e^{-Cs}$. By integration by parts and since $U(s,2,\omega)=0$, we have
\fm{&\int_{q}^2   \partial_q(U^2)(s,q',\omega)\cdot  (3-q')^{-1+C\eps^2}\,dq'\\
&=-U(s,q,\omega)^2(3-q)^{-1+C\eps^2}-\int_{q}^2   U(s,q',\omega)^2\cdot  (1-C\eps^2)(3-q')^{-2+C\eps^2}\,dq'\leq 0.}
Since $G(\omega)\geq 0$, we have
\fm{\partial_s\int_{q}^2(-\mu(s,q',\omega))^{-1}\cdot  (3-q')^{-1+C\eps^2}\,dq'&\leq C\min\{\eps^{-2}\lra{q}^{C\eps^2},\lra{q}\}e^{-(\eps^{-2}-C)s}+C\eps e^{-(\frac{1}{2}\eps^{-2}-C)s},}
\fm{&\int_{q}^2(-\mu(s,q',\omega))^{-1}\cdot  (3-q')^{-1+C\eps^2}\,dq'\\
&\leq\int_{q}^2(-\mu(s(q),q',\omega))^{-1}\cdot  (3-q')^{-1+C\eps^2}\,dq'\\
&\quad+\int_{s(q)}^sC\min\{\eps^{-2}\lra{q}^{C\eps^2},\lra{q}\}e^{-(\eps^{-2}-C)s'}+C\eps e^{-(\frac{1}{2}\eps^{-2}-C)s'}\,ds'\\
&\leq C\int_{q}^2e^{Cs(q)}\cdot  (3-q')^{-1+C\eps^2}\,dq' + C\min\{\lra{q}^{C\eps^2},\eps^{2}\lra{q}\}e^{-(\eps^{-2}-C)s(q)}+C\eps^3 e^{-(\frac{1}{2}\eps^{-2}-C)s(q)}\\
&\leq Ce^{Cs(q)}\min\{\eps^{-2}\lra{q}^{C\eps^2},\lra{q}\}\leq C\min\{\eps^{-2}\lra{q}^{C\eps^2},\lra{q}^{1+C\eps^2}\}\leq C\eps^{-1}\lra{q}^{\frac{1}{2}+C\eps^2}.}
Here we use $s(q)=\lambda^{-1}\eps^2\ln(3-q)$. In summary, we have
\fm{|\wh{U}^T-U|&\lesssim  M_0\eps^2e^{-(\frac{1}{2}\eps^{-2}-M_1) s} \cdot \eps^{-1}\lra{q}^{\frac{1}{2}+C\eps^2}+e^{-(\eps^{-2}-C)s}\lra{q} \lesssim (M_0\eps+1) e^{-(\frac{1}{2}\eps^{-2}-M_1) s} \lra{q}.}
By choosing $\eps\ll 1$, we obtain \eqref{asu:bootstrap:grs:T} with $M_0$ replaced by $\frac{M_0}{2}$. This finishes our bootstrap argument. 

Finally, we have
\fm{|\wh{U}_q^T-U_q|
&\lesssim (-\mu(s,q,\omega))^{-1}\cdot   \eps^2 \lra{q}^{-1+C\eps^2}e^{-(\frac{1}{2}\eps^{-2}-C) s}+e^{-(\eps^{-2}-C)s}\\
&\lesssim \eps^2 \lra{q}^{-1+C\eps^2}e^{-(\frac{1}{2}\eps^{-2}-C) s}+e^{-(\eps^{-2}-C)s},}
\fm{\partial_s\ln\frac{\mu}{\wh{\mu}^T} &=\partial_s\ln(-\mu)-\partial_s\ln(-\wh{\mu}^T)=\frac{1}{2}G(\omega)(U-\wh{U}^T)\Psi+O\kh{ e^{-(\eps^{-2}-C)s}  +\eps e^{-(\frac{1}{2}\eps^{-2}-C)s}\lra{q}^{-1} }\\
&=O\kh{\lra{q}^{C\eps^2} e^{-(\frac{1}{2}\eps^{-2}-C)s}}.}
Here we can take $M_0=M_1=C$ in \eqref{asu:bootstrap:grs:T}.
It follows from $\wh{\mu}^T=\mu$ at $s=T$ that
\fm{\abs{\ln\frac{\mu}{\wh{\mu}^T}(s,q,\omega)}&\lesssim \int_s^T\lra{q}^{C\eps^2} e^{-(\frac{1}{2}\eps^{-2}-C)s'}\,ds'\lesssim\eps^2 \lra{q}^{C\eps^2} e^{-(\frac{1}{2}\eps^{-2}-C)s}.}
As a result, in $\D^*\cap\{s\leq T\}$ we have
\fm{\frac{\mu}{\wh{\mu}^T}(s,q,\omega)&=\exp\kh{O\kh{\eps^2 \lra{q}^{C\eps^2} e^{-(\frac{1}{2}\eps^{-2}-C)s}}}=1+O\kh{\eps^2 \lra{q}^{C\eps^2} e^{-(\frac{1}{2}\eps^{-2}-C)s}}.}
Here we use $\lra{q}^{C\eps^2}\lesssim e^{Cs}$.
\end{proof}

Next, we show that $(\wh{\mu}^T,\wh{U}^T)$ converges to a solution to \eqref{sec:reduced:system:main:1.1} as $T\to\infty$.
\begin{lem}\label{lem:asym:beha:T:existence:limit}
There exists a solution $(\wh{\mu},\wh{U})$ to the geometric reduced system in $\D^*$ such that $\wh{U},\wh{U}_q,\wh{\mu}$ are continuous in $\D^*$ and that 
\fm{(\wh{U},\wh{U}_q,\wh{\mu})&=\lim_{T\to\infty}(\wh{U}^T,\wh{U}^T_q,\wh{\mu}^T),\qquad \wh{\mu}\wh{U}_q=\Psi,\qquad \wh{U}\restriction_{q\geq 2}\equiv 0.}
Moreover, we have
\eq{\label{lem:asym:beha:T:existence:est:limit}|\wh{U}-U|&\lesssim e^{-(\frac{1}{2}\eps^{-2}-C)s}\lra{q},\\
|\wh{U}_q-U_q|
&\lesssim \eps^2 \lra{q}^{-1+C\eps^2}e^{-(\frac{1}{2}\eps^{-2}-C) s}+e^{-(\eps^{-2}-C)s},\\
\frac{\mu}{\wh{\mu} }&=\exp\kh{O\kh{\eps^2 \lra{q}^{C\eps^2} e^{-(\frac{1}{2}\eps^{-2}-C)s}}}=1+O\kh{\eps^2 \lra{q}^{C\eps^2} e^{-(\frac{1}{2}\eps^{-2}-C)s}}.}
\end{lem}
\begin{proof}
Fix $T'\geq T\geq 100$. We can view $(\wh{\mu}^{T'},\wh{U}^{T'})$ as a solution to \eqref{eq:grs:asym:beha:T} with data \fm{(\wh{\mu}^{T'}(T,q,\omega),\wh{U}^{T'}(T,q,\omega))}
at $s=T$. In the proof of Lemma~\ref{lem:asym:beha:T:existence}, we have proved that
\fm{|(\wh{U}^{T'}-\wh{U}^{T})(T,q,\omega)|\lesssim |(\wh{U}^{T'}-U)(T,q,\omega)|+|(\wh{U}^{T}-U)(T,q,\omega)|\lesssim e^{-(\frac{1}{2}\eps^{-2}-C)T}\lra{q}}
whenever $(T,q,\omega)\in\D^*$. Moreover, by \eqref{est:asym:beha:muTmu:ratio}, we have
\fm{\ln\frac{\mu(T',q,\omega)}{\mu(T,q,\omega)}&=\frac{1}{2}G(\omega)\Psi(q,\omega)\int_{T}^{T'}  U (s',q,\omega)\,ds' + O\kh{ \eps^2e^{-(\eps^{-2}-C)T }  +\eps^3 e^{-(\frac{1}{2}\eps^{-2}-C)T  }\lra{q}^{-1} }.}
We also have
\fm{\wh{\mu}^T\wh{U}_q^T=\wh{\mu}^{T'}\wh{U}_q^{T'}=\Psi,\qquad \text{in }\D^*\cap\{s\leq T\}.}

We now use a bootstrap argument. Fix $ \Tboot\in[0,T]$. Suppose that 
\eq{\label{asu:bootstrap:grs:T:diff}|\wh{U}^{T'}-\wh{U}^T|&\leq M_0 e^{-\frac{1}{8}\eps^{-2}(s+T)}\lra{q},\qquad\text{ in }\D^*\cap\{s\in[\Tboot,T]\}.}
Here $M_0>1$ is a sufficiently large constant to be chosen. Previously we have proved that the bootstrap assumption holds for $\Tboot=T$. We will prove that \eqref{asu:bootstrap:grs:T:diff} holds with $M_0$ replaced by $\frac{M_0}{2}$.

Following the proof of Lemma~\ref{lem:asym:beha:T:existence}, we have
\fm{(\wh{U}_q^{T'}-\wh{U}_q^{T}) (s,q,\omega)&=\frac{\Psi(q,\omega)\wh{\kappa}^{T'}(s,q,\omega)}{\mu(T',q,\omega)}-\frac{\Psi(q,\omega)\wh{\kappa}^T(s,q,\omega)}{\mu(T,q,\omega)}\\
&=\frac{\Psi(q,\omega)\wh{\kappa}^T(s,q,\omega)}{\mu(T,q,\omega)}\kh{\frac{\mu(T,q,\omega)\wh{\kappa}^{T'}(s,q,\omega)}{\mu(T',q,\omega)\wh{\kappa}^T(s,q,\omega)}-1}.}
We have
\fm{\frac{\mu(T,q,\omega)}{\mu(T',q,\omega)}&=\exp\kh{-\frac{1}{2}G(\omega)\Psi(q,\omega)\int_{T}^{T'}  U (s',q,\omega)\,ds' + O\kh{ \eps^2e^{-(\eps^{-2}-C)T }  +\eps^3 e^{-(\frac{1}{2}\eps^{-2}-C)T  }\lra{q}^{-1} }},\\
\frac{ \wh{\kappa}^{T'}(s,q,\omega)}{ \wh{\kappa}^T(s,q,\omega)}&=\exp\kh{\frac{1}{2}G(\omega)\Psi(q,\omega)\kh{\int_{T}^{T'}\wh{U}^{T'}(s',q,\omega)\, ds'+ \int_{s}^T(\wh{U}^{T'}-\wh{U}^T)(s',q,\omega)\, ds'}}\\
&=\exp\kh{\frac{1}{2}G(\omega)\Psi(q,\omega) \int_{T}^{T'}\wh{U}^{T'}(s',q,\omega)\, ds'+O\kh{M_0\lra{q}^{C\eps^2}\eps^2 e^{-\frac{1}{8}\eps^{-2}(s+T)}}}.}
When we multiply these two estimates together, we use  \eqref{lem:asym:beha:T:existence:est} with $M_0=C$ to estimate $|U-\wh{U}^{T'}|$. This gives
\fm{\abs{\frac{\mu(T,q,\omega)\wh{\kappa}^{T'}(s,q,\omega)}{\mu(T',q,\omega)\wh{\kappa}^T(s,q,\omega)}-1}
&=\abs{\exp\kh{O\kh{M_0\lra{q}^{C\eps^2}\eps^2 e^{-\frac{1}{8}\eps^{-2}(s+T)}}}-1}\lesssim M_0\lra{q}^{C\eps^2}\eps^2 e^{-\frac{1}{8}\eps^{-2}(s+T)} .}
Moreover, we have
\fm{(\wh{\mu}^{T}\wh{\kappa}^{T})(s,q,\omega)&=(\wh{\mu}^{T}\wh{\kappa}^{T})(T,q,\omega)=\mu(T,q,\omega)} and thus in $\D^*\cap\{s\leq T\}$
\fm{\abs{(\wh{U}_q^{T'}-\wh{U}_q^{T}) (s,q,\omega)}&\lesssim M_0(-\wh{\mu}^T(s,q,\omega))^{-1}\lra{q}^{-1+C\eps^2}\eps^2 e^{-\frac{1}{8}\eps^{-2}(s+T)},\\
\abs{(\wh{U}^{T'}-\wh{U}^{T}) (s,q,\omega)}&\lesssim M_0\eps^2 e^{-\frac{1}{8}\eps^{-2}(s+T)}\int_q^2(-\wh{\mu}^T(s,q',\omega))^{-1}(3-q')^{-1+C\eps^2}\,dq'.}
Following the same argument as in the proofs of Lemmas~\ref{lem:reduced:global:qwe} and~\ref{lem:asym:beha:T:existence}, we have
\fm{&\int_q^2(-\wh{\mu}^T(s,q',\omega))^{-1}(3-q')^{-1+C\eps^2}\,dq'\lesssim\int_q^2(-\wh{\mu}^T(s(q),q',\omega))^{-1}(3-q')^{-1+C\eps^2}\,dq'\\
&\lesssim \lra{q}^{C\eps^2}\int_q^2\lra{q'}^{-1+C\eps^2}\,dq'\lesssim \min\{\eps^{-2}\lra{q}^{C\eps^2},\lra{q}\}\lesssim \eps^{-1}\lra{q}^{\frac{1}{2}+C\eps^2}.}
Here we have no remainder terms because $(\wh{\mu}^T,\wh{U}^T)$ is an exact solution to \eqref{sec:reduced:system:main:1.1}. We also use \eqref{lem:asym:beha:T:existence:est} to obtain $-\wh{\mu}^T\gtrsim e^{-Cs}$.
In summary, we have
\fm{\abs{(\wh{U}^{T'}-\wh{U}^{T}) (s,q,\omega)}&\lesssim M_0\eps  e^{-\frac{1}{8}\eps^{-2}(s+T)}\lra{q}.}
This improves \eqref{asu:bootstrap:grs:T:diff}.

In summary, in $\D^*\cap\{s\leq T\}$ we have
\fm{\abs{(\wh{U}^{T'}-\wh{U}^{T}) (s,q,\omega)}&\lesssim   e^{-\frac{1}{8}\eps^{-2}(s+T)}\lra{q},\\
\abs{(\wh{U}_q^{T'}-\wh{U}_q^{T}) (s,q,\omega)}&\lesssim \eps^2 e^{-\frac{1}{8}\eps^{-2}(s+T)+Cs}\lra{q}^{-1+C\eps^2},\\ \abs{\frac{\wh{\mu}^{T'}}{\wh{\mu}^T}(s,q,\omega)-1}&=\abs{\frac{\mu(T',q,\omega)\wh{\kappa}^T(s,q,\omega)}{\mu(T,q,\omega)\wh{\kappa}^{T'}(s,q,\omega)}-1}\lesssim \eps^2 e^{-\frac{1}{8}\eps^{-2}(s+T)}\lra{q}^{C\eps^2},\\
\abs{\kh{\ln(-\wh{\mu}^{T'})-\ln(-\wh{\mu}^{T})}(s,q,\omega)}&\lesssim \abs{\frac{\wh{\mu}^{T'}}{\wh{\mu}^T}(s,q,\omega)-1}\lesssim \eps^2 e^{-\frac{1}{8}\eps^{-2}(s+T)}\lra{q}^{C\eps^2}.}
Thus,
\fm{(\wh{U},\wh{V},\wh{\mu}):=\lim_{T\to\infty}(\wh{U}^T,\wh{U}^T_q,\wh{\mu}^T)} exists. The convergence is uniform in any compact subset of $\D^*$, so we have $\wh{U}_q=\wh{V}$ and $\wh{\mu}\wh{U}_q=\Psi$. Finally, we recall from \eqref{eq:grs:asym:beha:T} that in $\D^*\cap\{s\leq T\}$ 
\fm{\wh{\mu}^T(s,q,\omega) =\wh{\mu}^T(s(q),q,\omega)+\frac{1}{2}G(\omega)\Psi \int_{s(q)}^s(\wh{\mu}^T\wh{U}^T)(s',q,\omega)\,ds'.}
By sending $T\to\infty$ and using the uniform convergence in any compact set, we conclude that
\fm{\wh{\mu}(s,q,\omega) =\wh{\mu}(s(q),q,\omega)+\frac{1}{2}G(\omega)\Psi \int_{s(q)}^s(\wh{\mu} \wh{U} )(s',q,\omega)\,ds'.}
In summary, $(\wh{\mu},\wh{U})$ solves the geometric reduced system. Finally, we obtain \eqref{lem:asym:beha:T:existence:est:limit} by sending $T\to\infty$ in \eqref{lem:asym:beha:T:existence:est}.
\end{proof}

\subsection{Different asymptotic behavior}\label{sec:asym:beha:different}

\begin{lem}\label{lem:vanish:Psi:u}
Let $u$ be a global solution for $t\geq 0$ given by Theorem~\ref{thm:main}. Let $\Psi$ be the corresponding limit constructed in Lemma~\ref{lem:limit:muUq:Psi}. If $\Psi\equiv 0$, then $u\equiv 0$ as long as $\eps\ll1$\footnote{In this lemma, we will choose $\eps_1\in(0,1)$ depending on $g^{\alpha\beta},u_0,u_1$ and assume $\eps\in(0,\eps_1)$. This $\eps_1$ may be different from the $\eps_0$ chosen in Theorem~\ref{thm:main}, so we replace $\eps_0$ with $\min\{\eps_0,\eps_1\}$ there.}.
\end{lem}
\begin{proof}
Suppose that $\Psi\equiv 0$. It follows that in $\D$ we have
\fm{(\partial_ru)(t,x)&=\partial_r(\eps r^{-\frac{1}{2}}U(s,q,\omega))=-\frac{1}{2}r^{-1}u+\eps r^{-\frac{1}{2}}(q_rU_q)(s,q,\omega).}
Here $(s,q,\omega)=(\eps^2\ln t,q(t,x),\frac{x}{|x|})\in\D^*$. 
By Remark~\ref{rmk:lem:ptang:q:eikonal}, we have $q_r=(\frac{1}{2}+O(|u|^2))(-\mu)$ and thus
\fm{(\partial_ru)(t,x)&=\eps r^{-\frac{1}{2}}(-\frac{1}{2}+O(|u|^2))(\mu  U_q)(s,q,\omega)+O(r^{-1}|u|)\\
&=O(\eps r^{-\frac{1}{2}}\cdot \lra{t}^{-1+C\eps^2}+\eps r^{-1} \lra{t}^{-\frac{1}{2}+C\eps^2})=O(\eps \lra{t}^{-\frac{3}{2}+C\eps^2}).}
In the last step, we use  \eqref{est:diff:muUq:Psi} and $e^{-(\eps^{-2}-C)s}\lesssim \lra{t}^{-1+C\eps^2}$.
It follows that in $\D$
\fm{|\partial u|&\lesssim |\partial_ru|+|Lu|+r^{-1}|\Omega u|\lesssim \eps\lra{t}^{-\frac{3}{2}+C\eps^2}+\lra{r+t}^{-1}|Zu|\lesssim \eps\lra{t}^{-\frac{3}{2}+C\eps^2}.}
Here we use \eqref{est:new:ptb:sec:asym:beha}. For $t\geq 1$ and $r\leq t-t^\lambda+3$, we have $\lra{r-t}\gtrsim \lra{t}^\lambda$ and thus
\fm{|\partial u|&\lesssim \lra{r-t}^{-1}|Zu|\lesssim \lra{t}^{-\lambda}\cdot \eps\lra{t}^{-\frac{1}{2}+C\eps^2}\lesssim \eps\lra{t}^{-\frac{3}{2}+2\delta+C\eps^2}.}
Recall that $\delta=\frac{1}{1000}$. By the finite speed of propagation, for each $t\geq 1$, we have
\fm{\norm{\partial u(t)}_{L^2(\R^2)}\lesssim \eps\lra{t}^{-\frac{3}{2}+2\delta+C\eps^2}\cdot |B_{\R^2}(0,t+1)|^{\frac{1}{2}}\lesssim \eps\lra{t}^{-\frac{1}{2}+2\delta+C\eps^2}.}
We now apply the standard energy estimate for $\wt{\Box}_g$; see, e.g.,\ \cite[Proposition I.2.1]{MR2455195}. Since $\wt{\Box}_gu=0$ and $|\partial h|\lesssim |u||\partial u|\lesssim \eps^2\lra{t}^{-1}$ for all $(t,x)\in[1,\infty)\times\R^2$ by Lemma~\ref{lem:ptb:everywhere:u:du:ddu}, we have
\fm{\norm{\partial u(t)}_{L^2(\R^2)}&\lesssim \norm{\partial u(T)}_{L^2(\R^2)}\exp\kh{\int_t^T\norm{(\partial h^{\alpha\beta})(\tau)}_{L^\infty(\R^2)}\,d\tau}\\
&\lesssim  \norm{\partial u(T)}_{L^2(\R^2)}\exp\kh{\int_t^T\eps^{2}\lra{\tau}^{-1}\,d\tau}\lesssim \eps\lra{T}^{-\frac{1}{2}+2\delta+C\eps^2}\cdot \lra{T}^{C\eps^2}\\
&\lesssim \eps\lra{T}^{-\frac{1}{2}+2\delta+C\eps^2}.}
By choosing $\eps\ll1$, we have $-\frac12+2\delta+C\eps^2<0$.
The implicit constants are independent of $T$, so we conclude that $\norm{\partial u(t)}_{L^2(\R^2)}=0$ for all $t\geq 1$ by sending $T\to\infty$. Since $u\equiv 0$ for $r\geq t+1$, we conclude that $u\equiv 0$ for $t\geq 1$. By solving a Cauchy problem for $t\in[0,1]$ with zero data at $t=1$, we conclude that $u\equiv 0$ for all $t\geq 0$. 
\end{proof}

\begin{lem}
Let $u$ be a global solution to \eqref{qwe}--\eqref{init} for $t\geq 0$ constructed in Theorem~\ref{thm:main}. Suppose that $u\not\equiv0$ and that \eqref{asu:sign:G} holds but $G\not\equiv 0$. Then,
\eq{\lim_{t\to\infty}\lra{t}^{\frac{1}{2}}\norm{\partial_r^2u(t,\cdot)}_{L^\infty(\D_t)}=\infty.}
\end{lem}
\begin{proof}
We first claim that there exist $ [a,b]\subset(-\infty,1)$ with $a<b$, $\omega^0\in\mathbb{S}^1$, and $\zeta\in(0,1)$, such that $G(\omega^0)\geq \zeta $ and $|\Psi(q,\omega^0)|\geq \zeta $ whenever $q\in [a,b]$. To show this claim, we first apply Lemma~\ref{lem:vanish:Psi:u} to obtain $(q^1,\omega^1)\in (-\infty,1)\times\mathbb{S}^1$ such that $\Psi(q^1,\omega^1)\neq 0$. Since $\Psi$ is continuous, there exists an open set $\mcl{U}$ containing $(q^1,\omega^1)$ such that $\Psi\neq 0$ in $\mcl{U}$. Moreover, if we assume that $G\geq 0$ and that $G\not\equiv 0$, then we must have $G(\omega)\Psi(q,\omega)\neq 0$ at some point in $\mcl{U}$. In fact, the set $\{\omega\in\mathbb{S}^1:\ G(\omega)=0\}$ has zero surface measure on $\mathbb{S}^1$. This is because $G$ is a polynomial in $\omega$ of order at most $2$. Without loss of generality, we assume that $G(\omega^1)\Psi(q^1,\omega^1)\neq 0$. We now use the continuity of $G$ and $\Psi$ to prove the claim.

By changing $\zeta$ and $a$ if necessary, we may also assume that $\Psi(a,\omega^0)\neq\Psi(b,\omega^0)$, while $|\Psi(q,\omega^0)|\geq\zeta$ for all $q\in[a,b]$. This follows from the continuity of $\Psi$ and the decay $|\Psi(q,\omega^0)|\lesssim\lra{q}^{-1+C\eps^2}$.

We now recall from Lemma~\ref{lem:asym:beha:T:existence:limit} the exact solution $(\wh{\mu},\wh{U})$ to \eqref{sec:reduced:system:main:1.1} in $\D^*$. We set
\fm{\wh{Y}(s,q,\omega)&=2+\int_q^2\frac{2}{\wh{\mu}(s,q',\omega)}\,dq'.}
This $\wh{Y}$ is in fact the Lagrangian flow map as in Section~\ref{sec:reduced:system:3.2.3}.
Since
\fm{\partial_s\kh{\frac{2}{\wh{\mu}}}&=-\frac{2\wh{\mu}_s}{\wh{\mu}^2}=-G(\omega)\wh{U}\wh{U}_q=-\frac{1}{2}G(\omega)\partial_q\wh{U}^2,}
we have
\fm{\partial_s\wh{Y}(s,q,\omega)=-\int_{q}^2\frac{1}{2}G(\omega)\partial_q(\wh{U}^2)(s,q',\omega)\,dq'=\frac{1}{2}G(\omega)\wh{U}(s,q,\omega)^2\geq 0 }
whenever $(s,q,\omega)\in\D^*$.
In other words, for all fixed $(q,\omega)\in(-\infty,2]\times\mathbb{S}^1$, $\wh{Y}(s,q,\omega)$ is nondecreasing for $s\geq s(q)$. Since $\mu<0$ and $\frac{\mu}{\wh{\mu}}\geq \frac{1}{2}$ by \eqref{lem:asym:beha:T:existence:est:limit}, we have $\wh{\mu}<0$ and thus $\wh{Y}\leq 2$. It follows that  
$\lim_{s\to\infty} \wh{Y}(s,a,\omega)$ and $\lim_{s\to\infty} \wh{Y}(s,b,\omega) $ exist, where $a,b$ are chosen at the beginning of this proof. And since $\wh{Y}_s\geq 0$, we have \fm{\frac{1}{2}G(\omega)\liminf_{s\to\infty}\kh{\wh{U}(s,a,\omega)^2+\wh{U}(s,b,\omega)^2}=\liminf_{s\to\infty}(\wh{Y}_s(s,a,\omega)+\wh{Y}_s(s,b,\omega))=0.}
Since $G(\omega^0)>0$, we have
\eq{\label{est:liminf:whU}\liminf_{s\to\infty} \kh{\wh{U}(s,a,\omega^0)^2+\wh{U}(s,b,\omega^0)^2}=0.}

To continue, we recall that $\wh{\mu}\wh{U}_q=\Psi$. Without loss of generality, we assume that $\Psi(q,\omega^0)\geq \zeta$ for $q\in[a,b]$; the case where $\Psi(q,\omega^0)\leq- \zeta$ for $q\in[a,b]$ is similar. It follows that 
\fm{\wh{U}(s,b,\omega^0)-\wh{U}(s,a,\omega^0)&=\int_a^b\wh{U}_q(s,q',\omega^0)\,dq'=\int_a^b\frac{\Psi(q',\omega^0)}{\wh{\mu}(s,q',\omega^0)}\,dq'\\
&\leq-\frac{\zeta}{2} \int_a^b\frac{-2}{\wh{\mu}(s,q',\omega^0)}\,dq'<0.}
Note that \fm{\int_a^b\frac{-2}{\wh{\mu}(s,q',\omega^0)}\,dq'=\wh{Y}(s,b,\omega^0)-\wh{Y}(s,a,\omega^0).}
The right side has a limit as $s\to\infty$. We thus conclude that 
\eq{\label{est:liminf:whmu:int}\lim_{s\to\infty}\int_a^b\frac{-2}{\wh{\mu}(s,q',\omega^0)}\,dq'=0.}
Note that we have a usual limit instead of a limit inferior.

Next, recall the map $(t,x)\mapsto(\eps^2\ln t,q(t,x),\frac{x}{|x|})$ in Section~\ref{sec:asym:beha:setup}. It is invertible from $\D$ to $\D^*$ and the inverse is at least $C^2$.
Now, we view $r=|x|$ as a function of $(s,q,\omega)\in\D^*$. It gives $\partial_qr=q_r^{-1}$. By Remark~\ref{rmk:lem:ptang:q:eikonal}, we have $2q_r=-(1+O(|u|^2))\mu$, so in $\D$ we have
\fm{0<\frac{1}{q_r}&=\frac{-2}{(1+O(|u|^2))\mu}=\frac{(-2+O(|u|^2))}{\mu}=\frac{-2}{\mu}+O(\eps^2\lra{t}^{-1+C\eps^2})=\frac{-2}{\mu}+O(\eps^2e^{-(\eps^{-2}-C)s}).} 
By \eqref{lem:asym:beha:T:existence:est:limit}, we have in $\D^*$
\fm{\abs{\frac{1}{\wh{\mu}}-\frac{1}{\mu}}&\lesssim\eps^2\lra{q}^{C\eps^2}e^{-(\frac{1}{2}\eps^{-2}-C)s}|\mu|^{-1}\lesssim\eps^2 e^{-(\frac{1}{2}\eps^{-2}-C)s}, }
\fm{0<r(s,b,\omega^0)-r(s,a,\omega^0)&=\int_a^b\frac{-2}{\wh{\mu}(s,q',\omega^0)}+O(\eps^2e^{-(\frac{1}{2}\eps^{-2}-C)s})\,dq'\\
&=\int_a^b\frac{-2}{\wh{\mu}(s,q',\omega^0)}\,dq'+O((b-a)\eps^2e^{-(\frac{1}{2}\eps^{-2}-C)s}).}
By sending $s\to\infty$ and applying \eqref{est:liminf:whmu:int}, we conclude that
\eq{\label{est:liminf:ra:rb}\lim_{s\to\infty}(r(s,b,\omega^0)-r(s,a,\omega^0))=0.}

Now, we recall that in $\D$ 
\fm{(\partial_ru)(t,x)&=-\frac{1}{2}r^{-1}u+\eps r^{-\frac{1}{2}}q_rU_q=-\frac{1}{2}\eps r^{-\frac{1}{2}}(1+O(|u|^2))\mu U_q+O(\eps\lra{t}^{-\frac{3}{2}+C\eps^2})\\
&=-\frac{1}{2}\eps r^{-\frac{1}{2}}\Psi(q(t,x),\frac{x}{|x|})+O(\eps\lra{t}^{-\frac{3}{2}+C\eps^2}).}
It follows that
\fm{&(\eps^{-1}r^{\frac{1}{2}}\partial_ru)(e^{\eps^{-2}s},r(s,b,\omega^0)\omega^0)-(\eps^{-1}r^{\frac{1}{2}}\partial_ru)(e^{\eps^{-2}s},r(s,a,\omega^0)\omega^0)\\
&=-\frac{1}{2}(\Psi(b,\omega^0)-\Psi(a,\omega^0))+O( e^{-(\frac{1}{2}\eps^{-2}-C)s})\\
&=(-\frac{1}{2}+o_s(1))(\Psi(b,\omega^0)-\Psi(a,\omega^0))}
where $o_s(1)\to 0$ as $s\to\infty$. Also recall that our choice of $a,b$ guarantees that $\Psi(b,\omega^0)-\Psi(a,\omega^0)\neq 0$.
By the mean value theorem,
\fm{\eps^{-1}\norm{\partial_r(r^{\frac{1}{2}}\partial_ru)(t,\cdot)}_{L^\infty(\D_t)}\geq \frac{(\frac{1}{2}-o_s(1))|\Psi(b,\omega^0)-\Psi(a,\omega^0)| }{r(s,b,\omega^0)-r(s,a,\omega^0)}.}
Finally, we notice that
\fm{\abs{\partial_r(r^{\frac{1}{2}}\partial_r u)}&\lesssim r^{\frac{1}{2}}|\partial_r^2u|+Cr^{-\frac{1}{2}}|\partial u|\lesssim   \lra{t}^{\frac{1}{2}}|\partial_r^2u|+\eps\lra{t}^{-1+C\eps^2}.}
By \eqref{est:liminf:ra:rb}, we have
\fm{\lim_{t\to\infty}\lra{t}^{\frac{1}{2}}\norm{\partial_r^2u(t,\cdot)}_{L^\infty(\D_t)}=\infty.}
\end{proof}

\bibliographystyle{plain}
\bibliography{main}

\end{document}